\documentclass[11 pt]{article}
\usepackage{graphicx} % Required for inserting images
\usepackage{subcaption}
\usepackage{url}
\usepackage{tikz-cd}
\usepackage{amsmath}
\usepackage{amsthm}
\usepackage{amssymb}
\usepackage{xcolor}
\usepackage{titling}
\usepackage{diffcoeff,amssymb}
\usepackage{comment}
\usepackage{algorithm}
\usepackage{algpseudocode}
\usepackage{amssymb}
\usepackage{mathtools}
\usepackage{mathrsfs}
\usepackage{dsfont}
\usepackage{bm}
\usepackage{enumerate}
\usepackage{algorithm}
\usepackage{algpseudocode}
\usepackage[numbers]{natbib}
\usepackage{hyperref}
\usepackage[left=2.4cm,right=2.4cm,top=2.3cm,bottom=2cm]{geometry}
\def \proj {{\mathrm{proj}}}

\newtheorem{theorem}{Theorem}
\newtheorem{proposition}[theorem]{Proposition}
\newcommand{\probP}{\text{I\kern-0.15em P}}

\DeclareSymbolFont{extraup}{U}{zavm}{m}{n}
\DeclareMathSymbol{\varheart}{\mathalpha}{extraup}{86}
\DeclareMathSymbol{\vardiamond}{\mathalpha}{extraup}{87}

\newtheorem{lem}{\bf Lemma} 
\newtheorem{defn}{Definition}   
  
\newtheorem{rem}{\bf Remark}

\def \V {{\mathbf V}}

\def \R {{\mathbb R}}

\def \e {{\mathbf e}}

\def \v {{\mathbf v}}
\def \w {{\mathbf w}}
\def \u {{\mathbf u}}
\def \x {{\mathbf x}}

\def \z {{\mathbf z}}

\def \muv {\boldsymbol{\mu}}
\def \nuv {\boldsymbol{\nu}}

\def \cq {{\mathcal{Q}}}
\def \ci {{\mathcal{I}}}
\def \cm {{\mathcal{M}}}

\def \E {\mathbb{E}_{\z \sim \probP(\muv,\Lambda)}}

\def \Rn {{\mathbb{R}}}

\newcommand{\Dx}{\nabla_{\!x}}
\newcommand{\opnorm}[1]{%
  \left|\mkern-1.5mu\left|\mkern-1.5mu\left| #1 \right|\mkern-1.5mu\right|\mkern-1.5mu\right|
}

\newcommand{\wt}[1]{\widetilde{#1}}
\newcommand{\bs}{\mathbb{S}}
\newcommand{\cc}{\mathcal{C}}

\newcommand{\tr}{{\text{tr}}}
\newcommand{\cp}{{\Psi}}

\newcommand{\norm}[1]{\ensuremath{\left\|#1\right\|}}	
\newcommand{\abs}[1]{\left|#1 \right|}
\definecolor{blue-violet}{rgb}{0.54, 0.17, 0.89}

\definecolor{darkgreen}{rgb}{0.2,0.8,0.1}

\usepackage[parfill]{parskip}

\title{Riemannian ascent--descent for nonconvex nonconcave minimax landscapes:
convergence to basin saddle points and applications to distributionally robust optimization}

\author{
  Rishabh Dixit\thanks{Department of Mathematics, UC San Diego %, La Jolla, CA 92093, USA. 
  (ridixit@ucsd.edu).}
  \and
  Pranav Upadrashta\thanks{Department of Mathematics, UC San Diego %, La Jolla, CA 92093, USA. 
  (pupadrashta@ucsd.edu).}
  \and
  Alex Cloninger\thanks{Department of Mathematics and Hal{\i}c{\i}o{g}lu Data Science Institute, UC San Diego %, La Jolla, CA 92093, USA.  
  (acloninger@ucsd.edu).}
}
\date{\vspace{-4ex}}
\begin{document}

\maketitle
\begin{abstract}
We study a class of distributionally robust optimization (DRO) problems for the
statistical risk problem, formulated as minimax problems over the product of a
Euclidean space and a Riemannian manifold. Because the resulting minimax landscape
is nonconvex nonconcave in general, no globally convergent first order method is
known to be available. We instead introduce the notion of a \emph{basin saddle point}, a Nash equilibrium defined locally on the Cartesian product of a $\delta$ basin around a connected component of the local minima critical set and a geodesic ball on the measure
manifold. We develop an abstract convergence framework for a Riemannian gradient ascent multistep descent iteration to a basin saddle point under a local \L{}ojasiewicz type growth condition, with exponent $\beta \in (1,2]$, in the $\delta$ basin around connected components of the local minima critical sets. Under Lipschitz regularity of critical sets we establish linear convergence for $\beta = 2$ and
polynomial convergence for $\beta \in (1,2)$ to a basin saddle point, with
explicit dependence on the sectional curvature of the manifold. We then
instantiate this framework for the statistical risk DRO problem over Gaussian
measures, where the ambiguity set is naturally modeled as the product of
Euclidean space and the Bures Wasserstein manifold of covariance matrices, which we relax to a penalized DRO formulation. We
derive nonasymptotic Hessian estimates for the resulting Lagrangian,
establish existence and local uniqueness of its maximizer, and prove that an
alternating Riemannian gradient scheme (Algorithm \ref{algRGA:alternating_updates}) converges to a basin saddle
point of the penalized DRO problem, recovering the linear and polynomial rates of the abstract theory with all
constants explicit in terms of data dimension, loss moments, and the
reference covariance.
\end{abstract}

\begin{comment}
    \begin{abstract}
    
\end{abstract}

\subsection{Motivation}
\subsection{Related Work}
\subsection{Contributions}
\subsection{Preliminaries}
\subsubsection{Notation}
\end{comment}

\section{Introduction \& motivation}

This work studies a class of Distributionally Robust Optimization (DRO) problems arising in adversarial learning (GANs) \cite{bai2023wasserstein, li2025distributionally, staib2017distributionally}, robust training \cite{sinha2018certifying, madry2017towards, staib2017distributionally}, classification under distribution uncertainty \cite{shafieezadeh2015distributionally, mohajerin2018data, hu2018does}, etc. Fundamentally, the DRO problem is a special type of min-max problem formulated over a product metric space. In particular, a general min-max problem over some complete metric spaces $ \mathcal{X}, \mathcal{Y}$ is given by: 
\[  \min_{x \in \mathcal{X}} \max_{y \in\mathcal{Y}} f(x,y)\]
Then in the context of DRO problem, the function $f$ can be defined as the statistical risk $f(\w, \probP):= \mathbb{E}_{\z \sim \probP} [\ell(\w;\z)]$ for some suitable loss function $\ell$ and domain of $f$ is the product space of model parameters in $ \mathbb{R}^d \ni \w$ and data distributions in $\mathcal{P} \ni \probP$. The space of admissible distributions $\mathcal{P}$ is often referred to as the ambiguity set in the standard DRO literature \cite{kuhn2025distributionally}. While a lot of prior literature classifies the DRO problem based on the ambiguity sets arising as a consequence of using different metrics \cite{kuhn2025distributionally}, more recently emphasis has shifted toward developing numerical methods that can converge to the solutions of DRO problem (referred to as the saddle points of the DRO problem). 

The more general min-max optimization framework dates back several decades (see saddle point problem \cite{v1928theorie, brezzi1974existence, kuhn2013nonlinear}) and several efficient first order methods have been studied in this regard. Under some nice domains $\mathcal{X} \times \mathcal{Y}$ such as smooth Riemannian structure and geodesic convexity-concavity assumptions on the function $f$, the min-max problem enjoys existence of unique saddle points \cite{brezzi1974existence, sion1958general, zhang2023sion}. Such min-max problems are naturally solvable via gradient ascent descent type methods in linear time \cite{li2022convergence, jordan2022first}. As a consequence, the corresponding DRO problem modeled as a min-max problem is also solvable. More recently, the optimization community has focused on the more general nonconvex nonconcave function classes. There, under sufficient regularity conditions like weak convexity-weak concavity and with the Euclidean domain, the min-max problem is solvable in polynomial time \cite{liu2021first, grimmer2023landscape} provided the existence of saddle points is satisfied. However, extending the analysis of such solvers to the DRO problem is not straightforward since the domain of function in the DRO problem is not Euclidean. In fact, even under the ``nice" distributions in $\mathcal{P}$ (square integrable and absolutely continuous with respect to Lebesgue measure), the domain $\mathcal{Y}$ is at best a Riemannian length space (not even a manifold) \cite{otto2001geometry, ambrosio2005gradient}. Therefore, numerical methods for solving DRO problems cannot be directly lifted from the existing Euclidean min-max optimization machinery and need a separate analytical framework.

%\textcolor{red}{ML Paragraph attempt here}
The structural property we impose on the model fiber is a \emph{local} \L ojasiewicz
condition, and it is precisely the property that the modern loss-landscape literature
identifies for neural networks. In the overparameterized regime, the empirical risk of a
wide network satisfies a local Polyak--\L ojasiewicz inequality in a neighborhood of its
minimizers \cite{liu2022loss}, and gradient-based training operates entirely inside such a
region; the more general exponent $\beta \in (1,2)$ covers losses that are locally real
analytic but not PL \cite{kurdyka1998gradients}. Two features of this regime shape our
framework. First, the condition is \emph{local}: away from the basin the landscape is
genuinely nonconvex, which is why we target basin saddle points (Definition~\ref{basinsaddledef})
rather than global saddle points. Second, overparameterization renders the minimizers
\emph{non-isolated}, instead forming  positive-dimensional connected sets, which is why $\mathcal{S}^{*}(y)$ in \textbf{C3} is a connected component of the
local minima critical set, and why convergence is measured by
$\operatorname{dist}(x_k, \mathcal{S}^{*}(y_k))$ rather than by $\lVert x_k - x^{*}\rVert$.
Together these address the DRO problem for a trained network:
the reference measure $\mathbb{P}(\boldsymbol{\mu}^{*},\Lambda^{*})$ models the training
distribution, $\mathbf{w}_0$ is the trained model sitting in a basin of
$\mathcal{S}_c^{*}(\mathbb{P}_0)$, and the outer iteration asks how far that basin drifts as
the data distribution is perturbed within an ambiguity set. We make these requirements
precise as \textbf{A1'}--\textbf{A3'} in Section~\ref{sectionDROBW} and discuss
their empirical support there. We emphasize at the outset that our guarantees concern the
population risk, and that transferring a local PL property from the empirical to the
statistical landscape requires a separate concentration argument that we do not pursue here.

Within the Riemannian min-max optimization framework, several recent works develop convergence guarantees for first order Riemannian ascent descent type methods under general nonconvex-concave (or strongly concave) framework. However, such rates are usually weaker than the iterate convergence rates and the limit point may not necessarily be a saddle point for the problem \cite{huang2023gradient, xu2026riemannian}.
Recent advances such as \cite{wang2023local, garcia2024adversarial, cheng2025worst} take a step towards developing analytical convergence machinery for abstract manifold structure by employing implicit methods, such as the JKO scheme \cite{jordan1998variational}, over tangent spaces of infinite dimensional Riemannian length spaces\footnote{{Here, the tangent space terminology is abuse of definition since Riemannian length spaces are not manifolds. However, tangent spaces can be defined locally under some mild assumptions.}} \cite{cheng2025worst}. The numerical scheme in \cite{cheng2025worst} achieves convergence to a saddle point of the problem however the proof is build on some strong assumptions on the trajectory paths. Then \cite{garcia2024adversarial} addresses the minimax problem over an infinite dimensional product metric space by utilizing Wasserstein ascent descent type methods and shows convergence (in the weak sense) to an approximate or $\epsilon$-Nash equilibrium under the mean field particle theory model. They also impose the Polyak \L{}ojasiewicz (P\L{}) assumption on the minimax functional with respect to the measure couplings between points on the Borel probability space. However no rates of convergence are provided in \cite{garcia2024adversarial} due to the absence of a Lyapunov type analysis for functionals with non-unique minimizers on infinite dimensional metric spaces. Motivated by such recent developments pertaining to numerical schemes for the DRO and minimax problems on Wasserstein spaces, we develop an analytical machinery that provides convergence guarantees, under some ``nice" structural properties of the loss landscape, for the penalized statistical DRO problem \eqref{DRO1x}, which relaxes the constrained problem \eqref{DRO1}; Proposition \ref{prop:reduction} bounds the resulting gap. 

%{Anything by Jianqing Fan?}

The ambiguity set $\mathcal{P}$ is what gives the DRO problem its geometry, and
the natural way to build one is through optimal transport. The Wasserstein-2
distance is the natural distance on
the space $\mathcal{P}_2(\mathbb{R}^n)$ of square integrable measures
\cite{villani2009optimal,santambrogio2015optimal,peyre2019computational}, and it
carries considerably more structure than a distance alone. In
\cite{otto2001geometry}, the author equips $(\mathcal{P}_2, W_2)$ with a formal
Riemannian calculus, under which the geodesics are displacement interpolations
\cite{mccann} and the gradient flows are limits of JKO steps
\cite{jordan1998variational,ambrosio2005gradient}. Taking $\mathcal{P}$ to be a
$W_2$ ball around an empirical measure gives Wasserstein DRO, for which strong
duality and tractable convex reformulations are by now well understood
\cite{mohajerin2018data,blanchet2019quantifying,gao2023distributionally,kuhn2025distributionally}.
There are two main drawbacks to this approach for our purposes. The first is
that the reformulations are static: they replace the inner maximization by a
dual convex program whose tractability depends on the structure of the loss
$\ell$, and they say little about how to move through $\mathcal{P}$ iteratively.
The second is that $(\mathcal{P}_2, W_2)$ is at best a Riemannian length space
and not a manifold, so even when an iterative scheme is available it cannot
inherit the existing Euclidean or Riemannian optimization machinery. These
issues motivate the need for an ambiguity set that retains the transport
geometry but is finite dimensional.  There are a number of approaches to embedding transport geometry into finite dimensional space, including Linearized
Optimal Transportation (LOT) \cite{wang2013linear, moosmuller2023linear,khurana2023supervised}, embedding into negative Sobolev normed space \cite{peyre2018comparison, greengard2022linearization, robertson2024resistance}, and restrictions to finite parameter families of distributions, such as Gaussian measures.  When the measures are
Gaussian, $W_2$ restricted to the covariances is precisely the
Bures--Wasserstein (BW) metric
\cite{takatsu2011wasserstein,bhatia2019bures,malago2018wasserstein}. This is the
reduction we adopt. Rather than optimizing over the infinite dimensional length
space $(\mathcal{P}_2, W_2)$, we restrict $\mathcal{P}$ to the Gaussian family
and work on $\mathbb{R}^n \times \mathrm{BW}(\mathtt{SPD}_n)$, a finite
dimensional Riemannian manifold with bounded sectional curvature. We emphasize that our method does not optimize over the $W_2$ ball directly. Section \ref{sectionDROBW} replaces this ball with a squared $W_2$ penalty anchored at the reference measure, at a cost quantified in Proposition \ref{prop:reduction}.

Optimization directly over BW space is well studied, and it is the setting
closest to ours. The geometry is by now standard: $\mathtt{SPD}_n$ equipped with
the BW metric is a Riemannian manifold of nonnegative sectional curvature, with
an explicit exponential map and explicit curvature bounds
\cite{takatsu2011wasserstein,bhatia2019bures,massart2019curvature}. On the
algorithmic side, in \cite{chewi2020gradient} the authors analyze Riemannian
gradient descent for BW barycenters, and in \cite{altschuler2021averaging} this
is sharpened to dimension-free rates by showing that the barycenter functional,
while not geodesically convex, satisfies a variance type quadratic growth
inequality along geodesics. The same machinery drives Gaussian variational
inference, either through gradient descent in the BW metric
\cite{lambert2022variational} or through a forward--backward splitting of the
JKO step \cite{diao2023forward}. Gaussian ambiguity sets have also appeared in
DRO itself, in the context of covariance and mean square error estimation
\cite{nguyen2022distributionally,nguyen2023bridging}, though there the problem
is resolved through a convex reformulation rather than through a Riemannian
iteration. What these works have in common is that a single functional is
minimized, and that this functional either is geodesically convex or satisfies a
global growth condition. This is not the case for the DRO problem
\eqref{DRO1}: it is a min--max problem on the product $\mathbb{R}^d \times
(\mathbb{R}^n \times \mathrm{BW}(\mathtt{SPD}_n))$, nonconvex in $\mathbf{w}$ and
nonconcave in the measure, and the critical sets of the inner problem drift as the distribution $\probP$ varies. Geodesic convexity is therefore unavailable to us, and a convergence analysis for gradient ascent descent on the product of a Euclidean space and BW under purely local assumptions is still missing. Conditions
\textbf{C1}--\textbf{C4} are our replacement. They ask only for a
P\L{}-type growth around the connected components of the local minima
critical set together with a bounded drift of those components, which allows us to solve the penalized DRO problem \eqref{DRO1x} in these settings.

\subsection{Our contributions}
The contributions of this work can be summarized as follows:
\begin{itemize}
    \item In section \ref{jointrategeneralsection} we first develop an abstract framework for the min-max problem $$ \min_{x \in \mathcal{X}} \max_{y \in\mathcal{Y}} f(x,y)$$ on a product space of a Euclidean space $\mathcal{X}$ and a finite dimensional Riemannian manifold $\mathcal{Y}$ with bounded sectional curvature. We rigorously formulate the problem structure by imposing conditions \textbf{C1-C4} that describe the local properties of the min-max loss landscape. These conditions impose joint smoothness on $f$ over a compact product manifold $S_1 \times S_2 \subset \mathcal{X} \times \mathcal{Y}$, give local \L{}ojasiewicz type gradient growth around the connected components of local minima $\mathcal{S}^*(y)$ of $f(\cdot , y)$ for any $y \in S_2$ and ensure bounded perturbations of $ \mathcal{S}^*(y)$ as $y$ varies over $S_2$.
     \item Thereafter we define the basin saddle point (Definition \ref{basinsaddledef}) that serves as a local solution/ Nash equilibrium to the min-max problem over the compact product manifold $S_1 \times S_2$. We then define a Riemannian gradient ascent descent type iteration \eqref{RGA-MGD1}-\eqref{RGA-MGD2} that generates the sequence $\{(x_k, y_k)\}_k$ and it  converges to a basin saddle point of the problem under mild initialization conditions. The convergence is analyzed using a \eqref{lyapunov-RGA-MGD} $V_{\delta}(x,y)$. The ascent descent iteration is slightly non-standard in the sense that during any iteration $k$ the descent step \eqref{RGA-MGD2} is performed $J$ times on the fiber $S_1 \times y_k $ while the ascent step \eqref{RGA-MGD1} is performed only once on the fiber $ x_k \times S_2$. The $J$ step gradient descent ensures that $x_k$ gets sufficiently close to the critical set of $ f(\cdot, y_k)$ on the fiber $S_1 \times y_k $. Then after updating $y_k$ to $y_{k+1}$, even with the drift of critical set $ \mathcal{S}^*(y_k)$ to $ \mathcal{S}^*(y_{k+1})$, the iterate $x_k$ will be close to the new critical set $ \mathcal{S}^*(y_{k+1})$. This multi-step gradient descent iteration provides us boundedness of iterates in $S_1 \times S_2$. 
     \item We derive Theorems \ref{convergenceratethm_exactPL},  \ref{convergenceratethm_betaPL} that quantify the rates of convergence for the iteration \eqref{RGA-MGD1}-\eqref{RGA-MGD2} to a basin saddle point of $f$ under mild assumptions on initialization, inner gradient descent steps $J$ from \eqref{RGA-MGD2} and Lipschitz regularity of critical sets. The rates are a function of the \L{}ojasiewicz exponent $\beta \in (1,2]$ with linear convergence rates for $\beta = 2$ from Theorem \ref{convergenceratethm_exactPL} and sub-linear/ polynomial convergence rates for $\beta <2$ from Theorem \ref{convergenceratethm_betaPL}.
     \item We then apply this abstract convergence machinery to the penalized DRO problem \eqref{DRO1x} in Section \ref{sectionDROBW} for Gaussian measures. Rather than enforcing the
transport constraint directly, we work with the penalized problem
\eqref{DRO1x}, in which the ambiguity ball of radius $\xi$ is replaced by a
$W_2^2$ penalty of strength $1/\epsilon$; Proposition~\ref{prop:reduction}
establishes the correspondence $\epsilon = 2\xi/L_G$ and bounds the resulting
optimality gap by $L_G \xi$. All subsequent results --- the Hessian estimates
 (Theorems \ref{losshessianestimatethm}-\ref{lagrangehessianestimatethm}), the existence and uniqueness of the inner maximizer (Theorem~\ref{interiormaximaexistencethm}),
and the convergence rates for Algorithm~1 (Theorems~\ref{convergenceratethm_exactPL_specialcase},  \ref{convergenceratethm_betaPL_specialcase}) concern
\eqref{DRO1x}. Using tools from Riemannian geometry we derive explicit local Hessian estimates (Theorems \ref{losshessianestimatethm}-\ref{lagrangehessianestimatethm}) for the Lagrangian, i.e. the inner max functional in penalized DRO problem \ref{DRO1x}, and provide asymptotic scaling of the local strong geodesic concavity and Lipschitz parameters for the Lagrangian as a function of data dimension $n$, $p$-th moment of the loss function $\ell(\w;\z)$ and the covariance matrix of a reference Gaussian measure (Theorem \ref{lagrangehessianestimatethm}). Then in Theorem \ref{interiormaximaexistencethm} we show that the unique local maximizer of the Lagrangian with respect to the probability measure is uniformly bounded in the interior of a certain compact geodesic ball on the product of Euclidean and Bures Wasserstein (BW) manifold. We propose Algorithm \ref{algRGA:alternating_updates} that is a special case of the iteration \eqref{RGA-MGD1}-\eqref{RGA-MGD2} designed for the penalized DRO problem \eqref{DRO1x}. Finally, we define the loss landscape assumptions \textbf{A1'-A3'} along the lines of conditions \textbf{C1-C4} and derive Theorems \ref{convergenceratethm_exactPL_specialcase},  \ref{convergenceratethm_betaPL_specialcase} that quantify the rates of convergence for Algorithm \ref{algRGA:alternating_updates}. 
\end{itemize}

\paragraph{Relation to existing frameworks.}
{Our ascent--descent scheme instantiates the multi-step gradient descent--ascent
template introduced by \cite{sanjabi2018pl,nouiehed2019} for Euclidean
minimax problems in which one player's objective satisfies the
Polyak--{\L}ojasiewicz (PL) condition; specialized to a singleton critical set
and $\beta = 2$, our Theorem~\ref{convergenceratethm_exactPL} recovers the structure of their
guarantee. Three ingredients distinguish the present analysis. First, we work
with a {\L}ojasiewicz exponent $\beta \in (1,2]$, whereas the PL condition
corresponds to $\beta = 2$; the polynomial regime $\beta < 2$
(Theorem~\ref{convergenceratethm_betaPL}) has no counterpart in that literature. Second, the
analyses of \cite{nouiehed2019,yang2020} are driven by function-value
Lyapunov arguments built from $\Phi(x) = \max_y f(x,y)$ or the primal--dual
gap, so the set of minimizers never enters explicitly; we instead track
$\mathrm{dist}(x_k, S^*(y_k))$ directly, which is what permits a conclusion
localized to a single connected component rather than $\varepsilon$-stationarity
of $\Phi$, at the cost of controlling the drift
$\mathrm{dist}_H(S^*(y_k), S^*(y_{k+1}))$. Third, the ambient space is
Riemannian. We note that \cite{yang2020} address a different regime, with a single-loop
alternating GDA under a two-sided PL condition that yields global linear rates, and
that \cite{lin2020} treat the nonconvex--concave setting rather than the PL setting.}

{Definition~\ref{basinsaddledef} is a Nash-type (simultaneous-play) notion of local
optimality in which the min-player's optimality is required not merely in a
neighborhood of $x^*$ but throughout a $\delta$-basin around a connected
component of the local-minima critical set. Since that basin contains a
neighborhood of $x^*$, every basin saddle point is in particular a local Nash equilibrium point, and hence a local minimax point in the
sense of \cite{jin2020local}; see \cite{dai2020constrained} for the constrained
formulation appropriate to our setting $S_1 \times S_2$. Theorems~\ref{convergenceratethm_exactPL}, \ref{convergenceratethm_betaPL} certify that the iterates settle into a single
connected component of the critical set. Neither local Nash nor local minimax
convergence gives this. Our convergence guarantees are therefore stronger than convergence to a
local Nash or local minimax point.}

{Finally, the penalized formulation~\eqref{DRO1x} is the Gaussian
restriction of the penalized Wasserstein DRO problem of
\cite{sinha2018certifying}, where they obtain a strongly concave inner problem
by taking the penalty large, and observe the trade-off between the
penalty parameter and the level of robustness certified. Our contribution is not the formulation but the verification of geodesic strong concavity in Bures--Wasserstein geometry: the argument of \cite{sinha2018certifying} establishes Euclidean strong concavity of
$\ell(\w;\z) - \gamma c(\z,\z_0)$ in the data variable $\z$, whereas
Section~\ref{sectionDROBW} requires two-sided Hessian bounds for
$(\muv,\Lambda) \mapsto \mathbb{E}_{\z \sim \probP(\muv,\Lambda)}[\ell(\w;\z)]$ under
moment assumptions alone, combined with a comparison-theorem bound on
$\mathrm{Hess}\, d_{\mathrm{BW}}^2$.}

\subsection{Notations}

We use bold lower case letters to denote vectors in Euclidean space, i.e. we write $\x\in\mathbb{R}^d$ and all matrices are denoted by upper case letters for example $M\in\mathbb{R}^{n\times d}$ is a $n \times d$ matrix, $\mathbf{0}$ is used both for denoting a null vector and a null matrix, $I$ is the identity matrix. For any matrix $M$, its Frobenius norm is $ \| M\|_{F}$ and its operator norm is $\norm{M}_2 $, the operator norm in general is denoted by $\norm{\cdot}_{op}$ and for any vector $\w$ in Euclidean space its L-2 norm is $\norm{\w}$, the L-2 Euclidean metric is denoted by $l_2$. Next, $\text{tr}(\cdot)$ is the trace operator, $\text{det}(\cdot)$ is the determinant operator, $\text{int}(\cdot)$ is the interior operator, $\text{dist}_H(\cdot, \cdot)$ represents the Hausdorff distance, $\text{dist}(\cdot, \cdot)$ and $d(\cdot , \cdot)$ represent the metric distances, $\text{exp}(\cdot)$ is the exponential map, $\text{id}$ is the identity map, $\nabla^2 f$ represents the Hessian of $f$ in Euclidean space, for a Riemannian manifold $\cm$ and any $p \in \cm$, $\text{Hess}_{p} f$ represents the Riemannian Hessian of $f$ at $p$ and $\nabla_p f$ represents the Riemannian gradient of $f$ at $p$. 

We denote by $\mathcal{B}_r( \x)$ an open ball centered at $\x$ of radius $r$ and $\bar{\mathcal{B}}_r( \x)$ the closure of such ball. $\texttt{SPD}_n $ is the space of symmetric positive definite $n \times n$ matrices, $\mathbb{S}^n $ is the space of symmetric $n \times n$ matrices. For a set $A$ in some metric space $(\mathcal{X}, d)$, we denote its diameter by $\text{diam}(A) $. We write $\probP$ for a probability measure, $\mathbb{E}[\cdot]$ is the expectation operator and $\Gamma(\cdot)$ represents the Gamma function. For a product set $S_1 \times S_2$ and $x \in S_1, \, y \in S_2$, the sets $x \times S_2$, $S_1 \times y$ represent fibers of $S_2$, $S_1$ respectively.  

The symbol $\oplus$ represents the direct or orthogonal sum of vector spaces, $\otimes$ is used to for the Kronecker product. For two sets $A,B$ the set $A+B$ is their Minkowski sum. $\mathrm{dim}(\cdot)$ represents the dimension, and for any vector spaces $A,B$ with $A \subset B$, $A ^{\perp}$ represents the orthogonal complement of $A$ in $B$, where $B$ is understood from the context. The symbol $\mathcal{O}$ represents the Big-O notation, the symbol $o$ represents the little-o notation.   For any two sets $A,B$ in some topological space $X$, $A\Subset B$ means $A$ is compactly embedded in $B$ with respect to the topology on $X$. For a set $A$, $\partial A$ denotes its boundary when defined and for a continuous function $f$, $\partial f(x)$ represents the generalized sub-differential set of $f$ at $x$. Throughout the paper $\nabla f(\w;\z) $ denotes the gradient of $f$ with respect to the first argument $\w$. $\mathcal{N}(\mathbf{0}_n, I)$ is the standard normal distribution in $\mathbb{R}^n$, $\mathcal{C}^r$ represents the class of $r-$continuously differentiable functions.

{Given a Riemannian manifold $(\cm, g)$ we use $R$ to denote the Riemann curvature tensor. We use the symbol $D_t$ for the covariant derivative along a curve parameterized by $t$.  Let $\ci\subset \mathbb{R}$ be an interval, $t, t_0\in \ci$ and let $c:\ci \to \cm$ be a curve. Then we denote by $\Psi_{c(t_0)\to c(t)}$ the parallel transport map from $T_{c(t_0)}\cm$ to $T_{c(t)}\cm$ along the curve $c$. }

{Finally, as an exception, all the notations from Section \ref{jointrategeneralsection} and its proofs use lower case letters to denote points on the Euclidean vector space and the Riemannian manifold. This is done so as to abstract out the analytical machinery developed in Section \ref{jointrategeneralsection} as much as possible and reduce notational overload.    }

\section{Generalized min-max problem and convergence analysis}\label{jointrategeneralsection}

\subsection{Problem formulation}

Let $\mathcal{M}$ be a finite dimensional Riemannian manifold with sectional curvature $\mathbf{K}_{\mathcal{M}} \in [0,\Theta]$ for some $\Theta\geq 0$. Consider a function $f : S_1 \times S_2 \to \mathbb{R} $ where $S_1 \subseteq \mathbb{R}^d$, $S_2 \subseteq \mathcal{M} $ are compact, geodesically convex sets inheriting the metric and topologies from $\mathbb{R}^d, \mathcal{M}  $ respectively and the set $S_2$ is the closure of the normal neighborhood of a point in $\mathcal{M}$. (See subsection \ref{convdef} for precise definitions.) Further, $f$ satisfies the following conditions :
\begin{itemize}
    \item [\textbf{C1.}]  $f$ is jointly $\mathcal{C}^2$ smooth on some bounded open cover of $ S_1 \times S_2$ 
    \item [\textbf{C2.}] For every $x \in {\mathbb{R}^d} $, $f(x,\cdot)$ is $\mu$ geodesically strongly concave {locally} on $x \times S_2$ and uniformly for all $x$. Also, for all $x \in {\mathbb{R}^d} $ there exists a maximizer of $f(x,\cdot)$ in the interior of $x \times S_2$.
    \item [\textbf{C3.}] {For every $y \in S_2 $} the following hold:
    \begin{itemize}
    \item[\textbf{a.}] The function $f(\cdot,y)$ has at most finitely many, non-empty, connected components of critical points in the interior of the compact $S_1 \times y$ fiber. Within each such connected component $\mathcal{D}^*(y)$, the critical points can either all be local minima, or all be local maxima or all be strict saddle points\footnote{A point $x^* \in S_1 \times y$ is a strict saddle point of $f(\cdot,y)$ if $\nabla_x f(x^*,y)=0 $ and $\nabla_x^2 f(x^*,y) $ has at least one negative eigenvalue.} of the function $f(\cdot,y)$. Also, within every $S_1 \times y$ fiber, there exists at least one connected component $\mathcal{S}^*(y)$ of the local minima critical set of $f(\cdot,y)$.
        \item[\textbf{b.}]  For some $\beta \, \in \, (1,2]$, $f(\cdot,y)$ satisfies $\beta, C_{\beta}$ type \L{}ojasiewicz inequality in a closed $\delta>0$ neighborhood of every connected component of the local minima critical set of $f(\cdot,y)$  on $S_1 \times y$. More precisely for some $\beta \in (1,2]$ there exist $\delta,C_{\beta}>0$ such that
      $$   \norm{\nabla_x f(x,y) }^{\beta}  \geq C_{\beta}\bigg(f(x,y) - f_*(y)\bigg) \geq \left(\frac{\beta}{\beta-1} \right)^{\frac{-\beta}{\beta-1}}C_{\beta}^{\frac{\beta}{\beta-1}}\bigg( \mathrm{dist}(x, \mathcal{S}^*(y))\bigg)^{\frac{\beta}{\beta -1}} $$
      where $ \mathcal{S}^*(y)$ is a connected component of the local minima critical set of $f(\cdot,y)$ {in the interior of the compact} $S_1 \times y $ fiber, $$ f_*(y) := \inf_{x \in  \mathcal{S}^*(y) + \mathcal{B}_{\delta}(0)} f(x,y). $$ 
      
      %{I think we need to clarify the quantifiers here to not confuse the reader. There exist $\delta>0$, a $\beta\in (1,2]$ and a $C_\beta>0$.}
      
      \item[\textbf{c.}] Let $\wt y \in S_2$ be arbitrary, then there always exists a connected component $\mathcal{S}^*(\wt y) $  of the local minima critical set {in the interior of the compact} $S_1 \times \wt y $ fiber such that  $$ \mathrm{dist}_H(\mathcal{S}^*(\wt y) , \mathcal{S}^*(y) ) < \frac{\delta}{4} \, \, .$$ 
      \item[\textbf{d.}] Let $\mathcal{D}_1^*( y), \mathcal{D}_2^*( y) $ be any pair of distinct connected components of the critical set for the function $f(\cdot,y)$ {in the interior of the compact} $S_1 \times  \{y\} $ fiber. Then the sets $\mathcal{D}_1^*( y), \mathcal{D}_2^*( y) $ are at least $2 \delta$ -separated.
    \end{itemize}
   % \item [\textbf{C4.}] There exists at least one $(x^*, y^*) \in \text{int} (S_1 \times S_2)$ such that $(x^*, y^*) $ is a saddle point of $f$.
    \item [\textbf{C4.}] Let $L>0$ be the uniform Lipschitz constant for the functions $$f(\cdot, \cdot), \nabla_x f (x, \cdot), \nabla_y f (x, \cdot), \nabla_y f(\cdot,y) , \nabla_x f(\cdot,y) , \exp_y(\cdot), \exp_{(\cdot)}(v)$$ on $ S_1 \times S_2$ where $ \nabla_x, \nabla_y$ are gradient operators in the metric of $S_1, S_2$ respectively. Also, let $ \norm{\nabla_y \nabla_x f(\cdot, \cdot)}_{y} \le L$, $ \norm{\nabla_yf(\cdot, \cdot)}_y \le L$ uniformly on $ S_1 \times S_2$. 
\end{itemize}
   % {Changed `disjoint connected component' to `connected components'. Two distinct connected components are disjoint to begin with.}

 {We know that $(p,v)\mapsto \exp_{p}(v)$ is a smooth function on an open set of $T\cm$ (\cite[Theorem 5.19]{lee2018}), and hence $\exp_p(v)$ is Lipschitz continuous in both $p,v$ on some compact sets. Also note that the smoothness of $ \nabla_x f (x, \cdot), \nabla_y f (x, \cdot)$ is defined via parallel transport of gradient vector fields along geodesics. } 

\begin{defn}[Basin saddle point]\label{basinsaddledef}
    Under \textbf{C1-C4}, a point  $(x^*, y^*) \in \text{int} (S_1 \times S_2)$ is defined as a basin saddle point of the function $f$ if we have the following:
    \begin{itemize}
        \item   $x^* \in  \mathcal{S}^*(y^*)$ where $ \mathcal{S}^*(y^*)$ is a connected component of the local minima critical set of $f(\cdot,y^*)$ {in the interior of the compact} $S_1 \times y^* $ fiber
        \item   $  f(x^*,y) \le  f(x^*,y^*)  \le   f(x,y^*) \quad \forall \quad x \in (\mathcal{S}^*(y^*) + \mathcal{B}_{\delta}(0)) \times y^* \, \, , \, \, \forall \quad y \in x^* \times S_2  \, \, .$
    \end{itemize}
\end{defn}

We are interested in solving the following min-max problem over the domain $S_1 \times S_2$ :
\[
\min_{x \in S_1} \max_{y \in S_2} f(x,y) \, .
\]
However the above problem is in general not solvable in polynomial time by any first order iterative method since $ f(\cdot ,y)$ can be nonconvex on $S_1$ for any $y \in S_2$. This is because condition \textbf{C3} provides only local gradient growth around the critical sets and away from critical set neighborhoods $ f(\cdot ,y)$ can be arbitrary with non-benign nonconvex landscape. Hence we restrict ourselves to finding the basin saddle points of $f$ on $S_1 \times S_2$ (definition \ref{basinsaddledef}) and not any arbitrary min-max point /saddle point of $f$ on $S_1 \times S_2$ .
 To do so we need to build an analytical framework for a Riemannian gradient ascent descent type first order numerical method (described later). This analytical framework will use elements of differential geometry and a Lyapunov type convergence analysis that will be developed in the forthcoming sections.
 
 \subsection{Riemannian gradient descent under strong geodesic convexity}\label{sectionboumalref}

\begin{lem}[ Lemma 6 in \cite{zhang2016firstorder}]\label{alexandrovlem1}
    If $a,b ,c$ are the sides (i.e., side lengths) of a geodesic triangle in an Alexandrov space with curvature lower bounded by $\upsilon$, and $A$ is the angle between sides $b$ and $c$, then
    \begin{align*}
        a^2 \le  \frac{\sqrt{|\upsilon|}c}{\tanh (\sqrt{|\upsilon|}c)} b^2 + c^2 - 2bc \cos A \, .
    \end{align*}
\end{lem}

\begin{theorem}[11.29 Boumal adaptation \cite{boumal2023introduction}]\label{boumalthm1}
Let $f : \mathcal{M} \to \mathbb{R}$ be differentiable and $\mu$-strongly geodesically convex
on a compact, convex manifold $\mathcal{M}$ with bounded sectional curvature. Given $x_0 \in \mathcal{M}$, consider
the sublevel set
\[
S_0 = \{ x \in \mathcal{M} : f(x) \le f(x_0) \}.
\]
Assume $f$ has $L$-Lipschitz continuous gradient on a neighborhood of $\mathcal{M}$ and $f$ has a minimizer $x^*$ in the interior of $\mathcal{M}$.
Consider gradient descent with exponential retraction and constant step-size
$h \in (0,1/L)$ initialized at $x_0$:
\[
x_{k+1} = \operatorname{exp}_{x_k}
\left(
-h \nabla_{x} f(x_k)
\right),
\quad k = 0,1,2,\dots
\]
 Then the iterate sequence $\{x_k\}$ converges to $x^*$ at
least linearly. More precisely, with $\kappa =  \frac{1}{\mu h(2- Lh)} \ge 1$, the whole sequence
stays in $S_0$ and
\begin{equation}
f(x_k) - f(x^*)
\le \left(1 - \frac{1}{\kappa}\right)^k
\bigl(f(x_0) - f(x^*)\bigr),
\end{equation}
\begin{equation}
\operatorname{dist}(x_{k+1},x^*) \le  \sqrt{  (1 -2 \mu h + C_{\upsilon} L^2 h^2)} \operatorname{dist}(x_k,x^*), 
\end{equation}
\begin{equation}
\operatorname{dist}(x_k,x^*) \le   \min \bigg\{ \sqrt{\frac{L}{\mu}}
\left(1 - \frac{1}{\kappa}\right)^{k/2} , (1 -2 \mu h + C_{\upsilon} L^2 h^2)^{k/2}  \bigg\} \operatorname{dist}(x_0,x^*),
\end{equation}
for all $k \ge 0$ where $  C_{\upsilon} =\frac{\sqrt{|\upsilon|} \, \mathrm{diam}(\mathcal{M}) }{\tanh (\sqrt{|\upsilon|} \, \mathrm{diam}(\mathcal{M}))} $ and $\upsilon$ is the lower bound of the sectional curvature of $\mathcal{M}$.
\end{theorem}
The proof of Theorem \ref{boumalthm1} is in Appendix \ref{boumalappendix}.

\subsection{Contraction bounds for fixed point iterations along fibers of the product manifold}\label{sectionfiberrates}
The following lemmas hold:
\begin{lem}\label{contractionlema1}
    Under \textbf{C1-C4}, for any $x \in S_1$, the iterate sequence $\{y_k\}_k$ generated by the metric gradient ascent update $ y_{k+1} := \exp_{y_k}(h \nabla_y f(x, y_k)) $, with initialization $y_0 \in \text{int}(x \times S_2) $ and $Lh <1$, satisfies:
    \begin{align*}
     f^*(x) - f(x, y_{k+1})  &\leq \left(1 - \frac{1}{\kappa}\right)[ f^*(x) -f(x, y_{k})   ] \\
      \mathrm{dist}(y_{k+1}, y^*(x) ) & \leq \wt\gamma \, \mathrm{dist}(y_{k}, y^*(x) ) 
    \end{align*}
    where $ f^*(x) : = \sup_{y \in x \times S_2} f(x,y)  $, $  y^*(x) : = \arg\sup_{y \in x \times S_2} f(x,y)  $,  $ \kappa =  \frac{1}{\mu h(2- Lh)} \ge 1 $, $$ \wt\gamma = \sqrt{(1 -2 \mu h + C_{\upsilon} L^2 h^2)} <1 $$ for any sufficiently small $h$, $  C_{\upsilon} =\frac{\sqrt{|\upsilon|} \, \mathrm{diam}(S_2) }{\tanh (\sqrt{|\upsilon|} \, \mathrm{diam}(S_2))} $ and $\upsilon$ is the lower bound of the sectional curvature of $\mathcal{M}$.
\end{lem}
\begin{proof}
    Use Theorem \ref{boumalthm1} with the function $-f$ and the compact subset $S_2$ which is the closure of a geodesically convex normal neighborhood of a point in $\cm$. 
\end{proof}

\begin{lem}\label{contractionlema2}
    Under \textbf{C1-C4}, for any $y \in S_2$, the iterate sequence $\{x_k\}$ generated by the GD update $ x_{k+1} := x_k - h \nabla_x f(x_k, y)$, with initialization $x_0 \in \mathcal{S}^*(y) + \mathcal{B}_{\delta/2}(0)$ and $Lh <1$, stays bounded inside $\mathcal{S}^*(y) + \mathcal{B}_{\delta}(0) $ and satisfies:
    \begin{align*}
        f(x_{k+1}, y) - f_*(y) &\leq \rho_k[ f(x_{k}, y) - f_*(y) ] \\
         \bigg( \mathrm{dist}(x_{k+1}, \mathcal{S}^*(y))\bigg)^{\frac{\beta}{\beta -1}} & \leq \frac{\rho_k}{\left(\frac{\beta}{\beta-1} \right)^{\frac{-\beta}{\beta-1}}C_{\beta}^{\frac{1}{\beta-1}}} [f(x_k,y) -f_*(y)] \\
         \bigg( \mathrm{dist}(x_{k+1}, \mathcal{S}^*(y))\bigg)^{\frac{\beta}{\beta -1}} 
             & \leq \frac{ \prod_{j=0}^k \rho_j}{\left(\frac{\beta}{\beta-1} \right)^{\frac{-\beta}{\beta-1}}C_{\beta}^{\frac{1}{\beta-1}}} [f(x_0,y) -f_*(y)] \le \tau_k \bigg( \mathrm{dist}(x_{0}, \mathcal{S}^*(y))\bigg)^{{\beta}} 
    \end{align*}
    where $ f_*(y) : = \inf_{x \in \mathcal{S}^*(y) + \mathcal{B}_{\delta}(0)} f(x,y)  $, $ \rho_k = {\bigg( 1 - h(1 - \tfrac{Lh}{2}) C^{\tfrac{2}{\beta}}_{\beta}\bigg(f(x_k,y) - f_*(y)\bigg)^{\tfrac{2}{\beta}-1}\bigg)} \in [0,1] $ for any sufficiently small $h$ and 
    $ \tau_k := \frac{ L^{\beta} \prod_{j=0}^k \rho_j}{C_{\beta}\left(\frac{\beta}{\beta-1} \right)^{\frac{-\beta}{\beta-1}}C_{\beta}^{\frac{1}{\beta-1}}}$ and $ \tau_k < 1$ for sufficiently large $k$. 
    \\ Further, if $ \mathrm{dist}(x_{0}, \mathcal{S}^*(y)) <  \frac{C^{1/\beta}_{\beta}}{L}$ then we have for any $k >0$ that:
    \begin{align}
            \bigg( \mathrm{dist}(x_{k}, \mathcal{S}^*(y)) \bigg) &  \le  C_{L,\beta}  \frac{\bigg( \mathrm{dist}(x_{0}, \mathcal{S}^*(y)) \bigg)^{\frac{\beta (\beta -1) }{2}}}{\bigg(   1+ k \,C'h  \bigg)^{\frac{\beta-1}{2}}}
        \end{align}
        where $C_{L,\beta} :=   \left(\frac{\beta}{\beta-1} \right) \frac{L ^{\frac{\beta (\beta -1)}{2}}}{ C_{\beta}^{\frac{\beta-1}{2} + \frac{1}{\beta }} } $, $ C' := (1 - \tfrac{Lh}{2}) C^{2/\beta}_{\beta} $.
\end{lem}
The proof of Lemma \ref{contractionlema2} is in Appendix \ref{sectionfiberratesappendix}.
\begin{lem}\label{contractionlema3_supplement}
    Under \textbf{C1-C4} for the function $-f$, for any $x \in S_1 $, any $y_1, y_2 \in S_2$ the following inequality holds:
    \begin{align}
         L\, \mathrm{dist}^2(y_{1}, y_{2} ) \ge \langle  {\Psi}_{y_{2} \to y_{1}} (\nabla_y f(x, y_{2})) - \nabla_y f(x, y_1), \exp_{y_1}^{-1} (y_{2}) \rangle_{y_1} \ge \mu  \, \mathrm{dist}^2(y_{1}, y_{2} ) \label{comparisonineq1a}
    \end{align}
    where ${\Psi}_{y_{2} \to y_{1}} $ is the parallel transport map from $ y_{2} $ to $ y_{1}$.
\end{lem}
The proof of Lemma \ref{contractionlema3_supplement} is in Appendix \ref{sectionfiberratesappendix}.
{Observe that when $\mathcal{M}$ is Euclidean space, the parallel transport map reduces to identity map and the exponential map $\exp_{y_1}^{-1} (y_{2}) $ is simply $y_2 - y_1$. Then Lemma \ref{contractionlema3_supplement} recovers the standard inequality of \[L \norm{y_2 - y_1}^2 \ge \langle \nabla_y f(x, y_{2}) - \nabla_y f(x, y_1) , y_2 - y_1\rangle \ge \mu \norm{y_2 - y_1}^2  \] for locally $\mu$ strongly convex and $L$ smooth function $f(x, \cdot)$. }

\begin{lem}\label{contractionlema3}
    Under \textbf{C1-C4}, for any $x \in S_1$, the iterate sequence $\{y_k\}$ generated by the metric gradient ascent update $ y_{k+1} := \exp_{y_k}(h \nabla_y f(x, y_k)) = G(x,y_k) $, with initialization $y_0 \in \text{int}(x \times S_2) $ and $Lh <1$, satisfies:
    \begin{align*}
      \mathrm{dist}(y_{k+1}, y_{k} ) & \leq \eta \, \mathrm{dist}(y_{k}, y_{k-1} )  
    \end{align*}
    where $ \eta =  1-h\mu+\frac{3}{4}\Theta h^2\left(\sup_{p\in S_2} \norm{\nabla_yf (x,p)}_p^2\right) + \mathcal{O}(h^3) \approx 1-h\mu+\frac{3}{4}\Theta L^2 h^2$.
\end{lem}
The proof of Lemma \ref{contractionlema3} is in Appendix \ref{sectionfiberratesappendix}.
{Recall that when $\mathcal{M}$ is Euclidean space, the contraction factor from Lemma \ref{contractionlema3} is $\eta = 1 - h \mu$ since the sectional curvature of a Euclidean space is $0$. However, for a general Riemannian manifold $\mathcal{M}$ we get second and third order correction terms in $h$ from the expression of $\eta$ that depend on the sectional curvature $\mathbf{K}_{\mathcal{M}} \in [0,\Theta]$.     }

\begin{lem}[Finite path length on the fiber $ S_1 \times y$]\label{lem:length}
Under \textbf{C1-C4}, for any $y \in S_2$, the iterate sequence $\{x_k\}$ generated by the GD update $ x_{k+1} := x_k - h \nabla_x f(x_k, y)$ on the fiber $ S_1 \times y$ , with initialization $x_0 \in \mathcal{S}^*(y) + \mathcal{B}_{\delta/2}(0)$ and $Lh <1$, stays bounded inside $\mathcal{S}^*(y) + \mathcal{B}_{\delta}(0) $ and have finite path length:
\begin{equation}\label{eq:length}
  \sum_{k=0}^{\infty}\norm{x_{k+1}-x_{k}}
  \ \le\
  \frac{1}{(1-\tfrac{Lh}{2})}\cdot\frac{\beta}{\beta-1}
  \left(\frac{\bigl(f(x_{0},y)-f_{*}(y)\bigr)^{\beta-1}}{C_{\beta}}\right)^{\!1/\beta} \le \frac{\beta \, L^{\beta-1} \mathrm{dist}^{\beta-1}(x_0, \mathcal{S}^*(y))}{(1-\tfrac{Lh}{2})(\beta-1) C_{\beta}}  .
\end{equation}
\end{lem}
The proof of Lemma \ref{lem:length} is in Appendix \ref{sectionfiberratesappendix}.

\subsection{Riemannian gradient ascent- multi step gradient descent}
{Consider the RGA - multi GD update:
\begin{align}
        y_{k+1} &:= \exp_{y_k}(h_1 \nabla_y f(x_k, y_k)) = G(x_k,y_k) \label{RGA-MGD1}  \tag{RGA} \\
        x_{k+\frac{l+1}{J}} &:= x_{k+\frac{l}{J}} - h_2 \nabla_x f(x_{k+\frac{l}{J}}, y_{k+1}) \quad, \quad l \in  \{0,1,\cdots, J-1\} \quad  \label{RGA-MGD2}  \tag{J-step GD}
\end{align}
where $J$ can be constant or a function of $k$.
}\\
We measure the rate of convergence of the sequence $\{(x_k,y_k)\}_k$, generated from \eqref{RGA-MGD1}- \eqref{RGA-MGD2}, to a basin saddle point  $(x^*, y^*) \in \text{int} (S_1 \times S_2)$, via the duality gap: 
$$ f^*(x)  - f_*(y)  \equiv \sup_{y \in x \times S_2} f(x,y) - \inf_{x \in \mathcal{S}^*(y) + \mathcal{B}_{\delta}(0)} f(x,y)   \quad , $$
$$ f^*(x)  - f_*(y)      \ge 0 \quad \forall \quad  (x,y) \in  \mathcal{S}^*(y) + \mathcal{B}_{\delta}(0)\times S_2 . $$
We now present a seemingly trivial yet important property based on the definition of a basin saddle point.
\begin{lem}\label{RGA-MGD1-lem1}
    Under \textbf{C1-C4}, any point $(x^*, y^*) \in \text{int} (S_1 \times S_2)$ is a basin saddle point (Definition \ref{basinsaddledef}) of $f(\cdot ,\cdot)$ if and only if
    $$  f^*(x^*)  - f_*(y^*) = \sup_{y \in x^* \times S_2} f(x^*,y) - \inf_{x \in \mathcal{S}^*(y^*) + \mathcal{B}_{\delta}(0)} f(x,y^*) = 0 \, \, \quad \textbf{and} \quad x^* \in (\mathcal{S}^*(y^*) + \mathcal{B}_{\delta}(0)) \times y^* \, . $$
\end{lem}
\begin{proof}
    At any basin saddle point  $(x^*, y^*) \in \text{int} (S_1 \times S_2)$, by definition \ref{basinsaddledef} we have $x^* \in \mathcal{S}^*(y^*) \times y^*$ and
    $$  f(x^*,y) \le  f(x^*,y^*)  \le   f(x,y^*) \quad \forall \quad x \in (\mathcal{S}^*(y^*) + \mathcal{B}_{\delta}(0)) \times y^* \, \, , \, \, \forall \quad y \in x^* \times S_2  \, \, .$$
    Since $ y^* = \arg \sup_{y \in x^* \times S_2} f(x^*,y) $ by \textbf{C2} and $ f(x^*, y^*) = \inf_{x \in \mathcal{S}^*(y^*) + \mathcal{B}_{\delta}(0)} f(x,y^*)  $ by \textbf{C3}, the conclusion $ f^*(x^*)  = f_*(y^*) $ follows. In the other direction suppose for some point $(x^*, y^*) \in \text{int} (S_1 \times S_2) $ we have $\sup_{y \in x^* \times S_2} f(x^*,y) = \inf_{x \in \mathcal{S}^*(y^*) + \mathcal{B}_{\delta}(0)} f(x,y^*)  $ and $x^* \in (\mathcal{S}^*(y^*) + \mathcal{B}_{\delta}(0)) \times y^*$. Then by definition of sup and inf we write:
     $$    f(x^*,y^*)  \le \sup_{y \in x^* \times S_2} f(x^*,y) = \inf_{x \in \mathcal{S}^*(y^*) + \mathcal{B}_{\delta}(0)} f(x,y^*)\le   f(x,y^*) \quad \forall \quad x \in (\mathcal{S}^*(y^*) + \mathcal{B}_{\delta}(0)) \times y^* \, \, , \, \,  $$ and
     $$  f(x^*,y) \le \sup_{y \in x^* \times S_2} f(x^*,y) = \inf_{x \in \mathcal{S}^*(y^*) + \mathcal{B}_{\delta}(0)} f(x,y^*) \le  f(x^*,y^*)   \quad \forall \quad y \in x^* \times S_2  \, \, $$
     thus     $$  f(x^*,y) \le  f(x^*,y^*)  \le   f(x,y^*) \quad \forall \quad x \in (\mathcal{S}^*(y^*) + \mathcal{B}_{\delta}(0)) \times y^* \, \, , \, \, \forall \quad y \in x^* \times S_2  \, \, $$
     implying $x^* \in \mathcal{S}^*(y^*) \times y^*$ by \textbf{C3} which completes the proof.
\end{proof}
For some fixed reference $y_0 \in S_2$ consider the Lyapunov function $V_{\delta} : S_1 \times S_2 \to \mathbb{R}$ as follows:
\begin{align}
    V_{\delta}(x,y) &:=  \bigg( \sup_{ \wt y \in x \times S_2} f(x,\wt y) -  f(x, y) \bigg) + \bigg(f( x,y) - \inf_{\wt x \in (\mathcal{S}^*(y_0) + \mathcal{B}_{\delta/2}(0)) \times y } f(\wt x,y)\bigg)  \label{lyapunov-RGA-MGD}  \tag{Lyapunov function}
\end{align}
{Recall that $S_1 \times S_2$ is bounded. From \textbf{C4}, $ \sup_{(x,y) \in S_1 \times S_2} \abs{f(x,y)} < \infty $, the function $f(\cdot,y)$ is Lipschitz continuous in $x$ on $ S_1 \times y$ uniformly for all $y \in S_2$, the fiber $x \times S_2$ is isomorphic to $S_2$, and hence the family $\{f(\cdot,y)\}_{y \in S_2}$ is uniformly Lipschitz continuous (hence uniformly equicontinuous) on $ S_1$, we get that the function $f_1(x) := \sup_{y \in x \times S_2}  f(x,y)$ is uniformly Lipschitz continuous for all $x \in S_1 $. Similarly, by \textbf{C4}, $ \sup_{(x,y) \in S_1 \times S_2} \abs{f(x,y)} < \infty $, the function $f(x,\cdot)$ is Lipschitz continuous in $y$ on $ x \times S_2$ uniformly for all $x \in S_1$, the set $(\mathcal{S}^*(y_0) + \mathcal{B}_{\delta/2}(0)) \times y$ is isomorphic to $\mathcal{S}^*(y_0) + \mathcal{B}_{\delta/2}(0)$, and hence the family $\{f(x,\cdot)\}_{x \in \mathcal{S}^*(y_0) + \mathcal{B}_{\delta/2}(0)}$ is uniformly Lipschitz continuous (hence uniformly equicontinuous) on $S_2$, we get that the function $f_2(y) := \inf_{x \in (\mathcal{S}^*(y_0) + \mathcal{B}_{\delta/2}(0)) \times y } f(x,y)$ is uniformly Lipschitz continuous for all $y \in  S_2$. Then $V_{\delta}(x,y) := f_1(x) - f_2(y) $ is separable in $x,y$ and hence is jointly uniformly Lipschitz continuous on $S_1 \times S_2$.   \\ 
From \textbf{C3} there exists a connected component of local minima $\mathcal{S}^*(y) $ such that $ \mathrm{dist}_H(\mathcal{S}^*(y_0),\mathcal{S}^*(y)) < \delta/4 $ for any $y \in S_2$ and thus 
\[  \mathcal{S}^*(y) + \mathcal{B}_{\delta/8}(0) \subseteq \mathcal{S}^*(y_0) + \mathcal{B}_{\delta/2}(0) \subseteq \mathcal{S}^*(y) + \mathcal{B}_{3\delta/4}(0)  \]
in the sense of standard projections from the fiber $ S_1 \times y $ to the fiber $ S_1 \times y_0 $ isometrically and vice versa. But then we immediately get from \textbf{C3} that
\[  \inf_{\wt x \in (\mathcal{S}^*(y_0) + \mathcal{B}_{\delta/2}(0)) \times y} f(\wt x,y) = \inf_{\wt x \in \mathcal{S}^*(y) \times y} f(\wt x,y) = \inf_{\wt x \in (\mathcal{S}^*(y) + \mathcal{B}_{\delta/2}(0)) \times y} f(\wt x,y) \]
and thus
\begin{align}
     V_{\delta}(x,y) &:=  \bigg( \sup_{ \wt y \in x \times S_2} f(x,\wt y) -  f(x, y) \bigg) + \bigg(f( x,y) - \inf_{\wt x \in (\mathcal{S}^*(y_0) + \mathcal{B}_{\delta/2}(0)) \times y} f(\wt x,y)\bigg) \nonumber \\ &= \bigg( \sup_{ \wt y \in x \times S_2} f(x,\wt y) -  f(x, y) \bigg) + \bigg(f( x,y) - \inf_{\wt x \in (\mathcal{S}^*(y) + \mathcal{B}_{\delta/2}(0)) \times y} f(\wt x,y)\bigg) \, .
\end{align}
}
\\
The point of analyzing the limit of the sequence $\{V_{\delta}(x_k,y_k)\}_k$ is to obtain $$ \lim_{k \to \infty} V_{\delta}(x_k,y_k) = 0 $$ at a certain rate of convergence.
We first derive a few technical lemmas for error analysis on the $ V_{\delta}(x,y) $ from \eqref{lyapunov-RGA-MGD}.

\subsection{Technical results for Lyapunov convergence}\label{sectionLyapunovconvergence}

\begin{lem}\label{RGA-MGD1-lem0}
     Under \textbf{C1-C4} for any $k$, let $ \mathcal{S}^*(y_{k+1})$ be a local minima connected component on the fiber $ S_1 \times y_{k+1}$ that is at most $\delta/4$ separated from the local minima connected component $ \mathcal{S}^*(y_{k})$ on the fiber $ S_1 \times y_{k}$. Then we have the following H\"{o}lder bound :
     \begin{align}
       \mathrm{dist}_H(\mathcal{S}^*(y_{k+1}), \mathcal{S}^*(y_k))  & \le  L^{\beta -1} \, \left(\frac{\beta}{\beta-1} \right) C_{\beta}^{-1} \, \mathrm{dist}^{\beta - 1}(y_{k+1},y_k) .
     \end{align}
     Further we have the following uniform bound:
\begin{align}
    \mathrm{dist}_H(\mathcal{S}^*(y_{k+1}), \mathcal{S}^*(y_k)) 
          & \le  h_1^{\beta - 1} \left(\frac{\beta}{\beta-1} \right) C_{\beta}^{-1} L^{2\beta - 2} \, \mathrm{diam}^{\beta - 1}(S_2) \, .
\end{align}
Note that the Hausdorff distance $ \mathrm{dist}_H(\mathcal{S}^*(y_{k+1}), \mathcal{S}^*(y_k))$ makes sense as $\mathcal{S}^*(y_k)$ and $\mathcal{S}^*(y_{k+1})$ are both subsets of $S_1$.

\end{lem}
The proof of Lemma \ref{RGA-MGD1-lem0} is in Appendix \ref{sectionLyapunovconvergenceappendix}.
{Lemma \ref{RGA-MGD1-lem0} bounds the Hausdorff distance between the local minima connected components\footnote{We note that from here onward, the local minima connected components $\mathcal{S}^*(y_k)$ and $\mathcal{S}^*(y_{k+1})$ for any $k$ are understood in the sense of \textbf{C3}, i.e., these components are at most $\delta/4$ separated.} $\mathcal{S}^*(y_k)$ and $\mathcal{S}^*(y_{k+1})$ on the isomorphic fibers $ S_1 \times y_{k}$ and $ S_1 \times y_{k+1}$ respectively. In particular it shows that the distance function between critical sets on fibers $ S_1 \times y_k ,  S_1 \times y_{k+1} $ is H\"{o}lder continuous with respect to $\mathrm{dist}(y_k ,y_{k+1} )$ with exponent $\beta -1$. The proof comprises of the following steps: picking an arbitrary point $ x^*(y_{k+1})$ in the critical set $\mathcal{S}^*(y_{k+1})$, bounding the distance between $ x^*(y_{k+1})$ and $\mathcal{S}^*(y_{k})$ in terms of $\norm{\nabla_x f( x^*(y_{k+1}), y_{k})} $ using \textbf{C3} and finally locally linearizing the gradient $ \nabla_x f( x^*(y_{k+1}), y_{k}) $ on the fiber $ S_1 \times y_{k}$ with respect to the zero gradient $ \nabla_x f( x^*(y_{k+1}), y_{k+1}) $ on the fiber $ S_1 \times y_{k+1}$ via the mixed derivative operator $ \nabla_y \nabla_x f $. Since $f \in \mathcal{C}^2 (S_1 \times S_2)$, by Taylor's theorem the operator norm $\norm{ \nabla_y \nabla_x f}_{op} $ controls the local gradient perturbations across fibers. Lemma \ref{RGA-MGD1-lem0} supplies a local distance regularity condition between critical sets, with a Lipschitz bound for $\beta = 2$ case and a H\"{o}lder bound for $\beta \in (1,2)$ case. Locally, for small perturbations of $y_k$, a H\"{o}lder bound allows larger local variations in the critical sets compared to a Lipschitz bound. We however emphasize that these derived bounds may not necessarily be strict (both in the constant and the exponent) since we did not assume anything beyond the boundedness of distance between critical sets from \textbf{C3} (e.g., no continuity of the distance function was assumed). Further, it may not be possible to develop convergence guarantees in the absence of some form of continuity/sharper regularity of the distance between critical sets. Therefore, to get better regularity on the critical set variations and to derive convergence rates, we will later assume a Lipschitz bound on the distance between critical sets for any $\beta \in (1,2]$.      }

\begin{lem}\label{RGA-MGD1-lem2}
     Under \textbf{C1-C4} if the sequences $\{y_k\}_k$, $ \{x_{k+ j/J}\}_{k,j}$ generated from \eqref{RGA-MGD1}, \eqref{RGA-MGD2} stay bounded in $S_1 \times S_2$, then for any $k$, the following recursion holds:
    \begin{align}
        [f(x_{k+1}, y_{k+1}) - f_*( y_{k+1}) ] &  \le  \bigg(\prod_{j=0}^{J-1} \rho_{k +\frac{j}{J}} \bigg) [f(x_{k}, y_{k}) - f_*(y_{k}) ]  \nonumber \\ & +  L \, \mathrm{dist}(y_{k+1},y_k)\bigg(\prod_{j=0}^{J-1} \rho_{k +\frac{j}{J}} \bigg) \bigg( 2 + \frac{L^{\beta -1}}{C_{\beta}} \mathrm{dist}^{\beta -1}(y_{k+1},y_k) \bigg)
    \end{align}
    where $ f_*(y_{k+1}) : = \inf_{x \in \mathcal{S}^*(y_{k+1}) + \mathcal{B}_{\delta}(0)} f(x,y_{k+1})  $, $ f_*(y_{k}) : = \inf_{x \in \mathcal{S}^*(y_{k}) + \mathcal{B}_{\delta}(0)} f(x,y_{k})  $ and $$ \rho_{k+ \frac{j}{J}} = {\bigg( 1 - h_2(1 - \tfrac{Lh_2}{2}) C^{\tfrac{2}{\beta}}_{\beta}\bigg(f(x_{k+ \frac{j}{J}},y_{k+1}) - f_*(y_{k+1})\bigg)^{\tfrac{2}{\beta}-1}\bigg)} \in [0,1] $$ for any sufficiently small $h_2$. Furthermore, we also have the inequality:
    \begin{align*}
        \mathrm{dist}(x_{k+1}, \mathcal{S}^*(y_{k+1})) & {\le} \,\, \tau_k^{\frac{\beta -1}{\beta}} \mathrm{dist}^{\beta -1}(x_{k}, \mathcal{S}^*(y_k)) +   \tau_k^{\frac{\beta -1}{\beta}} \bigg( L^{\beta -1} \, \left(\frac{\beta}{\beta-1} \right) C_{\beta}^{-1} \, \mathrm{dist}^{\beta - 1}(y_{k+1},y_k) \bigg)^{\beta -1}  \, 
    \end{align*}
    where $ \tau_k := \frac{ L^{\beta} \prod_{j=0}^{J-1} \rho_{k+ \frac{j}{J}}}{C_{\beta}\left(\frac{\beta}{\beta-1} \right)^{\frac{-\beta}{\beta-1}}C_{\beta}^{\frac{1}{\beta-1}}}$ and $ \tau_k^{\frac{\beta -1}{\beta}}  < 1$ for sufficiently large $J$. \\
    Next, if $ \mathrm{dist}(x_{k}, \mathcal{S}^*(y_{k})) <  \frac{C^{1/\beta}_{\beta}}{L} -  h_1^{\beta - 1}\left(\frac{\beta}{\beta-1} \right) C_{\beta}^{-1} L^{2\beta - 2} \, \mathrm{diam}^{\beta - 1}(S_2)$ for $h_1$ sufficiently small and $L h_2 <1$, then we have that:
        \begin{align*}
        \mathrm{dist}(x_{k+1}, \mathcal{S}^*(y_{k+1}))
        & \le  C_{L,\beta}  \frac{\bigg( \mathrm{dist}(x_{k}, \mathcal{S}^*(y_k)) \bigg)^{\frac{\beta (\beta -1) }{2}}}{\bigg(   1+ J \,C'h_2  \bigg)^{\frac{\beta-1}{2}}}  + C_{L,\beta}  \frac{\bigg(  L^{\beta -1} \, \left(\frac{\beta}{\beta-1} \right) C_{\beta}^{-1} \, \mathrm{dist}^{\beta - 1}(y_{k+1},y_k) \bigg)^{\frac{\beta (\beta -1) }{2}}}{\bigg(   1+ J \,C'h_2  \bigg)^{\frac{\beta-1}{2}}} 
    \end{align*}
        where $C_{L,\beta} :=   \left(\frac{\beta}{\beta-1} \right) \frac{L ^{\frac{\beta (\beta -1)}{2}}}{ C_{\beta}^{\frac{\beta-1}{2} + \frac{1}{\beta }} } $, $ C' := (1 - \tfrac{Lh_2}{2}) C^{2/\beta}_{\beta} $.
\end{lem}
The proof of Lemma \ref{RGA-MGD1-lem2} is in Appendix \ref{sectionLyapunovconvergenceappendix}.
\begin{rem}\label{taukbounremark}
    For $\beta = 2$ the case reduces to the local Polyak \L{}ojasiewicz condition. Then 
    $$  \rho_{k+ \frac{j}{J}} := \rho = {\bigg( 1 - h_2(1 - \tfrac{Lh_2}{2}) C^{\tfrac{2}{\beta}}_{\beta}\bigg)} = {\bigg( 1 - h_2(1 - \tfrac{Lh_2}{2}) C_{\beta}\bigg)}< 1 \, \, \textbf{ for } h \textbf{ sufficiently small.} $$
    In particular, for $ h_2C_{\beta} <1$ (implied by $Lh_2<1$ since $ L > C_{\beta}$) we have the following bound on $ \tau_k$ for any $k$ :
    $$ \tau_k := \frac{ L^{2} \prod_{j=0}^{J-1} \rho_{k+ \frac{j}{J}}}{C_{\beta}\left(\frac{\beta}{\beta-1} \right)^{\frac{-\beta}{\beta-1}}C_{\beta}^{\frac{1}{\beta-1}}} = \frac{4 L^{2} \bigg( 1 - h_2(1 - \tfrac{Lh_2}{2}) C_{\beta}\bigg)^J }{C^2_{\beta}} \le  \frac{4 L^{2} \bigg( 1 - \frac{h_2 C_{\beta}}{2}\bigg)^J }{C^2_{\beta}}  \, \, .$$ 
\end{rem}

\begin{lem}\label{RGA-MGD1-lem3}
     Under \textbf{C1-C4}, if the sequences $\{y_k\}_k$, $ \{x_{k+ j/J}\}_{k,j}$ generated from \eqref{RGA-MGD1}, \eqref{RGA-MGD2} stay bounded in $S_1 \times S_2$, then for any $k$, the following recursion holds:
    \begin{align}
        \mathrm{dist}(y_{k+1},y^*(x_{k+1}))  &\le  \wt\gamma \, \mathrm{dist}(y_{k},y^*(x_k))   + \frac{ L}{\mu} \norm{x_{k+1} - x_k}
    \end{align}
    where $$ \wt\gamma = \sqrt{(1 -2 \mu h_1 + C_{\upsilon} L^2 h_1^2)} <1  \quad , \quad y^*(x_k) = \arg\max_{y \in S_2} f(x_k, y) $$ for any sufficiently small $h_1$, $  C_{\upsilon} =\frac{\sqrt{|\upsilon|} \, \mathrm{diam}(S_2) }{\tanh (\sqrt{|\upsilon|} \, \mathrm{diam}(S_2))} $ and $\upsilon$ is the lower bound of the sectional curvature of $\mathcal{M}$. 
\end{lem}
The proof of Lemma \ref{RGA-MGD1-lem3} is in Appendix \ref{sectionLyapunovconvergenceappendix}.
Note that Lemmas \ref{RGA-MGD1-lem2}, \ref{RGA-MGD1-lem3} provide inexact contraction bounds for the $J$ step gradient descent \eqref{RGA-MGD2} and the Riemannian gradient ascent \eqref{RGA-MGD1} respectively. 

\begin{lem}\label{RGA-MGD1-lem4z}
     Under \textbf{C1-C4} and for $L h_2 <1$, $J \ge 1$, if the sequences $\{y_k\}_k$, $ \{x_{k+ j/J}\}_{k,j}$ generated from \eqref{RGA-MGD1}, \eqref{RGA-MGD2} stay bounded in $S_1 \times S_2$, then for any $k>0$, the following bound holds uniformly for any $J$ :
    \begin{align}
        \norm{x_k - x_{k-1}}  &\le    \frac{\beta \, L^{\beta-1} }{(1-\tfrac{Lh_2}{2})(\beta-1) C_{\beta}} \bigg( \mathrm{dist}^{\beta-1}(x_{k-1}, \mathcal{S}^*(y_{k-1})) \nonumber \\ & +  \mathrm{dist}_H^{\beta-1}(\mathcal{S}^*(y_{k-1}), \mathcal{S}^*(y_k)) \bigg) \, .
    \end{align}
\end{lem}
\begin{proof}
    Applying Lemma \ref{lem:length} on the fiber $S_1 \times y_{k}$ with initial iterate $x_{k-1}$ yields the bound:
    \[
    \norm{x_k - x_{k-1}}  \le \sum_{j=0}^{J-1}\norm{x_{k-1+j/J}-x_{k-1+ (j+1)/J}} \le \frac{\beta \, L^{\beta-1} \mathrm{dist}^{\beta-1}(x_{k-1}, \mathcal{S}^*(y_k))}{(1-\tfrac{Lh_2}{2})(\beta-1) C_{\beta}}  \, .
    \]
    Applying triangle inequality on the right hand side using Lemma \ref{lem:dist-hausdorff} followed by sub-additivity of the concave function $t \mapsto t^{\beta-1}$ for $\beta \in (1,2]$, $t \ge 0$ gives:
    \[  \mathrm{dist}^{\beta-1}(x_{k-1}, \mathcal{S}^*(y_k)) \le  \mathrm{dist}^{\beta-1}(x_{k-1}, \mathcal{S}^*(y_{k-1})) +  \mathrm{dist}_H^{\beta-1}(\mathcal{S}^*(y_{k-1}), \mathcal{S}^*(y_k))  \]
    which completes the proof.
\end{proof}
Lemma \ref{RGA-MGD1-lem4z} controls the distance between the iterates $x_{k-1}$ and $x_k$ where $x_{k-1}$ is the initial iterate on the fiber $S_1 \times y_k$, and $x_{k}$ is the iterate generated by performing $J$ steps of gradient descent with initialization $x_{k-1}$ on the fiber $S_1 \times y_k$. The upper bound from Lemma \ref{RGA-MGD1-lem4z} is $J$ independent due to finite path length of gradient descent on the fiber $S_1 \times y_k$ as a consequence of the \L{}ojasiewicz inequality from \textbf{C3} (see Lemma \ref{lem:length}).   

\begin{lem}\label{RGA-MGD1-lem4}
     Under \textbf{C1-C4}, suppose the sequences $\{y_k\}_k$, $ \{x_{k+ j/J}\}_{k,j}$ generated from \eqref{RGA-MGD1}, \eqref{RGA-MGD2} stay bounded in $S_1 \times S_2$. Also suppose that $\mathrm{dist}_H(\mathcal{S}^*(y_{k-1}), \mathcal{S}^*(y_k)) $, for any $\beta \in (1,2]$, satisfies a Lipschitz estimate instead of the H\"{o}lder estimate from Lemma \ref{RGA-MGD1-lem0} as follows:
    \begin{align}
        \mathrm{dist}_H(\mathcal{S}^*(y_{k-1}), \mathcal{S}^*(y_k)) \le  D_{\beta} \, \mathrm{dist}(y_{k-1},y_k)
    \end{align} 
    for some constant with $ D_{\beta} >0$.
    Then for any $k$, the following recursions hold:
     \begin{enumerate}
         \item [a.] For $\beta = 2$, $L h_2 <1$, $J \ge 1$ we have
         \begin{align}
       \mathrm{dist}(y_{k+1},y_k)
           & \le  \, \bigg(\eta +  \frac{2 \, L^{3}  D_{\beta} h_1 }{(1-\tfrac{Lh_2}{2}) C_{\beta}} \bigg) \, \,\mathrm{dist}(y_{k-1},y_k)   +    \frac{2 \, L^{3}h_1 }{(1-\tfrac{Lh_2}{2}) C_{\beta}}  \mathrm{dist}(x_{k-1}, \mathcal{S}^*(y_{k-1})) .
    \end{align}
    \item [b.]  For arbitrary $\beta \in (1,2)$ with $Lh_2 <1$, $J \ge 1$ we have
    \begin{align}
        \mathrm{dist}(y_{k+1},y_k)
           & \le  \, \eta \, \,\mathrm{dist}(y_{k-1},y_k) +     \frac{\beta \, L^{\beta+1} h_1 }{(1-\tfrac{Lh_2}{2})(\beta-1) C_{\beta}} \bigg( \mathrm{dist}^{\beta-1}(x_{k-1}, \mathcal{S}^*(y_{k-1})) + \nonumber \\ &   D_{\beta}^{(\beta - 1)}  \, \mathrm{dist}^{(\beta - 1)}(y_{k-1},y_k) \bigg) 
    \end{align}
     \end{enumerate}
    where $ \eta  = 1-h_1\mu+\frac{3}{4}\Theta L^2 h_1^2  + \mathcal{O}(h_1^3) $ and $ \bigg(\eta +  \frac{2 \, L^{3}  D_{\beta} h_1 }{(1-\tfrac{Lh_2}{2}) C_{\beta}} \bigg) <1$ for sufficiently small $h_1$ provided $D_{\beta}  < \frac{(1-\tfrac{Lh_2}{2}) C_{\beta}} {2 \, L^{3}   }\mu $.
\end{lem}
 The proof of Lemma \ref{RGA-MGD1-lem4} is in Appendix \ref{sectionLyapunovconvergenceappendix}.

\subsection{Main theorems}\label{sectionmaintheorems}

\begin{theorem}\label{convergenceratethm_exactPL}
 Under \textbf{C1-C4}, for $\beta =2$ suppose the iteration \eqref{RGA-MGD1} - \eqref{RGA-MGD2} is initialized with $ (x_0, y_0) \in \text{int}(S_1 \times S_2)$ with $  \mathrm{dist}(x_0, \mathcal{S}^*(y_0)) \le \frac{\delta}{2} \,  $. Suppose $Lh_2 <1$, $h_1\ll 1$ with $\frac{2 \, L^{3}h_1 }{(1-\tfrac{Lh_2}{2}) C_{\beta}} \ll 1 $ and $J>0$ satisfies the following bound:
\begin{align*}
    \frac{2L}{C_{\beta}} \cdot \frac{2 L ( 1 -  \frac{h_2 C_{\beta}}{2} )^{J/2} }{C_{\beta}}  &<\frac{2 \, L^{3}h_1 }{(1-\tfrac{Lh_2}{2}) C_{\beta}}  \, .
\end{align*}
Also suppose that $\mathrm{dist}_H(\mathcal{S}^*(y_{k-1}), \mathcal{S}^*(y_k)) $, for all $k \ge 0$, satisfies a Lipschitz estimate from Lemma \ref{RGA-MGD1-lem0} with a smaller constant $ D_{\beta}$ with $ 0< D_{\beta}  < \frac{(1-\tfrac{Lh_2}{2}) C_{\beta} \mu}{2 \, L^{3}  } $ as follows:
    \begin{align*}
        \mathrm{dist}_H(\mathcal{S}^*(y_{k-1}), \mathcal{S}^*(y_k)) \le  D_{\beta} \, \mathrm{dist}(y_{k-1},y_k).
    \end{align*}
Then there exists $h_1 >0$ such that the iteration $\{(x_k, y_k)\}_k$ from \eqref{RGA-MGD1} - \eqref{RGA-MGD2} converges to a basin saddle point $ (x^*,y^*) \in \text{int}(S_1 \times S_2)$ of $f(\cdot , \cdot )$ at a linear rate $\mathcal{O}( a^k)$ where 
 $$ a = \max \bigg\{ \bigg(1 - \frac{h_1}{4}\bigg(\mu - \frac{2 \, L^{3}  D_{\beta} }{(1-\tfrac{Lh_2}{2}) C_{\beta}} \bigg)    \bigg), \sqrt{(1 -2 \mu h_1 + C_{\upsilon} L^2 h_1^2)} \bigg\} \in (0,1)  \, . $$
\end{theorem}
The proof of Theorem \ref{convergenceratethm_exactPL} is in Appendix \ref{sectionmaintheoremsappendix}.

\begin{rem}\label{rmk:Dbeta}
The hypothesis $D_\beta < (1-\tfrac{Lh_2}{2})C_\beta\mu/(2L^3)$ in Theorem \ref{convergenceratethm_exactPL} is a smallness
condition, but it is not vacuous. It enters the analysis only through the
requirement
$\eta + 2L^3D_\beta h_1/\big((1-\tfrac{Lh_2}{2})C_\beta\big) < 1$, with $\eta = 1-h_1\mu+\tfrac34\Theta L^2h_1^2+O(h_1^3)$.
The drift of the critical set must not destroy the contraction of the ascent
step \eqref{RGA-MGD1}. Some condition of this shape is unavoidable, since if $\mathcal{S}^*(y)$ recedes
faster than the iteration contracts there is no fixed point to converge to.
Moreover $D_\beta$ bounds the rate at which $\mathcal{S}^*(y)$ varies with $y$ and not the
distance it travels, so the hypothesis does not require $\mathcal{S}^*(y)$ to be invariant or
nearly so. We do not claim the constant $2L^3/C_\beta$ is sharp.
\end{rem}

\begin{theorem}\label{convergenceratethm_betaPL}
 Under \textbf{C1-C4}, for $\beta \in (1,2)$ suppose the iteration \eqref{RGA-MGD1} - \eqref{RGA-MGD2} is initialized with $ (x_0, y_0) \in \text{int}(S_1 \times S_2)$ and for some large enough $K_0 \gg 1$ we have that $$  \mathrm{dist}(x_{K_0}, \mathcal{S}^*(y_{K_0})) \le \min \bigg \{  \,\frac{\delta}{2},  \frac{C^{1/\beta}_{\beta}}{L} -  h_1^{\beta - 1} \left(\frac{\beta}{\beta-1} \right) C_{\beta}^{-1} L^{2\beta - 2}\, \mathrm{diam}^{\beta - 1}(S_2) \bigg\} \, , $$ with $h_1>0$ sufficiently small such that
    $ h_1^{\beta - 1}\left(\frac{\beta}{\beta-1} \right) C_{\beta}^{-1} L^{2\beta - 2} \, \mathrm{diam}^{\beta - 1}(S_2) < \frac{C^{1/\beta}_{\beta}}{L} $ and $ \frac{1+\eta}{2} \le 1-\frac{\mu h_1 }{3}$ where $\eta  = 1-h_1\mu+\frac{3}{4}\Theta L^2 h_1^2  + \mathcal{O}(h_1^3)$. Suppose for all $k$ we have the parameter $J:= J(k) = k^{\frac{2\alpha}{(\beta -1)^2}} $ for any constant $\alpha > {1-(\beta -1 )^2\theta} $ , $\theta = \frac{\beta (\beta -1)}{2} \in (0,1) $ and $J(K_0)$ satisfies the lower bound 
    $$ \bigg(   1+ J(K_0) \,C'h_2  \bigg)^{\frac{\beta-1}{2}} > \frac{2\Bigg(C_{L,\beta} \, \delta^{\frac{\beta (\beta -1) }{2}} +   C_{L,\beta} \bigg(  L^{\beta -1} \, \left(\frac{\beta}{\beta-1} \right) C_{\beta}^{-1} \, \bigg(\mathrm{diam}(S_2) \bigg)^{\beta - 1} \bigg)^{\frac{\beta (\beta -1) }{2}} \Bigg)}{\delta}  \, $$
     where $0 < h_2 < 1/L$, $C_{L,\beta} :=   \left(\frac{\beta}{\beta-1} \right) \frac{L ^{\frac{\beta (\beta -1)}{2}}}{ C_{\beta}^{\frac{\beta-1}{2} + \frac{1}{\beta }} } $, $ C' := (1 - \tfrac{Lh_2}{2}) C^{2/\beta}_{\beta} \le C^{2/\beta}_{\beta} $. Also suppose that $\mathrm{dist}_H(\mathcal{S}^*(y_{k-1}), \mathcal{S}^*(y_k)) $, for all $k \ge 0$, satisfies a Lipschitz estimate for any $\beta \in (1,2)$ instead of the H\"{o}lder estimate from Lemma \ref{RGA-MGD1-lem0} as follows:
    \begin{align*}
        \mathrm{dist}_H(\mathcal{S}^*(y_{k-1}), \mathcal{S}^*(y_k)) \le  D_{\beta} \, \mathrm{dist}(y_{k-1},y_k).
    \end{align*}
    Further, let the iterates  $\{(x_k, y_k)\}_k$ from \eqref{RGA-MGD1} - \eqref{RGA-MGD2} satisfy the following growth condition uniformly for all $k \ge K_0$ :
     $$\mathrm{dist}_H(\mathcal{S}^*(y_{k+1}), \mathcal{S}^*(y_k))  = \mathcal{O}(\mathrm{dist}(x_{k}, \mathcal{S}^*(y_{k}))) \, .$$ Then the iteration $\{(x_k, y_k)\}_k$ from \eqref{RGA-MGD1} - \eqref{RGA-MGD2} converges to a basin saddle point $ (x^*,y^*) \in \text{int}(S_1 \times S_2)$ of $f(\cdot, \cdot)$ at a rate $\mathcal{O}( k^{- \frac{\alpha}{1-(\beta -1 )^2\theta}})$.  
\end{theorem}
The proof of Theorem \ref{convergenceratethm_betaPL} is in Appendix \ref{sectionmaintheoremsappendix}.
\begin{rem}
        {We note that our growth rate assumption $\mathrm{dist}_H(\mathcal{S}^*(y_{k-1}), \mathcal{S}^*(y_k)) \le C_0 \mathrm{dist}(x_{k}, \mathcal{S}^*(y_{k})) $ for any large enough $k$, does not contradict the convergence rate. From Theorem \ref{convergenceratethm_betaPL} proof we have $ \mathrm{dist}^{\beta -1}(x_{k}, \mathcal{S}^*(y_{k}))  \sim k^{- \frac{\alpha}{1-(\beta -1 )^2\theta}} $ and $ \mathrm{dist}(y_{k-1}, y_k)  \sim k^{- \frac{\alpha}{1-(\beta -1 )^2\theta}} $. From the Lipschitz estimate on $\mathrm{dist}(\mathcal{S}^*(y_{k-1}), \mathcal{S}^*(y_k))$ we have that 
    \begin{align}
        \mathrm{dist}_H(\mathcal{S}^*(y_{k-1}), \mathcal{S}^*(y_k))  = \mathcal{O}(\mathrm{dist}(y_{k-1}, y_k)  ) = \mathcal{O} (k^{- \frac{\alpha}{1-(\beta -1 )^2\theta}}) \label{holderestimaterate1y}
    \end{align}
  and from our growth condition we obtained 
  \begin{align}
      \mathrm{dist}_H(\mathcal{S}^*(y_{k-1}), \mathcal{S}^*(y_k))  = \mathcal{O}(\mathrm{dist}(x_{k}, \mathcal{S}^*(y_{k}))) = \mathcal{O} (k^{- \frac{\alpha}{(1-(\beta -1 )^2\theta)(\beta -1 )}}) \, .  \label{holderestimaterate1z}
  \end{align}
    The two Big-O terms do not contradict one another and in fact the second analytically obtained estimate \eqref{holderestimaterate1z} implies the first Lipschitz type estimate \eqref{holderestimaterate1y} holds for $\beta \in (1,2)$. Furthermore, the growth condition is not vacuous since functions $f(\cdot, y)$ with invariant critical sets in $S_1$ under variations of $y \in S_2$ will trivially satisfy the growth condition. 
    }
\end{rem}

\begin{rem}
    We remind the reader that conditions \textbf{C1-C4} never assume existence of a basin saddle point of the function $f(\cdot, \cdot)$ on the product manifold $S_1 \times S_2$. When $S_2$ is a compact convex subset of a topological vector space, $f(\cdot,y)$ is lower semi-continuous and convex or quasi-convex on $S_1$ for all $y \in S_2$, $f(x,\cdot)$ is upper semi-continuous and concave or quasi-concave on $S_2$ for all $x \in S_1$ then by Sion minimax theorem \cite{sion1958general} $$ \min_{x \in S_1} \max_{y \in S_2} f(x,y) = \max_{y \in S_2} \min_{x \in S_1} f(x,y) $$ and the existence of a saddle/ min-max point for $f$ is guaranteed on $S_1 \times S_2$. This version of Sion minimax theorem can be extended beyond vector spaces to smooth manifolds \cite{zhang2023sion}. While conditions \textbf{C1-C4} do not assume the quasi-convex quasi-concave structure, these conditions are sufficient to conclude that the convergence limit in Theorem \ref{convergenceratethm_exactPL} or Theorem \ref{convergenceratethm_betaPL} is a basin saddle point. The analytical machinery for providing saddle point existence resides in the theorem proofs where we show Cauchy convergence of iterate sequence $\{(x_k, y_k)\}_k$ along with convergence of the Lyapunov function $V_{\delta}(x_k,y_k)$. Then condition \textbf{C3} of bounded drift of critical sets along with Lemma \ref{RGA-MGD1-lem1} and the basin saddle point definition \ref{basinsaddledef} proves the existence. In principle, the proof structure of Theorem \ref{convergenceratethm_exactPL} or Theorem \ref{convergenceratethm_betaPL} is similar to the Banach fixed point theorem proof under contractive maps \cite{banach1922operations} and therefore existence of a basin saddle point is never a required assumption. Our results however do not imply the uniqueness of basin saddle point inside $S_1 \times S_2$ as uniqueness requires stronger global assumptions.
\end{rem}

\section{Special case: DRO on the product of Euclidean and BW manifold}\label{sectionDROBW}
Consider the following constrained DRO problem:
\begin{align}
    \min_{\w \in \mathbb{R}^d}  \max_{ \substack{\Lambda \in \mathbb{R}^{n \times n} \hspace{0.1cm}; \hspace{0.1cm}\Lambda = \Lambda^T \hspace{0.1cm};  \hspace{0.1cm} \Lambda \succ \mathbf{0} \\ \muv \in \mathbb{R}^n \\ d(\Lambda, \Lambda^*) < \xi  \hspace{0.1cm}; \hspace{0.1cm} \muv \in \mathcal{B}_{\xi}(\muv^*) }  }   \mathbb{E}_{\z \sim \probP(\muv,\Lambda)} [\ell(\w;\z) ]   \label{DRO1}
\end{align}
where $\probP(\muv, \Lambda)$ is a Gaussian distribution with density $$ p(\z) = \frac{1}{ \sqrt{\det( 2 \pi\Lambda)}}\exp\bigg(- \frac{1}{2}\langle (\z -\muv),\Lambda^{-1}(\z - \muv)\rangle\bigg) \quad , \quad \z \in \mathbb{R}^n $$ and $d(\cdot , \cdot)$ is some metric on the space of symmetric positive definite matrices left unspecified at this stage. Also, $\probP(\muv^*,\Lambda^*) $ is some base measure that is known apriori.

Next, for any $(\muv, \Lambda) \in  \mathbb{R}^{n} \times \mathbb{R}^{n \times n}$ let $\V = (\muv, \Lambda)$ where $\V \in \mathbb{R}^{n} \times \mathbb{R}^{n \times n}$. Since $ \mathbb{R}^{n} \times \mathbb{R}^{n \times n} \cong \mathbb{R}^{n^2+n} $ we can vectorize $\V$ to some $\v \in \mathbb{R}^{n^2+n}$ such that $\v = [\muv^T, [\text{vec}(\Lambda)]^T]^T$ and hence the distribution $\probP(\muv, \Lambda)$ can be embedded in a finite dimensional vector space $ \mathbb{R}^{n^2+n}$. In particular, the space of such distributions $ \mathcal{P}$ with an appropriate metric is isomorphic to the product Riemannian manifold $ \mathbb{R}^n\times  BW(\texttt{SPD}_n)$ where the first manifold is the standard Euclidean space for $\muv$ with the Euclidean metric $l_2$ and the second manifold $BW(\texttt{SPD}_n)$ is the symmetric positive definite cone $\texttt{SPD}_n$ equipped with the Bures-Wasserstein metric $BW$. This metric is a natural choice for the $ \texttt{SPD}_n $ cone since the Wasserstein flow when projected orthogonally on the tangent space of $ \texttt{SPD}_n $ gives the finite dimensional Bures-Wasserstein flow. Then discretizing the new flow via an implicit method, i.e. the JKO step gives the iterative method in the Bures-Wasserstein metric. The Bures-Wasserstein flow exhibits some nice convergence properties under certain mild regularity assumptions on the energy functional such as geodesic convexity \cite{ambrosio2005gradient, chewi2025statistical}. We therefore use the $BW$ metric for $d(\cdot , \cdot)$ in our DRO formulation \eqref{DRO1}.

Solving the DRO problem \eqref{DRO1} is equivalent to solving a constrained nonconvex-nonconcave minimax optimization problem over the product space $ \mathbb{R}^d \times  \mathbb{R}^n \times BW(\texttt{SPD}_n)$, where $\mathbb{R}^d$ is a finite dimensional Euclidean space and $ \mathbb{R}^n \times BW(\texttt{SPD}_n)$ is a metric space (product of Euclidean and a Riemannian length space). This problem in its generality is NP hard with no known optimization solvers. Therefore we first need to simplify and ease the problem structure/assumptions in order to make any progress towards building a provable numerical method for solving the DRO problem \eqref{DRO1}. Even solving the inner max sub-problem in \eqref{DRO1} via an iterative numerical method (such as the JKO method) entails significant effort due to the Riemannian length space structure of the space of probability measures and {without additional assumptions such as geodesic convexity of the functional $\mathcal{F}(\probP) := \mathcal{F}(\muv, \Lambda) = -\mathbb{E}_{\z \sim \probP(\muv,\Lambda)} [\ell(\w;\z) ]  $, numerical methods such as the JKO scheme (defined in the next section) have no known provable convergence guarantees to the best of our knowledge \cite{hraivoronska2026convergence}.} We therefore first reformulate the inner max sub-problem in \eqref{DRO1} so that it is at least becomes solvable by the implicit JKO scheme. Thereafter we adapt the JKO scheme (from an implicit to an easier explicit method) for solving the reformulated version of the inner max sub-problem in \eqref{DRO1}.

\subsection{The JKO step for distribution update}
For an energy functional $\mathcal{F}$ defined on the metric space $(\mathcal{X},g)$ of absolutely continuous probability measures (with respect to Lebesgue measure) with bounded second moments, the JKO step \cite{jordan1998variational} with $h>0$ satisfies:
\begin{align}
    \probP_{k+1} =  \arg\min_{ \probP \in \mathcal{X}} \bigg\{  \mathcal{F}(\probP) + \frac{1}{2h}d_g(\probP,\probP_k)^2 \bigg\} \label{JKO1} 
\end{align}
The above discretization is the implicit Euler step of the Wasserstein gradient flow on $\mathcal{X}$. The metric $g$ is often taken to be the $W_2$ Wasserstein metric and the existence of a unique minimizer from \eqref{JKO1} is guaranteed under certain mild conditions on $ \mathcal{F}$ such as convexity and bounded gradient growth (see \cite{ambrosio2005gradient, jordan1998variational} ). Further, as $h \to 0$ the above scheme converges weakly to the Wasserstein gradient flow for $\mathcal{F}$ on $\mathcal{X}$ (see \cite{jordan1998variational}).

For the multivariate Gaussian case, we can restrict the measures to the finite\footnote{{The dimension of $\mathcal{P}$ is its dimension as a smooth manifold.}} dimensional space $\mathcal{P} \subset \mathcal{X}$ such that $ (\mathcal{P},g) \cong  (\mathbb{R}^n, l_2) \times (BW(\texttt{SPD}_n  ),BW)$ where $g$ is a push forward of the metric $ l_2 \oplus BW$ under some local diffeomorphism. We then identify $\probP \in \mathcal{P}$ with\footnote{For sake of brevity, from here onward, we will refer to the $BW(\texttt{SPD}_n)$ manifold  simply as the $BW$ manifold.} $(\muv, \Lambda) \in \mathbb{R}^n \times BW$. For $\mathcal{F}(\probP) := \mathcal{F}(\muv, \Lambda) = -\mathbb{E}_{\z \sim \probP(\muv,\Lambda)} [\ell(\w;\z) ] $ we redefine the JKO step \eqref{JKO1} as follows:
\begin{align}
    (\muv_{k+1},\Lambda_{k+1}) \in  \arg\min_{\substack{ \probP(\muv, \Lambda ) \in \mathcal{P} \\ d_{BW}(\Lambda ,\Lambda^*)  \leq \xi  \hspace{0.1cm}; \hspace{0.1cm}  \norm{\muv - \muv^*} \leq \xi  }} \bigg\{  -\mathbb{E}_{\z \sim \probP(\muv,\Lambda)} [\ell(\w;\z) ] + \frac{1}{2h} \bigg(  \norm{\muv - \muv_k}^2 + d^2_{BW}(\Lambda ,\Lambda_k)\bigg) \bigg\} \label{JKO2} 
\end{align}
Note that the max step from \eqref{DRO1} is equivalent to minimizing $ \mathcal{F}(\muv, \Lambda)$ over the finite dimensional space $\mathcal{P} \subset \mathcal{X}$. {Iterating the JKO step \eqref{JKO2} under some mild regularity conditions on the functional will then yield convergence to the solution of inner max sub-problem in \eqref{DRO1}.}
Since any arbitrary JKO iteration from \eqref{JKO2} involves solving a constrained and possibly nonconvex problem (see Appendix \ref{appendixcountereg1}) with respect to the probability measures over the BW manifold\footnote{{Note that $\mathcal{F}(\probP) $ can be a geodesic nonconvex functional with respect to the metric on probability measures ( assuming the space of measures is CAT(0) or is at least a geodesic metric space) thereby posing significant challenges for faster convergence of the JKO scheme \eqref{JKO2}.}}, even for sufficiently small values of $h$, it may not be possible for the JKO iteration from \eqref{JKO2} to converge in polynomial time. Therefore, instead of solving the much harder problem in \eqref{JKO2} iteratively via the implicit JKO scheme, we solve the following unconstrained minimization problem on the product manifold $\mathbb{R}^n \times BW $ where the functional $\wt{\mathcal{F}}(\probP)$ defined below is strong geodesically convex under some mild assumptions (as shown later) :
\begin{align}
    \min_{\substack{ \probP(\muv, \Lambda ) \in \mathcal{P} \\ }} \wt{\mathcal{F}}(\probP) :=  \bigg\{  -\mathbb{E}_{\z \sim \probP(\muv,\Lambda)} [\ell(\w;\z) ] + \frac{1}{\epsilon} \bigg(  \norm{\muv - \muv^*}^2 + d^2_{BW}(\Lambda ,\Lambda^*)\bigg) \bigg\} \label{JKO3} 
\end{align}
Under mild assumptions on the loss $\ell$ and sufficiently small $\xi$ it can be shown that $ \wt{\mathcal{F}} $ in \eqref{JKO3} will be minimized in the interior of the ambiguity/ constraint set from the inner max problem in \eqref{DRO1}. {Proposition~\ref{prop:reduction} (Appendix \ref{sectionDROBWappendix}) makes this precise: choosing $\epsilon=2\xi/L_G$ where $L_G$ is a Lipschitz constant, the solution of the penalized problem \eqref{JKO3} is feasible for the inner max problem in \eqref{DRO1}, this solution lies within $(1+\sqrt2)\xi$ of any maximizer of the inner problem in \eqref{DRO1}, and attains its optimal value up to $L_G\xi$. The JKO steps \eqref{JKO1}--\eqref{JKO2} play no role in this comparison and the step size $h$ does not enter; they are recorded only to motivate why the implicit scheme is intractable here.} Thus under mild assumptions and a certain choice of parameters, solving \eqref{JKO3} approximately solves the max step of the constrained DRO problem \eqref{DRO1}. From here onward we will therefore work with the following unconstrained and a more tractable penalized DRO problem:
\begin{align}
    \min_{\w \in \mathbb{R}^d}  \max_{\substack{ \probP(\muv, \Lambda ) \in \mathcal{P} \\ }} \bigg\{  \mathbb{E}_{\z \sim \probP(\muv,\Lambda)} [\ell(\w;\z) ] -\frac{1}{\epsilon} \bigg(  \norm{\muv - \muv^*}^2 + d^2_{BW}(\Lambda ,\Lambda^*)\bigg) \bigg\}  \label{DRO1x}
\end{align}
where $\epsilon>0$ will be specified later.
We then consider the explicit Riemannian gradient ascent (RGA) to solve \eqref{JKO3} iteratively instead of solving the significantly harder problem from the implicit JKO step \eqref{JKO2}. 
{The RGA iteration is as follows:}
\[
 (\muv_{k+1},\Lambda_{k+1})
    =
    \exp_{(\muv_k,\Lambda_k)}
    \!\left(
    h_1 \nabla_{(\muv, \Lambda)} \mathcal{L}(\muv_k,\Lambda_k;\w_k)
    \right)
\]
where
\[ \mathcal{L}(\muv,\Lambda;\w) :=  \mathbb{E}_{\z \sim \probP(\muv,\Lambda)} [\ell(\w;\z) ] - \frac{1}{\epsilon} \bigg(  \norm{\muv - \muv^*}^2 + d^2_{BW}(\Lambda ,\Lambda^*)\bigg) \, .\] We will later refer to this functional $\mathcal{L}$ as the Lagrangian in  definition \ref{Lagrangiandefn}. The explicit RGA iteration exhibits linear convergence under assumptions such as geodesic strong concavity of the functional $\mathcal{L}$ (see \cite{boumal2023introduction}). To derive convergence rates for the RGA iteration we first need to establish certain properties of the functional's hessian with respect to probability measure defined in \eqref{JKO3}.

We will use the following assumptions throughout the remaining paper.

\subsection{Assumptions}
\begin{enumerate}
\item[\textbf{A1.}]\textbf{(Smoothness and coercivity)} The function $\ell$ is jointly $\mathcal{C}^2$ smooth on $ \mathbb{R}^d \times \mathbb{R}^n $, is non-negative and $\ell(\cdot;\z)$ is coercive for all $\z$, i.e. $ \lim_{\norm{\w} \to \infty} \ell(\w;\z) = + \infty$ for all $\z$.
 \item[\textbf{A2.}]\textbf{(Bounded moments)} Let $\mathcal{P} \ni \probP(\muv,\Lambda)$ be the space of Gaussian measures with $\muv \in \mathbb{R}^n $, $ \Lambda \in \texttt{SPD}_n$. The $p$-{th} moments of $\ell(\w;\cdot)$, $\norm{\nabla_{\w} \ell(\w; \cdot)}$ and $\norm{\nabla^2_{\w} \ell(\w; \cdot)}_F$ , with respect to some reference Gaussian probability measure $\probP(\muv^*,\Lambda^*) \in \mathcal{P} $, are bounded for any $\w \in \mathbb{R}^d$ and for all $p \leq P < \infty$ for some $P \geq 2$, i.e., 
 $$ \big(\mathbb{E}_{\z \sim \probP(\muv^*,\Lambda^*) } [\abs{\ell(\w;\z)}^p ]\big)^{1/p} < \infty    $$
 $$ \big(\mathbb{E}_{\z \sim \probP(\muv^*,\Lambda^*) } [ \norm{\nabla_{\w}\ell(\w;\z)}^p] \big)^{1/p} < \infty $$
  $$ \big(\mathbb{E}_{\z \sim \probP(\muv^*,\Lambda^*) } [\norm{\nabla^2_{\w}\ell(\w;\z)}_F^p] \big)^{1/p} < \infty \, . $$
 
\end{enumerate}

We now derive Hessian estimates crucial for generating the strong geodesic concavity of $\mathcal{L}(\muv,\Lambda;\w)$ locally on the product Riemannian manifold $ \mathbb{R}^n \times  BW $ and for any fixed $\w$. Before that we introduce the necessary definitions and some crucial properties in relation to the BW geometry.

\subsection{Differential geometry preliminaries}
\subsubsection{Results from the BW geometry}

\begin{defn}
    For the Riemannian manifold $\texttt{SPD}_n$ equipped with the Bures-Wasserstein metric $ BW(\cdot,\cdot)$ the following are explicitly provided:
    \begin{itemize}
    \item Tangent space: $T_{\Lambda} \texttt{SPD}_n \cong \mathbb{S}^n $, then for any $V,W \in T_{\Lambda} \cong \mathbb{S}^n $ where $\mathbb{S}^n $ is the space of symmetric $n \times n$ matrices, the Bures Wasserstein metric is defined as: 
    {$ g_{\Lambda}(V,W) := \langle  V,W\rangle_{g} = tr(L_{\Lambda}(V) \Lambda L_{\Lambda}(W))  $ where $L_{\Lambda}(U) $, for any $U \in T_{\Lambda} \texttt{SPD}_n$, is the unique, symmetric\footnote{Uniqueness follows from the fact that $ \Lambda \succ \mathbf{0}$ and the solution is symmetric since if $X$ satisfies the given Sylvester equation then $X^T$ must also be a solution.} solution of the Sylvester equation 
    \[  \Lambda X + X \Lambda  = U .\] }
     { Equivalently the metric can be defined at $\Lambda$ by considering $A,B $ to be the solutions of Sylvester equations, i.e. $A =L_{\Lambda}(U) $, $B= L_{\Lambda}(V) $ where $U,V \in T_{\Lambda} \texttt{SPD}_n$ and hence $$ g_{\Lambda}(U,V) = g_{\Lambda}(\Lambda A + A \Lambda,\Lambda B + B \Lambda)  := tr(A \Lambda B)  .$$} 
    \item Exponential map at $\Lambda $ for any $U \in T_{\Lambda} \texttt{SPD}_n$ where $L_{\Lambda}(U) = S $ : $$\exp_{\Lambda}(U) = (I +S) \Lambda (I +S) .$$ 
    \item Optimal transport map joining $ \Lambda_1, \Lambda_2$ : $ J_{\Lambda_1, \Lambda_2} := \Lambda_1^{-1/2} ( \Lambda_1^{1/2} \Lambda_2  \Lambda_1^{1/2})^{1/2} \Lambda_1^{-1/2}  $
    \item Geodesic joining $ \Lambda_1, \Lambda_2$ : $ \gamma(t) := ((1-t)I + tJ_{\Lambda_1, \Lambda_2})\Lambda_1 ((1-t)I + tJ_{\Lambda_1, \Lambda_2})  $ with $t \in [0,1]$
        \item Geodesic distance: $ d_{BW}(\Lambda_1,\Lambda_2) := \sqrt{tr\bigg(\Lambda_1 + \Lambda_2 - 2( \Lambda_1^{1/2} \Lambda_2  \Lambda_1^{1/2})^{1/2}  \bigg)} $  
       % \item Riemannian gradient w.r.t. metric $BW(\cdot,\cdot)$ : $ \nabla_{BW} f(\Lambda) := \text{proj}_{T_{\Lambda}( \texttt{SPD}_n) } (Df(\Lambda))  =  2 Df (\Lambda)$ where $Df (\Lambda) $ is the Euclidean gradient ({need to verify independently!!})
    \end{itemize}
\end{defn}

\begin{lem}
    Let $(\cm_1 , g_1)$ and $(\cm_2, g_2)$ be Riemannian manifolds. The geodesic distance between two points on the product Riemannian manifold $(\cm_1 \times \cm_2, g_1\oplus g_2) $ satisfies:
    $$ d_g^2((x_1,x_2), (y_1,y_2)) = d^2_{g_1}(x_1,y_1)+ d^2_{g_2}(x_2,y_2). $$
\end{lem}

\subsubsection{Convex subsets of a Riemannian manifold}\label{convdef}

The following definition is central to the our analysis of the rate of convergence gradient descent on Riemannian manifolds. 

\begin{defn}
    Let $(\cm, g)$ be a Riemannian manifold. A subset $S\subset \cm$ is said to be \emph{geodesically convex} if for each $p,q\in S$, there is a unique minimizing geodesic segment from $p$ to $q$ in $\cm$, and the image of this geodesic segment lies entirely in $S$. 
\end{defn}
Boumal (\cite[Definition 11.17]{boumal2023introduction}) calls such a subset geodesically strongly convex, we choose the name geodesically convex following Lee (\cite[pg. 166]{lee2018}).

\begin{defn}\label{normaldef}
    Let $(\cm, g)$ be a Riemannian manifold and let $p\in \cm$ be a point. A \emph{normal neighborhood of $p$} in $\cm$ is an open subset $S$ that is the diffeomorphic image under the exponential map of an open subset $V\subset T_p\cm$.
\end{defn}
While our definition \ref{normaldef} is adopted from Lee (\cite[pg. 133]{lee2018}), it differs from his definition in that $V$ is not required to star-shaped with respect to $0\in T_p \cm$.

\subsubsection{Curvature estimates for the BW manifold}

To apply the analysis of Section \ref{jointrategeneralsection} to the Riemannian manifold $\R^n\times BW $, we need estimates on the sectional curvatures of BW metric on $\texttt{SPD}_n$. We recall the following result. 

\begin{proposition}(\cite[Propositions 1,2]{massart2019curvature})\label{secbound} Let $\Lambda \in \texttt{SPD}_n$ and let $\lambda_1 \geq \dots \geq \lambda_{n-1} \geq \lambda_n >0$ be the eigenvalues of $\Lambda$. Then the sectional curvature $\mathbf{K}_{BW}(\Lambda)$ of the BW metric at the point $\Lambda$ satisfies 
\[
0 \leq \mathbf{K}_{BW}(\Lambda) \leq \frac{3}{(\lambda_n + \lambda_{n-1})^2}.
\]
The bounds are sharp.    
\end{proposition}

By using Proposition \ref{secbound}, the following lemma gives bounds on the sectional curvature of BW manifold on geodesic balls centered at a $\Lambda \in \texttt{SPD}_n$.

\begin{lem}\label{secboundlemma}
    Let $\Lambda \in \texttt{SPD}_n$ and let $\lambda_1 \geq \lambda_2 \geq \cdots \geq \lambda_n >0$ be the eigenvalues of $\Lambda$. Given $\gamma\in (0,1)$, let $r = (1-\gamma)\sqrt{\lambda_n}$ and consider the ball 
    \[
    \bar{\mathcal{B}}_r(0) = \{ V\in T_{\Lambda}\texttt{SPD}_n | g_{\Lambda}(V,V) \leq r^2\}, \text{ and let } \bar{\mathcal{B}}_r(\Lambda) = \exp_{\Lambda}(\bar{\mathcal{B}}_r(0)).
    \] 
    Then the sectional curvature $\mathbf{K}_{BW}$ of the BW metric  satisfies 
    \[
        0\leq \mathbf{K}_{BW}(\Lambda') \leq \frac{3}{4\gamma^4\lambda_n^2}
    \]
    for all $\Lambda' \in \bar{\mathcal{B}}_r(\Lambda)$.
\end{lem}

\begin{proof}
    Given $V\in \bar{\mathcal{B}}_r(0)$, let $\sigma_1, \dots, \sigma_n$ be eigenvalues of the matrix $L_{\Lambda}(V)$ and let $\{z_1,\dots, z_n\}$ be an orthonormal basis for $\Rn^n$ corresponding to these eigenvalues. Then 
    \begin{align*}
      (1-\gamma)^2\lambda_n = r^2 &\geq \text{tr}(L_{\Lambda}(V) \Lambda L_{\Lambda}(V)) \\
           &= \sum_{i=1}^n z_i^T L_{\Lambda}(V) \Lambda L_{\Lambda}(V)z_i \\
           &=\sum_{i=1}^n (L_{\Lambda}(V)z_i)^T\Lambda (L_{\Lambda}(V)z_i) \\
           &= \sum_{i=1}^n \sigma_i^2 z_i^T\Lambda z_i \geq \sum_{i=1}^n \sigma_i^2 \lambda_n.
    \end{align*}
    Thus, we see that if $V\in \bar{\mathcal{B}}_r(0)$ then the eigenvalues $\sigma_i$ of $L_{\Lambda}(V)$ satisfy $\abs{\sigma_i}\leq 1-\gamma$. Consequently, the eigenvalues of $I+L_{\Lambda}(V)$ are in the interval $[\gamma, 2-\gamma]$. Thus, the smallest eigenvalue $\lambda_{\min}$ of $\exp_{\Lambda}(V)$ satisfies
    \begin{align*}
    \lambda_{\min} &= \inf_{z\neq 0}\frac{\langle (I+L_{\Lambda}(V))\Lambda (I+L_{\Lambda}(V))z,z\rangle}{\langle z,z\rangle} \\
    &= \inf_{z\neq 0}\frac{\langle \Lambda (I+L_{\Lambda}(V))z, (I+L_{\Lambda}(V))z\rangle}{\langle z,z\rangle} \\
    & \ge \lambda_n \inf_{z\neq 0} \frac{\langle (I+L_{\Lambda}(V))z, (I+L_{\Lambda}(V))z\rangle}{\langle z,z\rangle} \\
    &\geq \gamma^2\lambda_n.
    \end{align*}
    By Proposition \ref{secbound}, the sectional curvature of BW metric at $\Lambda' =\exp_{\Lambda}(V)$ has upper bound
    \[
    \mathbf{K}_{BW}(\Lambda') \leq \frac{3}{4\lambda_{\min}^2} \leq \frac{3}{4\gamma^4\lambda_n^2}.
    \]
    The lower bound for the sectional curvature is always $0$ and is attained, see \cite{massart2019curvature} for a proof.
\end{proof}
We defer the reader to Appendix section \ref{riemtools} for additional results from Riemannian geometry crucial for deriving explicit estimates of the Hessian in BW geometry.

\subsection{Hessian Calculation}\label{sectionhessiancalc}

\begin{lem}
    Let $(\cm,g)$ be a Riemannian manifold and let $f:\cm\to \R$ be a $\cc^2$-function. Then for all $p\in \cm$ and $X\in T_p\cm$ we have
    \[
    \text{Hess}_p(f)(X,X) = \frac{d^2}{dt^2}\Big|_{t=0} f(c(t))
    \]
    where $c:(-\epsilon,\epsilon)\to \cm$ is a geodesic with $c(0)=p$ and $c'(0)=X$. 
\end{lem}

\begin{proof}
    Given any vectors $X,Y\in T_p\cm$, the Hessian of $f$ is defined as 
    \[
    \text{Hess}_p(f)(X,Y) = XYf(p) - df_p(\nabla_XY),
    \]
    where $\nabla$ denotes the Levi-Civita connection for $g$. If $X=Y=c'(0)$, the second term vanishes since $\nabla_{c'}c'=0$ at all points of the geodesic $c$. Letting $\wt{X}$ be the vector field along the curve $c$ given by $\wt{X}_{c(t)}=dc_t\left(\frac{d}{dt}\right)$, we have
    \[
    X(Xf)(p) = \frac{d}{dt}\Big|_{t=0} \left( \wt{X}f(c(t))\right)= \frac{d}{dt}\Big|_{t=0}\frac{d}{dt}f(c(t)) = \frac{d^2}{dt^2}\Big|_{t=0}f(c(t)). 
    \]
\end{proof}

Given $\muv\in \R^n$, let $T_{\muv}:\R^n\to \R^n$ denote the translation $T_{\muv}(\x)=\x+\muv$. Then the geodesic $c:[0,1]\to \mathcal{P}$ joining $\rho_{\Lambda_0, \muv_0}$ and $\rho_{\Lambda_1,\muv_1}$ is given (see \cite[Example 1.7]{mccann}) by 
\[
c(t) = T_{(1-t)\muv_0+t\muv_1,\#}[(1-t)I+tS]_{\#}\rho_{\Lambda_0,\muv_0}, \text{ where } S= \Lambda_1^{1/2}(\Lambda_1^{1/2}\Lambda_0\Lambda_1^{1/2})^{-1/2}\Lambda_1^{1/2}.
\]
One can readily check that each $c(t)$ is a Gaussian measure, specifically $c(t) = \rho_{\Lambda_t, \muv_t}$ where 
\[
\Lambda_t = [(1-t)I+tS]\Lambda_0[(1-t)I+tS] \quad \text{ and }\quad \muv_t = (1-t)\muv_0+t\muv_1.
\]

\subsubsection{Hessian estimate of the statistical loss}

\begin{theorem}\label{losshessianestimatethm}
   Under \textbf{A1-A2}, the Hessian $\text{Hess}_{(\muv,\Lambda)}f$ for the function
\[
f(\w; \Lambda, \muv) = \int_{\z \in \mathbb{R}^n} \ell(\w;\z)   \exp\bigg(-\frac{1}{2}\langle(\z -\muv),\Lambda^{-1}(\z - \muv)\rangle\bigg) \frac{d\z}{\sqrt{\det(2\pi \Lambda)}} 
\]
evaluated with respect to $ (\muv, \Lambda)$ for any fixed $\w$ satisfies the following uniform estimate:
\begin{align}
  \sup_{\substack{\nuv \in \mathbb{R}^n, \norm{S}_F < 1 \\ {S \in \mathbb{S}^n} }} \frac{ \abs{\text{Hess}_{(\muv,\Lambda)}f((\nuv,  S\Lambda +\Lambda S),(\nuv,S\Lambda +\Lambda S))}}{ \norm{\nuv}^2 + tr(S \Lambda S)}  &\leq    C_{\Lambda, q, n} (C''_{p,\Lambda^*,n} \times \big( \mathbb{E}_{\z \sim \probP(\muv^*,\Lambda^*)}\left[\abs{\ell(\w;\z)}^{pq}\right] \big)^{1/q}  )^{1/p} \nonumber \\ & < \infty  
\end{align}
on the {product of closed balls of radius $r = (1- \gamma) \norm{(\Lambda^*)^{-1}}_2^{-1/2}$, with $\gamma \in (0,1)$ ,} centered at $(\muv^*,\Lambda^*)$ for any $ {\w \in \mathbb{R}^d}$, for some universal constants $C''_{p,\Lambda^*,n}, \, C_{\Lambda, q, n} \, ,\, 1/p + 1/q = 1 , \, p \ge 1$ , where $$ C_{\Lambda, q, n} \sim_{q} \bigg(\frac{\Gamma(\frac{n+4q}{2}) }{\Gamma(\frac{n}{2}) }\bigg)^\frac{1}{q} \norm{(\Lambda^*)^{1/2}}^4_F \bigg( \frac{\sqrt{n} \norm{(\Lambda^*)^{-1}}_F}{(1-r \norm{(\Lambda^*)^{-1/2}}_F)}\bigg)^{6} $$
and 
\begin{align*}
C''_{p,\Lambda^*,n}  & \lesssim_p n \max\{ 1, r^{2}\} \norm{(\Lambda^*)^{-1/2}}_F^{4} \norm{(\Lambda^*)^{1/2}}_F^{4} \bigg(\frac{\Gamma(\frac{n+4p}{2}) }{\Gamma(\frac{n}{2}) }\bigg)^{1/p}
\end{align*}
for $ \norm{\Lambda^*}_F \ge 1 \, , \, \norm{(\Lambda^*)^{-1/2} }_F \ge 1 $. 
\end{theorem}
The proof of Theorem \ref{losshessianestimatethm} is in Appendix \ref{sectionDROBWappendix}.
\begin{rem}
   Observe that the uniform estimate from Theorem \ref{losshessianestimatethm} is evaluated on the tangent space at the point $(\muv, \Lambda)$ for any tangent vectors (or equivalently direction vectors) $(\nuv, S \Lambda + \Lambda S) $ satisfying $\nuv \in \mathbb{R}^n, \norm{S}_F < 1 $. The estimate is scale invariant on the tangent space $ T_{(\muv, \Lambda)} (\mathbb{R}^n \times BW)$, i.e. by scaling $(\nuv , S)$ as $(c \nuv , cS)$ for any $c \neq 0$ the estimate remains unchanged due to the bilinear map $\text{Hess}_{(\muv,\Lambda)}f (\cdot , \cdot) $ and the cone $T_{(\muv, \Lambda)} (\mathbb{R}^n \times BW) \cong  \mathbb{R}^n \times \mathbb{S}^n$. The statement immediately after requires $(\muv, \Lambda)$ to be inside the {product of closed balls of radius $r = (1- \gamma) \norm{(\Lambda^*)^{-1}}_2^{-1/2}$, with $\gamma \in (0,1)$,} centered at $(\muv^*,\Lambda^*)$. It is worth noting that the uniform estimate computed over the restriction $\nuv \in \mathbb{R}^n, \norm{S}_F < 1 $ will still hold on the product of closed balls of radius $r$ around $(\muv^*,\Lambda^*)$ due to scaling invariance of the estimate on the tangent space $ T_{(\muv, \Lambda)} (\mathbb{R}^n \times BW)$. By scaling $(\nuv , S)$ as $(c \nuv , cS)$, one can extend geodesics from $ (\muv, \Lambda)$ to any point on the boundary of the product of closed balls of radius $r$ around $(\muv^*,\Lambda^*)$. The condition $r = (1- \gamma) \norm{(\Lambda^*)^{-1}}_2^{-1/2}$ only implies the existence of unique geodesics between any two points inside the product of closed balls of radius $r$ around $(\muv^*,\Lambda^*)$. A simple calculation yields this estimate $r < \norm{(\Lambda^*)^{-1}}_2^{-1/2}$ for the existence of unique geodesic $ \Lambda_t := (I +t \wt S)\Lambda^* (I + t \wt S)$ for any $t \in [0,1]$ (see Appendix \ref{appendixunifestimatenormlambda}). 
\end{rem}

\subsubsection{Hessian estimates of the distance squared function in BW metric}

\begin{theorem}\label{distancehessianestimatethm}
   Under \textbf{A1-A2} and for $\gamma \in (0,1)$ satisfying $ (1-\gamma)\norm{(\Lambda^*)^{-1}}^{-1/2}_2 < \frac{  \gamma^2 \pi \norm{(\Lambda^*)^{-1}}^{-1}_2}{\sqrt{3} }$, the squared distance in BW metric between $\Lambda, \Lambda^*  $ given as follows:
$$ d_{BW}^2(\Lambda,\Lambda^*) := {tr\bigg(\Lambda + \Lambda^* - 2( \Lambda^{1/2} \Lambda^*  \Lambda^{1/2})^{1/2}  \bigg)} \, , $$
admits a Hessian $\text{Hess}_{(\muv,\Lambda)}d_{BW}^2(\Lambda,\Lambda^*)$ that satisfies the following uniform estimate : 
\begin{align}
    0 < (1- \gamma)\norm{(\Lambda^*)^{-1}}^{1/2}_2\frac{\sqrt{3} }{\gamma^2}  \cot \Bigg((1- \gamma)\norm{(\Lambda^*)^{-1}}^{1/2}_2\frac{\sqrt{3} }{2\gamma^2} \Bigg)   &\leq  \frac{\text{Hess}_{(\muv,\Lambda)} d_{BW}^2(\Lambda,\Lambda^*)(V,V)}{\text{tr} (L_{\Lambda}(V) \Lambda L_{\Lambda}(V)) } \leq  2 
\end{align}
for any $V \in T_{\Lambda}(\texttt{SPD}_n )$ provided $ d_{BW}(\Lambda,\Lambda^*) \le  (1-\gamma)\norm{(\Lambda^*)^{-1}}^{-1/2}_2 $.
\end{theorem}
The proof of Theorem \ref{distancehessianestimatethm} is in Appendix \ref{sectionDROBWappendix}.

\subsubsection{Hessian estimates of the Lagrangian}

\begin{defn}\label{Lagrangiandefn}
 The Lagrangian from the unconstrained minimization problem \eqref{JKO3} can be defined as:
$$ \mathcal{L}(\muv,\Lambda;\w) := \mathbb{E}_{\z \sim \probP(\muv,\Lambda)} [\ell(\w;\z) ]  - \frac{1}{\epsilon} \bigg(  \norm{\muv - \muv^*}^2 + d^2_{BW}(\Lambda ,\Lambda^*)\bigg) \, . $$   
\end{defn}

\begin{theorem}\label{lagrangehessianestimatethm}
   Under \textbf{A1-A2} and for $\gamma \in (0,1)$, the Lagrangian $\mathcal{L}(\muv,\Lambda;\w) $ from definition \ref{Lagrangiandefn}, for any fixed $\w$, admits a Hessian $\text{Hess}_{(\muv,\Lambda)}\mathcal{L}(\muv,\Lambda;\w)$ that satisfies the following uniform estimates : 
\begin{align}
\sup_{\substack{\nuv \in \mathbb{R}^n, \norm{S}_F < 1 \\ {S \in  \mathbb{S}^n} }}  \frac{ {\text{Hess}_{(\muv,\Lambda)}\mathcal{L}((\nuv,  S\Lambda +\Lambda S),(\nuv,S\Lambda +\Lambda S))}}{ \norm{\nuv}^2 + tr(S \Lambda S)}    & \nonumber \\ & \hspace{-7cm} \le C_{\Lambda^*, q, n} (C''_{p,\Lambda^*,n} \times \big( \mathbb{E}_{\z \sim \probP(\muv^*,\Lambda^*)}\left[\abs{\ell(\w;\z)}^{pq}\right] \big)^{1/q}  )^{1/p} \nonumber \\ & \hspace{-7cm} - \frac{1}{\epsilon}(1- \gamma)\norm{(\Lambda^*)^{-1}}^{1/2}_2\frac{\sqrt{3} }{\gamma^2}  \cot \Bigg((1- \gamma)\norm{(\Lambda^*)^{-1}}^{1/2}_2\frac{\sqrt{3} }{2\gamma^2} \Bigg) \, ,
\end{align}
and
\begin{align}
  \sup_{\substack{\nuv \in \mathbb{R}^n, \norm{S}_F < 1 \\ {S \in  \mathbb{S}^n} }}  \frac{ \abs{\text{Hess}_{(\muv,\Lambda)}\mathcal{L}((\nuv,  S\Lambda +\Lambda S),(\nuv,S\Lambda +\Lambda S))}}{ \norm{\nuv}^2 + tr(S \Lambda S)}   & \nonumber \\ & \hspace{-7cm} \le C_{\Lambda^*, q, n} (C''_{p,\Lambda^*,n} \times \big( \mathbb{E}_{\z \sim \probP(\muv^*,\Lambda^*)}\left[\abs{\ell(\w;\z)}^{pq}\right] \big)^{1/q})^{1/p}  + \frac{2}{\epsilon}  < \infty \, ,
\end{align}
for some universal constants $C''_{p,\Lambda^*,n}, \, C_{\Lambda^*, q, n} \, ,\, 1/p + 1/q = 1 , \, p \ge 1$ ,
provided $$ d_{BW}(\Lambda,\Lambda^*) \le (1-\gamma)\norm{(\Lambda^*)^{-1}}^{-1/2}_2 := r \quad , \quad  (1-\gamma)\norm{(\Lambda^*)^{-1}}^{-1/2}_2 < \frac{  \gamma^2 \pi \norm{(\Lambda^*)^{-1}}^{-1}_2}{\sqrt{3} } \, .$$ Then for any $$ 0 < \epsilon \le \frac{(1- \gamma)\norm{(\Lambda^*)^{-1}}^{1/2}_2\frac{\sqrt{3} }{\gamma^2}  \cot \Bigg((1- \gamma)\norm{(\Lambda^*)^{-1}}^{1/2}_2\frac{\sqrt{3} }{2\gamma^2} \Bigg)}{2C_{\Lambda^*, q, n}  (C''_{p,\Lambda^*,n} \times \big( \mathbb{E}_{\z \sim \probP(\muv^*,\Lambda^*)}\left[\abs{\ell(\w;\z)}^{pq}\right] \big)^{1/q})^{1/p}} $$ the Lagrangian $\mathcal{L}$ is:
\begin{itemize}
    \item at least $C_{\Lambda^*, q, n}   (C''_{p,\Lambda^*,n} \times \big( \mathbb{E}_{\z \sim \probP(\muv^*,\Lambda^*)}\left[\abs{\ell(\w;\z)}^{pq}\right] \big)^{1/q})^{1/p} := \mu $ geodesically strongly concave and
    \item at most $C_{\Lambda^*, q, n}  (C''_{p,\Lambda^*,n} \times \big( \mathbb{E}_{\z \sim \probP(\muv^*,\Lambda^*)}\left[\abs{\ell(\w;\z)}^{pq}\right] \big)^{1/q})^{1/p} + \frac{2}{\epsilon} := L  $ gradient Lipschitz continuous 
\end{itemize}
 in $ (\muv,\Lambda)$ on the {product of closed balls of radius $r$} centered at $(\muv^*,\Lambda^*)$ for any $ {\w \in \mathbb{R}^d}$ and $$ C_{\Lambda^*, q, n} \sim_{q}  \bigg(\frac{\Gamma(\frac{n+4q}{2}) }{\Gamma(\frac{n}{2}) }\bigg)^\frac{1}{q}\norm{(\Lambda^*)^{1/2}}^4_F \bigg( \frac{\sqrt{n} \norm{(\Lambda^*)^{-1}}_F}{(1-r \norm{(\Lambda^*)^{-1/2}}_F)}\bigg)^{6} \, , $$  
\begin{align*}
C''_{p,\Lambda^*,n}  & \lesssim_p n \max\{ 1, r^{2}\} \norm{(\Lambda^*)^{-1/2}}_F^{4} \norm{(\Lambda^*)^{1/2}}_F^{4} \bigg(\frac{\Gamma(\frac{n+4p}{2}) }{\Gamma(\frac{n}{2}) }\bigg)^{1/p}
\end{align*}
for $ \norm{\Lambda^*}_F \ge 1 \, , \, \norm{(\Lambda^*)^{-1/2} }_F \ge 1 $.
\end{theorem}
The proof of Theorem \ref{lagrangehessianestimatethm} is in Appendix \ref{sectionDROBWappendix}.
\begin{rem}
Observe that the Lagrangian $\mathcal{L}$ is only locally geodesically strongly concave for the choice of parameter $\epsilon$ from Theorem \ref{lagrangehessianestimatethm}. In general, outside the closed ball product $\bar{\mathcal{B}}_r(\muv^*) \times \bar{\mathcal{B}}_r(\Lambda^*) $, the Lagrangian $\mathcal{L}$ can be geodesically nonconcave in $ (\muv,\Lambda)$. We defer the reader to section \ref{sectiongradientcalc} in appendix for explicit gradient calculations in the BW metric. Although well known in literature, we provide these calculations just for completeness and for the ease of reader to follow the derivations.  
\end{rem}

\section{Local maxima existence inside the injectivity radius ball}
We now show that the Lagrangian $\mathcal{L}(\cdot, \cdot ; \w) $ for any sufficiently small $\epsilon>0$ will have a unique maximizer in the interior of the product of closed balls of radius $r$ centered at $ (\muv^*,\Lambda^*)$ for any given $\w \in \mathbb{R}^d$. For that we first need the following supporting lemma.
{
\begin{lem} \label{localmaxexistencelem}
Let $\mathcal{M}$ be a smooth Riemannian manifold of finite dimension, let $x^* \in \mathcal{M}$, 
and let $B = \bar{B}_\delta(x^*)$ be a closed geodesic ball of radius $\delta > 0$. Suppose:
\begin{enumerate}
    \item $F : \mathcal{M} \to \mathbb{R}$ is $\mu$-strongly convex on $B$ with minimizer 
    $x^* \in \mathrm{int}(B)$,
    \item $G := F + \epsilon H$ is strongly convex on $B$ for some $\epsilon > 0$,
    \item $H : \mathcal{M} \to \mathbb{R}$ is differentiable on $B$ with 
    $\sup_{x \in B}\|\mathrm{grad}\, H(x)\| \leq L < \infty$.
\end{enumerate}
Then $G$ has a unique minimizer $x_\epsilon \in B$ satisfying
\begin{equation}
    d(x_\epsilon, x^*) \leq \frac{2\epsilon L}{\mu}.
\end{equation}
In particular, if $\epsilon < \mu\delta / (2L)$, then $x_\epsilon \in \mathrm{int}(B)$ 
and is therefore an unconstrained local minimizer of $G$ on $\mathcal{M}$.
\end{lem}

\begin{proof}
Since $B$ is compact and $G$ is continuous, $G$ attains 
its minimum on $B$ at some point $x_\epsilon \in B$ by the extreme value theorem. 
Uniqueness of $x_\epsilon$ follows from strong convexity of $G$ on $B$: if there were 
two distinct minimizers, the value of $G$ at their midpoint would be strictly less than 
the minimum, a contradiction.

 Since $F$ is $\mu$-strongly convex on $B$ with minimizer $x^*$, 
it satisfies the quadratic growth condition
\begin{equation}
    F(x) \geq F(x^*) + \frac{\mu}{2} d(x, x^*)^2 \quad \forall\, x \in B.
\end{equation}
Since $x_\epsilon$ minimizes $G$ on $B$ and $x^* \in B$, we have 
$G(x_\epsilon) \leq G(x^*)$, i.e.,
\begin{equation}
    F(x_\epsilon) + \epsilon H(x_\epsilon) \leq F(x^*) + \epsilon H(x^*).
\end{equation}
Rearranging and applying the quadratic growth bound:
\begin{equation}
    \frac{\mu}{2} d(x_\epsilon, x^*)^2 
    \leq F(x_\epsilon) - F(x^*) 
    \leq \epsilon\bigl(H(x^*) - H(x_\epsilon)\bigr).
\end{equation}
By the mean value inequality on $\mathcal{M}$,
\begin{equation}
    \bigl|H(x^*) - H(x_\epsilon)\bigr| \leq L \cdot d(x_\epsilon, x^*).
\end{equation}
If $x_\epsilon = x^*$ the bound holds trivially. Otherwise, dividing both sides 
by $d(x_\epsilon, x^*) > 0$ gives
\begin{equation}
    \frac{\mu}{2} d(x_\epsilon, x^*) \leq \epsilon L,
\end{equation}
hence $d(x_\epsilon, x^*) \leq 2\epsilon L / \mu$.

Hence if $\epsilon < \mu\delta/(2L)$, then 
$d(x_\epsilon, x^*) < \delta$, so $x_\epsilon \in \mathrm{int}(B)$. 
Since $G$ is differentiable and $x_\epsilon$ is an interior minimizer, 
the first-order optimality condition holds without constraint, and $x_\epsilon$ 
is an unconstrained local minimizer of $G$ on $\mathcal{M}$.
\end{proof}
}

\begin{theorem}\label{interiormaximaexistencethm}
   Under \textbf{A1-A2}, for any $\w \in \mathbb{R}^d$ and any $\Lambda \in \bar{\mathcal{B}}_{r}(\Lambda^*)$ where $r = (1- \gamma) \norm{(\Lambda^*)^{-1}}_2^{-1/2}$ , $\gamma \in (0,1)$ with $$  (1-\gamma)\norm{(\Lambda^*)^{-1}}^{-1/2}_2 < \frac{  \gamma^2 \pi \norm{(\Lambda^*)^{-1}}^{-1}_2}{\sqrt{3} } \, ,$$ let $C''_{p,\Lambda^*,n}$, $ C_{\Lambda^*, q, n}  $ be the universal constants as defined in Theorem \ref{lagrangehessianestimatethm} where $ 1/p + 1/q = 1 , \, p \ge 1$, and let
    $$ 0<\epsilon < \frac{(1- \gamma)\norm{(\Lambda^*)^{-1}}^{1/2}_2\frac{\sqrt{3} }{\gamma^2}  \cot \Bigg((1- \gamma)\norm{(\Lambda^*)^{-1}}^{1/2}_2\frac{\sqrt{3} }{2\gamma^2} \Bigg)}{2C_{\Lambda^*, q, n}  (C''_{p,\Lambda^*,n} \times \big( \mathbb{E}_{\z \sim \probP(\muv^*,\Lambda^*)}\left[\abs{\ell(\w;\z)}^{pq}\right] \big)^{1/q})^{1/p}} \times \min \{\tfrac{r}{2} ,1\} \, . $$
   Then there exists a unique maximizer $(\muv^{\#}, \Lambda^{\#})$ of the Lagrangian $ \mathcal{L}(\cdot ,\cdot \, ;\w)$ inside the {product of closed balls of radius $r$} centered at $(\muv^*, \Lambda^*)$ such that  $$ d_{l_2 \oplus BW}\bigg((\muv^{\#}, \Lambda^{\#}), (\muv^*, \Lambda^*)\bigg) < \frac{4C_{\Lambda^*, q, n}  (C''_{p,\Lambda^*,n} \times \big( \mathbb{E}_{\z \sim \probP(\muv^*,\Lambda^*)}\left[\abs{\ell(\w;\z)}^{pq}\right] \big)^{1/q})^{1/p}} {(1- \gamma)\norm{(\Lambda^*)^{-1}}^{1/2}_2\frac{\sqrt{3} }{\gamma^2}  \cot \Bigg((1- \gamma)\norm{(\Lambda^*)^{-1}}^{1/2}_2\frac{\sqrt{3} }{2\gamma^2} \Bigg)} \epsilon \, . $$
\end{theorem}
The proof of Theorem \ref{interiormaximaexistencethm} is in Appendix \ref{sectionDROBWappendix}.
\begin{rem}
    Using the upper bound on $\epsilon$ from Theorem \ref{interiormaximaexistencethm} one can construct the desired ambiguity sets for the DRO problem so that the Lagrangian $ \mathcal{L}(\cdot, \cdot; \w)$ remains strongly geodesically concave on those ambiguity sets with respect to the probability measure and also that a unique maximizer $\probP^{\#} := \probP(\muv^{\#}, \Lambda^{\#})$ exists for the Lagrangian $ \mathcal{L}(\cdot, \cdot; \w)$ within the joint compact search region for any model parameter $\w$ that is uniformly bounded.  
\end{rem}

\section{Alternating Riemannian gradient ascent descent on the product manifold}\label{RGAconvergencesection}

\begin{algorithm}[H]
\caption{Alternating Riemannian Measure and Model Updates}
\label{algRGA:alternating_updates}
\begin{algorithmic}[1]
\Require Initial measure parameters $(\muv_0,\Lambda_0)$, initial model $\w_0$, step sizes $h_1,h_2$, inner steps $J$, outer iterations $K$
\For{$k = 0,1,\ldots,K-1$}
    \State \textbf{Measure update:}
    \State
    \[
    \textbf{Compute} \quad \nabla_{(\muv, \Lambda)} \mathcal{L}(\muv_k,\Lambda_k;\w_k) \quad \text{explicitly from Appendix \ref{sectiongradientcalc}} 
    \]
    \[
  \hspace{-2cm}  (\muv_{k+1},\Lambda_{k+1})
    =
    \exp_{(\muv_k,\Lambda_k)}
    \!\left(
    h_1 \nabla_{(\muv, \Lambda)} \mathcal{L}(\muv_k,\Lambda_k;\w_k)
    \right).
    \]
    \State Construct $\probP_{k+1}$ from $(\muv_{k+1},\Lambda_{k+1})$.
    \State \textbf{Model updates:}
    \For{$t=0,\ldots,J-1$}
        \State
        \[
        \w_{k+\frac{t+1}{J}}
        =
        \w_{k+\frac{t}{J}}
        -
        h_2
        \nabla_{\w}
        \mathbb{E}_{\z\sim\probP_{k+1}}
        \bigl[
        \ell(\w_{k+\frac{t}{J}};\z)
        \bigr].
        \]
    \EndFor
    \State Set $\w_{k+1} \gets \w_{k+\frac{J}{J}}$.
\EndFor
\end{algorithmic}
\end{algorithm}

\begin{rem}[\textbf{Validity of assumption \textbf{A2}}]\label{cutofffunremark}
    We emphasize that assumption \textbf{A2} of bounded moments for the loss $\ell(\w, \cdot)$ with respect to the base measure $\probP(\muv^*, \Lambda^*)$ can be satisfied by a large class of coercive loss functions (for example functions with at most polynomial growth in $\w$). For convenience, we can also satisfy \textbf{A2} uniformly for any $\w \in \mathbb{R}^d$ if required. For example, define the following mollified/tempered statistical loss $ \mathbb{E}_{\z \sim \probP(\muv^*,\Lambda^*)} [\ell(\w;\z)  \Phi_{\mathcal{K}}(\w) ] $ where
 $\Phi_{\mathcal{K}}(\w) \in \mathcal{C}^{\infty}(\mathbb{R}^d ) $ is a deterministic smooth bump function supported on a compact set $\mathcal{K} \subset \mathbb{R}^d $ such that $ \Phi_{\mathcal{K}} \equiv 1 $ on $\mathcal{K}$, $ 0 \leq \Phi_{\mathcal{K}} \leq 1 $ on $ V \backslash \mathcal{K}$ for some compact $V$ with $ \mathcal{K} \subset  V \subset  \mathbb{R}^d $ and $\Phi_{\mathcal{K}} \equiv 0 $ otherwise. The compact set ${\mathcal{K}}$ is convex, and is sufficiently large so that it contains the global minima $ \w^* \in \arg\min \mathbb{E}_{\z \sim \probP(\muv^*,\Lambda^*)} [\ell(\w;\z)] $ in its interior by bounded sublevel sets of $\mathbb{E}_{\z \sim \probP(\muv^*,\Lambda^*)} [\ell(\w;\z)]$ from \textbf{A1} (since pointwise mean of coercive functions is also coercive). Then the $p$-th moments of the function $\ell(\w;\z)  \Phi_{\mathcal{K}}(\w)$ needed in \textbf{A2} have the form 
 \begin{align*}
     &  [\mathbb{E}_{\z \sim \probP(\muv^*,\Lambda^*) } [\ell(\w;\z) \Phi_{\mathcal{K}}(\w) ]^p]^{1/p}  \quad , \quad   [\mathbb{E}_{\z \sim \probP(\muv^*,\Lambda^*) } \norm{\Phi_{\mathcal{K}}(\w) \nabla_{\w}\ell(\w;\z) + \ell(\w;\z) \nabla_{\w} \Phi_{\mathcal{K}}(\w) } ^p]^{1/p}  \quad ,\\
  & [\mathbb{E}_{\z \sim \probP(\muv^*,\Lambda^*) } \norm{\Phi_{\mathcal{K}}(\w) \nabla^2_{\w}\ell(\w;\z) + (\nabla_{\w}\Phi_{\mathcal{K}}(\w) \nabla^T_{\w}\ell(\w;\z) + \nabla_{\w} \ell(\w;\z) \nabla^T_{\w} \Phi_{\mathcal{K}}(\w) )+ \ell(\w;\z) \nabla^2_{\w} \Phi_{\mathcal{K}}(\w) }_F ^p]^{1/p} \, ,
 \end{align*}
and all these moments are bounded uniformly for any $\w$ by the cutoff property of smooth bump function $\Phi_{\mathcal{K}}(\w) $. Thus \textbf{A2} is satisfied for the function $\ell(\w;\z)  \Phi_{\mathcal{K}}(\w)$ uniformly for any $\w$.
\end{rem}

\begin{rem}\label{jointsmoothremark}
    We note that the Lagrangian $\mathcal{L}(\muv,\Lambda ;\w)$ (from definition \ref{Lagrangiandefn}) is jointly $\mathcal{C}^2$ on the product manifold $ \bar{\mathcal{B}}_r(\muv^*) \times  \bar{\mathcal{B}}_r(\Lambda^*) \times \mathbb{R}^d  $ by virtue of \textbf{A1-A2} for $r = (1- \gamma) \norm{(\Lambda^*)^{-1}}_2^{-1/2}$ where $\gamma \in (0,1)$. From the Riemannian gradient computations (\eqref{Taylor13}, \eqref{Taylor23}) of $ \nabla_{(\muv,\Lambda)} \mathcal{L}(\muv,\Lambda ;\w)$ in Appendix \ref{taylormetricappendix}, one can apply $ \nabla_{\w}$ operator to $\nabla_{(\muv,\Lambda)} \mathcal{L}$ and since the loss $\ell(\w; \z)$ does not depend on ${(\muv,\Lambda)}$, the operator $ \nabla_{\w}$ can be passed inside the expectation operators in \eqref{Taylor13}, \eqref{Taylor23} by the Dominated convergence theorem and \textbf{A2} to get $ \nabla_{\w}\nabla_{(\muv,\Lambda)} \mathcal{L}(\muv,\Lambda ;\w) = \nabla_{(\muv,\Lambda)} \nabla_{\w} \mathcal{L}(\muv,\Lambda ;\w)$ on the product manifold $ \bar{\mathcal{B}}_r(\muv^*) \times  \bar{\mathcal{B}}_r(\Lambda^*) \times \mathbb{R}^d$ . In fact, the Riemannian gradient computations (\eqref{Taylor13}, \eqref{Taylor23}) can be performed by working with $ \nabla_{\w} \ell(\w; \z)$ instead of $\ell(\w;\z)$ and it can be easily verified that the expressions $ \nabla_{\w}\nabla_{(\muv,\Lambda)} \mathcal{L}(\muv,\Lambda ;\w) $, $ \nabla_{(\muv,\Lambda)} \nabla_{\w} \mathcal{L}(\muv,\Lambda ;\w)$ are identical. Since the mixed derivatives commute we immediately get that the Lagrangian $\mathcal{L}(\muv,\Lambda ;\w)$ (from definition \ref{Lagrangiandefn}) is jointly $\mathcal{C}^2$ on the product manifold $ \bar{\mathcal{B}}_r(\muv^*) \times  \bar{\mathcal{B}}_r(\Lambda^*) \times \mathbb{R}^d  $. For exact mixed derivative computations refer Appendix \ref{mixedderivativeappendix}.
\end{rem}
We now briefly describe the two major steps from Algorithm \ref{algRGA:alternating_updates}.

\subsection{Distribution update rule via Riemannian ascent \& exponential map}

Consider the first order {Riemannian ascent iteration} on the Lagrangian $\mathcal{L}$ (see definition \ref{Lagrangiandefn}) along the fiber $\mathbb{R}^n \times BW \times \{\w\} $ of the product space $ \mathbb{R}^n \times BW \times \mathbb{R}^d $ for any $k \geq 0$ and {$ h_1 > 0$ sufficiently small :}
\begin{align}
    (\muv_{k+1},\Lambda_{k+1}) = \exp_{(\muv_k, \Lambda_k)} ( h_{1} \nabla_{(\muv, \Lambda)} \mathcal{L}(\muv_k,\Lambda_k;\w)) \label{measureupdate1}
\end{align}
For any $\w \in \mathcal{K}$ the Lagrangian $ \mathcal{L}(\cdot \, , \cdot \, ;\w)$ will have a unique local maxima $ (\hat\muv_{\w},\hat\Lambda_{\w})$ on the fiber $\mathbb{R}^n \times BW \times \{\w\} $ provided $\epsilon$ is sufficiently small (Theorem \ref{interiormaximaexistencethm}). Since $ -\mathcal{L}(\muv,\Lambda;\w)$ is locally strongly geodesically convex around its minimizer with parameter $\mu$ for any given $\w$ (Theorem \ref{lagrangehessianestimatethm}), the iteration \eqref{measureupdate1} the fiber $\mathbb{R}^n \times BW \times \{\w\} $ will achieve a linear convergence rate provided $h_1 \in (0, 1/L) $ where $ L$ is the local uniform metric gradient Lipschitz parameter for $\mathcal{L}(\cdot ;\w)$ (Theorem \ref{lagrangehessianestimatethm}). 
%\newpage

\subsection{Model update rule via GD}

For any $k \ge 0$ consider the following $J$-multi-step GD update on the Lagrangian $\mathcal{L}$ along the fiber $ \{\muv_k\} \times \{\Lambda_k\}  \times \mathbb{R}^d $ of the product space $ \mathbb{R}^n \times BW \times \mathbb{R}^d $ for any small enough $h_2 >0$, $J:= J(k) $ large enough (constant or a function of $k$) and any fixed index $k$ : 
\begin{align}
    \w_{k + (t+1)/J} & = \w_{k + t/J} - h_2 \nabla_{\w} \mathbb{E}_{\z \sim \probP_{k+1}} [\ell(\w_{k + t/J};\z) ] \quad , \quad 0 \le t \le J-1  \label{modelupdate1}
\end{align}
where the measure $ \probP_{k+1}$ is parameterized by the tuple $(\muv_{k+1},\Lambda_{k+1})$. Since the quadratic regularization of measures does not depend on $\w$ we only need to update $\w$ using the gradients of statistical risk rather than the gradients of Lagrangian $ \mathcal{L}(\muv,\Lambda;\w)$. Under \textbf{A1-A2} by the estimate \eqref{Taylor13eo} we have that $ \mathbb{E}_{\z \sim \probP_{k+1}} [\ell(\cdot \, ;\z) ] $ will be $L$-gradient Lipschitz continuous for some constant $L >0$ for any uniformly bounded $\w$ and any probability measure $ \probP_{k+1}$ sufficiently close to the reference probability measure parameterized by $(\muv^*,\Lambda^*)$. Then for $h_2 \in (0, 1/L) $ we can invoke the monotonic descent lemma (monotonic $ \{ \mathbb{E}_{\z \sim \probP_{k+1}} [\ell(\w_{k + t/J};\z) ]  \}_t$), i.e., $  \mathbb{E}_{\z \sim \probP_{k+1}} [\ell(\w_{k + t/J};\z) ] $ decreases monotonically with $t$. If the gradient of the expected loss is not available then we may modify \eqref{modelupdate1} to an SGD update as follows:
\begin{align}
    \w_{k + (t+1)/J} & = \w_{k + t/J} - h_2 \nabla_{\w} \ell(\w_{k + t/J};\z_{t,k})  \quad , \quad \z_{t,k} \overset{i.i.d.}{\sim} \probP_{k+1}   \quad ; \quad  0 \le t \le J-1 \label{modelupdate2}
\end{align}
however the convergence analysis with SGD update \eqref{modelupdate2} will be much more tedious and therefore not pursued in this work.

\subsection{Population loss landscape assumptions for local strongly geodesically concave-local \L{}ojasiewicz type Lagrangian}

To analyze the convergence of Algorithm \ref{algRGA:alternating_updates} we make the following assumptions on critical sets and gradient regularity of the statistical/population loss function:
\begin{enumerate}
\item[\textbf{A1'.}]\textbf{(Existence of nice critical sets)} Let $U$ be a compact neighborhood of $\probP_0 \equiv \probP(\muv^*, \Lambda^*)$ in the topology induced by the product of BW and Euclidean metric such that $\probP \in U$ is non-degenerate. For every $\probP \in U$, the set of critical points of the statistical risk function  $\mathbb{E}_{\z \sim \probP} [\ell(\cdot;\z) ]$ are a finite union of closed, connected components and is uniformly bounded in $U$. Further, for every $\probP \in U$, any connected component of the critical points of the statistical risk function $\mathbb{E}_{\z \sim \probP} [\ell(\cdot;\z) ]$ either contains all local minima, or all local maxima or all strict saddle points and there is at least one connected component of local minima of $\mathbb{E}_{\z \sim \probP} [\ell(\cdot;\z) ]$ for  every $\probP \in U$.   
\item[\textbf{A2'.}]\textbf{(Local \L{}ojasiewicz type condition)} For every $\probP$ in the compact neighborhood $U$ of $\probP_0$, the statistical risk function $\mathbb{E}_{\z \sim \probP} [\ell(\cdot;\z) ]$ satisfies a $C_\beta, \beta$-type P\L{} inequality for any $\w$ in some uniform $\delta$ open neighborhood $\mathcal{W}$ of every connected component of $\mathcal{S}^*(\probP)$ for some $\beta \in (1,2]$, $C_\beta >0$ :
\[
\norm{\nabla_{\w}\mathbb{E}_{\z \sim \probP} [\ell(\w;\z) ] }^{\beta} \geq C_\beta \bigg( \mathbb{E}_{\z \sim \probP} [\ell(\w;\z) ] - \inf_{\w \in \mathcal{W} }\mathbb{E}_{\z \sim \probP} [\ell(\w;\z) ]\bigg)
\]
where 
\[
\mathcal{S}^*(\probP) : = \{\w :  \nabla_{\w} f(\w, \probP)  = \mathbf{0} \quad, \quad \nabla^2_{\w} f(\w, \probP)  \succeq \mathbf{0}  \} .
 \]
 The $\delta >0 $ is such that the $\delta$ open neighborhood $\mathcal{W}$ of any given connected component of $\mathcal{S}^*(\probP)$ is disjoint, and in particular $2 \delta$ away from other connected components of $\mathcal{S}^*(\probP)$ by Hausdorff property and by finitely many connected critical components of $\mathcal{S}^*(\probP)$.
 \item[\textbf{A3'.}]\textbf{(Critical set variations under distributional shift)} There exists a compact neighborhood $W_{\delta} \subset U$ of $\probP_0 \equiv \probP(\muv^*, \Lambda^*)$ in the topology induced by the product of BW and Euclidean metric, independent of $\w$, such that for any pair $\probP_1, \probP_2 \in W_{\delta} $ and for any local minima connected component $\mathcal{S}_c^*(\probP_1) $ on the fiber $\probP_1$, there always exists a local minima connected component $\mathcal{S}_c^*(\probP_2) $ on the fiber $\probP_2$ such that $ \mathrm{dist}(\mathcal{S}_c^*(\probP_1), \mathcal{S}_c^*(\probP_2)) < \delta/4$ .
\end{enumerate}
{Notice that assumption \textbf{A1'} imposes a very general qualitative structure on the critical sets of the statistical loss within compact sets. The finite union of closed, connected critical components assumption appears in recent works \cite{fehrman2020convergence, arora2022understanding, azizian2025global} as well as Morse Bott theory\footnote{A Morse-Bott function is a smooth map on a manifold where the critical set is a disjoint union of smooth, closed submanifolds rather than isolated points.} \cite{bott1982lectures}. \textbf{A1'} also removes pathological cases such as existence of connected critical components that contain both local minima and strict saddle points. Last from \textbf{A1'}, the uniform boundedness of the critical sets of statistical risk function $\mathbb{E}_{\z \sim \probP} [\ell(\w;\z) ]$, as $\probP$ varies in $U$, is possible via the coercivity of the loss $\ell(\w;\z)$  from \textbf{A1} that allows the coercivity of statistical loss $\mathbb{E}_{\z \sim \probP} [\ell(\w;\z) ]$ in $\w$. Coercivity of the map $\w \mapsto \ell(\w;\z)$  from \textbf{A1} can be satisfied by mollifying the loss $\ell(\w;\z)$ to something like $\ell(\w;\z) \Phi_{\mathcal{K}}(\w) + (1- \Phi_{\mathcal{K}}(\w))\norm{\w}^2 $ for a suitable smooth bump function $ \Phi_{\mathcal{K}}(\w)$ as defined in Remark \ref{cutofffunremark}. Next, \textbf{A2'} for the statistical loss has been empirically observed in \cite{banerjee2024loss}. In particular, the case of $\beta = 2$ or the local Polyak \L{}ojasiewicz property is satisfied by the empirical loss landscape for wide neural networks in the over-parameterized regime (Neural Tangent Kernel regime) \cite{liu2022loss}. While such local P\L{} property may not necessarily transfer from the empirical to the statistical regime directly, the loss landscapes in the two regimes could be closely approximated if the empirical loss is constructed with sufficiently large number of samples. For instance, the local separation between the empirical and statistical loss landscapes can still be bounded arbitrary small as a function of number of i.i.d. samples with high probability via central limit theorems \cite{mei2018landscape}. Hence, some approximate version of the local P\L{} property may still hold in the statistical regime\footnote{Developing a proof machinery that can transfer local landscape properties between the empirical and the statistical regimes is beyond the scope of current paper and hence left for future work.} . Finally, \textbf{A3'} controls the drift of critical sets of the statistical loss landscape locally as distributions are perturbed. Note that without \textbf{A3'} the critical sets can drift significantly far apart on the joint search space. Then even ``small" ambiguity sets may result in sharp variations of the optimal model parameter $\w^*(\probP)$. Thus without \textbf{A3'}, the models solving the DRO problem can be extremely sensitive to even small noise in the data distributions (see Appendix \ref{appendixcountereg2} for a toy example on bounded critical set drift).} \\
\begin{lem}\label{suplemPLbound0}
Under \textbf{A1-A2, A1'-A2'}, for any $\w$ in some uniform $\delta$ open neighborhood $\mathcal{W}$ of any connected component of $\mathcal{S}^*(\probP)$, the gradient of the statistical loss $ f(\w; \probP) := \mathbb{E}_{\z \sim \probP} [\ell(\w;\z) ] $ satisfies the following inequality:
\begin{align}
    \norm{\nabla_{\w} f(\w;\probP) } \geq C_{\beta}^{\frac{1}{\beta-1}}\bigg(\frac{\beta}{\beta-1}\bigg)^{\frac{-1}{\beta-1}}\bigg( \mathrm{dist}(\w, \mathcal{S}^*(\probP))\bigg)^{\frac{1}{\beta -1}} \, . \label{QGtypegrowth1a}
\end{align}
\end{lem}
The proof of Lemma \ref{suplemPLbound0} is in Appendix \ref{sectionRGAconvergenceappendix}.
\begin{rem}
    Lemma \ref{suplemPLbound0} bound can be interpreted as the \L{}ojasiewicz inequality (also known as the Kurdyka \L{}ojasiewicz (K\L{}) inequality) with exponent $(\beta - 1) \in (0,1]$ and is satisfied for functions that are locally real analytic around their critical points \cite{kurdyka1998gradients}. 
\end{rem}

\subsection{Convergence rates}\label{convergenceratesBWsection}

We define the following setup for convenience before stating the main theorems in this section. Under \textbf{A1-A2}, let $ r :=(1-\gamma)\norm{(\Lambda^*)^{-1}}^{-1/2}_2$ where $\gamma \in (0,1)$ satisfies $$  (1-\gamma)\norm{(\Lambda^*)^{-1}}^{-1/2}_2 < \frac{  \gamma^2 \pi \norm{(\Lambda^*)^{-1}}^{-1}_2}{\sqrt{3} } \, .$$  
Suppose Algorithm \ref{algRGA:alternating_updates} is initialized with $(\w_0 , \probP_0) $ and $ \mathcal{S}_c^*(\probP_0) $ is a local minima connected component of the Lagrangian $\mathcal{L}( \probP_0 \, ; \, \cdot )$ on the $\probP_0$ fiber. Define $ S_1 := \bar{\mathcal{B}}_{\mathrm{diam}(\mathcal{K})}(\w_0)  \subset \mathbb{R}^d$ for some sufficiently large compact ball $\mathcal{K}$ centered at $\w_0$ with $\mathrm{diam}(\mathcal{K}) \gg \delta $. Then by coercivity of $\mathcal{L}$ in $\w$ from \textbf{A1} we get that $ \mathcal{S}_c^*(\probP_0) + \bar{\mathcal{B}}_{\delta}(\mathbf{0}) \Subset \mathrm{int}( \mathcal{K}) \subset S_1$. By \textbf{A3'} for every $\probP \in U$ there exists a local minima connected component $\mathcal{S}_c^*(\probP) $ such that $ \mathrm{dist}(\mathcal{S}_c^*(\probP), \mathcal{S}_c^*(\probP_0)) < \delta/4$ and by taking $\mathcal{K}$ large enough we get that $ \mathcal{S}_c^*(\probP) + \bar{\mathcal{B}}_{\delta}(\mathbf{0}) \Subset \mathrm{int}( \mathcal{K}) \subset S_1 $. Define $S_2 := \bar{\mathcal{B}}_r(\muv^*) \times \bar{\mathcal{B}}_r(\Lambda^*) $.

Let $C''_{p,\Lambda^*,n}$, $ C_{\Lambda^*, q, n}  $ be the universal constants as defined in Theorem \ref{lagrangehessianestimatethm} for $ 1/p + 1/q = 1  \, , p \ge 1 $.
Let $\epsilon$ satisfy the following uniform bound   $$ 0 <\epsilon < \inf_{\w \in S_1} \frac{(1- \gamma)\norm{(\Lambda^*)^{-1}}^{1/2}_2\frac{\sqrt{3} }{\gamma^2}  \cot \Bigg((1- \gamma)\norm{(\Lambda^*)^{-1}}^{1/2}_2\frac{\sqrt{3} }{2\gamma^2} \Bigg)}{2C_{\Lambda^*, q, n}  (C''_{p,\Lambda^*,n} \times \big( \mathbb{E}_{\z \sim \probP(\muv^*,\Lambda^*)}\left[\abs{\ell(\w;\z)}^{pq}\right] \big)^{1/q})^{1/p}} \times \min \{\tfrac{r}{2} ,1\} \, . $$
Define $$ \mu := \inf_{\w \in S_1} C_{\Lambda^*, q, n}   (C''_{p,\Lambda^*,n} \times \big( \mathbb{E}_{\z \sim \probP(\muv^*,\Lambda^*)}\left[\abs{\ell(\w;\z)}^{pq}\right] \big)^{1/q})^{1/p} > 0 \, . $$ Note that the function $ \w  \mapsto \mathbb{E}_{\z \sim \probP(\muv^*,\Lambda^*)}\left[\abs{\ell(\w;\z)}^{pq}\right]$ is continuous in $\w$, is strictly positive and bounded on $S_1$ hence achieves a positive infimum and supremum on $S_1$.
    Let $L := L_{n,d,p,\Lambda^*, r, \epsilon}$ be the single uniform Lipschitz constant for the Lagrangian $\mathcal{L} $ (definition \ref{Lagrangiandefn}) and its mixed derivatives on $S_1 \times S_2$ from Appendix \ref{uniformlipschitzfinalappendix}.
    
\begin{theorem}\label{convergenceratethm_exactPL_specialcase}
 Under \textbf{A1-A2} for some $p>1$ and \textbf{A1'-A3'} for $\beta =2$, suppose Algorithm \ref{algRGA:alternating_updates} is initialized with $(\w_0 , \probP_0) $ where $\probP_0 := \probP(\muv^*, \Lambda^* )$ such that $  \mathrm{dist}(\w_0, \mathcal{S}_c^*(\probP_0)) \le \frac{\delta}{2} \,  $.  
 Suppose $Lh_2 <1$, $h_1\ll 1$ with $\frac{2 \, L^{3}h_1 }{(1-\tfrac{Lh_2}{2}) C_{\beta}} \ll 1 $ and $J>0$ satisfies the following bound:
\begin{align*}
    \frac{2L}{C_{\beta}} \cdot \frac{2 L ( 1 - \frac{ h_2 C_{\beta}}{2} )^{J/2} }{C_{\beta}}  &<\frac{2 \, L^{3}h_1 }{(1-\tfrac{Lh_2}{2}) C_{\beta}}  \, .
\end{align*}
Also suppose that $\mathrm{dist}_H(\mathcal{S}_c^*(\probP_{k-1}), \mathcal{S}_c^*(\probP_k)) $, for all $k \ge 0$, satisfies a Lipschitz estimate with constant $ 0< D_{\beta}  < \frac{(1-\tfrac{Lh_2}{2}) C_{\beta} \mu}{2 \, L^{3}  } $ as follows:
    \begin{align*}
        \mathrm{dist}_H(\mathcal{S}_c^*(\probP_{k-1}), \mathcal{S}_c^*(\probP_k)) \le  D_{\beta} \, \mathrm{dist}(\probP_{k-1},\probP_k).
    \end{align*}
Then there exists $h_1 >0$ such that the iteration $\{(\w_k, \probP_k)\}_k$ from Algorithm \ref{algRGA:alternating_updates} converges to a basin saddle point $ (\w^{\vardiamond},\probP^{\vardiamond}) $ (definition \ref{basinsaddledef}) of the function $ \wt f(\w ,\probP) := \mathcal{L}( \muv, \Lambda ; \w  )$ at a linear rate $\mathcal{O}( a^k)$ where 
 $$ a = \max \bigg\{ \bigg(1 - \frac{h_1}{4}\bigg(\mu - \frac{2 \, L^{3}  D_{\beta} }{(1-\tfrac{Lh_2}{2}) C_{\beta}} \bigg)    \bigg), \sqrt{(1 -2 \mu h_1 +  L^2 h_1^2)} \bigg\} \in (0,1)  \, . $$
\end{theorem}

\begin{proof}
    The function $ \wt f(\w ,\probP) := \mathcal{L}( \muv, \Lambda ; \w  )$ for $\probP:= \probP(\muv, \Lambda) $ is jointly $\mathcal{C}^2$ smooth in $(\w , \muv, \Lambda) $ from remark \ref{jointsmoothremark}, is $\mu$ geodesically strongly concave in $(\muv, \Lambda) \in  \bar{\mathcal{B}}_r(\muv^*) \times \bar{\mathcal{B}}_r(\Lambda^*) \subset   \mathbb{R}^n \times BW $ for the given range of $\epsilon$ uniformly in $\w \in S_1$ from Theorem \ref{lagrangehessianestimatethm} and $ \wt f(\w ,\cdot)$ has a maximizer in the interior of $ \bar{\mathcal{B}}_r(\muv^*) \times \bar{\mathcal{B}}_r(\Lambda^*) $ uniformly for all $\w \in S_1$ from Theorem \ref{interiormaximaexistencethm}. Consider the product Riemannian manifold $   \bar{\mathcal{B}}_r(\muv^*) \times \bar{\mathcal{B}}_r(\Lambda^*) \subset  \mathbb{R}^n \times BW $. The sectional curvature $\mathbf{K}_{\bar{\mathcal{B}}_r(\muv^*) \times \bar{\mathcal{B}}_r(\Lambda^*)} $ of the product manifold $\bar{\mathcal{B}}_r(\muv^*) \times \bar{\mathcal{B}}_r(\Lambda^*) $ satisfies the inequality
    $$ \min \{ \mathbf{K}_{\bar{\mathcal{B}}_r(\muv^*)}, \, \mathbf{K}_{\bar{\mathcal{B}}_r(\Lambda^*)} ,\,  0  \} \le  \mathbf{K}_{\bar{\mathcal{B}}_r(\muv^*) \times \bar{\mathcal{B}}_r(\Lambda^*)} \le \max \{ \mathbf{K}_{\bar{\mathcal{B}}_r(\muv^*)}, \, \mathbf{K}_{ \bar{\mathcal{B}}_r(\Lambda^*)} ,\,  0  \}$$
    where $ \mathbf{K}_{\bar{\mathcal{B}}_r(\Lambda^*)} \in \bigg[0, \frac{3 \norm{(\Lambda^*)^{-1}}^2_2}{4 \gamma^4} \bigg] $ from Lemma \ref{secboundlemma} (the lower bound is sharp) and $\mathbf{K}_{\bar{\mathcal{B}}_r(\muv^*)} = 0$ trivially. Hence $ \mathbf{K}_{\bar{\mathcal{B}}_r(\muv^*) \times \bar{\mathcal{B}}_r(\Lambda^*)} \in \bigg[0, \frac{3 \norm{(\Lambda^*)^{-1}}^2_2}{4 \gamma^4} \bigg] $ and the lower bound is attained. Then for $ \upsilon := \inf_{\bar{\mathcal{B}}_r(\muv^*) \times \bar{\mathcal{B}}_r(\Lambda^*)} \mathbf{K}_{\bar{\mathcal{B}}_r(\muv^*) \times \bar{\mathcal{B}}_r(\Lambda^*)} $, the constant
    $  C_{\upsilon} :=\frac{\sqrt{|\upsilon|} \, \mathrm{diam}(S_2) }{\tanh (\sqrt{|\upsilon|} \, \mathrm{diam}(S_2))} $ must be $1$ since $\upsilon = 0$. 
    
    Then under \textbf{A1-A2}, \textbf{A1'-A3'} and Lemma \ref{suplemPLbound0} for $\beta =2$, all the conditions \textbf{C1-C4} from section \ref{jointrategeneralsection} are satisfied for the function $ \wt f(\w ,\probP) := \mathcal{L}( \muv, \Lambda ; \w  )$ on $S_1 \times S_2$. Invoking Theorem \ref{convergenceratethm_exactPL} completes the proof.
\end{proof}

\begin{theorem}\label{convergenceratethm_betaPL_specialcase}
 Under \textbf{A1-A2} for some $p>1$ and \textbf{A1'-A3'} for $\beta \in (1,2)$, suppose Algorithm \ref{algRGA:alternating_updates} is initialized with $(\w_0 , \probP_0) $ where $\probP_0 := \probP(\muv^*, \Lambda^* )$ such that $  \mathrm{dist}(\w_0, \mathcal{S}_c^*(\probP_0)) \le \frac{\delta}{2} \,  $. Let $h_1>0$ be sufficiently small so that
    $h_1^{\beta - 1}\left(\frac{\beta}{\beta-1} \right) C_{\beta}^{-1} L^{2\beta - 2} \, (2\sqrt{2} r )^{\beta - 1} < \frac{C^{1/\beta}_{\beta}}{L} $ and $ \frac{1+\eta}{2} \le 1-\frac{\mu h_1 }{3}$ where $\eta  = 1-h_1\mu+\frac{9 \norm{(\Lambda^*)^{-1}}^2_2}{16 \gamma^4} L^2 h_1^2  + \mathcal{O}(h_1^3) <1$. 
Suppose for some large enough $K_0 \gg 1$ and a local minima connected component $ \mathcal{S}_c^*(\probP_{K_0}) $ we have that $$  \mathrm{dist}(\w_{K_0}, \mathcal{S}_c^*(\probP_{K_0})) \le \min \bigg \{  \,\frac{\delta}{2},  \frac{C^{1/\beta}_{\beta}}{L} -  h_1^{\beta - 1}\left(\frac{\beta}{\beta-1} \right) C_{\beta}^{-1} L^{2\beta - 2} \, (2\sqrt{2} r )^{\beta - 1} \bigg\} \, . $$  For all $k \ge 0$ define $J:= J(k) = k^{\frac{2\alpha}{(\beta -1 )^2}} $ for any constant $\alpha > {1-(\beta -1 )^2\theta} $ , $\theta = \frac{\beta (\beta -1)}{2} \in (0,1) $, and $K_0 \gg 1$ is large enough so that 
    $$ \bigg(   1+ J(K_0) \,C'h_2  \bigg)^{\frac{\beta-1}{2}} > \frac{2\Bigg(C_{L,\beta} \, \delta^{\frac{\beta (\beta -1) }{2}} +   C_{L,\beta} \bigg(  L^{\beta -1} \, \left(\frac{\beta}{\beta-1} \right) C_{\beta}^{-1} \, (2\sqrt{2} r )^{\beta - 1} \bigg)^{\frac{\beta (\beta -1) }{2}} \Bigg)}{\delta}  \, $$
     where $0 < h_2 < 1/L$, $C_{L,\beta} :=   \left(\frac{\beta}{\beta-1} \right) \frac{L ^{\frac{\beta (\beta -1)}{2}}}{ C_{\beta}^{\frac{\beta-1}{2} + \frac{1}{\beta }} } $, $ C' := (1 - \tfrac{Lh_2}{2}) C^{2/\beta}_{\beta} \le C^{2/\beta}_{\beta} $. Then the iterate sequence $\{\w_k\}_{k \ge K_0}$ stays bounded in $ \{ \mathcal{S}_c^*(\probP_k) + \mathcal{B}_{\delta/2}(\mathbf{0})\}_{k \ge K_0}$ where the local minima connected components in the sequence $\{ \mathcal{S}_c^*(\probP_k)\}_{k \ge K_0} $ are at most $\delta/4$ separated. Also suppose that $\mathrm{dist}_H(\mathcal{S}_c^*(\probP_{k-1}), \mathcal{S}_c^*(\probP_k)) $, for all $k \ge 0$, satisfies a Lipschitz estimate for any $\beta \in (1,2) $ as follows:
    \begin{align*}
        \mathrm{dist}_H(\mathcal{S}_c^*(\probP_{k-1}), \mathcal{S}_c^*(\probP_k)) \le  D_{\beta} \, \mathrm{dist}(\probP_{k-1},\probP_k).
    \end{align*} Further, let the iterates $\{(\w_k , \probP_k)\}_k$ from Algorithm \ref{algRGA:alternating_updates} satisfy the following growth condition uniformly for all $k \ge K_0$ :
     $$\mathrm{dist}(\mathcal{S}_c^*(\probP_{k+1}), \mathcal{S}_c^*(\probP_k))  = \mathcal{O}(\mathrm{dist}(\w_{k}, \mathcal{S}_c^*(\probP_{k}))) \, .$$
     Then the iteration $\{(\w_k, \probP_k)\}_k$ from Algorithm \ref{algRGA:alternating_updates} converges to a basin saddle point $ (\w^{\vardiamond},\probP^{\vardiamond}) $ (definition \ref{basinsaddledef}) of the function $ \wt f(\w ,\probP) := \mathcal{L}( \muv, \Lambda ; \w  )$ at a rate $\mathcal{O}( k^{- \frac{\alpha}{1-(\beta -1 )^2\theta}})$. 
\end{theorem}
\begin{proof}
    Since $S_2 = \bar{\mathcal{B}}_r(\muv^*) \times \bar{\mathcal{B}}_r(\Lambda^*) $ we get that $\mathrm{diam}(S_2) = 2\sqrt{2} r$. Then the sectional curvature $\mathbf{K}_{\bar{\mathcal{B}}_r(\muv^*) \times \bar{\mathcal{B}}_r(\Lambda^*)} $ of the product manifold $\bar{\mathcal{B}}_r(\muv^*) \times \bar{\mathcal{B}}_r(\Lambda^*) $ satisfies $ \mathbf{K}_{\bar{\mathcal{B}}_r(\muv^*) \times \bar{\mathcal{B}}_r(\Lambda^*)} \in \bigg[0, \frac{3 \norm{(\Lambda^*)^{-1}}^2_2}{4 \gamma^4} \bigg] $ as shown in Theorem \ref{convergenceratethm_exactPL_specialcase} proof. Then under \textbf{A1-A2}, \textbf{A1'-A3'} and Lemma \ref{suplemPLbound0} for $\beta \in (1,2)$, all the conditions \textbf{C1-C4} from section \ref{jointrategeneralsection} are satisfied for the function $ \wt f(\w ,\probP) := \mathcal{L}( \muv, \Lambda ; \w  )$ on $S_1 \times S_2$. Invoking Theorem \ref{convergenceratethm_betaPL} completes the proof.
\end{proof}
We now make several remarks in relation to Theorems \ref{convergenceratethm_exactPL_specialcase} and \ref{convergenceratethm_betaPL_specialcase}.
\begin{rem}[Iteration complexity and growth condition]
    Observe that the convergence rate of $ \mathcal{O}( k^{- \frac{\alpha}{1-(\beta -1 )^2\theta}})$ from Theorem \ref{convergenceratethm_betaPL_specialcase} depends on a free parameter $\alpha$ that is only lower bounded with $\alpha > {1-(\beta -1 )^2\theta} $. Without any upper bound on $\alpha$ one can take $\alpha$ arbitrary large and improve the exponent of convergence rate arbitrarily. However, the convergence rate of $ \mathcal{O}( k^{- \frac{\alpha}{1-(\beta -1 )^2\theta}})$ from Theorem \ref{convergenceratethm_betaPL_specialcase} only describes the error rate for the outer loop iterations of Algorithm \ref{algRGA:alternating_updates}. The  Algorithm \ref{algRGA:alternating_updates} still performs $J:= J(k)  = k^{\frac{2\alpha}{(\beta -1 )^2}}$ gradient descent steps or $J$ inner loop iterations within each outer loop iteration $k$. Hence, after $k$ outer iterations of Algorithm \ref{algRGA:alternating_updates} the total iterations (inner and outer) performed by Algorithm \ref{algRGA:alternating_updates} are $ \mathcal{O}(k^{\frac{2\alpha}{(\beta -1 )^2} + 1} + k)$ which scales exponentially with $\alpha$. Therefore, an arbitrary large $\alpha$ creates a natural trade-off between convergence rates and total iteration complexity. Next, in Theorem \ref{convergenceratethm_betaPL_specialcase}, the growth condition, on the Hausdorff distance between the local minima connected components from different statistical loss landscapes, could be hard to verify in general but is not vacuous. In fact, statistical loss landscapes with invariant minimizer sets under distribution shifts will trivially satisfy this growth condition. 
\end{rem}

\begin{rem}[Optimality of rates and initialization conditions]
    Note that in Theorems \ref{convergenceratethm_exactPL_specialcase} and \ref{convergenceratethm_betaPL_specialcase}, the convergence rate depends on the parameter $p>1$ that controls the $p$-th moment of the absolute loss function from \textbf{A2} with respect to a known reference Gaussian measure $\probP(\muv^*, \Lambda^*)$. In particular all the crucial parameters governing the rate of convergence depend on $p$ such as $\epsilon$ from the Lagrangian $\mathcal{L}$, the local geodesic strong concavity parameter $\mu$ and the uniform Lipschitz constant $L$. For $p>1$ all these constants are bounded and so are the constants appearing within the convergence rate expressions. However computing the optimal $p>1$, that minimizes the constants within the convergence rates from Theorems \ref{convergenceratethm_exactPL_specialcase}, \ref{convergenceratethm_betaPL_specialcase} and that simultaneously maximizes the upper bound on $\epsilon$ to give the largest possible ambiguity set for our DRO problem, is beyond the scope of this work. {Last, the initialization conditions on $\w_0, \w_{K_0}$ from Theorems \ref{convergenceratethm_exactPL_specialcase}, \ref{convergenceratethm_betaPL_specialcase} respectively may appear to be restrictive in the sense of their dependencies on constants $\delta, C_{\beta}, L$ which in turn are controlled by the population loss landscape, problem dimension and a reference covariance matrix. However, such initialization constraints make practical sense while solving a DRO problem. For instance, treating the reference measure $\probP(\muv^*, \Lambda^*) $ as a distribution on training data and $\w_0$ as the trained model will imply that $\w_0$ is sufficiently close to some local minima connected component $\mathcal{S}^*_c(\probP_0)$ of the population loss/Lagrangian\footnote{The critical points of the population or statistical loss and the Lagrangian on $\mathbb{R}^d$ for any fixed probability measure $\probP(\muv, \Lambda)$ are identical.}. The closeness of $\w_0$ to $\mathcal{S}^*_c(\probP_0)$ will obviously depend on the training algorithm, in our case we may assume that the training was performed on population loss rather than empirical \footnote{Training model on empirical loss (instead of population loss) using SGD as proxy for statistical gradients will converge in probability (under certain assumptions on loss landscape) to a local minima connected component $\hat{\mathcal{S}}^*_c(\probP_0)$ of the empirical loss. The gap between the empirical minima $\hat{\mathcal{S}}^*_c(\probP_0)$ and population minima $\mathcal{S}^*_c(\probP_0)$ can possibly be controlled via central limit theorems along the lines of \cite{mei2018landscape} under some strong assumptions which is left for future.}. In that scenario $\w_0$ will trivially satisfy the initialization condition with respect to $\mathcal{S}^*_c(\probP_0)$.}  
\end{rem}
\begin{rem}[Lipschitz continuity of the distance function]
    The Lipschitz continuity of the Hausdorff distance, between at most $\delta/4$ separated local minima connected components $ \mathcal{S}_c^*(\probP_{k-1}), \mathcal{S}_c^*(\probP_k)$, given by $$ \mathrm{dist}_H(\mathcal{S}_c^*(\probP_{k-1}), \mathcal{S}_c^*(\probP_k)) \le  D_{\beta} \, \mathrm{dist}(\probP_{k-1},\probP_k) \, , $$ and required in Theorems \ref{convergenceratethm_exactPL_specialcase} and \ref{convergenceratethm_betaPL_specialcase}, holds for $\beta = 2$ case with a Lipschitz constant possibly different from $D_{\beta}$ (apply Lemma \ref{RGA-MGD1-lem0} to $\wt f(\w,\probP):= \mathcal{L}( \muv, \Lambda ; \w  )$). In fact the Hausdorff distance between at most $\delta/4$ separated local minima connected components is at least H\"{o}lder continuous with exponent $(\beta-1)$ for $\beta \in (1,2)$, i.e., $$ \mathrm{dist}_H(\mathcal{S}_c^*(\probP_{k-1}), \mathcal{S}_c^*(\probP_k)) \le  D_{\beta} \, \mathrm{dist}^{\beta -1}(\probP_{k-1},\probP_k)$$ from Lemma \ref{RGA-MGD1-lem0} (for the function $\wt f(\w,\probP)$). But H\"{o}lder continuity by itself is not sufficient to generate iterate convergence rate in Theorem \ref{convergenceratethm_betaPL_specialcase} so we assume Lipschitz continuity instead to derive rates. The Lipschitz continuity condition will be trivially satisfied in statistical loss landscapes with invariant minimizer sets under distribution shifts and under some Morse-Bott constructions like Appendix \ref{appendixcountereg2}. Proving Lipschitz continuity of the Hausdorff distance for general class of population losses is an avenue for future work. 
\end{rem}

\begin{rem}[Scope of Lipschitz regularity and pathological cases]\label{rmk:drift-scope}
 Our proof technique allowed us to derive convergence rates when the Hausdorff distance between local minima set of statistical loss is Lipschitz continuous in the probability measure. The Lipschitz continuity of the distance function of the solution map $y \mapsto \mathcal{S}^*(y)$ relates to the pseudo-Lipschitz regularity of critical points (zeros of a parameterized family of set valued maps) under local perturbations of parameters from \citet{aubin1984lipschitz} and is along the lines of the strong regularity condition (Theorem 2.1) of the zeros of a parametric family of Fr\'{e}chet differentiable maps in \citet{robinson1980strongly}. In general, the solution map $y \mapsto \mathcal{S}^*(y)$ may be merely outer semicontinuous, in which case the distance function need not even be H\"{o}lder continuous. Another could be a genuine deterioration of regularity, for instance consider a problem where a connected component of the critical set, that is governed by an analytic equation and is a smooth curve, bifurcates into multiple curves (a pitchfork bifurcation) as distribution changes. Such constructions will need a separate analysis altogether and possibly a generalized definition of the basin saddles to work with. We leave both extensions to future work.
\end{rem}

{We remind that our guarantees concern the population risk and its landscape under measure perturbation; a Monte Carlo study would instead probe a finite-sample/SGD proxy governed by a separate concentration analysis, which we leave to future work (Section \ref{futureworksection}). Appendix \ref{appendixcountereg2} gives a closed-form example verifying that assumptions (\textbf{C1–C4, A1'–A3'}) are non-vacuous.}

\section{Conclusion \& future work}\label{futureworksection}
We presented a convergence analysis of a Riemannian gradient-ascent type method {for} the min-max optimization problem $\min_{x \in \mathcal{X}} \max_{y \in\mathcal{Y}} f(x,y) $, that is defined over the product of a Euclidean space $\mathcal{X}$ and a Riemannian manifold $\mathcal{Y}$ with bounded sectional curvature. Under certain regularity conditions on the function $f$ and bounded initializations, we obtained linear convergence rates (local P\L{} condition on $f(\cdot,y)$) and polynomial convergence rates (local K\L{} condition on $f(\cdot,y)$) to a saddle point of the problem. We then implemented this abstract convergence framework on the penalized DRO problem with Gaussian measures. We derived explicit estimates for the Hessian of a certain Lagrangian in the BW geometry that supplied bounds for the local strong geodesic concavity and Lipschitz parameters. Finally we showed that under reasonable assumptions on the statistical loss landscape (nice critical sets with bounded drift under distribution variation and local \L{}ojasiewicz property), algorithm \ref{algRGA:alternating_updates} recovers the linear and polynomial convergence rates from the abstract convergence framework. This paper opens up several interesting avenues for future work.

First, the penalized DRO problem that we solved for Gaussian measures can be extended to more complex settings such as mixture of Gaussian measures, parametric family of distributions, Gaussian measures supported on smooth manifolds such as the $n$-dimensional sphere instead of $\mathbb{R}^n$, etc., provided these measures induce some Riemannian structure like the BW manifold geometry. It will be also interesting to see explicit calculations in the Riemannian metric for these families of distributions and also what type of Lagrangian formulations will generate local geodesic strong concavity from the Riemannian Hessian calculations. Second, the empirical version of the penalized DRO problem can be studied by adapting this framework to empirical/counting measures. Under that setting the abstract convergence framework may need to be re-derived so that the space $\mathcal{Y}$ can accommodate counting measures. Techniques from {linear optimal transport} could be leveraged for transporting finite set of data-points using empirical measures. Third, the local \L{}ojasiewicz property on the fibers of $f(\cdot, y)$ can be relaxed to a more general nonconvex structure, such as the Morse property. Then the $\delta$ basin around the set of local minima of $f(\cdot, y)$ can possibly have strict saddle points. Such non-benign geometry will cover a very large class of statistical loss landscapes. Last, proving rigorous upper bounds on the distance between local critical sets of statistical loss landscapes as a function of distribution shifts still remains an open problem. We plan to address these in our future work.

\color{black}{}

\bibliographystyle{plainnat}
\bibliography{reference.bib}

\newpage

\appendix

\section{Proofs for section \ref{sectionboumalref} }\label{boumalappendix}

\subsection{Proof of Theorem \ref{boumalthm1} }
\begin{proof}
By Lipschitz continuity of the gradient, for a geodesic $c(t) = \operatorname{exp}_{x_k}
\left(
-t \nabla_{x} f(x_k)
\right)$ with $c(0) = x_k$ and all $t$ in a suitable interval:
\[
f(c(t)) \le f(x_k) - t\left(1 - \frac{tL}{2}\right)
\|\nabla_{x} f(x_k)\|_{x_k}^2.
\]
Suppose $x_k \in S_0$ then for $ t < 2/L$, we have that $f(c(t)) = f(x_{k+1})\le f(x_k)  $ implying $x_{k+1} \in S_0$. 
Then $ \{x_k\} \subset S_0$ as long as $ t < 2/L$.
Setting $t = h < 1/L$ yields:
\begin{equation}
f(x_{k+1})
\le f(x_k) - h\left(1 - \frac{hL}{2}\right)
\|\nabla_{x} f(x_k)\|_{x_k}^2.
\end{equation}

Next, it can be shown using MVT that the unique minimizer $x^*$ of $f$ in the interior of $\mathcal{M}$ is a critical point of $f$. Subtracting $f(x^*)$:
\begin{equation}
f(x_{k+1}) - f(x^*)
\le f(x_k) - f(x^*)
- h\left(1 - \frac{hL}{2}\right)\|\nabla_{x} f(x_k)\|_{x_k}^2.
\end{equation}

Geodesic strong convexity of $f$ on $S_0$ gives:
\begin{equation}
\|\nabla_{x} f(x_k)\|_{x_k}^2
\ge 2\mu \bigl(f(x_k) - f(x^*)\bigr).
\end{equation}

Combining:
\begin{equation}
f(x_{k+1}) - f(x^*)
\le \left(1 - 2 \mu h\left(1 - \frac{hL}{2}\right)\right)
\bigl(f(x_k) - f(x^*)\bigr).
\end{equation}

Furthermore, by geodesic strong convexity of $f$ at $x^*$ we also have
\begin{equation}
\operatorname{dist}(x_k,x^*)^2
\le \frac{2}{\mu}
\bigl(f(x_k) - f(x^*)\bigr).
\end{equation}

 Since $v = \exp_{x_k}^{-1}(x_{k+1}) \in T_{x_k} \mathcal{M}$ and parallel transport ${\Psi}$ preserves inner product, we also have the following inequality along similar lines as in Lemma \ref{contractionlema3_supplement} :
    \begin{align}
      L\, \mathrm{dist}^2(x_k, x_{k+1} ) \ge \langle  {\Psi}_{x_{k+1} \to x_{k}} ( \nabla_{x} f(x_{k+1})) - \nabla_{x} f(x_{k}), \exp_{x_{k}}^{-1} (x_{k+1}) \rangle_{x_{k}} \ge \mu  \, \mathrm{dist}^2(x_{k}, x_{k+1} ) \label{comparisonineq0a}
    \end{align}

Hence, for $ \kappa = \frac{1}{\mu h(2- Lh)}$ we get
\begin{equation}
f(x_k) - f(x^*)
\le \left(1 - \frac{1}{\kappa}\right)^k
\bigl(f(x_0) - f(x^*)\bigr),
\end{equation}
and
\begin{equation}
\operatorname{dist}(x_k,x^*)
\le \sqrt{\frac{2(f(x_0)-f(x^*))}{\mu}}
\left(1 - \frac{1}{\kappa}\right)^{k/2},
\end{equation}
where $  \kappa = \frac{1}{\mu h(2- Lh)} \ge 1$ for  $0< h < 1/L$.

Using gradient Lipschitz continuity at $x^*$:
\begin{equation}
f(x_0) \le f(x^*)
+ \frac{L}{2}\operatorname{dist}(x_0,x^*)^2 \, ,
\end{equation}
we have:
\begin{align}
 \operatorname{dist}(x_k,x^*)
\le \sqrt{\frac{L}{\mu}}
\left(1 - \frac{1}{\kappa}\right)^{k/2}  \operatorname{dist}(x_0,x^*).   \label{rateestimate1}
\end{align}

Alternatively, if the sectional curvature of $\mathcal{M}$ is lower bounded by $\upsilon$ then using Lemma \ref{alexandrovlem1} for a geodesic triangle with vertices $x_k , x_{k+1}, x^*$, edge lengths $ a = \operatorname{dist}(x_{k+1},x^*), b= \operatorname{dist}(x_k,x_{k+1}), c= \operatorname{dist}(x_k,x^*)$ and using compactness of $ \mathcal{M}$, continuity and monotonicity of $g(y) := \frac{y}{\tanh (y)} > 0$ for $y \ge 0$ to uniformly bound $$\frac{\sqrt{|\upsilon|}c}{\tanh (\sqrt{|\upsilon|}c)} \le \sup_{x_k \in \mathcal{M}} \frac{\sqrt{|\upsilon|}c}{\tanh (\sqrt{|\upsilon|}c)} = \frac{\sqrt{|\upsilon|} \, \mathrm{diam}(\mathcal{M}) }{\tanh (\sqrt{|\upsilon|} \, \mathrm{diam}(\mathcal{M}))} =  C_{\upsilon}  $$ on $\mathcal{M}$ gives the following:
\begin{align}
    \operatorname{dist}^2(x_{k+1},x^*) &\le \operatorname{dist}^2(x_{k},x^*) +  \frac{\sqrt{|\upsilon|}c}{\tanh (\sqrt{|\upsilon|}c)}\operatorname{dist}^2(x_k,x_{k+1}) \nonumber \\ & -2  \operatorname{dist}(x_{k},x^*) \operatorname{dist}(x_k,x_{k+1}) \cos ( \angle \exp_{x_k}^{-1}(x^*), \exp_{x_k}^{-1}(x_{k+1})   ) \\
    & = \operatorname{dist}^2(x_{k},x^*) +  \frac{\sqrt{|\upsilon|}c}{\tanh (\sqrt{|\upsilon|}c)}\norm{h \nabla_{x} f(x_k) }^2_{x_k}   -2 \langle   \exp_{x_k}^{-1}(x^*),- h \nabla_{x} f(x_k) \rangle_{x_k} \\
    & \le \operatorname{dist}^2(x_{k},x^*) + C_{\upsilon}\norm{h \nabla_{x} f(x_k) }^2_{x_k}   -2 \langle   \exp_{x_k}^{-1}(x^*), - h \nabla_{x} f(x_k) \rangle_{x_k} \\ 
     & = \operatorname{dist}^2(x_{k},x^*) + C_{\upsilon}\norm{h \nabla_{x} f(x_k) - h \underbrace{ {\Psi}_{x^* \to x_k} (\nabla_{x} f(x^*))}_{= \mathbf{0}} }^2_{x_k} \nonumber \\ & -2 \langle   \exp_{x_k}^{-1}(x^*),h{\Psi}_{x^* \to x_k} (\nabla_{x} f(x^*)) - h \nabla_{x} f(x_k) \rangle_{x_k} \\
     & \hspace{-2cm}\underbrace{\le}_{\textbf{Lipschitzness of } \nabla_{x} f(\cdot), \, \eqref{comparisonineq0a}} \operatorname{dist}^2(x_{k},x^*) + C_{\upsilon} L^2 h^2 \operatorname{dist}^2(x_{k},x^*)  -2 \mu h \operatorname{dist}^2(x_{k},x^*) \\ 
     & = (1 -2 \mu h + C_{\upsilon} L^2 h^2) \operatorname{dist}^2(x_{k},x^*) 
\end{align}
where $ 1 - 2 \mu h + C_{\upsilon} L^2 h^2< 1$ for any sufficiently small $h$. Therefore, combining \eqref{rateestimate1} we can write the bound:
\begin{align}
    \operatorname{dist}(x_k,x^*)
\le   \min \bigg\{ \sqrt{\frac{L}{\mu}}
\left(1 - \frac{1}{\kappa}\right)^{k/2} , (1 -2 \mu h + C_{\upsilon} L^2 h^2)^{k/2}  \bigg\} \operatorname{dist}(x_0,x^*).
\end{align}
\end{proof}

\section{Proofs for section \ref{sectionfiberrates} }\label{sectionfiberratesappendix}

\subsection{Proof of Lemma \ref{contractionlema2} }
\begin{proof}
We first show that the sequence $\{x_k\}$ stays bounded in $  \mathcal{S}^*(y) + \mathcal{B}_{\delta}(0)$. Consider the set 
$$ W = (\mathcal{S}^*(y) + \mathcal{B}_{\delta}(0) ) \, \cap \, \{ x : f(x,y) - f_*(y)  \leq  (\tfrac{\beta}{\beta -1})^{\frac{-\beta}{\beta -1}} C_{\beta}^{\frac{1}{\beta - 1}} (\tfrac{\delta}{2})^{\frac{\beta}{\beta -1}} \} $$
Then by \textbf{C3} we have for any $u \in W$ that :
   \begin{align}
       (\tfrac{\beta}{\beta -1})^{\frac{-\beta}{\beta -1}} C_{\beta}^{\frac{1}{\beta - 1}} (\mathrm{dist}(u, \mathcal{S}^*(y)))^{\frac{\beta}{\beta -1}} \leq  f(u,y) - f_*(y) & \leq (\tfrac{\beta}{\beta -1})^{\frac{-\beta}{\beta -1}} C_{\beta}^{\frac{1}{\beta - 1}} (\tfrac{\delta}{2})^{\frac{\beta}{\beta -1}} \\
       \implies \mathrm{dist}(u, \mathcal{S}^*(y)) & \leq \frac{\delta}{2} \quad \forall \, \, u\in W \label{levelsetcontainment1}
   \end{align}
Define the set 
$$ Z = (\mathcal{S}^*(y) + \mathcal{B}_{\delta}(0) )^c \, \cap \, \{ x : f(x,y) - f_*(y)  \leq  (\tfrac{\beta}{\beta -1})^{\frac{-\beta}{\beta -1}} C_{\beta}^{\frac{1}{\beta - 1}} (\tfrac{\delta}{2})^{\frac{\beta}{\beta -1}} \} $$
and thus for any $v \in Z$ we must have $\mathrm{dist}(v, \mathcal{S}^*(y))  \geq {\delta} $.
Then $W$ is separated from $Z$ as follows:
\begin{align*}
    \mathrm{dist}_H(W,Z) &\ge \abs{\mathrm{dist}_H(W,\mathcal{S}^*(y)) - \mathrm{dist}_H(Z,\mathcal{S}^*(y))} \ge \tfrac{\delta}{2},
\end{align*}
by the triangle inequality for the Hausdorff distance.

   We have that $x_0 \in W$. Suppose for some $k>0$, $ x_k \in W $ and thus $ f(x_k) \leq f_*(y) + (\tfrac{\beta}{\beta -1})^{\frac{-\beta}{\beta -1}} C_{\beta}^{\frac{1}{\beta - 1}} (\tfrac{\delta}{2})^{\frac{\beta}{\beta -1}}$. Next, we have the following inequality from gradient Lipschitz continuity of $f(\cdot,y)$ on $ S_1 \times y$ :
    \begin{align}
        f(x_{k+1},y) &\leq f(x_k,y) + \langle \nabla_x f(x_k,y) , x_{k+1} - x_k \rangle + \frac{L}{2}\norm{x_{k+1} - x_k}^2   \\
        &= f(x_k,y) + \langle \nabla_x f(x_k,y) , - h \nabla_x f(x_k,y) \rangle + \frac{L}{2}\norm{h \nabla_x f(x_k,y)}^2   \\
        &= f(x_k,y) - h(1 - \tfrac{Lh}{2})\norm{ \nabla_x f(x_k,y)}^2  \underbrace{\leq}_{Lh<1} f(x_k , y) \label{errorrecursionsublinear0} 
        \end{align}
        and thus $ f(x_{k+1}{,y}) \leq  f(x_k{,y}) \leq  f_*(y) + (\tfrac{\beta}{\beta -1})^{\frac{-\beta}{\beta -1}} C_{\beta}^{\frac{1}{\beta - 1}} (\tfrac{\delta}{2})^{\frac{\beta}{\beta -1}}$. Hence, 
        $$x_{k+1} \in \{ x : f(x,y) - f_*(y)  \leq  (\tfrac{\beta}{\beta -1})^{\frac{-\beta}{\beta -1}} C_{\beta}^{\frac{1}{\beta - 1}} (\tfrac{\delta}{2})^{\frac{\beta}{\beta -1}} \} = W \cup Z$$ where $W,Z$ are $\tfrac{\delta}{2}$ separated. It must be that $x_{k+1} $ either belongs to $W$ or it belongs to $Z$. 
        
   From gradient Lipschitz continuity of $f(\cdot,y)$ on $ S_1 \times y$ we also have that:
        \begin{align}
            \norm{x_{k+1} - x_k } = h \norm{\nabla_x f(x_{k} ,y)} &\le L h \mathrm{dist}(x_{k}, \mathcal{S}^*(y)) \\
            \implies \mathrm{dist}(x_{k+1}, \mathcal{S}^*(y))  \le \mathrm{dist}(x_{k}, \mathcal{S}^*(y)) + \norm{x_{k+1} - x_k } &\le (1+Lh) \mathrm{dist}(x_{k}, \mathcal{S}^*(y)) \\ & \underbrace{\leq}_{\eqref{levelsetcontainment1}} (1+Lh)\tfrac{\delta}{2} \\ & \underbrace{<}_{Lh <1} \delta
        \end{align}
        hence $x_{k+1} \in W$ and thus by induction we have the sequence $ \{x_{k}\}_{k} \subset W \subset  \mathcal{S}^*(y) + \mathcal{B}_{\delta}(0)$. 
        From \eqref{errorrecursionsublinear0} we then have for any $k$ that:
        \begin{align}
    f(x_{k+1},y)    & \underbrace{\leq}_{\textbf{C3}}  f(x_k,y) - h(1 - \tfrac{Lh}{2}) C^{2/\beta}_{\beta}\bigg(f(x_k,y) - f_*(y)\bigg)^{2/\beta} \leq f(x_k, y) \\
        \implies f(x_{k+1},y) - f_*(y) &\leq    f(x_k,y) -f_*(y) - h(1 - \tfrac{Lh}{2}) C^{2/\beta}_{\beta}\bigg(f(x_k,y) - f_*(y)\bigg)^{2/\beta} \le f(x_k,y) -f_*(y)  \label{errorrecursionsublinear1} \\
                 \implies f(x_{k+1},y) - f_*(y) &\leq   \underbrace{\bigg( 1 - h(1 - \tfrac{Lh}{2}) C^{\tfrac{2}{\beta}}_{\beta}\bigg(f(x_k,y) - f_*(y)\bigg)^{\tfrac{2}{\beta}-1}\bigg)}_{\rho_k \le 1 \textbf{ for $h \ll 1$}} [f(x_k,y) -f_*(y)]  \label{errorrecursionsublinear1a}  \\
           & \leq \prod_{j=0}^{k} \rho_j [f(x_0,y) -f_*(y)] .\label{errorrecursionsublinear2}
        \end{align}    
        Using \textbf{C3} in the inequality \eqref{errorrecursionsublinear1a} we get that:
         \begin{align*}
           \left(\frac{\beta}{\beta-1} \right)^{\frac{-\beta}{\beta-1}}C_{\beta}^{\frac{1}{\beta-1}}\bigg( \mathrm{dist}(x_{k+1}, \mathcal{S}^*(y))\bigg)^{\frac{\beta}{\beta -1}}  &\leq  f(x_{k+1},y) - f_*(y)  \leq  \rho_k [f(x_k,y) -f_*(y)]   \\
            \implies \bigg( \mathrm{dist}(x_{k+1}, \mathcal{S}^*(y))\bigg)^{\frac{\beta}{\beta -1}} & \leq \frac{\rho_k}{\left(\frac{\beta}{\beta-1} \right)^{\frac{-\beta}{\beta-1}}C_{\beta}^{\frac{1}{\beta-1}}} [f(x_k,y) -f_*(y)] \\
             & \leq \frac{ \prod_{j=0}^k \rho_j}{\left(\frac{\beta}{\beta-1} \right)^{\frac{-\beta}{\beta-1}}C_{\beta}^{\frac{1}{\beta-1}}} [f(x_0,y) -f_*(y)] \\
             & \hspace{-2.5cm} \underbrace{\le}_{\textbf{C3 and Lipschitz $\nabla_x f(\cdot,y)$}} \, \, \frac{ L^{\beta} \prod_{j=0}^k \rho_j}{C_{\beta}\left(\frac{\beta}{\beta-1} \right)^{\frac{-\beta}{\beta-1}}C_{\beta}^{\frac{1}{\beta-1}}} \bigg( \mathrm{dist}(x_{0}, \mathcal{S}^*(y))\bigg)^{{\beta}}   
    \end{align*}
   which completes first part of the lemma.
   
  Next, for any $ k>0$, define $e_k := f(x_k,y) -f_*(y)$, $ C' := (1 - \tfrac{Lh}{2}) C^{2/\beta}_{\beta} $, $ r := \tfrac{2}{\beta} \in [1,2)$  for simplicity in \eqref{errorrecursionsublinear1} to write :
    \begin{align}
        e_{k+1} & \leq e_k -  C'h \, e_k^{r} \\ 
        \iff -(e_{k+1} - e_k) & \geq    C'h \, e_k^{r} \label{errorrecursionsublinear3}
    \end{align}
    Then for a convex function $\phi(t) := t^{-r}$ for $t>0$, $r >1$ and using monotonicity of $\{e_k\}$, i.e., $ e_0 \ge e_1 \ge \cdots \ge e_k \ge  \cdots   $ from \eqref{errorrecursionsublinear1} we have :
    \begin{align}
        \implies \phi(e_{k+1}) - \phi(e_k) &\ge \phi'(e_k) (e_{k+1} - e_k) = -r e^{-r-1}_k (e_{k+1} - e_k) \\
        & \underbrace{\ge}_{\eqref{errorrecursionsublinear3}}  r e^{-r-1}_k \times  C'h \, e_k^{r} =  C'h \, r e_k^{-1} \geq  C'h \, r e_0^{-1} \\
        \implies \sum_{j=0}^{k-1} (\phi(e_{j+1}) - \phi(e_j)) = \phi(e_{k}) - \phi(e_0) & \ge k \,C'h \,  r e_0^{-1} \\
        \implies e^{-r}_k & \ge e^{-r}_0 + k \,C'h \, r e_0^{-1}  \underbrace{\ge}_{e_0 <1} e_0^{-1} (1+ k \,C'h \, r) \\
        \implies e_k & \le \frac{e_0^{\frac{1}{r}}}{\bigg( 1+ k \,C'h \, r \bigg)^{\frac{1}{r}}} 
    \end{align}
    where we used that $ e_0 = f(x_0,y) -f_*(y) \le \frac{(L \mathrm{dist}(x_{0}, \mathcal{S}^*(y)))^{\beta}} {C_{\beta}}  <1 $ which holds true for $\mathrm{dist}(x_{0}, \mathcal{S}^*(y)) <  \frac{C^{1/\beta}_{\beta}}{L}$.
    Finally, using \textbf{C3} in the last inequality we get
        \begin{align}
                 \left(\frac{\beta}{\beta-1} \right)^{\frac{-\beta}{\beta-1}}C_{\beta}^{\frac{1}{\beta-1}}\bigg( \mathrm{dist}(x_{k}, \mathcal{S}^*(y))\bigg)^{\frac{\beta}{\beta -1}}  &\leq  e_k \leq  \frac{e_0^{\frac{1}{r}}}{\bigg( 1+ k \,C'h \, r \bigg)^{\frac{1}{r}}} \leq  \frac{(L \mathrm{dist}(x_{0}, \mathcal{S}^*(y)))^{\frac{\beta}{r}}}{\bigg( C_{\beta} ( 1+ k \,C'h \, r) \bigg)^{\frac{1}{r}}}
        \end{align}
        and with $r = 2/\beta $, $ r \in [1,2)$ we get that
        \begin{align}
            \bigg( \mathrm{dist}(x_{k}, \mathcal{S}^*(y)) \bigg)^{\frac{\beta }{\beta -1}} & \le \left(\frac{\beta}{\beta-1} \right)^{\frac{\beta}{\beta-1}} \frac{(L )^{\frac{\beta}{r}}}{ C_{\beta}^{\frac{1}{r} + \frac{1}{\beta -1 }} \bigg(   1+ k \,C'h \, r \bigg)^{\frac{1}{r}}} \bigg( \mathrm{dist}(x_{0}, \mathcal{S}^*(y)) \bigg)^{\frac{\beta }{r}} \\
            & \le \left(\frac{\beta}{\beta-1} \right)^{\frac{\beta}{\beta-1}} \frac{(L )^{\frac{\beta^2}{2}}}{ C_{\beta}^{\frac{\beta}{2} + \frac{1}{\beta -1 }} \bigg(   1+ k \,C'h \, r \bigg)^{\frac{\beta}{2}}} \bigg( \mathrm{dist}(x_{0}, \mathcal{S}^*(y)) \bigg)^{\frac{\beta^2 }{2}} \\
            \implies \bigg( \mathrm{dist}(x_{k}, \mathcal{S}^*(y)) \bigg) & \leq C_{L,\beta}  \frac{\bigg( \mathrm{dist}(x_{0}, \mathcal{S}^*(y)) \bigg)^{\frac{\beta (\beta -1) }{2}}}{\bigg(   1+ k \,C'h \, r \bigg)^{\frac{\beta-1}{2}}}  \le  C_{L,\beta}  \frac{\bigg( \mathrm{dist}(x_{0}, \mathcal{S}^*(y)) \bigg)^{\frac{\beta (\beta -1) }{2}}}{\bigg(   1+ k \,C'h  \bigg)^{\frac{\beta-1}{2}}}
        \end{align}
        where $C_{L,\beta} :=   \left(\frac{\beta}{\beta-1} \right) \frac{L ^{\frac{\beta (\beta -1)}{2}}}{ C_{\beta}^{\frac{\beta-1}{2} + \frac{1}{\beta }} } $, $ C' := (1 - \tfrac{Lh}{2}) C^{2/\beta}_{\beta} $.
\end{proof}

\subsection{Proof of Lemma \ref{contractionlema3_supplement} }
\begin{proof}
        We have the following inequalities by \textbf{C1-C2} for the function $-f$ (section 11.5 of \cite{boumal2023introduction}) for any $x \in S_1 $, any $y_1, y_2 \in S_2$ and $\forall \quad t\in [0,1] $:
    \begin{align*}
        f(x,\exp_{y_1}(tv)) &\ge f(x,y_1) + t\langle \nabla_y f(x,y_1) ,v\rangle_{y_1} + \frac{t^2}{2}\mu \norm{v}^2_{y_1}  \quad  ; \quad \exp_{y_1}(v) = y_2 \\
         f(x,\exp_{y_1}(tv)) &\le f(x,y_1) + t\langle \nabla_y f(x,y_1) ,v\rangle_{y_1} + \frac{t^2}{2}L \norm{v}^2_{y_1} \quad  ; \quad \exp_{y_1}(v) = y_2 
    \end{align*}
    where $ \norm{v}^2_{y_1} = \mathrm{dist}^2(y_{1}, y_{2} ) \, \, $ by a property of the Riemannian exponential since $\exp_{y_1}(v) = y_2  $.
    
    Interchanging $y_1, y_2$ in both inequalities followed by substituting $v$ with ${\Psi}_{y_{1} \to y_{2}}(-v)$ where ${\Psi}_{y_{1} \to y_{2}} $ is the parallel transport map from $ y_{1} $ to $ y_{2}$ along the curve $c_1(t)=\exp_{y_1}(tv)$ yields:
        \begin{align*}
        f(x,\exp_{y_2}(t{\Psi}_{y_{1} \to y_{2}}(-v))) &\ge f(x,y_2) + t\langle \nabla_y f(x,y_2) ,{\Psi}_{y_{1} \to y_{2}}(-v)\rangle_{y_2} + \frac{t^2}{2}\mu \underbrace{\norm{{\Psi}_{y_{1} \to y_{2}}(-v)}^2_{y_2}}_{=\mathrm{dist}^2(y_{1}, y_{2} )} \quad \forall \quad t\in [0,1] ; \\
         f(x,\exp_{y_2}(t{\Psi}_{y_{1} \to y_{2}}(-v))) &\le f(x,y_2) + t\langle \nabla_y f(x,y_2) ,{\Psi}_{y_{1} \to y_{2}}(-v)\rangle_{y_2} + \frac{t^2}{2}L \underbrace{\norm{{\Psi}_{y_{1} \to y_{2}}(-v)}^2_{y_2}}_{=\mathrm{dist}^2(y_{1}, y_{2} )} \quad \forall \quad t\in [0,1]  
    \end{align*}
    where we used the fact that $ \exp_{y_2}({\Psi}_{y_{1} \to y_{2}}(-v)) = y_1$.
     Setting $t= 1$, then adding the corresponding inequalities in $\mu ,L $ gives:
     \begin{align*}
       0 &\ge -\langle \nabla_y f(x,y_2) ,{\Psi}_{y_{1} \to y_{2}}(v)\rangle_{y_2} +\langle \nabla_y f(x,y_1) ,v\rangle_{y_1} + \mu \,\mathrm{dist}^2(y_{1}, y_{2} )  \\
      0 &\le -\langle \nabla_y f(x,y_2) ,{\Psi}_{y_{1} \to y_{2}}(v)\rangle_{y_2} +\langle \nabla_y f(x,y_1) ,v\rangle_{y_1} + L \, \mathrm{dist}^2(y_{1}, y_{2} )  
    \end{align*}
    Since $v = \exp_{y_1}^{-1}(y_2) \in T_{y_1}M$ and parallel transport preserves inner product, we can simplify the above inequalities as :
    \begin{align}
      L\, \mathrm{dist}^2(y_{1}, y_{2} ) \ge \langle  {\Psi}_{y_{2} \to y_{1}} (\nabla_y f(x, y_{2})) - \nabla_y f(x, y_1), \exp_{y_1}^{-1} (y_{2}) \rangle_{y_1} \ge \mu  \, \mathrm{dist}^2(y_{1}, y_{2} ) 
    \end{align}
    where ${\Psi}_{y_{2} \to y_{1}} $ is the parallel transport map from $ y_{2} $ to $ y_{1}$.
    
    \emph{Proof that } $\exp_{y_2}(\Psi_{y_1\to y_2}(-v)) = y_1$. Since $c_1(t) = \exp_{y_1}(tv)$ is a geodesic, its velocity is parallel so $\Psi_{y_1\to y_2}v = \Psi_{y_1\to y_2}c_1'(0)=c_1'(1)$. Letting $c_2(t)=c_1(1-t)$ gives $c_2'(t) = -c_1'(1-t)$. As a result
    \[
    \nabla_{c_2'(t)}c_2'(t) = \nabla_{c'(1-t)}c'(1-t) = 0 \text{ and } c_2'(0)=-c_1'(1) = \Psi_{y_1\to y_2}(-v).
    \]
    Thus, $t\mapsto c_2(t)$ and $t\mapsto \exp_{y_2}(t\Psi_{y_1\to y_2}(-v))$ both solve the same initial value problem so by uniqueness, we must have $c_2(t) = \exp_{y_2}(t\Psi_{y_1\to y_2}(-v))$. Thus, we get 
    \[
    \exp_{y_2}(\Psi_{y_1\to y_2}(-v)) = c_2(1) = c_1(0) = y_1.
    \]
\end{proof}

\subsection{Proof of Lemma \ref{contractionlema3} }
\begin{proof}
     Let $c$ be a path such that $c(0)=y_{k-1}$ and $c(1) =y_k$. Then  
    \begin{align*}
    \rm{dist}(y_{k+1},y_k) &= \rm{dist}(G(y_{k}), G(y_{k-1})) \\
                           &\leq \int_0^1 \norm{\frac{d}{dt}G\circ c(t)}_{c(t)} dt \\
                           &= \int_0^1 \norm{dG_{c(t)}(c'(t))}_{c(t)}dt  \\
                           &\leq \sup_{t\in [0,1]}\norm{dG_{c(t)}}_{op} \int_0^1 \norm{c'(t)}_{c(t)}dt    
    \end{align*}
    where $\norm{dG_{c(t)}}_{op}$ is the operator norm of $dG$. Taking $c$ to be $c(t)= \exp_{y_{k-1}}(th\nabla_yf(x,y_{k-1}))$  gives the identity
    \[
    \rm{dist}(G(y_k), G(y_{k-1})) \leq \sup_{t\in [0,1]}\norm{dG_{c(t)}}_{op} \rm{dist}(y_k, y_{k-1}).
    \]
    Fix a $t_0\in [0,1]$ and let $c(t_0)=p$. Let $c_1(t) = \exp_p(th\nabla_yf(x,p))$, then for $w\in T_p\cm$ the chain rule gives
    \begin{equation}\label{dG}
    dG(w) = d_1\exp_p(h\nabla_y f(x,p))(w) + d_2\exp_p(h\nabla_y f(x,p))(\nabla_w\nabla_y f(x,p)).    
    \end{equation}
    By Proposition \ref{expder} we can rewrite equation \eqref{dG} as
    \begin{align*}
        dG(w) = J(1)
    \end{align*}
    where $J$ is a Jacobi field along $c_1(t) = \exp_p(th\nabla_yf(x,p))$ satisfying 
    \[
    J(0) = w, \quad \text{ and }\quad D_tJ(0) = h\nabla_w\nabla_y f(x,p) =h\text{Hess}_yf(x,p)(w).
    \]
    Departing from the usual notation of the paper, here $\nabla_w$ is the covariant derivative in the direction $w$. Thus,
    \begin{align*}
        \norm{dG_p}_{op} &= \sup_{w\in T_p\cm} \frac{\norm{dG_p(w)}_{G(p)}}{\norm{w}_p}\\
                    &= \sup_{w\in T_p\cm } \frac{\norm{J(1)}_{G(p)}}{\norm{w}_p}\\
                    &= \sup_{w\in T_p\cm} \frac{\norm{\cp_{c_1(0)\to c_1(1)}^{-1}J(1)}_p}{\norm{w}_p}\\
                    &= \sup_{w\in T_p\cm} \frac{\norm{w+h\text{Hess}_yf(w) - \frac{1}{2}R(w,h\nabla_yf)h\nabla_yf + \mathcal{O}(h^3 \norm{w}_p)}_p}{\norm{w}_p} \quad (\text{Proposition }\ref{tayexp})\\
                    &\leq 1-h\mu + \mathcal{O}(h^2),
    \end{align*}
    since $\text{Hess}_pf \leq -\mu \cdot \text{Id}_{T_p\cm}$. By the symmetries of the Riemannian curvature tensor, we have
    \[
    \langle R(\nabla_y f,\nabla_yf)\nabla_yf, w\rangle_p = \langle R(w,\nabla_y f)\nabla_yf, \nabla_yf\rangle_p =0 \text{ for all } w\in T_p\cm.
    \]
    Let $W\subset T_p\cm$ be the linear subspace consisting of all vectors orthogonal to $\nabla_yf(x,p)$. For fixed $u,w\in W$, we write $u= \proj_w(u) + \proj_{w^{\perp}}(u) $ where $\proj_{w^{\perp}}(u)$ denotes the projection of $u$ onto the orthogonal complement of $w$. Consequently, 
    \begin{align}
    \frac{\abs{\langle R(w,\nabla_yf)\nabla_yf,u\rangle_p}}{\norm{w}_p\norm{u}_p} &\leq  \frac{\abs{\langle R(w,\nabla_yf)\nabla_yf,\proj_w(u)\rangle_p}}{\norm{w}_p\norm{u}_p} +  \frac{\abs{\langle R(w,\nabla_yf)\nabla_yf,\proj_{w^{\perp}}(u)\rangle_p}}{\norm{w}_p\norm{u}_p} \nonumber \\
    &\leq \frac{\abs{\langle R(w,\nabla_yf)\nabla_yf,\proj_{w}(u)\rangle_p}}{\norm{w}_p\norm{\proj_{w}(u)}_p} +  \frac{\abs{\langle R(w,\nabla_yf)\nabla_yf,\proj_{w^{\perp}}(u)\rangle_p}}{\norm{w}_p\norm{\proj_{w^{\perp}}(u)}_p} \label{secrest}\\
    &\leq \Theta\norm{\nabla_y f}_p^2 + \frac{1}{2}\Theta \norm{\nabla_y f}_p^2.
    \end{align}
    Here the estimate on the first term of \eqref{secrest} follows from the definition of sectional curvature, while the estimate on the second term follows from Berger's curvature estimate (equation (3) of \cite{kar70}).  Thus,
    \begin{align*}
        \sup_{w\neq 0\in T_pM} \frac{\norm{R(w,\nabla_yf)\nabla_yf}_p}{\norm{w}_p}
        &= \sup_{w\neq 0\in W} \frac{\norm{R(w,\nabla_yf)\nabla_yf}_p}{\norm{w}_p} \\
        &= \sup_{w\neq 0\in W} \sup_{u\neq 0\in W} \frac{\abs{\langle R(w,\nabla_yf)\nabla_yf,u\rangle_p}}{\norm{w}_p\norm{u}_p} \\
        &\leq \frac{3}{2}\Theta \norm{\nabla_y f}_p^2.
    \end{align*}
    Hence we get that \[ \eta =  1-h\mu + \frac{3}{4}\Theta h^2 \sup_{p \in S_2}\norm{\nabla_y f}_p^2 + \mathcal{O}(h^3) \approx 1-h\mu + \frac{3}{4}\Theta h^2 L^2  \,  \]
    where the constant in Big-O term above is supremum of the norm of third order terms from the expansion in Proposition \ref{tayexp}. The norm of such remainder terms will be bounded uniformly in $S_2$. 
    This completes the proof.
\end{proof}

\subsection{Proof of Lemma \ref{lem:length}}
\begin{proof}
We fix a $y$ and suppress it for brevity: we write $g_{k}:=\Dx f(x_{k},y)$ and
$\Delta_{k}:=f(x_{k},y)-f_*(y)\ge0$. Since $\Dx f(\cdot,y)$ is $L$-Lipschitz,
the descent lemma applied to gradient descent on the fiber $ S_1 \times y$ gives the
estimate
\begin{equation}\label{eq:sd}
  \Delta_{k}-\Delta_{k+1}\ \ge\ h\Bigl(1-\tfrac{Lh}{2}\Bigr)\norm{g_{k}}^{2}
  \ =\ h(1-\tfrac{Lh}{2})\,\norm{g_{k}}^{2}\ \ge\ 0 ,
\end{equation}
so $\{\Delta_{k}\}$ is nonincreasing. 

If $\Delta_{m}=0$ for some $m$, then $x_{m}$ minimizes $f(\cdot,y)$, hence
$g_{m}=0$ and $x_{k}=x_{m}$ for all $k\ge m$; so $\sum_{k\geq m} \norm{x_{k+1}-x_k} =0$. We may therefore assume $\Delta_{k}>0$ for all $k$.

Define the function $\varphi: [0,\infty) \to \mathbb{R}$
\begin{equation}\label{eq:phi}
  \varphi(s):=\frac{\beta}{\beta-1}\left(\frac{s^{\beta-1}}{C_{\beta}}\right)^{1/\beta},
  \qquad\text{so that}\qquad
  \varphi'(s)=\bigl(C_{\beta}\,s\bigr)^{-1/\beta}.
\end{equation}
Since $\beta\in(1,2]$, the exponent $\tfrac{\beta-1}{\beta}$ lies in
$(0,\tfrac12]$, so $\varphi$ is nonnegative, increasing and concave on
$[0,\infty)$. Condition \textbf{C3} for $x=x_k$ gives
$\norm{g_{k}}\ge(C_{\beta}\Delta_{k})^{1/\beta}$, which means
\begin{equation}\label{eq:desing}
  \varphi'(\Delta_{k})\,\norm{g_{k}}\ \ge\ 1 .
\end{equation}

Concavity of $\varphi$ together with $\Delta_{k+1}\le\Delta_{k}$ yields
$\varphi(\Delta_{k})-\varphi(\Delta_{k+1})\ge\varphi'(\Delta_{k})\bigl(\Delta_{k}-\Delta_{k+1}\bigr)$.
Combining this with \eqref{eq:sd} and then \eqref{eq:desing},
\begin{equation}\label{eq:key}
  \varphi(\Delta_{k})-\varphi(\Delta_{k+1})
  \ \ge\ h(1-\tfrac{Lh}{2})\,\varphi'(\Delta_{k})\norm{g_{k}}^{2}
  \ \ge\ h(1-\tfrac{Lh}{2})\,\norm{g_{k}}
  \ =\ (1-\tfrac{Lh}{2})\,\norm{x_{k+1}-x_{k}} ,
\end{equation}
the last equality being the definition of the gradient descent. Summing
\eqref{eq:key} over $k=0,1,\dots,N$ and using $\varphi\ge0$ gives
\begin{equation*}
 (1-\tfrac{Lh}{2})\sum_{k=0}^{N}\norm{x_{k+1}-x_{k}}
  \ \le\ \varphi(\Delta_{0})-\varphi(\Delta_{N+1})
  \ \le\ \varphi(\Delta_{0}) .
\end{equation*}
Letting $N\to\infty$, recalling \eqref{eq:phi} and using the \textbf{C3} bound as 
$$ L^{\beta}\mathrm{dist}^{\beta}(x, \mathcal{S}^*(y))  \ge \norm{\nabla_x f(x,y) }^{\beta}  \geq C_{\beta}\bigg(f(x,y) - f_*(y)\bigg) $$
proves \eqref{eq:length}.
\end{proof}

\section{Proofs for section \ref{sectionLyapunovconvergence} }\label{sectionLyapunovconvergenceappendix}

{\begin{lem}[Lipschitz continuity of the set--distance function]\label{lem:dist-hausdorff}
Let $(\mathcal{M},d)$ be a metric space and let $A,B\subseteq\mathcal{M}$ be nonempty. For $x\in\mathcal{M}$ and $S\subseteq\mathcal{M}$ nonempty write
\[
  \operatorname{dist}(x,S):=\inf_{s\in S}d(x,s),
  \qquad
  e(S,T):=\sup_{s\in S}\operatorname{dist}(s,T),
  \qquad
  \operatorname{dist}_H(A,B):=\max\bigl\{e(A,B),\,e(B,A)\bigr\},
\]
so that $e$ is the directed excess and $\operatorname{dist}_H$ the Hausdorff distance. Then for every $x\in\mathcal{M}$,
\begin{equation}\label{eq:dist-directed}
  \operatorname{dist}(x,A)\;\le\;\operatorname{dist}(x,B)+e(B,A),
\end{equation}
and consequently
\begin{equation}\label{eq:dist-lip}
  \bigl|\operatorname{dist}(x,A)-\operatorname{dist}(x,B)\bigr|\;\le\;\operatorname{dist}_H(A,B).
\end{equation}
In particular $S\mapsto\operatorname{dist}(x,S)$ is $1$-Lipschitz with respect to $\operatorname{dist}_H$, uniformly in $x$.
\end{lem}

\begin{proof}
Fix $\varepsilon>0$. Since $B\neq\emptyset$, the definition of the infimum yields $b\in B$ with
\[
  d(x,b)\;\le\;\operatorname{dist}(x,B)+\varepsilon .
\]
Because $b\in B$ and $e(B,A)$ is the supremum of $\operatorname{dist}(\,\cdot\,,A)$ over $B$, we have $\operatorname{dist}(b,A)\le e(B,A)$. Since $A\neq\emptyset$, a second appeal to the definition of the infimum yields $a\in A$ with
\[
  d(b,a)\;\le\;\operatorname{dist}(b,A)+\varepsilon\;\le\;e(B,A)+\varepsilon .
\]
The triangle inequality in $(\mathcal{M},d)$ now gives
\[
  \operatorname{dist}(x,A)\;\le\;d(x,a)\;\le\;d(x,b)+d(b,a)\;\le\;\operatorname{dist}(x,B)+e(B,A)+2\varepsilon .
\]
Letting $\varepsilon\downarrow 0$ proves \eqref{eq:dist-directed}. Exchanging the roles of $A$ and $B$ gives $\operatorname{dist}(x,B)\le\operatorname{dist}(x,A)+e(A,B)$, and since $\operatorname{dist}_H(A,B)$ dominates both $e(A,B)$ and $e(B,A)$, the two bounds combine to \eqref{eq:dist-lip}.
\end{proof}
}

\subsection{Proof of Lemma \ref{RGA-MGD1-lem0} }
\begin{proof}
    We have that:
    \begin{align*}
       \mathrm{dist}_H(\mathcal{S}^*(y_{k+1}), \mathcal{S}^*(y_k)) 
        & = \max \{ \sup_{x^*(y_{k+1}) \in  \mathcal{S}^*(y_{k+1})}  \mathrm{dist}(x^*(y_{k+1}), \mathcal{S}^*(y_k)) , \sup_{x^*(y_{k}) \in  \mathcal{S}^*(y_{k})}  \mathrm{dist}(x^*(y_{k}), \mathcal{S}^*(y_{k+1})) \} 
        \end{align*}
        Bounding the first supremum (both are symmetric) we get:
        \begin{align*}
      \sup_{x^*(y_{k+1}) \in  \mathcal{S}^*(y_{k+1})}  \mathrm{dist}(x^*(y_{k+1}), \mathcal{S}^*(y_k))  & \underbrace{\le}_{\textbf{C3}} \sup_{x^*(y_{k+1}) \in  \mathcal{S}^*(y_{k+1})}   \left(\frac{\beta}{\beta-1} \right) C_{\beta}^{-1}\norm{\nabla_x f(x^*(y_{k+1}),y_{k}) }^{\beta - 1} .
    \end{align*}
Applying the fundamental theorem of calculus to each component of $t\mapsto \nabla_x(x^*(y_{k+1}), c(t))$ we get :
\[  \nabla_x f(x^*(y_{k+1}),y_{k}) ={\underbrace{\nabla_x f(x^*(y_{k+1}), y_{k+1})}_{=\mathbf{0}} +  \int_{0}^1 \langle\nabla_y \nabla_x f(x^*(y_{k+1}), c(t))  , c'(t) \rangle_{c(t)} dt }\]
for the geodesic $c(t)$ with $c(1) = y_k$, $c(0) = y_{k+1}$. Then we have the uniform bound for any $ {x^*(y_{k+1}) \in  \mathcal{S}^*(y_{k+1})}$:
\begin{align*}
    \norm{\nabla_x f(x^*(y_{k+1}),y_{k}) }  &= \norm{\int_{0}^1 \langle\nabla_y \nabla_x f(x^*(y_{k+1}), c(t))  , c'(t) \rangle_{c(t)} dt } \\
    & \le \sup_{t \in [0,1]} \norm{\nabla_y \nabla_x f(x^*(y_{k+1}), c(t))}_{c(t)}  \int_{0}^1  \norm{ c'(t)}_{c(t)} dt \\
    & \underbrace{\le}_{\textbf{C4}} L  \, \mathrm{dist}(y_{k+1},y_k)
\end{align*}
which completes first part of the proof. For the second part we have the bound:
\begin{align*}
    \mathrm{dist}_H(\mathcal{S}^*(y_{k+1}), \mathcal{S}^*(y_k)) 
          & \underbrace{\le}_{\textbf{C3}} \sup_{x^*(y_{k+1}) \in  \mathcal{S}^*(y_{k+1})}  \left(\frac{\beta}{\beta-1} \right) C_{\beta}^{-1}\norm{\nabla_x f(x^*(y_{k+1}),y_{k}) }^{\beta - 1} \\
          & \le \left(\frac{\beta}{\beta-1} \right) C_{\beta}^{-1} L^{\beta - 1} \, { \mathrm{dist}^{\beta - 1}(y_{k+1},y_k) } \\
          & =  \left(\frac{\beta}{\beta-1} \right) C_{\beta}^{-1} L^{\beta - 1} \,  \norm{ \exp^{-1}_{y_k}(y_{k+1}) }^{\beta - 1}_{y_k} \\
           & =  \left(\frac{\beta}{\beta-1} \right) C_{\beta}^{-1} L^{\beta - 1} h_1^{\beta - 1}\,  \norm{ \nabla_y f(x_k, y_{k}) }^{\beta - 1}_{y_k} \\
            & \le h_1^{\beta - 1} \left(\frac{\beta}{\beta-1} \right) C_{\beta}^{-1} L^{2\beta - 2} \, \mathrm{diam}^{\beta - 1}(S_2)
\end{align*}
where in the last step we used that $y_k \in x_k \times S_2$ and so $  \norm{ \nabla_y f(x_k, y_{k}) } \le L \, \mathrm{dist}(y_k, y^*(x_k))  \le L \, \mathrm{diam}(S_2) $.
This completes the second part of the proof.
\end{proof}

\subsection{Proof of Lemma \ref{RGA-MGD1-lem2} }
\begin{proof}
    From Lemma \ref{contractionlema2}, non-negativity of $ f(x_{k}, y_{k}) - f_*(y_{k})$ and triangle inequality we can write:
    \begin{align}
        [f(x_{k+1}, y_{k+1}) - f_*( y_{k+1}) ] &\le \bigg(\prod_{j=0}^{J-1} \rho_{k +\frac{j}{J}} \bigg) [f(x_{k}, y_{k}) - f_*(y_{k}) ]  +   \bigg(\prod_{j=0}^{J-1} \rho_{k +\frac{j}{J}} \bigg) |f(x_{k}, y_{k+1}) - f(x_{k}, y_{k})|  \nonumber \\ & + \bigg(\prod_{j=0}^{J-1} \rho_{k +\frac{j}{J}} \bigg)| f_*( y_{k+1}) - f_*( y_{k})|
    \end{align}
    Using Lipschitzness of $f(x_k, \cdot)$ the second summand is bounded $ |f(x_{k}, y_{k+1}) - f(x_{k}, y_{k})| \le L \, \mathrm{dist}(y_{k+1},y_k) $. The third summand will have the bound:
    \begin{align}
        | f_*( y_{k+1}) - f_*( y_{k})| &= \bigg\lvert \inf_{x \in \mathcal{S}^*(y_{k+1}) + \mathcal{B}_{\delta}(0)} f(x, y_{k+1}) - \inf_{x \in \mathcal{S}^*(y_k) + \mathcal{B}_{\delta}(0)} f(x, y_{k}) \bigg\rvert  \\
         & =  \bigg\lvert  f(\mathcal{S}^*(y_{k+1}), y_{k+1}) -  f(\mathcal{S}^*(y_{k}), y_{k}) \bigg\rvert  \\
         & \hspace{-3cm} \underbrace{\le}_{\textbf{for any } x^*(y_{k+1}) \in \mathcal{S}^*(y_{k+1})} \bigg\lvert  f(x^*(y_{k+1}), y_{k+1}) -  f(x^*(y_{k+1}), y_{k}) \bigg\rvert + \bigg\lvert  f(x^*(y_{k+1}), y_{k}) -  f(\mathcal{S}^*(y_{k}), y_{k}) \bigg\rvert \\
         & \hspace{-3cm} \underbrace{\le}_{\textbf{C3. and Lipschitz } f(x,\cdot) } L \, \mathrm{dist}(y_{k+1},y_k)  \, \, + \, \, \frac{1}{C_{\beta}}  \norm{\nabla_x f(x^*(y_{k+1}), y_{k}) }^{\beta} \\
        & \hspace{-3cm} = L \, \mathrm{dist}(y_{k+1},y_k)  \, \, + \, \, \frac{1}{C_{\beta}}  \norm{\underbrace{\nabla_x f(x^*(y_{k+1}), y_{k+1})}_{=\mathbf{0}} +  \int_{0}^1 \langle\nabla_y \nabla_x f(x^*(y_{k+1}), c(t))  , c'(t) \rangle_{c(t)} dt }^{\beta} \\
        & \hspace{-3cm} \underbrace{\le}_{\textbf{Cauchy Schwarz}} L \, \mathrm{dist}(y_{k+1},y_k)  \, \, + \, \, \frac{1}{C_{\beta}}  \norm{  \bigg(\sup_{t \in [0,1]} \norm{\nabla_y \nabla_x f(x^*(y_{k+1}), c(t))}_{c(t)} \bigg) \int_{0}^1 \norm{c'(t)}_{c(t)} dt }^{\beta} \\
        & \hspace{-3cm} \le L \, \mathrm{dist}(y_{k+1},y_k)  \, \, + \, \, \frac{L^{\beta}}{C_{\beta}} \mathrm{dist}^{\beta}(y_{k+1},y_k) 
    \end{align}
    where in the second summand of the third last step, we used the fundamental theorem of calculus to each component of $t\mapsto \nabla_x(x^*(y_{k+1}),c(t))$ and in the last step we used that $t\mapsto c(t)$ is the distance minimizing geodesic in the geodesically convex set $S_2$. Note that intuitively $ \nabla_y \nabla_x f(x^*(y_{k+1}), c(t))$ behaves like a Jacobian matrix with each row in the tangent space of $c(t)$ and the inner product $\langle\nabla_y \nabla_x f(x^*(y_{k+1}), c(t))  , c'(t) \rangle_{c(t)} $ is a vector defined coordinate wise by taking inner products in the tangent space of $c(t)$. Putting everything together completes the proof:
        \begin{align}
        [f(x_{k+1}, y_{k+1}) - f_*( y_{k+1}) ] &\le \bigg(\prod_{j=0}^{J-1} \rho_{k +\frac{j}{J}} \bigg) [f(x_{k}, y_{k}) - f_*(y_{k}) ]   +  L \bigg(\prod_{j=0}^{J-1} \rho_{k +\frac{j}{J}} \bigg) \mathrm{dist}(y_{k+1},y_k) \nonumber \\ &  + \bigg(\prod_{j=0}^{J-1} \rho_{k +\frac{j}{J}} \bigg) \bigg(L \, \mathrm{dist}(y_{k+1},y_k)  \, \, + \, \, \frac{L^{\beta}}{C_{\beta}} \mathrm{dist}^{\beta}(y_{k+1},y_k) \bigg) \\
        & \le  \bigg(\prod_{j=0}^{J-1} \rho_{k +\frac{j}{J}} \bigg) [f(x_{k}, y_{k}) - f_*(y_{k}) ]  \nonumber \\ & +  L \, \mathrm{dist}(y_{k+1},y_k)\bigg(\prod_{j=0}^{J-1} \rho_{k +\frac{j}{J}} \bigg) \bigg( 2 + \frac{L^{\beta -1}}{C_{\beta}} \mathrm{dist}^{\beta -1}(y_{k+1},y_k) \bigg) . 
    \end{align}  
    The last part follows from Lemma \ref{contractionlema2}, Lemma \ref{RGA-MGD1-lem0}, Lemma \ref{lem:dist-hausdorff} (triangle inequality) and concavity (hence sub-additivity) of function $ g(t) := t^{\beta -1}$ for $\beta \in (1,2]$ as follows:
    \begin{align*}
        \mathrm{dist}(x_{k+1}, \mathcal{S}^*(y_{k+1})) & \le  \tau_k^{\frac{\beta -1}{\beta}} \mathrm{dist}^{\beta -1}(x_{k}, \mathcal{S}^*(y_{k+1})) \\
        & \le  \tau_k^{\frac{\beta -1}{\beta}} \bigg( \mathrm{dist}(x_{k}, \mathcal{S}^*(y_k)) + \mathrm{dist}_H(\mathcal{S}^*(y_{k+1}), \mathcal{S}^*(y_k))\bigg)^{\beta -1} \\
        & \le  \tau_k^{\frac{\beta -1}{\beta}} \mathrm{dist}^{\beta -1}(x_{k}, \mathcal{S}^*(y_k)) +   \tau_k^{\frac{\beta -1}{\beta}} \mathrm{dist}_H^{\beta -1}(\mathcal{S}^*(y_{k+1}), \mathcal{S}^*(y_k)) \\
         & {\le} \tau_k^{\frac{\beta -1}{\beta}} \mathrm{dist}^{\beta -1}(x_{k}, \mathcal{S}^*(y_k)) +   \tau_k^{\frac{\beta -1}{\beta}} \bigg( L^{\beta -1} \, \left(\frac{\beta}{\beta-1} \right) C_{\beta}^{-1} \, \mathrm{dist}^{\beta - 1}(y_{k+1},y_k) \bigg)^{\beta -1}  \, .
    \end{align*}
    Further, if $ \mathrm{dist}(x_{k}, \mathcal{S}^*(y_{k})) <  \frac{C^{1/\beta}_{\beta}}{L} -   h_1^{\beta - 1}\left(\frac{\beta}{\beta-1} \right) C_{\beta}^{-1} L^{2\beta - 2}\, \mathrm{diam}^{\beta - 1}(S_2)$ for $h_1$ sufficiently small, then from triangle inequality in Lemma \ref{lem:dist-hausdorff} and Lemma \ref{RGA-MGD1-lem0} we get
    \[
    \mathrm{dist}(x_{k}, \mathcal{S}^*(y_{k+1})) \le \mathrm{dist}(x_{k}, \mathcal{S}^*(y_{k}))  + \mathrm{dist}_H(\mathcal{S}^*(y_{k+1}), \mathcal{S}^*(y_k))  < \frac{C^{1/\beta}_{\beta}}{L} .
    \]
   Hence from Lemma \ref{contractionlema2}, Lemma \ref{RGA-MGD1-lem0} and concavity (hence sub-additivity) of function $ g(t) := t^{\frac{\beta(\beta-1)}{2}}$ for $\beta \in (1,2]$ we have that:
        \begin{align*}
        \mathrm{dist}(x_{k+1}, \mathcal{S}^*(y_{k+1})) & \le  C_{L,\beta}  \frac{\bigg( \mathrm{dist}(x_{k}, \mathcal{S}^*(y_{k+1})) \bigg)^{\frac{\beta (\beta -1) }{2}}}{\bigg(   1+ J \,C'h_2  \bigg)^{\frac{\beta-1}{2}}} \\
       & \le C_{L,\beta}  \frac{\bigg( \mathrm{dist}(x_{k}, \mathcal{S}^*(y_k)) + \mathrm{dist}_H(\mathcal{S}^*(y_{k+1}), \mathcal{S}^*(y_k)) \bigg)^{\frac{\beta (\beta -1) }{2}}}{\bigg(   1+ J \,C'h_2  \bigg)^{\frac{\beta-1}{2}}} 
       \end{align*}
       \begin{align*}
        & \le  C_{L,\beta}  \frac{\bigg( \mathrm{dist}(x_{k}, \mathcal{S}^*(y_k)) \bigg)^{\frac{\beta (\beta -1) }{2}}}{\bigg(   1+ J \,C'h_2  \bigg)^{\frac{\beta-1}{2}}}  + C_{L,\beta}  \frac{\bigg(  \mathrm{dist}_H(\mathcal{S}^*(y_{k+1}), \mathcal{S}^*(y_k)) \bigg)^{\frac{\beta (\beta -1) }{2}}}{\bigg(   1+ J \,C'h_2  \bigg)^{\frac{\beta-1}{2}}} \\
        & \le  C_{L,\beta}  \frac{\bigg( \mathrm{dist}(x_{k}, \mathcal{S}^*(y_k)) \bigg)^{\frac{\beta (\beta -1) }{2}}}{\bigg(   1+ J \,C'h_2  \bigg)^{\frac{\beta-1}{2}}}  + C_{L,\beta}  \frac{\bigg(  L^{\beta -1} \, \left(\frac{\beta}{\beta-1} \right) C_{\beta}^{-1} \, \mathrm{dist}^{\beta - 1}(y_{k+1},y_k) \bigg)^{\frac{\beta (\beta -1) }{2}}}{\bigg(   1+ J \,C'h_2  \bigg)^{\frac{\beta-1}{2}}} 
    \end{align*}
        where $C_{L,\beta} :=   \left(\frac{\beta}{\beta-1} \right) \frac{L ^{\frac{\beta (\beta -1)}{2}}}{ C_{\beta}^{\frac{\beta-1}{2} + \frac{1}{\beta }} } $, $ C' := (1 - \tfrac{Lh_2}{2}) C^{2/\beta}_{\beta} $.
    This completes the proof.
\end{proof}

\subsection{Proof of Lemma \ref{RGA-MGD1-lem3} }
\begin{proof}
 Applying triangle inequality followed by Lemma \ref{contractionlema1} gives:
       \begin{align}
        \mathrm{dist}(y_{k+1},y^*(x_{k+1}))  & \le   \, \mathrm{dist}(y_{k+1},y^*(x_{k})) +  \, \mathrm{dist}(y^*(x_{k+1}),y^*(x_k)) \\ &\le  \wt\gamma \, \mathrm{dist}(y_{k},y^*(x_k)) +  \, \mathrm{dist}(y^*(x_{k+1}),y^*(x_k))
    \end{align}
    Next, upper bounding $  \mathrm{dist}(y^*(x_{k+1}),y^*(x_k)) $ gives:
    \begin{align}
         \mathrm{dist}(y^*(x_{k+1}),y^*(x_k)) &\underbrace{\le}_{\eqref{comparisonineq1a} \textbf{, Cauchy Schwarz}} \frac{1}{\mu} \norm{  {\Psi}_{y^*(x_{k+1}) \to y^*(x_{k})} (\nabla_y f(x_k, y^*(x_{k+1}))) - \underbrace{\nabla_y f(x_k ,y^*(x_k))}_{=\mathbf{0}} }_{y^*(x_k)}  \\
         &\underbrace{\le}_{\textbf{isometry of } {\Psi}} \frac{1}{\mu} \norm{   \nabla_y f(x_k, y^*(x_{k+1}))  }_{y^*(x_{k+1})} \\
         & \hspace{-4cm} \underbrace{=}_{\substack{\textbf{parallel transport via $\Phi $ } \\ \textbf{between isomorphic fibers at $x_k, x_{k+1}$ ,} \\  \Phi \, \, : \, \, x_k \times S_2 \, \, \to \, \, x_{k+1} \times S_2 }} \frac{1}{\mu} \norm{   {\Phi}_{x_k \to x_{k+1}}  (\nabla_y f(x_k, y^*(x_{k+1})) ) - \underbrace{\nabla_y f(x_{k+1} ,y^*(x_{k+1}))}_{= \mathbf{0}} }_{y^*(x_{k+1})} \\
         & \hspace{-2cm}\underbrace{\le}_{\textbf{Lipschitzness of } \nabla_y f(\cdot,y^*(x_{k+1}) )} \frac{L}{\mu} \norm{x_{k+1} - x_k}
    \end{align}
    and then putting everything together completes the proof. 
\end{proof}

\subsection{Proof of Lemma \ref{RGA-MGD1-lem4} }
\begin{proof}
    From the definition \eqref{RGA-MGD1} we have that $ \mathrm{dist}(y_{k+1},y_k) = \mathrm{dist}(G(x_k,y_{k}),G(x_{k-1},y_{k-1})) $. Using triangle inequality, contractiveness of $G(x, y) := \exp_{y}(h_1 \nabla_y f(x, y))$ in $y$ we can write:
    {
     \begin{align}
           \mathrm{dist}(G(x_k,y_{k}),G(x_{k-1},y_{k-1})) &\le \mathrm{dist}(G(x_k,y_{k}),G(x_{k},y_{k-1})) + \mathrm{dist}(G(x_k,y_{k-1}),G(x_{k-1},y_{k-1})) \\
           & \hspace{-1cm}\underbrace{\le}_{\textbf{Lemma \ref{contractionlema3}, C4}}   \eta \, \,\mathrm{dist}(y_{k-1},y_k)  +  L h_1 \norm{ \nabla_y f(x_k, y_{k-1}) -  \nabla_y f(x_{k-1}, y_{k-1})}_{y_{k-1}} \\
           & \underbrace{\le}_{\textbf{C4}}  \, \eta \, \,\mathrm{dist}(y_{k-1},y_k)  +  L^2 h_1 \norm{x_k - x_{k-1}} \label{beta1_2tempineq1} \\
           & \hspace{-1cm} \underbrace{\le}_{\textbf{Lemma } \ref{RGA-MGD1-lem4z}}  \, \eta \, \,\mathrm{dist}(y_{k-1},y_k) +   L^2 h_1  \frac{\beta \, L^{\beta-1} }{(1-\tfrac{Lh_2}{2})(\beta-1) C_{\beta}} \bigg( \mathrm{dist}^{\beta-1}(x_{k-1}, \mathcal{S}^*(y_{k-1}))  \nonumber \\ & +  \mathrm{dist}_H^{\beta-1}(\mathcal{S}^*(y_{k-1}), \mathcal{S}^*(y_k)) \bigg)  \label{I1'sourceqn}
    \end{align}
    }
    where $ \eta =  1-h_1\mu+\frac{3}{4}\Theta h_1^2\left(\sup_{p\in S_2} \norm{\nabla_yf}_p^2\right) \le 1-h_1\mu+\frac{3}{4}\Theta L^2 h_1^2 $ from \textbf{C4}. 
    Last, suppose $\mathrm{dist}_H(\mathcal{S}^*(y_{k-1}), \mathcal{S}^*(y_k)) $ satisfies a sharper Lipschitz estimate as opposed to the H\"{o}lder estimate in Lemma \ref{RGA-MGD1-lem0} with constant $ D_{\beta} >0$:
    \begin{align}
        \mathrm{dist}_H(\mathcal{S}^*(y_{k-1}), \mathcal{S}^*(y_k)) \le  D_{\beta} \, \mathrm{dist}(y_{k-1},y_k).
    \end{align}
    Then we have the following bound:
       \begin{align}
        \mathrm{dist}(G(x_k,y_{k}),G(x_{k-1},y_{k-1}))
           & \le  \, \eta \, \,\mathrm{dist}(y_{k-1},y_k) +   L^2 h_1  \frac{\beta \, L^{\beta-1} }{(1-\tfrac{Lh_2}{2})(\beta-1) C_{\beta}} \bigg( \mathrm{dist}^{\beta-1}(x_{k-1}, \mathcal{S}^*(y_{k-1})) + \nonumber \\ &   D_{\beta}^{(\beta - 1)}  \, \mathrm{dist}^{(\beta - 1)}(y_{k-1},y_k) \bigg) \, 
    \end{align}
    and for $\beta = 2$ we get:
        \begin{align}
        \mathrm{dist}(G(x_k,y_{k}),G(x_{k-1},y_{k-1}))
           & \le  \, \bigg(\eta +  \frac{2 \, L^{3}  D_{\beta} h_1 }{(1-\tfrac{Lh_2}{2}) C_{\beta}} \bigg) \, \,\mathrm{dist}(y_{k-1},y_k) +     \frac{2 \, L^{3}h_1 }{(1-\tfrac{Lh_2}{2}) C_{\beta}}  \mathrm{dist}(x_{k-1}, \mathcal{S}^*(y_{k-1})) .
    \end{align}
    The condition $ \bigg(\eta +  \frac{2 \, L^{3}  D_{\beta} h_1 }{(1-\tfrac{Lh_2}{2}) C_{\beta}} \bigg) <1$ for sufficiently small $h_1$ follows immediately from the definition of $\eta$ and the bound
    \[D_{\beta}  < \frac{(1-\tfrac{Lh_2}{2}) C_{\beta}} {2 \, L^{3}   }\mu   \, .\] 
\end{proof}

\section{Proofs for section \ref{sectionmaintheorems} }\label{sectionmaintheoremsappendix}

\subsection{Some technical bounds and supporting results}
We write 4 inequalities as follows:
    \begin{enumerate}
        \item [\textbf{I1.}] From Lemma \ref{RGA-MGD1-lem4}(a) for $\beta =2$, $Lh_2 <1$ we have that:
            \begin{align}
       \mathrm{dist}(y_{k+1},y_k)
           & \le  \, \bigg(\eta +  \frac{2 \, L^{3}  D_{\beta} h_1 }{(1-\tfrac{Lh_2}{2}) C_{\beta}} \bigg) \, \,\mathrm{dist}(y_{k-1},y_k)   +    \frac{2 \, L^{3}h_1 }{(1-\tfrac{Lh_2}{2}) C_{\beta}}  \mathrm{dist}(x_{k-1}, \mathcal{S}^*(y_{k-1}))
    \end{align}
    where $ \eta   = 1-h_1\mu+\frac{3}{4}\Theta L^2 h_1^2  + \mathcal{O}(h_1^3) $.
        \item [\textbf{I2.}] From Lemma \ref{RGA-MGD1-lem2}, remark \ref{taukbounremark} for $\beta =2$, $Lh_2 <1$ we have that:
        \begin{align}
        \mathrm{dist}(x_{k}, \mathcal{S}^*(y_{k})) & \le     \frac{2 L ( 1 - \frac{h_2 C_{\beta}}{2} )^{J/2} }{C_{\beta}} \mathrm{dist}(x_{k-1}, \mathcal{S}^*(y_{k-1}))  +   \frac{4 L^2 ( 1  - \frac{h_2 C_{\beta}}{2} )^{J/2} }{C^2_{\beta}} \, \mathrm{dist}(y_{k-1},y_k)    
    \end{align}
    where $ \frac{2 L ( 1  - \frac{h_2 C_{\beta}}{2} )^{J/2} }{C_{\beta}} < 1$ for any sufficiently large $J$. 
    \item  [\textbf{I1'.}] For $L h_2 <1 $ and arbitrary $\beta \in (1,2)$, from Lemma \ref{RGA-MGD1-lem4}(b) proof (inequality \eqref{I1'sourceqn}) for $J := J(k)$ we have that:
   \begin{align}
        \mathrm{dist}(y_{k+1},y_k)
           & \le  \, \eta \, \,\mathrm{dist}(y_{k-1},y_k) +     \frac{\beta \, L^{\beta+1} h_1 }{(1-\tfrac{Lh_2}{2})(\beta-1) C_{\beta}} \bigg( \mathrm{dist}^{\beta-1}(x_{k-1}, \mathcal{S}^*(y_{k-1}))  \nonumber \\ & +  \mathrm{dist}_H^{\beta-1}(\mathcal{S}^*(y_{k-1}), \mathcal{S}^*(y_k)) \bigg) \, .
    \end{align}
    \item [\textbf{I2'.}]   Let $ \mathrm{dist}(x_{k-1}, \mathcal{S}^*(y_{k-1})) <  \frac{C^{1/\beta}_{\beta}}{L} - h_1^{\beta - 1} \left(\frac{\beta}{\beta-1} \right) C_{\beta}^{-1} L^{2\beta - 2} \, \mathrm{diam}^{\beta - 1}(S_2)$ for $h_1$ sufficiently small, then for $L h_2 <1 $ and arbitrary $\beta \in (1,2)$, from Lemma \ref{RGA-MGD1-lem2} for $J := J(k)$ we have that:
        \begin{align}
        \mathrm{dist}(x_{k}, \mathcal{S}^*(y_{k}))
        & \le  C_{L,\beta}  \frac{\bigg( \mathrm{dist}(x_{k-1}, \mathcal{S}^*(y_{k-1})) \bigg)^{\frac{\beta (\beta -1) }{2}}}{\bigg(   1+ J(k) \,C'h_2  \bigg)^{\frac{\beta-1}{2}}}  + C_{L,\beta}  \frac{\bigg(  L^{\beta -1} \, \left(\frac{\beta}{\beta-1} \right) C_{\beta}^{-1} \, \mathrm{dist}^{\beta - 1}(y_{k-1},y_k) \bigg)^{\frac{\beta (\beta -1) }{2}}}{\bigg(   1+ J(k) \,C'h_2  \bigg)^{\frac{\beta-1}{2}}} 
    \end{align}
        where $C_{L,\beta} :=   \left(\frac{\beta}{\beta-1} \right) \frac{L ^{\frac{\beta (\beta -1)}{2}}}{ C_{\beta}^{\frac{\beta-1}{2} + \frac{1}{\beta }} } $, $ C' := (1 - \tfrac{Lh_2}{2}) C^{2/\beta}_{\beta} $.
    \end{enumerate}

\begin{theorem}\citep[Theorem 6.3.12]{horn2012matrix}\label{theomatperturb}
     Let $X, E \in \mathbb{R}^{n \times n}$ and let $q$ be a simple eigenvalue of $X$. Let $\mathbf{v}$ and $\mathbf{u}$ be, respectively, the right and left eigenvectors of $X$ corresponding to the eigenvalue $q$. Then,
     \begin{enumerate}
         \item  for each $\epsilon>0$, there exists a $\delta>0$ such that, $\forall p \in \mathbb{C}$ with $|p|<\delta$, there is a unique eigenvalue $q(p)$ of $ X+p E$ such that $\left|q(p)-q-p \frac{\mathbf{u}^H E \mathbf{v}}{\mathbf{u}^H \mathbf{v}}\right| \leq|p| \epsilon$,
         \item  $q(p)$ is continuous at $p=0$, and $\lim _{p \rightarrow 0} q(p)=q$,
         \item $q(p)$ is differentiable at $p=0,\left.\frac{d q(p)}{d p}\right|_{p=0}=\frac{\mathbf{u}^H E \mathbf{v}}{\mathbf{u}^H \mathbf{v}}$,
     \end{enumerate}
        where $(\cdot)^{H}$ is Hermitian operator.
\end{theorem}

\begin{lem}\label{RGA-MGD1-lem4zz}
    Let $a_k \, , b_k \ge 0$ be sequences that satisfy the recursion:
    \[ a_{k+1} \le \alpha \, a_k \,  +  \,  B_1 \, b_{k+1} \quad \forall \, \, k\ge 0 \]
    for $\alpha \in (0,1)$, a constant $B_1>0$ and that 
    \[ b_{k} \le \sigma^k B_2 \, b_0  \quad \forall \, \, k\ge 1 \]
     for $\sigma \in (0,1)$, a constant $B_2>0$. Then for a constant $B_3 = \max \bigg\{ a_0, \frac{B_1  B_2 b_0 }{\max\{\alpha, \sigma\} - \min\{\alpha, \sigma\}} \bigg\} $ we have
     \[ a_k \le (\max\{\alpha, \sigma\})^{k} \, B_3 \quad \forall \, \, k\ge 1 .\]
\end{lem}
\begin{proof}
    We can recursively write:
    \begin{align}
        a_{k+1} &\le \alpha \, a_k \,  +  \,  B_1 \, b_{k+1} \\
   \implies a_k     & \le \alpha^k \, a_0 +  B_1  B_2 \sum_{j=0}^{k-1} \alpha^{k-1-j} \sigma^{j} \, b_0 \\
   & \le \alpha^k \, a_0 +  B_1  B_2 b_0  \alpha^{k-1}\sum_{j=0}^{k-1} \bigg(\frac{\sigma}{\alpha}\bigg)^{j} \,  = \alpha^k \, a_0 +  B_1  B_2 b_0  \sigma^{k-1}\sum_{j=0}^{k-1} \bigg(\frac{\alpha}{\sigma}\bigg)^{j} \, \\
   & = \alpha^k \, a_0 +  B_1  B_2 b_0  (\max\{\alpha, \sigma\})^{k-1}\sum_{j=0}^{k-1} \bigg(\frac{(\min\{\alpha, \sigma\})}{(\max\{\alpha, \sigma\})}\bigg)^{j} \\
   & \le \alpha^k \, a_0 +  B_1  B_2 b_0   \frac{(\max\{\alpha, \sigma\})^{k}}{\max\{\alpha, \sigma\} - \min\{\alpha, \sigma\}} \\
   \implies a_k &\le (\max\{\alpha, \sigma\})^{k} \, \max \bigg\{ a_0, \frac{B_1  B_2 b_0 }{\max\{\alpha, \sigma\} - \min\{\alpha, \sigma\}} \bigg\} .
    \end{align}
\end{proof}

\subsection{Proof of Theorem \ref{convergenceratethm_exactPL} }
\begin{proof}
   { \textbf{Iterate boundedness for $\beta = 2$ :}} \\ 
    We note that the sequence $\{y_k\}$ always stays bounded in the compact, geodesic convex set\footnote{We can assume without loss of generality that $x_k \times S_2$ is a closed geodesic ball with center $\wt y$, $x_k \times S_2$ contains $y_k, y^*(x_k)$ in its interior (\textbf{C2}) and that $\mathrm{dist}(y^*(x_k),\wt y)$ is sufficiently small uniformly for all $k$. Then one step of Riemannian ascent with small enough $h_1$ will stay inside $S_2$ due to contractive property (Lemma \ref{contractionlema1}) and hence $\{y_k\}_k$ will be bounded in $S_2$. } $x_k \times S_2$ by simple induction and contractiveness of the map $\exp_{\cdot}(h \nabla_y f(x, \cdot) ) $ from Lemma \ref{contractionlema1} for any $x$. Hence $ \mathrm{dist}(y_{k-1},y_k) \le \mathrm{diam}(S_2)$ for all $k$. We only need to show boundedness of $ \{x_k\}$ and in particular $\mathrm{dist}(x_{k}, \mathcal{S}^*(y_k)) < \delta/2$ for all $k$. We have that $\mathrm{dist}(x_{0}, \mathcal{S}^*(y_0)) < \delta/2$. Suppose for some $k$ we have $\mathrm{dist}(x_{k}, \mathcal{S}^*(y_k)) < \delta/2$ then from Lemma \ref{RGA-MGD1-lem2} we have 
       \begin{align*}
        \mathrm{dist}(x_{k+1}, \mathcal{S}^*(y_{k+1})) & {\le} \,\, \sqrt{\tau_k} \,  \mathrm{dist}(x_{k}, \mathcal{S}^*(y_k)) +   \sqrt{\tau_k} \,  \bigg( 2L C_{\beta}^{-1} \, \mathrm{dist}(y_{k+1},y_k) \bigg)  \\ 
        & {\le} \,\, \sqrt{\tau_k} \,  \frac{\delta}{2} +   \sqrt{\tau_k} \,  \bigg( 2L C_{\beta}^{-1} \, \mathrm{diam}(S_2) \bigg) \\
        & \le \frac{2 L ( 1 - \frac{h_2 C_{\beta}}{2} )^{J/2} }{C_{\beta}} \bigg(  \frac{\delta}{2} +  2L C_{\beta}^{-1} \, \mathrm{diam}(S_2)\bigg) < \delta/2
    \end{align*}
    where in the last step we used the bound $\sqrt{\tau_k} \le  \frac{2 L ( 1 - \frac{h_2 C_{\beta}}{2} )^{J/2} }{C_{\beta}} <1 $ for any large $J$ (see Remark \ref{taukbounremark}). Hence, by induction $ \{x_k\}_k$ stays bounded in $\{  \mathcal{S}^*(y_k)+\mathcal{B}_{\delta/2}( 0)\}_k$. Note that the local minima connected components in the sequence $\{ \mathcal{S}^*(y_k)\}_{k} $ are at most $\delta/4$ separated. 
    \\ \\
   {\textbf{Convergence analysis for $\beta = 2$ :}} \\
   Suppose $h_1\ll 1$ with $\frac{2 \, L^{3}h_1 }{(1-\tfrac{Lh_2}{2}) C_{\beta}} < 1 $ and $J$ satisfies the following bound:
\begin{align*}
   \frac{2 L ( 1 - \frac{h_2 C_{\beta}}{2} )^{J/2} }{C_{\beta}} \le \frac{2L}{C_{\beta}} \cdot \frac{2 L ( 1 - \frac{h_2 C_{\beta}}{2} )^{J/2} }{C_{\beta}}  &<\frac{2 \, L^{3}h_1 }{(1-\tfrac{Lh_2}{2}) C_{\beta}} \ll 1 \, .
\end{align*}
     Writing inequalities from \textbf{I1, I2} in matrix form and substituting {the bound for $h_1 \ll 1$ : $$\eta +  \frac{2 \, L^{3}  D_{\beta} h_1 }{(1-\tfrac{Lh_2}{2}) C_{\beta}}\le 1-h_1 \bigg(\mu - \frac{2 \, L^{3}  D_{\beta} }{(1-\tfrac{Lh_2}{2}) C_{\beta}} \bigg)+\frac{3}{4}\Theta L^2 h_1^2 + \mathcal{O}(h_1^3) < 1-  \frac{h_1}{2}\bigg(\mu - \frac{2 \, L^{3}  D_{\beta} }{(1-\tfrac{Lh_2}{2}) C_{\beta}} \bigg) < 1 $$ gives the following system:}  
\begin{align}
   \Bigg[\begin{array}{c}
          \mathrm{dist}(y_{k+1},y_k) \\
           \mathrm{dist}(x_{k}, \mathcal{S}^*(y_{k}))
    \end{array} \Bigg]   & \le  \Bigg[ \begin{array}{cc}
1-  \frac{h_1}{2}\bigg(\mu - \frac{2 \, L^{3}  D_{\beta} }{(1-\tfrac{Lh_2}{2}) C_{\beta}} \bigg)  &    \frac{2 \, L^{3}h_1 }{(1-\tfrac{Lh_2}{2}) C_{\beta}} \\
       \frac{4 L^2 ( 1 - \frac{h_2 C_{\beta}}{2})^{J/2} }{C^2_{\beta}} & \frac{2 L ( 1   - \frac{h_2 C_{\beta}}{2})^{J/2} }{C_{\beta}}
    \end{array} \Bigg]    \Bigg[\begin{array}{c}
         \mathrm{dist}(y_{k},y_{k-1}) \\
           \mathrm{dist}(x_{k-1}, \mathcal{S}^*(y_{k-1}))
    \end{array} \Bigg] 
    \end{align}
    \begin{align}
    & \le \Bigg[ \begin{array}{cc}
1-  \frac{h_1}{2}\bigg(\mu - \frac{2 \, L^{3}  D_{\beta} }{(1-\tfrac{Lh_2}{2}) C_{\beta}} \bigg)  &    \frac{2 \, L^{3}h_1 }{(1-\tfrac{Lh_2}{2}) C_{\beta}} \\
     \frac{2 \, L^{3}h_1 }{(1-\tfrac{Lh_2}{2}) C_{\beta}}  & \frac{2 \, L^{3}h_1 }{(1-\tfrac{Lh_2}{2}) C_{\beta}} 
    \end{array} \Bigg]    \Bigg[\begin{array}{c}
         \mathrm{dist}(y_{k},y_{k-1}) \\
           \mathrm{dist}(x_{k-1}, \mathcal{S}^*(y_{k-1}))
    \end{array} \Bigg]  \\
     & = \Bigg ( \underbrace{\Bigg[ \begin{array}{cc}
        1 &   0 \\
         0 &  0  
    \end{array} \Bigg] }_{X} + \, h_1 \underbrace{\Bigg[ \begin{array}{cc}
       -  \frac{1}{2}\bigg(\mu - \frac{2 \, L^{3}  D_{\beta} }{(1-\tfrac{Lh_2}{2}) C_{\beta}} \bigg)     &   \frac{2 \, L^{3} }{(1-\tfrac{Lh_2}{2}) C_{\beta}}    \\
           \frac{2 \, L^{3}}{(1-\tfrac{Lh_2}{2}) C_{\beta}}  \, &   \frac{2 \, L^{3} }{(1-\tfrac{Lh_2}{2}) C_{\beta}} 
    \end{array} \Bigg] }_{ E} \Bigg )  \Bigg[\begin{array}{c}
         \mathrm{dist}(y_{k},y_{k-1}) \\
           \mathrm{dist}(x_{k-1}, \mathcal{S}^*(y_{k-1}))
    \end{array} \Bigg] 
\end{align}
Observe that by using Theorem \ref{theomatperturb} on $X+ h_1 E $ above, for any $\epsilon>0$ there exists $\delta' >0$ such that the largest absolute eigenvalue $\lambda$ of the matrix $X+ h_1 E$ in the above equation, corresponding to the diagonal entry $1$ of the matrix $X$, satisfies the bound:
\[  \bigg| \lambda  -  \bigg(1 - h_1   \frac{ \langle \e_1, E \e_1 \rangle }{\langle \e_1, \e_1 \rangle }   \bigg) \bigg|  = \bigg| \lambda  -  \bigg(1 - \frac{h_1}{2}\bigg(\mu - \frac{2 \, L^{3}  D_{\beta} }{(1-\tfrac{Lh_2}{2}) C_{\beta}} \bigg)   \bigg) \bigg| \le h_1 \, \epsilon \]
where $\e_1 = [1,0]^T$ and $ h_1 < \delta'$. Then choosing $ \epsilon :=\frac{1}{4}\bigg(\mu - \frac{2 \, L^{3}  D_{\beta} }{(1-\tfrac{Lh_2}{2}) C_{\beta}} \bigg)   $ there exists $\delta' >0$ sufficiently small such that if $h_1 < \delta' $ then $|\lambda | \le \bigg(1 - \frac{h_1}{4}\bigg(\mu - \frac{2 \, L^{3}  D_{\beta} }{(1-\tfrac{Lh_2}{2}) C_{\beta}} \bigg)   \bigg) < 1 $. Since the spectral radius of the matrix $X+{ h_1} E$ given by $ \rho(X+{h_1} E) \le \bigg(1 - \frac{h_1}{4}\bigg(\mu - \frac{2 \, L^{3}  D_{\beta} }{(1-\tfrac{Lh_2}{2}) C_{\beta}} \bigg)    \bigg) $ is strictly less than $1$, we get that:
\begin{align}
  \opnorm{\Bigg[\begin{array}{c}
          \mathrm{dist}(y_{K+1},y_K) \\
           \mathrm{dist}(x_{K}, \mathcal{S}^*(y_{K}))
    \end{array} \Bigg]  }  & \le  \rho^K(X+{h_1} E) \opnorm{\Bigg[\begin{array}{c}
         \mathrm{dist}(y_{1},y_{0}) \\
           \mathrm{dist}(x_{0}, \mathcal{S}^*(y_{0}))
    \end{array} \Bigg]} \\ 
    & \le \bigg(1 - \frac{h_1}{4}\bigg(\mu - \frac{2 \, L^{3}  D_{\beta} }{(1-\tfrac{Lh_2}{2}) C_{\beta}} \bigg)    \bigg)^K \opnorm{\Bigg[\begin{array}{c}
         \mathrm{dist}(y_{1},y_{0}) \\
           \mathrm{dist}(x_{0}, \mathcal{S}^*(y_{0}))
    \end{array} \Bigg]}  \xrightarrow{ K \to \infty} 0  \label{rateiteratePL1*}
\end{align}
where $\opnorm{\cdot}$ is some subordinate norm induced by diagonalizing the matrix $X+{h_1} E $. Thus, 
\begin{align}
    \mathrm{dist}(x_{K}, \mathcal{S}^*(y_{K})) \le \bigg(1 - \frac{h_1}{4}\bigg(\mu - \frac{2 \, L^{3}  D_{\beta} }{(1-\tfrac{Lh_2}{2}) C_{\beta}} \bigg)    \bigg)^K C_{L, \mu,  C_{\beta}} \,\,  \mathrm{dist}(x_{0}, \mathcal{S}^*(y_{0})) \label{rateiteratePL1}
\end{align}
 for all $K >0$ and some constant $ C_{L, \mu,  C_{\beta}} >0$. Then from Lemma \ref{RGA-MGD1-lem4z} and \eqref{rateiteratePL1} we immediately have that $ \norm{x_{K}- x_{K-1}} \to 0$ as $K \to \infty$ at a linear rate, i.e., 
 $$ \norm{x_{K}- x_{K-1}} = \mathcal{O}\bigg( \bigg(1 - \frac{h_1}{4}\bigg(\mu - \frac{2 \, L^{3}  D_{\beta} }{(1-\tfrac{Lh_2}{2}) C_{\beta}} \bigg)    \bigg)^K\bigg)  $$
 and by Cauchy criteria of convergence $ \{x_k\}_k$ converges. Similarly $ \{y_k\}_k$ converges by \eqref{rateiteratePL1*} and Cauchy criteria of convergence.

Next, from Lemmas \ref{RGA-MGD1-lem3}, \ref{RGA-MGD1-lem4z} we have that:
\begin{align}
        \mathrm{dist}(y_{k+1},y^*(x_{k+1}))  &\le  \wt\gamma \, \mathrm{dist}(y_{k},y^*(x_k))   + \frac{ L}{\mu} \norm{x_{k+1} - x_k} \\
         & \le \wt\gamma \, \mathrm{dist}(y_{k},y^*(x_k))   + \frac{ L}{\mu} \, \bigg( \frac{2 \, L }{(1-\tfrac{Lh_2}{2}) C_{\beta}} \bigg( \mathrm{dist}(x_{k-1}, \mathcal{S}^*(y_{k-1})) \nonumber \\ & +  D_{\beta} \, \mathrm{dist}(y_{k}, y_{k-1}) \bigg)\bigg) \label{rateiteratePL2}
    \end{align}
    where $$ \wt\gamma = \sqrt{(1 -2 \mu h_1 + C_{\upsilon} L^2h_1^2)} <1 $$ for any sufficiently small $h_1$, $  C_{\upsilon} =\frac{\sqrt{|\upsilon|} \, \mathrm{diam}(S_2) }{\tanh (\sqrt{|\upsilon|} \, \mathrm{diam}(S_2))} $ and $\upsilon$ is the lower bound of the sectional curvature of $\mathcal{M}$. Then using Lemma \ref{RGA-MGD1-lem4zz} for \eqref{rateiteratePL1}, \eqref{rateiteratePL2} immediately gives the convergence:
    \begin{align}
        \mathrm{dist}(y_{K},y^*(x_K)) \le \alpha^K D_0 
    \end{align}
    for some constant $D_0$ and $\alpha = \max \bigg\{ \bigg(1 - \frac{h_1}{4}\bigg(\mu - \frac{2 \, L^{3}  D_{\beta} }{(1-\tfrac{Lh_2}{2}) C_{\beta}} \bigg)    \bigg), \sqrt{(1 -2 \mu h_1 + C_{\upsilon} L^2h_1^2)} \bigg\} < 1$. Then from \eqref{lyapunov-RGA-MGD} :
    \begin{align}
        V_{\delta}(x_K, y_K) &= \bigg( \sup_{y \in x_K \times S_2} f(x_K,y) -  f(x_K,y_K) \bigg) + \bigg(f(x_K,y_K) - \inf_{x \in (\mathcal{S}^*(y_K) + \mathcal{B}_{\delta}(0)) \times y_K} f(x,y_K)\bigg) \\
         & \underbrace{\le}_{\textbf{C4}}  L  \, \mathrm{dist}(y_{K},y^*(x_K)) + L \,  \mathrm{dist}(x_{K}, \mathcal{S}^*(y_{K})) \\
         & \le L\bigg(\max \bigg\{ \bigg(1 - \frac{h_1}{4}\bigg(\mu - \frac{2 \, L^{3}  D_{\beta} }{(1-\tfrac{Lh_2}{2}) C_{\beta}} \bigg)    \bigg), \sqrt{(1 -2 \mu h_1 + C_{\upsilon} L^2h_1^2)} \bigg\} \bigg)^K D_1 \label{VdeltaPLconvergencerate}
    \end{align}
    where $D_1 = \bigg( C_{L, \mu,  C_{\beta}} \,\,  \mathrm{dist}(x_{0}, \mathcal{S}^*(y_{0})) \, + \, D_0 \bigg)$. Next, $V_{\delta}(x_K, y_K) \ge 0 $ for all $K$ by iterate boundedness in $ ( \mathcal{S}^*(y_K)+\mathcal{B}_{\delta/2}( 0)) \times S_2 \subseteq  ( \mathcal{S}^*(y_0)+\mathcal{B}_{3\delta/4}( 0)) \times S_2  $ as a consequence of \textbf{C3}, $ \lim_{K\to \infty} V_{\delta}(x_K, y_K) = 0 $ from \eqref{VdeltaPLconvergencerate} and thus from {joint uniform continuity of $V_{\delta}(\cdot,\cdot)$ on $S_1 \times S_2$, the Cauchy convergence of the sequence $\{(x_k, y_k)\}_k$ from \eqref{RGA-MGD1} - \eqref{RGA-MGD2},} the iteration $\{(x_k, y_k)\}_k$ from \eqref{RGA-MGD1} - \eqref{RGA-MGD2} converges to some $(x^*,y^*) \in \text{int}((\mathcal{S}^*(y_0)+\mathcal{B}_{3\delta/4}( 0)) \times S_2) $ with $ V_{\delta}(x^*, y^*) = 0 $. \\
    {From \textbf{C3} there exists a connected component of local minima $\mathcal{S}^*(y^*) $ such that $ \mathrm{dist}_H(\mathcal{S}^*(y_0),\mathcal{S}^*(y^*)) < \delta/4$ and so $(x^*,y^*) \in \text{int}((\mathcal{S}^*(y^*)+\mathcal{B}_{\delta}( 0)) \times S_2) $. }
    \\ From $ V_{\delta}(x^*, y^*) = 0 $, $(x^*,y^*) \in \text{int}((\mathcal{S}^*(y^*)+\mathcal{B}_{\delta}( 0)) \times S_2) $ and Lemma \ref{RGA-MGD1-lem1}, $ (x^*,y^*) \in \text{int}(S_1 \times S_2)$ is a basin saddle point of $f$ and hence the sequence $\{(x_k, y_k)\}_k$ from \eqref{RGA-MGD1} - \eqref{RGA-MGD2} converges linearly to a basin saddle point of $f$ with a rate $\mathcal{O}( a^k)$ and the constant $a \in (0,1)$ is defined from \eqref{VdeltaPLconvergencerate}.
\end{proof}

\subsection{Proof of Theorem \ref{convergenceratethm_betaPL} }
\begin{proof} 
     { \textbf{Iterate boundedness for $\beta \in (1,2)$ :}} \\ 
      Similar to Theorem \ref{convergenceratethm_exactPL} the sequence $\{y_k\}$ always stays bounded in $\{x_k\} \times S_2$ by simple induction and contractiveness of the map $\exp_{\cdot}(h_1 \nabla_y f(x, \cdot) ) $ from Lemma \ref{contractionlema1} for any $x$. Hence $ \mathrm{dist}(y_{k-1},y_k) \le \mathrm{diam}(S_2)$ for all $k$. Without loss of generality or by shrinking $\delta$ assume 
    $$ \frac{\delta}{2} <  \frac{C^{1/\beta}_{\beta}}{L} -   h_1^{\beta - 1}\left(\frac{\beta}{\beta-1} \right) C_{\beta}^{-1} L^{2\beta - 2}\, \mathrm{diam}^{\beta - 1}(S_2) $$ with $h_1$ sufficiently small such that
    $ h_1^{\beta - 1}\left(\frac{\beta}{\beta-1} \right) C_{\beta}^{-1} L^{2\beta - 2} \, \mathrm{diam}^{\beta - 1}(S_2) < \frac{C^{1/\beta}_{\beta}}{L} $ . Suppose we have the parameter $J:= J(k) = k^{\frac{2\alpha}{(\beta-1)^2}} $ for some constant $\alpha >   0$ and $J(k)$ satisfies the following uniform lower bound for all $k \ge K_0$ and some large enough $K_0 \gg 1$ :
    \begin{align*}
        \bigg(   1+ J(k) \,C'h_2  \bigg)^{\frac{\beta-1}{2}} & \ge  \bigg(   1+ J(K_0) \,C'h_2  \bigg)^{\frac{\beta-1}{2}} \nonumber \\ & > \frac{2\Bigg(C_{L,\beta} \, \delta^{\frac{\beta (\beta -1) }{2}} +   C_{L,\beta} \bigg(  L^{\beta -1} \, \left(\frac{\beta}{\beta-1} \right) C_{\beta}^{-1} \, \bigg(\mathrm{diam}(S_2) \bigg)^{\beta - 1} \bigg)^{\frac{\beta (\beta -1) }{2}} \Bigg)}{\delta}  \, .
    \end{align*}
    The above condition is not vacuous because the L.H.S. in the above inequality is of the order $(J(K_0))^{\frac{\beta-1}{2}} \sim K_0^{\frac{\alpha}{\beta -1}}$ while the R.H.S. is bounded above by some constant.
   Assume for some $k > K_0$ we have $\mathrm{dist}(x_{k}, \mathcal{S}^*(y_k)) < \delta/2$, then from \textbf{I2'.} we get that for $k+1$ :
    \begin{align}
     \mathrm{dist}(x_{k+1}, \mathcal{S}^*(y_{k+1}))
        &  \le C_{L,\beta}  \frac{\bigg( \mathrm{dist}(x_{k}, \mathcal{S}^*(y_k)) \bigg)^{\frac{\beta (\beta -1) }{2}}}{\bigg(   1+ J(k+1) \,C'h_2  \bigg)^{\frac{\beta-1}{2}}}    + C_{L,\beta}  \frac{\bigg(  L^{\beta -1} \, \left(\frac{\beta}{\beta-1} \right) C_{\beta}^{-1} \, \mathrm{dist}^{\beta - 1}(y_{k-1},y_k) \bigg)^{\frac{\beta (\beta -1) }{2}}}{\bigg(   1+ J (k+1)\,C'h_2  \bigg)^{\frac{\beta-1}{2}}}  \\
        & \hspace{-2cm} <    \frac{ C_{L,\beta}\delta^{\frac{\beta (\beta -1) }{2}}}{\bigg(   1+ J(k+1) \,C'h_2  \bigg)^{\frac{\beta-1}{2}}}   \nonumber  +   \frac{  C_{L,\beta} \bigg(  L^{\beta -1} \, \left(\frac{\beta}{\beta-1} \right) C_{\beta}^{-1} \, \bigg(\mathrm{diam}(S_2) \bigg)^{\beta - 1} \bigg)^{\frac{\beta (\beta -1) }{2}}  }{\bigg(   1+ J (k+1)\,C'h_2  \bigg)^{\frac{\beta-1}{2}}} < \frac{\delta}{2} \\ \label{unifestimatelem0}
    \end{align}
  where the last bound holds for $ J(k+1) \gg 1$ whenever $ k \ge K_0 \gg 1$. Since $\mathrm{dist}(x_{K_0}, \mathcal{S}^*(y_{K_0})) < \delta/2$, \eqref{unifestimatelem0} completes the induction that $ \{x_k\}_{k \ge K_0}$ stays bounded in $\{  \mathcal{S}^*(y_k)+\mathcal{B}_{\delta/2}( 0)\}_{k \ge K_0}$. Note that the local minima connected components in the sequence $\{ \mathcal{S}^*(y_k)\}_{k \ge K_0} $ are at most $\delta/4$ separated.\\ \\
     {\textbf{Convergence analysis for $\beta \in (1, 2) $ :}} \\
     We use the inequalities \textbf{I1', I2'} for constant $h_2 < \frac{1}{L}$, $\eta = 1-h_1\mu+\frac{3}{4}\Theta L^2 h_1^2 + \mathcal{O}(h_1^3) < 1   $ for $ h_1 $ sufficiently small. 
Suppose for any large enough $k$ we uniformly have $\mathrm{dist}_H(\mathcal{S}^*(y_{k+1}), \mathcal{S}^*(y_k))  = \mathcal{O}(\mathrm{dist}(x_{k}, \mathcal{S}^*(y_{k})))  $ where $\mathcal{O}$ is the Big-O notation meaning $\mathrm{dist}_H(\mathcal{S}^*(y_{k-1}), \mathcal{S}^*(y_k))$ decays faster or equal to $\mathrm{dist}(x_{k}, \mathcal{S}^*(y_{k}))$. Then :
\begin{align}
   \mathrm{dist}_H(\mathcal{S}^*(y_{k-1}), \mathcal{S}^*(y_k)) & 
  \le C_2 \, \mathrm{dist}(x_{k-1},\mathcal{S}^*(y_{k-1}))
\end{align}
    for some constant $C_2>0$ independent of $k$. Using this bound in \textbf{I1'} yields:
       \begin{align}
     \mathrm{dist}(y_{k+1},y_k)
           & \le  \, \eta \, \,\mathrm{dist}(y_{k-1},y_k) +     \frac{\beta \, L^{\beta+1}  (1 + C_2^{\beta -1}) h_1 }{(1-\tfrac{Lh_2}{2})(\beta-1) C_{\beta}}  \mathrm{dist}^{\beta-1}(x_{k-1}, \mathcal{S}^*(y_{k-1})) 
         \label{beta12ratebound1}
    \end{align}
    and raising \textbf{I2'} to the power $\beta -1$ both sides followed by sub-additivity of concave function $t \mapsto t^{\beta -1}$ yields:
            \begin{align}
        \mathrm{dist}^{(\beta -1)}(x_{k}, \mathcal{S}^*(y_{k}))
        & \le  C^{(\beta -1)}_{L,\beta}  \frac{\bigg( \mathrm{dist}^{(\beta -1)}(x_{k-1}, \mathcal{S}^*(y_{k-1})) \bigg)^{\frac{\beta (\beta -1) }{2}}}{\bigg(   1+ J(k) \,C'h_2  \bigg)^{\frac{(\beta-1)^2}{2}}} \nonumber  \\ & + C^{(\beta -1)}_{L,\beta}  \frac{\bigg(  L^{\beta -1} \, \left(\frac{\beta}{\beta-1} \right) C_{\beta}^{-1} \, \mathrm{dist}^{\beta - 1}(y_{k-1},y_k) \bigg)^{\frac{\beta (\beta -1)^2 }{2}}}{\bigg(   1+ J(k) \,C'h_2  \bigg)^{\frac{(\beta-1)^2}{2}}} \label{beta12ratebound2}
    \end{align}
    Define $a_k := \mathrm{dist}^{(\beta -1)}(x_{k}, \mathcal{S}^*(y_{k}))$, $b_k := \mathrm{dist}(y_{k+1},y_k) $, $\theta = \frac{\beta (\beta -1)}{2} \in (0,1) $, a Lyapunov function $$F_k := a_k   + \lambda b_k $$ for some constant $\lambda >0$ specified later. Then $ a_k \le F_k$, $ b_k \le \frac{F_k}{\lambda} $ for all $k$. Let $J := J(k) = k^{\frac{2\alpha}{(\beta -1)^2}} $ for some constant $\alpha > 0$. From \eqref{beta12ratebound1}, \eqref{beta12ratebound2} we have then for some large enough constant $C>0$ and any large enough $k$ :
    \begin{align}
   \lambda b_k & \le \eta \,\lambda b_{k-1} + C \lambda h_1 a_{k-1}  \le \eta F_{k-1} + C \lambda h_1 F_{k-1}   \\
        a_k & \le \frac{C }{  k^{\alpha} } \bigg( a^{\theta}_{k-1} + b^{(\beta -1 )^2\theta}_{k-1} \bigg) \le \frac{C }{ k^{\alpha}} \bigg( F^{\theta}_{k-1} +   \bigg(\frac{F_{k-1}}{\lambda}\bigg)^{(\beta -1 )^2\theta} \bigg)
    \end{align}
    Adding the last 2 bounds gives the estimate:
    \begin{align}
        F_{k} = a_{k}   + \lambda b_{k}  & \le \eta F_{k-1} + C \lambda h_1 F_{k-1} + \frac{C }{ k^{\alpha}} \bigg( F^{\theta}_{k-1} +   \bigg(\frac{F_{k-1}}{\lambda}\bigg)^{(\beta -1 )^2\theta} \bigg) \\
    \implies  F_{k}  & \le (\eta  + C \lambda h_1 ) F_{k-1} + \frac{C (1+ \lambda)}{k^{\alpha}} F^{\theta}_{k-1} + \frac{C}{\lambda^{(\beta -1 )^2\theta}  k^{\alpha}}    F_{k-1}^{(\beta -1 )^2\theta} 
    \end{align}
    where we choose $\lambda = \frac{\mu}{2C} - \frac{3}{8C}\Theta L^2 h_1+ \mathcal{O}(h^2_1)  $ so that $$ (\eta + C \lambda h_1 )  \le  1- \mu h_1 +\frac{3}{4}\Theta L^2 h_1^2 + \mathcal{O}(h^3_1) + \frac{\mu}{2} h_1 - \frac{3}{8}\Theta L^2 h_1^2 = \frac{1+\eta}{2} $$ and thus we have:
    \begin{align}
       F_{k}  & \le \frac{1+\eta}{2}  F_{k-1} + \frac{C+\mu}{k^{\alpha}} F^{\theta}_{k-1} + \frac{C}{\lambda^{(\beta -1 )^2\theta}  k^{\alpha}}    F_{k-1}^{(\beta -1 )^2\theta} \\ 
       & \le \frac{1+\eta}{2}  F_{k-1}  + \frac{D}{  k^{\alpha}}    F_{k-1}^{(\beta -1 )^2\theta} \label{beta12ratebound3}
    \end{align}
    for some constant $D>0$ where $ \frac{1+\eta}{2} < 1$ and we assumed $F_{k-1} < 1$ in the last step without loss of generality and recalled $\beta - 1 \in (0,1)$. 
    \paragraph{Rate exponent estimation via ODE approximation:} Using ODE approximation of Euler discretization with $\frac{d F}{dt} \approx F_{k}  - F_{k-1} $, $F_k := F(t_k) \approx F(t)$, $k := t_k \approx t$ in the last bound, invoking $ h_1 \ll 1$ and dropping higher order terms of $h_1$ gives:
    \begin{align}
        \frac{d F}{dt} \le -\frac{1-\eta}{2}  F  + \frac{D}{  t^{\alpha}}    F^{(\beta -1 )^2\theta} \le -\frac{\mu h_1}{2}  F  + \frac{D}{  t^{\alpha}}    F^{(\beta -1 )^2\theta}   < 0 \quad \textbf{locally for large } t
    \end{align}
    since $ \alpha > 0 $ and $\beta \in (1,2)$. Thus $F$ decreases with large $t$ so $F_k$ must decrease for large $k$. From a Gronwall's type inequality the above ODE's solution satisfies $F \le H $ for all $t$ where $ H := H(t)$ is the solution of the ODE:
     \begin{align}
        \frac{d H}{dt}  = -\frac{\mu h_1}{2}  H  + \frac{D}{  t^{\alpha}}    H^{p}  
    \end{align}
    where $p = {(\beta -1 )^2\theta} \in (0,1)$ and thus the above ODE is Bernoulli type equation. Then letting $u = H^{1-p}$ so that $  \frac{d u}{dt}  = (1-p) H^{-p} \frac{d H}{dt}  $ we get that:
    \begin{align}
       H^{-p} \frac{d H}{dt}  &=  -\frac{\mu h_1}{2} H^{1-p} + \frac{D}{  t^{\alpha}}    \\
  \iff     \frac{d u}{dt} &=  -\frac{\mu h_1 (1-p)}{2}  u + \frac{D (1-p)}{  t^{\alpha}} 
    \end{align}
    which is linear first order ODE with integrating factor $q (t)= e^{\int  \frac{\mu h_1 (1-p)}{2} dt } = e^{ \frac{\mu h_1 (1-p)t}{2}} $ and the ODE solution is:
    \begin{align}
        u(t)  &=  e^{ -\frac{\mu h_1 (1-p)t}{2}} \int \frac{D}{  t^{\alpha}}  e^{ \frac{\mu h_1 (1-p)t}{2}} dt  \sim \frac{1}{  t^{\alpha}} \quad \textbf{ for large } t \textbf{ and $ \alpha >0$} \\
        \implies F(t) \le H(t) &= u^{\frac{1}{1-p}}(t) \sim \frac{1}{  t^{\frac{\alpha }{1-p}}} = t^{- \frac{\alpha}{1-(\beta -1 )^2\theta}} \quad \textbf{ for large } t \textbf{ and $ \alpha >0$}
    \end{align}
    where in the first step we used the asymptotics $  \int f(t) e^{g(t)} dt \sim \frac{f(t)}{g'(t)}e^{g(t)}  $ for large $t$, $f(t) > 0$ and increasing $ g(t) \geq 0$ and $ (\frac{f}{g'})' \to 0 $ as $t \to \infty$. From the last step, and for $J(k) = k^{\frac{2\alpha}{(\beta -1)^2}}$ we get from the discrete nonlinear recursion \eqref{beta12ratebound3} that $F_k \sim F(t)$ for large enough $k,t$ and hence:
    \begin{align}
        F_k \sim k^{- \frac{\alpha}{1-(\beta -1 )^2\theta}} \, .
    \end{align}
    Remarkably the same rate is obtained from the discrete nonlinear recursion \eqref{beta12ratebound3} if one guesses a polynomial decay rate. Suppose $F_k \sim k^{-c}$ for some $c>0$ and large enough $k$. Then equating the rate order from both sides of \eqref{beta12ratebound3} gives:
    \begin{align}
        k^{-c} &= k^{-c(\beta -1 )^2\theta - \alpha } \\
        \implies c & = \frac{\alpha}{1- (\beta -1)^2\theta} \\
        \implies  F_k &\sim k^{- \frac{\alpha}{1-(\beta -1 )^2\theta}} \, .
    \end{align}
    \paragraph{Formal derivation of convergence rate:}
    We now derive the polynomial convergence rate of $F_k \lesssim k^{- \frac{\alpha}{1-(\beta -1 )^2\theta}} $ obtained from the ODE discretization analysis using an induction argument. \\
     \textbf{Base case:} Let $F_{K_0} \le {C_0}{ {K_0}^{ -\frac{\alpha}{1-(\beta -1 )^2\theta}}}$ for some large enough constant $C_0 \ge 1$. This condition holds since for all $k \ge 0$ uniformly we have $F_k = a_k + \lambda b_k \le \delta^{\beta-1} + \lambda \mathrm{diam}(S_2) $ and hence one can choose $C_0$ large enough for any given $K_0$ such that $ {C_0}{ {K_0}^{ -\frac{\alpha}{1-(\beta -1 )^2\theta}}} \ge \delta^{\beta-1} + \lambda \mathrm{diam}(S_2) $. We also require that constants $C_0, K_0$ satisfy the condition:
     $$ \bigg( 1-\frac{\mu h_1 }{3}\bigg)  {C_0}{ \bigg(\frac{K_0}{K_0+1}\bigg)^{ -\frac{\alpha}{1-(\beta -1 )^2\theta}}} + {D}   \bigg({C_0}{ { \bigg(\frac{K_0}{K_0+1}\bigg)^{ -\frac{\alpha}{1-(\beta -1 )^2\theta}}} }\bigg)^{(\beta -1 )^2\theta} < C_0 \, ,$$
     where the constant $D$ comes from \eqref{beta12ratebound3}. This condition is non-vacuous and can be satisfied for large enough constants $C_0, K_0$ . For instance suppose $K_0$ is large enough so that $$ \bigg( 1-\frac{\mu h_1 }{3}\bigg) { \bigg(1+ \frac{1}{K_0}\bigg)^{ \frac{\alpha}{1-(\beta -1 )^2\theta}}} <  \bigg( 1-\frac{\mu h_1 }{4}\bigg) $$ and $ C_0$ satisfies 
     $ {{ { 2^{ \frac{\alpha(\beta -1 )^2\theta}{1-(\beta -1 )^2\theta}}} } D} < \frac{\mu h_1}{8} {C_0^{1-(\beta -1 )^2\theta} } $, $  {C_0} \ge { {K_0}^{ \frac{\alpha}{1-(\beta -1 )^2\theta}}}\bigg(\delta^{\beta-1} + \lambda \mathrm{diam}(S_2) \bigg) \, .$ Then
     \begin{align*}
         \bigg( 1-\frac{\mu h_1 }{3}\bigg)  {C_0}{ \bigg(\frac{K_0}{K_0+1}\bigg)^{ -\frac{\alpha}{1-(\beta -1 )^2\theta}}} + {D}   \bigg({C_0}{ { \bigg(\frac{K_0}{K_0+1}\bigg)^{ -\frac{\alpha}{1-(\beta -1 )^2\theta}}} }\bigg)^{(\beta -1 )^2\theta} & \nonumber \\  & \hspace{-8cm}=   \bigg( 1-\frac{\mu h_1 }{3}\bigg)  { \bigg(1+ \frac{1}{K_0}\bigg)^{ \frac{\alpha}{1-(\beta -1 )^2\theta}}} {C_0}+ \frac{{ { \bigg(1+ \frac{1}{K_0}\bigg)^{ \frac{\alpha(\beta -1 )^2\theta}{1-(\beta -1 )^2\theta}}} } D}{C_0^{1-(\beta -1 )^2\theta} }     C_0 
         \end{align*}
         \begin{align*}
         & \le   \bigg( 1-\frac{\mu h_1 }{3}\bigg)  { \bigg(1+ \frac{1}{K_0}\bigg)^{ \frac{\alpha}{1-(\beta -1 )^2\theta}}} {C_0}+ \frac{{ { 2^{ \frac{\alpha(\beta -1 )^2\theta}{1-(\beta -1 )^2\theta}}} } D}{C_0^{1-(\beta -1 )^2\theta} }     C_0 \\
         &  \le  \bigg(  1-\frac{\mu h_1 }{4} + \frac{\mu h_1 }{8} \bigg)C_0 = \bigg(1-\frac{\mu h_1 }{8}\bigg) C_0 < C_0 .
     \end{align*}
     \\ \\
     \textbf{Induction hypothesis:} Let $F_k \le {C_0}{ k^{ -\frac{\alpha}{1-(\beta -1 )^2\theta}}}$ for some $k \ge K_0$. \\ \\
     \textbf{Inductive step:} Since $ \mathrm{dist}(\mathcal{S}^*(y_{k}), \mathcal{S}^*(y_{k+1}))  \le C_2 \mathrm{dist}(x_{k}, \mathcal{S}^*(y_{k})) $, from \eqref{beta12ratebound3} for $h_1$ small enough so that $ \frac{1+\eta}{2} \le 1-\frac{\mu h_1 }{3}$ we have that:
     \begin{align}
       F_{k+1}   & \le \frac{1+\eta}{2}  F_{k}  + \frac{D}{  (k+1)^{\alpha}}    F_{k}^{(\beta -1 )^2\theta} \\
       & \hspace{-2cm} \le \bigg( 1-\frac{\mu h_1 }{3}\bigg)  {C_0}{ \bigg(\frac{k}{k+1}\bigg)^{ -\frac{\alpha}{1-(\beta -1 )^2\theta}}} { (k+1)^{ -\frac{\alpha}{1-(\beta -1 )^2\theta}}} + \frac{D}{  (k+1)^{\alpha}}    \bigg({C_0}{ { \bigg(\frac{k}{k+1}\bigg)^{ -\frac{\alpha}{1-(\beta -1 )^2\theta}}} (k+1)^{ -\frac{\alpha}{1-(\beta -1 )^2\theta}}}\bigg)^{(\beta -1 )^2\theta} \\
       & \hspace{-1cm} = \bigg(\bigg( 1-\frac{\mu h_1 }{3}\bigg)  {C_0}{ \bigg(\frac{k}{k+1}\bigg)^{ -\frac{\alpha}{1-(\beta -1 )^2\theta}}} + {D}   \bigg({C_0}{ { \bigg(\frac{k}{k+1}\bigg)^{ -\frac{\alpha}{1-(\beta -1 )^2\theta}}} }\bigg)^{(\beta -1 )^2\theta} \bigg) { (k+1)^{ -\frac{\alpha}{1-(\beta -1 )^2\theta}}} \\
       & \hspace{-1cm} \le C_0 { (k+1)^{ -\frac{\alpha}{1-(\beta -1 )^2\theta}}}
    \end{align}
    where the inequality $$ \bigg( 1-\frac{\mu h_1 }{3}\bigg)  {C_0}{ \bigg(\frac{k}{k+1}\bigg)^{ -\frac{\alpha}{1-(\beta -1 )^2\theta}}} + {D}   \bigg({C_0}{ { \bigg(\frac{k}{k+1}\bigg)^{ -\frac{\alpha}{1-(\beta -1 )^2\theta}}} }\bigg)^{(\beta -1 )^2\theta} < C_0 $$ is satisfied for any $k \ge K_0$ from the base case condition and the fact that the L.H.S. in the above inequality is a decreasing function of $k$ since $(\beta -1 )^2\theta \in (0,1) $. This completes the induction and we have $ F_k \sim k^{- \frac{\alpha}{1-(\beta -1 )^2\theta}} $ for all $k \ge K_0$. \\ \\
    Recalling that $F_k := a_k   + \lambda b_k $ where  $a_k := \mathrm{dist}^{\beta-1}(x_{k}, \mathcal{S}^*(y_{k}))$, $b_k := \mathrm{dist}(y_{k+1},y_k) $, $\theta = \frac{\beta (\beta -1)}{2} \in (0,1) $,  $\lambda >0$ is a constant, we immediately get that $\mathrm{dist}(x_{k}, \mathcal{S}^*(y_{k})) \lesssim k^{- \frac{\alpha}{(1-(\beta -1 )^2\theta)(\beta -1 )}}  $ and also $ \mathrm{dist}(y_{k+1},y_k)  \lesssim k^{- \frac{\alpha}{1-(\beta -1 )^2\theta}}  $ for all $k \ge K_0$. The constant $ \alpha$ is a free parameter that controls $J := J(k) = k^{\frac{2\alpha}{(\beta -1)^2}} $ with $\alpha  > 0$. From Lemma \ref{RGA-MGD1-lem4z} we have that for any $k \ge K_0$ :
    \begin{align*}
        \mathrm{dist}(x_{k}, x_{k-1}) & \lesssim  \mathrm{dist}^{\beta-1}(x_{k-1}, \mathcal{S}^*(y_{k-1})) +  \mathrm{dist}_H^{\beta-1}(\mathcal{S}^*(y_{k-1}), \mathcal{S}^*(y_k))  \\
        & \lesssim  \mathrm{dist}^{\beta-1}(x_{k-1}, \mathcal{S}^*(y_{k-1}))   + C_2^{\beta -1} \, \mathrm{dist}^{\beta-1}(x_{k}, \mathcal{S}^*(y_{k})) \lesssim k^{- \frac{\alpha}{1-(\beta -1 )^2\theta}}  
    \end{align*}
    and hence for $\alpha > {1-(\beta -1 )^2\theta}  $, by Cauchy convergence, the sequences $ \{x_k\}_k$, $ \{y_k\}_k$ converge.
    Next, we have that for any $k \ge K_0$ by triangle inequality:
    \begin{align*}
        \mathrm{dist}(y_{k}, y^*(x_{k})) & \le \mathrm{dist}(y_{k}, y_{k+1}) + \mathrm{dist}(y_{k+1}, y^*(x_{k+1})) + \mathrm{dist}(y^*(x_{k+1}), y^*(x_{k})) \\
        & \underbrace{\le}_{\textbf{Lemma \ref{RGA-MGD1-lem3}}} \mathrm{dist}(y_{k}, y_{k+1}) + \wt\gamma \, \mathrm{dist}(y_{k},y^*(x_k))   + \frac{ L}{\mu} \mathrm{dist}(x_{k}, x_{k+1})  + \frac{L}{\mu} \mathrm{dist}(x_{k}, x_{k+1})\\
  \implies     \mathrm{dist}(y_{k}, y^*(x_{k}))   & \le \frac{1}{1-\wt\gamma} \mathrm{dist}(y_{k}, y_{k+1})   + \frac{2L }{\mu(1- \wt\gamma)} \mathrm{dist}(x_{k}, x_{k+1}) \lesssim k^{- \frac{\alpha}{1-(\beta -1 )^2\theta}}  
    \end{align*}
    and hence for all $K \ge K_0$ from \eqref{lyapunov-RGA-MGD} we get :
    \begin{align}
        V_{\delta}(x_K, y_K) &= \bigg( \sup_{y \in x_K \times S_2} f(x_K,y) -  f(x_K,y_K) \bigg) + \bigg(f(x_K,y_K) - \inf_{x \in (\mathcal{S}^*(y_K) + \mathcal{B}_{\delta}(0)) \times y_K} f(x,y_K)\bigg) \\
         & \underbrace{\le}_{\textbf{C4}}  L  \, \mathrm{dist}(y_{K},y^*(x_K)) + L \,  \mathrm{dist}(x_{K}, \mathcal{S}^*(y_{K})) \\
         & \lesssim K^{- \frac{\alpha}{1-(\beta -1 )^2\theta}} . \label{VdeltaKLconvergencerate}
    \end{align}
    Further, $V_{\delta}(x_K, y_K) \ge 0 $ for all $K \ge K_0$ by iterate boundedness in $ ( \mathcal{S}^*(y_K)+\mathcal{B}_{\delta/2}( 0)) \times S_2 \subseteq  ( \mathcal{S}^*(y_{K_0})+\mathcal{B}_{3\delta/4}( 0)) \times S_2  $ as a consequence of \textbf{C3}, $ \lim_{K\to \infty} V_{\delta}(x_K, y_K) = 0 $ from \eqref{VdeltaKLconvergencerate} and thus from {joint uniform continuity of $V_{\delta}(\cdot,\cdot)$ on $S_1 \times S_2$, the Cauchy convergence of the sequence $\{(x_k, y_k)\}_k$ from \eqref{RGA-MGD1} - \eqref{RGA-MGD2},} the iteration $\{(x_k, y_k)\}_k$ from \eqref{RGA-MGD1} - \eqref{RGA-MGD2} converges to some $(x^*,y^*) \in \text{int}((\mathcal{S}^*(y_{K_0})+\mathcal{B}_{3\delta/4}( 0)) \times S_2) $ with $ V_{\delta}(x^*, y^*) = 0 $. \\
    {From \textbf{C3} there exists a connected component of local minima $\mathcal{S}^*(y^*) $ such that $ \mathrm{dist}_H(\mathcal{S}^*(y_{K_0}),\mathcal{S}^*(y^*)) < \delta/4$ and so $(x^*,y^*) \in \text{int}((\mathcal{S}^*(y^*)+\mathcal{B}_{\delta}( 0)) \times S_2) $. }
    \\ From $ V_{\delta}(x^*, y^*) = 0 $, $(x^*,y^*) \in \text{int}((\mathcal{S}^*(y^*)+\mathcal{B}_{\delta}( 0)) \times S_2) $ and Lemma \ref{RGA-MGD1-lem1}, $ (x^*,y^*) \in \text{int}(S_1 \times S_2)$ is a basin saddle point of $f$ and hence the sequence $\{(x_k, y_k)\}_k$ from \eqref{RGA-MGD1} - \eqref{RGA-MGD2} converges to a basin saddle point of $f$ at a rate 
    $\mathcal{O}(k^{- \frac{\alpha}{1-(\beta -1 )^2\theta}}) $.  
\end{proof}

\section{Tools from Riemannian geometry}\label{riemtools}
In this section we collect the necessary statements and results from Riemannian geometry that are needed in calculations.

\begin{theorem}(\citealp[Theorem 6.6.1]{jost})\label{jostdistthm} Let $(\cm,g)$ be a Riemannian $n$-manifold, let $p\in \cm$ and let $\exp_p:T_p\cm\to \cm$ be a diffeomoprhism on $\{v\in T_p\cm: \norm{v}_p \leq \rho\}$. Let the sectional curvature $\mathbf{K}$ of  $\cm$ in the closed ball
\[
 \bar{\mathcal{B}}_{\rho}(p) := \{ q\in \cm | d(p,q) \leq \rho\}
\]
satisfy 
\[
\lambda \leq \mathbf{K} \leq \mu, \text{ with } \lambda \leq 0, \mu \geq 0
\]
and suppose that $\rho < \frac{\pi}{2\sqrt{\mu}}$ in case $\mu >0$. Let $r(x) := d(x,p), \, k(x) :=\frac{1}{2}d^2(x,p)$. Then $k$ is smooth on $ \bar{\mathcal{B}}_{\rho}(p)$ and
\[
\nabla_pk(x) = -\exp_x^{-1}(p),
\]
and therefore $\norm{\nabla_pk(x)}_x = r(x)$. Moreover, we have 
\[
\sqrt{\mu} r(x) \cot(\sqrt{\mu}r(x)) \norm{v}_x^2 \leq \text{Hess }k(v,v) \leq \sqrt{-\lambda}r(x) \coth (\sqrt{-\lambda}r(x))\norm{v}_x^2
\]
for $x\in  \bar{\mathcal{B}}_{\rho}(p), v\in T_x M$.
\end{theorem}

\textbf{Definition.} Let $\ci_1, \ci_2\subset\R$ be intervals, and let $c:\ci_1 \to \cm$ be a geodesic. We say that a smooth map $\wt{\Gamma}: \ci_2\times \ci_1\to \cm$ is a \emph{variation through geodesics} if the curve $\wt{\Gamma}_s(t) = \wt{\Gamma}(s,t)$ is a geodesic for all $s\in \ci_2$. We define the \emph{variation field} $\partial_s\wt{\Gamma}: \ci_1\to T\cm$ as
\[
\partial_s\wt{\Gamma}(t) = \frac{\partial}{\partial s}\Big|_{s=0} \wt \Gamma(s,t).
\]

\textbf{Fact.} (\cite[Theorem 10.1]{lee2018}) Let $c:\ci_1\to \cm$ be a geodesic. If $\wt{\Gamma}:\ci_2\times \ci_1\to \cm$ is a variation of $c$ through geodesics then the variation field $\partial_s\wt{\Gamma}$ satisfies the \emph{Jacobi equation} 
\[
D_t^2J + R(J,c')c' = 0.
\]
Here $D_t$ denotes the covariant derivative along the curve $c$ and $R$ denotes the Riemann curvature tensor. A vector field $J$ along the geodesic $c$ is called a \emph{Jacobi field}.

The general existence theory of solutions to ODEs implies the following result.

\begin{proposition}(\cite[Proposition 10.2]{lee2018}) Let $(\cm,g)$ be a Riemannian manifold. Suppose $\ci\subset \R$ is an interval, $c:\ci\to \cm$ is a geodesic, $a\in \ci$, and $p=\ci(a)$. For every pair of vectors $v,w\in T_p\cm$, there is a unique Jacobi field $J$ along $c$ satisfying the initial conditions
\[
J(a) = v, \quad D_tJ(a) =w.
\]    
\end{proposition}

Let $\Omega \subset T\cm$ be the open set such that for each $p\in \cm$, $\Omega_p := \Omega\cap T_p\cm$ is the set consisting of all vectors $v\in T_p\cm$ such that there is a geodesic $c$ with 
\[
c(0) = p, c'(0) =v \quad \text{ and } c(1) \text{ is defined.}
\]
We know that the map $\exp:\Omega \to \cm$ given by $\exp(p,v) = \exp_p(v)$ is smooth.

\begin{proposition}\label{expder} Write the differential $d\exp: T_{(p,v)}\Omega\to T_{\exp(p,v)}\cm$ as\footnote{Here $\exp:\Omega\to \cm$ is the exponential map defined on an open subset of the tangent bundle $T\cm$, not to be confused with $\exp_p$ which is only defined on $T_p\cm$.}
\[
d\exp_{(p,v)} = (d_1\exp_{(p,v)}, d_2\exp_{(p,v)}). 
\]
Let $c:\ci\to \cm$ given by $c(t) = \exp_p(tv)$ be the geodesic with $c(0)= p$ and $c'(0) = v$ (where $\ci$ is some interval containing $0$). 
\begin{enumerate}[(a)]
    \item $d_1\exp_{(p,v)}(u) = J_1(1)$ where $J_1$ is a Jacobi field along $c$ such that $J_1(0) = u$ and $D_tJ_1(0) =0$.

    \item $d_2\exp_{(p,v)}(w) = J_2(1)$ where $J_2$ is a Jacobi field along $c$ such that $J_2(0) = 0$ and $D_tJ_2(0) = w$
\end{enumerate}
\end{proposition}

\begin{proof}
    Let $c_1:(-\epsilon,\epsilon)\to \cm$ be a curve such that 
    \[
    c_1(0) = p, \quad \text{ and }\quad c_1'(0) = u.
    \]
    Let $v(s)=\cp_{0\to s}v$ where $\cp_{0\to s}:T_{c_1(0)}\cm\to T_{c_1(s)}\cm$ is the parallel transport of $v$ along $c_1$. Then the map $\alpha:(-\epsilon,\epsilon)\times \ci\to \cm$ given by
    \[
    \alpha(s,t) = \exp_{c_1(s)}(tv(s))
    \]
    is a variation of $c$ through geodesics. Consequently, its variation field 
    \[
    J_1(t) = \frac{\partial}{\partial s}\Big|_{s=0} \alpha(s,t)
    \]
    is a Jacobi field with 
    \[
    J_1(0) = \frac{\partial}{\partial s}\Big|_{s=0}c_1(s) = u \quad \text{ and } \quad D_tJ_1(0) =D_t\partial_s\alpha|_{(0,0)} = D_s\partial_t\alpha|_{(0,0)}= D_sv(0) =0.
    \]
    The map $s\mapsto (c_1(s),v(s))$ is a curve from $(-\epsilon,\epsilon)\to \Omega$ whose velocity at $s=0$ is $(u,0)$. Thus
    \[
    d_1\exp_{(p,v)}(u) = \frac{\partial}{\partial s}\Big|_{s=0}\exp(c_1(s),v(s)) = \frac{\partial}{\partial s}\Big|_{s=0}\exp_{c_1(s)}v(s) =  \frac{\partial}{\partial s}\Big|_{s=0}\alpha(s,1) = J_1(1).
    \]
    This proves part (a).

    Let $c_2:(-\epsilon,\epsilon)\to T_p\cm$ be the map $c(s)=v+sw$. Then the map $\beta:(-\epsilon,\epsilon)\times \ci\to \cm$ given by
    \[
    \beta(s,t) = \exp_p(t(v+sw))
    \]
    is a variation of $c$ through geodesics. Consequently, its variation field 
    \[
    J_2(t) = \frac{\partial}{\partial s}\Big|_{s=0} \beta(s,t)
    \]
    is a Jacobi field with 
    \[
    J_2(0) = \frac{\partial}{\partial s}\Big|_{s=0}\beta(s,0) = 0 \quad \text{ and } \quad D_tJ_2(0) =D_t\partial_s\beta|_{(0,0)} = \partial_s D_t\beta|_{(0,0)} =\partial_s|_{s=0}(v+sw) =w.
    \]
    The map $s\mapsto (p,v+sw)$ is a curve from $(-\epsilon,\epsilon)\to \Omega$ whose velocity at $s=0$ is $(0,w)$. Thus
    \[
    d_2\exp_{(p,v)}(w)  = \frac{\partial}{\partial s}\Big|_{s=0}\exp_p(v+sw) =  \frac{\partial}{\partial s}\Big|_{s=0}\beta(s,1) = J_2(1).
    \]
    This proves part (b).
\end{proof}

\begin{proposition}
    Let $(\cm,g)$ be a Riemannian manifold, and let $c: \ci\to \cm$ be a geodesic defined on an interval $\ci\subset \R$. Then for any vector field $v$ along $c$ we have  
    \begin{equation}\label{cder}
    \frac{d^j}{dt^j}\cp_{c(t_0)\to c(t)}^{-1}v = \cp_{c(t_0)\to c(t)}^{-1} D_t^jv, \quad k = 1,2,\dots,
    \end{equation}
    for all $t_0,t\in I$. 
\end{proposition}

\begin{proof} Let $c(t_0) = p$, and pick coordinates $\{x^i\}$ centered at $p$. Let $t\mapsto e_i(t)$ be the vector field obtained by parallel transporting $\frac{\partial}{\partial x^i}\big|_{p}$ along $c$, i.e.,
\[
e_i(t) = \cp_{c(t_0)\to c(t)}\left(\frac{\partial}{\partial t}\Big|_{c(t_0)}\right).
\]
Since $\left\{\frac{\partial}{\partial x^i}\Big|_{p}\right\}$ are linearly independent, the vectors $\{e_i(t)\}$ are linearly independent in $T_{c(t)}M$. Thus, we can write 
\[
v(t) = v^i(t)e_i(t)
\]
and consequently 
\[
\cp_{c(t_0)\to c(t)}^{-1}v(t) = v^i(t) \frac{\partial}{\partial x^i}\Big|_p.
\]
So we have 
\begin{align*}
    \frac{d}{dt}\cp_{c(t_0)\to c(t)}^{-1}v(t) &= \frac{dv^i}{dt}(t)\frac{\partial}{\partial x^i}\Big|_p\\
    &= \frac{dv^i}{dt}(t)\cp_{c(t_0)\to c(t)}^{-1}\left(e_i(t)\right)\\
    &= \cp_{c(t_0)\to c(t)}^{-1}\left(\frac{dv^i}{dt}(t) e_i(t)\right)\\
    &= \cp_{c(t_0)\to c(t)}^{-1}\left(D_tv(t)\right),
\end{align*}
since $D_te_i =0$ for all $i=1,2\dots, n$. Applying the same reasoning to $D_t^{j-1}v$ gives \eqref{cder}.
\end{proof}

\begin{proposition}\label{tayexp}
    Let $(\cm,g)$ be a Riemannian manifold. Suppose $\ci\subset \R$ is an interval, $c:\ci\to \cm$ is a geodesic, $t_0\in \ci$, satisfying 
    \[
    c(t_0) = p, \quad c'(t_0) = v.
    \]
    Let $J$ be a Jacobi field along $c$ satisfying the initial conditions
    \[
    J(t_0) = u, \quad D_tJ(t_0) = w.
    \]
    Then we have the Taylor expansion
    \[
    \cp_{c(t_0)\to c(t)}^{-1}J(t) = u + (t-t_0)w - \frac{1}{2}(t-t_0)^2 R(u,v)v - \frac{1}{3!}\left((\nabla_vR)(u,v)v + R(w,v)v \right)(t-t_0)^3 + \mathcal{O}((t-t_0)^4).
    \]
    Moreover, the coefficients of the terms that appear in $\mathcal{O}((t-t_0)^4)$ are of degree three or higher in $v$. 
\end{proposition}

\begin{proof} The map $\wt{J}(t):= \cp_{c(t_0)\to c(t)}^{-1}J(t)$ is a map of the interval $\ci$ into a fixed vector space $T_p\cm$. So we have the Taylor series
\begin{align*}
    \wt{J}(t) &= \wt{J}(t_0) + \frac{d\wt{J}}{dt}(t_0)(t-t_0) + \frac{1}{2}\frac{d^2\wt{J}}{dt^2}(t_0) + \frac{1}{3!}\frac{d^3\wt{J}}{dt^3}(t_0)(t-t_0)^3 + \cdots \\
    &\underset{\text{by }\eqref{cder}}{=} J(t_0) + D_tJ(t_0)(t-t_0) + \frac{1}{2}D_t^2J(t_0)(t-t_0)^2 + \frac{1}{3!}D_t^3J(t_0)(t-t_0)^3 + \cdots. 
\end{align*}
By the initial conditions on $J$ we have
\[
J(t_0) = u, \quad \text{ and } \quad D_tJ(t_0) = w.
\]
Since $J$ is a Jacobi field it satisfies 
\begin{align*}
    D_t^2J &= - R(J, c')c' \\
    \implies D_t^2J(t_0) &= - R(J(t_0), c'(t_0))c'(t_0) = -R(u,v)v\\
    D_t^3J &= -D_t\left(R(J,c')c'\right)\\
           &= - (\nabla_{c'}R)(J,c')c' - R(\nabla_{c'}J,c')c' \quad \qquad (\text{since }\nabla_{c'}c'=0)\\
    \implies D_t^3J(t_0) &= - (\nabla_vR)(u,v)v - R(w,v)v\\
    D_t^4J &= - D_t\left( (\nabla_{c'}R)(J,c')c'\right) - D_t(R(\nabla_{c'}J,c')c') \\
    &= -(\nabla_{c'}\nabla_{c'}R)(J,c')c' - 2(\nabla_{c'}R)(\nabla_{c'}J,c')c' - R(D_t^2J,c')c' \\
    &= -(\nabla_{c'}\nabla_{c'}R)(J,c')c' - 2(\nabla_{c'}R)(\nabla_{c'}J,c')c' + R(R(J,c')c',c')c'\\
    \implies D_t^4J(t_0) &= -(\nabla_{v}\nabla_{v}R)(u,v)v - 2(\nabla_{v}R)(w,v)v + R(R(u,v)v,v)v.
\end{align*}
Thus, we see that $D_t^4J(t_0)$ is of degree $3$ in $v$. Repeated differentiation of the equation $D_t^2J +R(J,\gamma')\gamma'=0$ gives other Taylor coefficients and one can then see that these coefficients are at least of degree 3 in $v$. 
\end{proof}

\section{Proofs for section \ref{sectionDROBW} }\label{sectionDROBWappendix}
\subsection{\texorpdfstring{On the relaxation of constrained DRO problem \eqref{DRO1} to the penalized DRO problem \eqref{DRO1x}}{}}
Write $\V=(\muv,\Lambda)$, $\V^*=(\muv^*,\Lambda^*)$, and
\[
  \psi(\V):=\|\muv-\muv^*\|^2+d_{BW}^2(\Lambda,\Lambda^*),\qquad
  G_{\w}(\V):=\mathbb E_{\z\sim\mathbb P(\muv,\Lambda)}[\ell(\w;\z)] .
\]
Let $K_\xi:=\bar{\mathcal{B}}_\xi(\muv^*)\times\{\Lambda: d_{BW}(\Lambda,\Lambda^*)\le\xi\}$ denote the
ambiguity set of \eqref{DRO1}, and note $\{\psi\le\xi^2\}\subseteq K_\xi\subseteq\{\psi\le 2\xi^2\}$.
Let $\V_\epsilon(\w)$ be the maximizer of the
inner problem in \eqref{DRO1x}. Set
\[
  L_G:=\sup_{\w\in S_1,\;\V\in S_2}\big\|\operatorname{grad}_{\V} G_{\w}(\V)\big\|_{l_2 \oplus BW}<\infty 
\]
where $S_1, S_2$ are defined as in section \ref{convergenceratesBWsection} and $ \operatorname{grad}$ is the Riemannian gradient operator.
\begin{proposition}\label{prop:reduction}
Assume $K_\xi\subseteq S_2$ and that $\mathcal{L}(\cdot\,;\w)$ is $\mu$-strongly geodesically
concave on $S_2$ (Theorem~\ref{lagrangehessianestimatethm}). Fix $\epsilon>0$ and set $\xi:=\epsilon L_G/2$,
equivalently $\epsilon=2\xi/L_G$. Then, uniformly in $\w\in S_1$:
\begin{enumerate}
\item[(i)] $d\big(\V_\epsilon(\w),\V^*\big)\le \epsilon L_G/2=\xi$; in particular
      $\V_\epsilon(\w)\in K_\xi$, so it is feasible for \eqref{DRO1}.
\item[(ii)] Every maximizer $\V_\xi^\star(\w)$ of the inner problem in \eqref{DRO1} satisfies
      $d\big(\V_\epsilon(\w),\V_\xi^\star(\w)\big)\le(1+\sqrt2)\,\xi$.
\item[(iii)] $0\;\le\;\max_{\V \in K_{\xi}} G_{\w}\big(\V\big) -G_{\w}\big(\V_\epsilon(\w)\big)\;\le\;2\xi^2/\epsilon=\xi L_G$.
\end{enumerate}
\end{proposition}

\begin{proof}
(i) First-order optimality for the inner problem of \eqref{DRO1x} gives
$\operatorname{grad}G_{\w}(\V_\epsilon)=\tfrac1\epsilon\operatorname{grad}\psi(\V_\epsilon)$.
Since $S_2$ is contained in a normal neighborhood of $\V^*$, $\psi$ is smooth there with
$\|\operatorname{grad}\psi(\V)\|=2\,d(\V,\V^*)$, so
$2d(\V_\epsilon,\V^*)/\epsilon=\|\operatorname{grad}G_{\w}(\V_\epsilon)\|\le L_G$.
Feasibility follows since $\psi(\V_\epsilon)\le\xi^2$ implies $\V_\epsilon\in K_\xi$.

(ii) $\V_\xi^\star\in K_\xi$ gives $\psi(\V_\xi^\star)\le 2\xi^2$, i.e.
$d(\V_\xi^\star,\V^*)\le\sqrt2\,\xi$; using this with (i) and the triangle inequality proves (ii).

(iii) The left inequality is (i) together with optimality of $\V_\xi^\star$. For the right,
optimality of $\V_\epsilon$ for the penalized objective and feasibility of $\V_\xi^\star$ give
\[
  G_{\w}(\V_\epsilon)-\tfrac1\epsilon\psi(\V_\epsilon)\;\ge\;
  G_{\w}(\V_\xi^\star)-\tfrac1\epsilon\psi(\V_\xi^\star)\;\ge\;\max_{\V \in K_{\xi}} G_{\w}\big(\V\big)-\tfrac{2\xi^2}{\epsilon},
\]
and $\psi(V_\epsilon)\ge0$. 
\end{proof}
\subsection{Supporting results}
{
\begin{lem}\label{hessiancalc1a}
    The following bound holds:
    \[  \langle (\z-\muv),S\Lambda^{-1}(\z-\muv)\rangle \le  \norm{S}_F \norm{  \Lambda^{-1}}_2 \norm{\Lambda}_2  (\z - \muv)^T {\Lambda}^{-1} (\z - \muv) \]
\end{lem}
\begin{proof}
    \begin{align*}
    \langle (\z-\muv),S\Lambda^{-1}(\z-\muv)\rangle & = \langle S (\z-\muv),\Lambda^{-1}(\z-\muv)\rangle \\
    & = \text{tr}( \langle S (\z-\muv),\Lambda^{-1}(\z-\muv)\rangle ) \\
     & = \text{tr}(  S  (\z-\muv)(\z-\muv)^T \Lambda^{-1}  ) \\
    & \le \norm{S}_F \norm{(\z-\muv)(\z-\muv)^T \Lambda^{-1}}_F \\
      & = \norm{S}_F \norm{(\z-\muv)} \norm{  \Lambda^{-1}(\z-\muv)}\\
      & = \norm{S}_F \norm{(\z-\muv)}^2 \norm{  \Lambda^{-1}}_2 \\
      & \le \norm{S}_F \norm{  \Lambda^{-1}}_2 \norm{\Lambda}_2  (\z - \muv)^T\Lambda^{-1}(\z - \muv)  
\end{align*}
where we used in the last step that $\norm{(\z-\muv)}^2 \le \norm{\Lambda}_2  (\z - \muv)^T\Lambda^{-1}(\z - \muv)$. This completes the proof.
\end{proof}
}

\subsection{Proof of Theorem \ref{losshessianestimatethm}}
\begin{proof}
  We want to calculate $\text{Hess}_{(\muv,\Lambda)}f$ where $f$ is the function
\[
f(\w; \Lambda, \muv) = \int_{\z \in \mathbb{R}^n} \ell(\w;\z)   \exp\bigg(-\frac{1}{2}\langle(\z -\muv),\Lambda^{-1}(\z - \muv)\rangle\bigg) \frac{d\z}{\sqrt{\det(2\pi \Lambda)}} .
\]
Let $(\nuv, S)\in  \Rn^n\oplus \bs^n \cong T_{(\muv,\Lambda)}(\Rn^n\times BW) $ be given so that $\text{tr}(S^TS)<1$ $\iff$ $ \norm{S}_F <1$. Then $I+tS$ is invertible for all $t\in [0,1]$ (since operator norm is bounded above by the Frobenius norm, $(I+tS)^{-1}=\sum_{k=0}^{\infty} (-1)^kt^kS^k$), and the curve $\tau:[0,1]\to \Rn^n\oplus \text{BW}$ given by
\begin{equation}\label{taudef}
\tau(t) = (\muv_t,\Lambda_t), \quad \Lambda_t = (I+tS)\Lambda(I+tS), \muv_t = \muv+t\nuv
\end{equation}
is a geodesic beginning at $(\muv,\Lambda)$ whose velocity at $t=0$ is $(\nuv, S\Lambda+\Lambda S)$.
Next, we have that
\[
\text{Hess}_{(\muv,\Lambda)}f((\nuv,  S\Lambda +\Lambda S),(\nuv,S\Lambda +\Lambda S)) = \frac{d^2}{dt^2}\Big|_{t=0}f(\tau(t)).
\]
{
We want to show that
\[
\frac{ \abs{\text{Hess}_{(\muv,\Lambda)}f((\nuv,  S\Lambda +\Lambda S),(\nuv,S\Lambda +\Lambda S))}}{\langle (\nuv,  S\Lambda +\Lambda S) , (\nuv,  S\Lambda +\Lambda S)\rangle_{l_2 \oplus BW}} = \frac{ \abs{\text{Hess}_{(\muv,\Lambda)}f((\nuv,  S\Lambda +\Lambda S),(\nuv,S\Lambda +\Lambda S))}}{ \norm{\nuv}^2 + tr(S \Lambda S)}  \leq C < \infty 
\]
uniformly for all $\nuv, S$. Since $S \Lambda S \succeq \mathbf{0}$, $\Lambda  \succ \mathbf{0}$  and
\[ tr(S \Lambda S) = tr(S \Lambda^{1/2} \Lambda^{1/2} S) = \norm{S \Lambda^{1/2} }^2_F \ge  \norm{S }^2_F  \norm{ \Lambda^{-1/2} }^{-2}_F \]
it suffices to show that 
\begin{align}
  \sup_{\nuv \in \mathbb{R}^n, \norm{S}_F < 1} \frac{ \abs{\text{Hess}_{(\muv,\Lambda)}f((\nuv,  S\Lambda +\Lambda S),(\nuv,S\Lambda +\Lambda S))}}{ \norm{\nuv}^2 +  \norm{S }^2_F  \norm{ \Lambda^{-1/2} }^{-2}_F} =  \sup_{\nuv \in \mathbb{R}^n, \norm{S}_F < 1} \frac{ \abs{\frac{d^2}{dt^2}\Big|_{t=0}f(\tau(t))}}{ \norm{\nuv}^2 +  \norm{S }^2_F  \norm{ \Lambda^{-1/2} }^{-2}_F}  \leq C < \infty    \label{holderestimate001}
\end{align}
}
So we need to calculate 
\begin{align*}
    &\frac{d^2}{dt^2}\Big|_{t=0} \exp\bigg(-\frac{1}{2}\langle(\z -\muv_t),\Lambda_t^{-1}(\z - \muv_t)\rangle\bigg) \frac{1}{\sqrt{\det(2\pi \Lambda_t)}}  \\
    &=\frac{d^2}{dt^2}\Big|_{t=0} \exp\bigg(-\frac{1}{2}\langle(\z -\muv_t),\Lambda_t^{-1}(\z - \muv_t)\rangle -\frac{1}{2}\log \det(2\pi\Lambda_t)\bigg) 
\end{align*}
Since $\frac{d^2}{dt^2}e^{-u(t)} = e^{-u(t)}(-u''(t)+u'(t)u'(t))$ we let 
\begin{equation}\label{utdef}
u(t) = \frac{1}{2}\langle(\z -\muv_t),\Lambda_t^{-1}(\z - \muv_t)\rangle +\frac{1}{2}\log \det(2\pi\Lambda_t),
\end{equation}
and calculate\\ 
{
\begin{align*}
    u'(t) &= \langle (\z-\muv_t)',\Lambda_t^{-1}(\z-\muv_t)\rangle - \frac{1}{2}\langle (\z-\muv_t),\Lambda_t^{-1}\Lambda_t'\Lambda_t^{-1}(\z-\muv_t)\rangle + \frac{1}{2} \text{tr}(\Lambda_t^{-1}\Lambda'_t) \\
    u''(t) &= \langle (\z-\muv_t)',\Lambda_t^{-1}(\z-\muv_t)'\rangle - \langle (\z-\muv_t)',\Lambda_t^{-1}\Lambda_t'\Lambda_t^{-1}(\z-\muv_t)\rangle \\
    &- \langle (\z-\muv_t)',\Lambda_t^{-1}\Lambda_t'\Lambda_t^{-1}(\z-\muv_t)\rangle -\frac{1}{2} \langle (\z-\muv_t), \Lambda_t^{-1}(-2\Lambda'_t\Lambda_t^{-1}\Lambda'_t + \Lambda_t'')\Lambda_t^{-1}(\z-\muv_t)\rangle \\
    &- \frac{1}{2} \text{tr}(\Lambda_t^{-1}\Lambda'_t\Lambda_t^{-1}\Lambda'_t) + \frac{1}{2} \text{tr}(\Lambda_t^{-1}\Lambda''_t)
\end{align*}
So 
\begin{align*}
    u'(0) &=  \langle -\nuv,\Lambda^{-1}(\z-\muv)\rangle - \frac{1}{2}\langle (\z-\muv),\Lambda^{-1}(S\Lambda+\Lambda S)\Lambda^{-1}(\z-\muv)\rangle + \frac{1}{2} \text{tr}(\Lambda^{-1}(S\Lambda+\Lambda S))\\
          &= \langle -\nuv,\Lambda^{-1}(\z-\muv)\rangle - \frac{1}{2}\langle (\z-\muv),(\Lambda^{-1} S + S\Lambda^{-1})(\z-\muv)\rangle + \text{tr}(S)\\
          &= -\langle \nuv, \Lambda^{-1}(\z -\muv)\rangle -\langle (\z-\muv),S\Lambda^{-1}(\z-\muv)\rangle + \text{tr}(S) \\
          \implies \abs{u'(0)} & \underbrace{\leq}_{\textbf{Lemma \ref{hessiancalc1a}}} \norm{\Lambda^{-1/2} \nuv}\norm{\Lambda^{-1/2}(\z-\muv)} + \norm{S}_F \norm{  \Lambda^{-1}}_2 \norm{\Lambda}_2\langle (\z-\muv), \Lambda^{-1}(\z-\muv)\rangle + \text{tr}(S) \\
          \implies \abs{u'(0)} & \underbrace{\leq}_{\norm{S}_F^2 = \text{tr}(S^2)} \norm{\Lambda^{-1/2} \nuv}\norm{\Lambda^{-1/2}(\z-\muv)} + \norm{S}_F \norm{  \Lambda^{-1}}_2 \norm{\Lambda}_2\langle (\z-\muv), \Lambda^{-1}(\z-\muv)\rangle + \sqrt{n}\norm{S}_F .
    \end{align*}     
    Then for $ \Lambda^{-1/2}(\z-\muv) \mapsto \z $ where $ \z \sim \mathcal{N}(\mathbf{0}_n, I)$ and using Jensen's inequality for $p > 1$ we get :
    \begin{align*}
          \abs{u'(0)}^p & \leq 3^{p-1}\bigg(\norm{\Lambda^{-1/2} }_F^p \norm{\nuv}^p\norm{\z}^p + \norm{S}^p_F \norm{\Lambda^{-1}}^p_2 \norm{\Lambda}^p_2\norm{\z}^{2p} + \sqrt{n}^p\norm{S}^p_F  \bigg)
    \end{align*}
    Then taking expectation operator both sides in the above bound w.r.t. the random vector $ \z \sim \mathcal{N}(\mathbf{0}_n, I)$ and using the fact that $\E[\norm{\z}^{q}] = 2^{q/2} \frac{\Gamma(\frac{n+q}{2}) }{\Gamma(\frac{n}{2}) } $ for any $q>-n$ (see \cite{vershynin2018high}) yields:
       \begin{align}
          \E[\abs{u'(0)}^p] & \leq 3^{p-1}\bigg(\norm{\Lambda^{-1/2} }_F^p \norm{\nuv}^p 2^{p/2} \frac{\Gamma(\frac{n+p}{2}) }{\Gamma(\frac{n}{2}) } + \norm{S}^p_F \norm{\Lambda^{-1}}^p_2 \norm{\Lambda}^p_2 2^{p} \frac{\Gamma(\frac{n+2p}{2}) }{\Gamma(\frac{n}{2}) } + \sqrt{n}^p\norm{S}^p_F  \bigg) \label{bwhessianestimate1a}
    \end{align}
    Further, using the last bound \eqref{bwhessianestimate1a} for $p=2q >2$ followed by the subadditivity of nonnegative concave function $g(t) := t^{1/q}$ for $q>1$ and $t\ge 0$, i.e. $ g(\sum_i t_i ) \le \sum_i g(t_i)$ for $t_i \ge 0$, we have the following estimate:
    \begin{align}
         \sup_{\nuv \in \mathbb{R}^n, \norm{S}_F < 1} \frac{ (\E[\abs{u'(0)}^{2q}])^{1/q}}{ \norm{\nuv}^2 +  \norm{S }^2_F  \norm{ \Lambda^{-1/2} }^{-2}_F}  &\le  3^{2-\frac{1}{q}}\Bigg( 2 \bigg(\frac{\Gamma(\frac{n+2q}{2}) }{\Gamma(\frac{n}{2}) } \bigg)^\frac{1}{q} \sup_{\nuv \in \mathbb{R}^n, \norm{S}_F < 1} \frac{ \norm{\Lambda^{-1/2} }_F^2 \norm{\nuv}^2 }{ \norm{\nuv}^2 +  \norm{S }^2_F  \norm{ \Lambda^{-1/2} }^{-2}_F} \nonumber \\ & + 2^{2} \bigg(\frac{\Gamma(\frac{n+4q}{2}) }{\Gamma(\frac{n}{2}) }\bigg)^\frac{1}{q} \norm{\Lambda^{-1}}^2_2 \norm{\Lambda}^2_2\sup_{\nuv \in \mathbb{R}^n, \norm{S}_F < 1} \frac{\norm{S}^2_F}{ \norm{\nuv}^2 + \norm{S }^2_F  \norm{ \Lambda^{-1/2} }^{-2}_F} \nonumber \\   &+ n\sup_{\nuv \in \mathbb{R}^n, \norm{S}_F < 1} \frac{\norm{S }^2_F}{ \norm{\nuv}^2 +  \norm{S }^2_F  \norm{ \Lambda^{-1/2} }^{-2}_F}  \Bigg) \\
         & \hspace{-0.2cm}\le 3^{2-\frac{1}{q}} \Bigg( 2 \norm{ \Lambda^{-1/2} }^{2}_F \bigg(\frac{\Gamma(\frac{n+2q}{2}) }{\Gamma(\frac{n}{2}) } \bigg)^\frac{1}{q}  + 4 \norm{\Lambda^{-1}}^2_2 \norm{\Lambda}^2_2\bigg(\frac{\Gamma(\frac{n+4q}{2}) }{\Gamma(\frac{n}{2}) }\bigg)^\frac{1}{q} + n  \Bigg)  \label{bwhessianestimate1a0}
    \end{align}
    }

   { Next, we have the estimate on $ u''(0)$
    \begin{align*}      
    u''(0)&= \langle \nuv,\Lambda^{-1} \nuv\rangle - \langle -\nuv,\Lambda^{-1}(S\Lambda + \Lambda S)\Lambda^{-1}(\z-\muv)\rangle \\
    &- \langle -\nuv,\Lambda^{-1}(S\Lambda+\Lambda S)\Lambda^{-1}(\z-\muv)\rangle \nonumber \\ & -\frac{1}{2} \langle (\z-\muv), \Lambda^{-1}(-2(S\Lambda+\Lambda S)\Lambda^{-1}(S\Lambda + \Lambda S) + 2S\Lambda S)\Lambda^{-1}(\z-\muv)\rangle \\
    &- \frac{1}{2} \text{tr}(\Lambda^{-1}(S\Lambda + \Lambda S)\Lambda^{-1}(S\Lambda + \Lambda S)) + \frac{1}{2} \text{tr}(\Lambda^{-1}2S\Lambda S) \\
   u''(0) &= \langle \nuv,\Lambda^{-1} \nuv\rangle +2\langle \Lambda^{-1/2}(S\Lambda + \Lambda S)\Lambda^{-1} \nuv,\Lambda^{-1/2}(\z-\muv)\rangle \nonumber \\ & -\frac{1}{2} \langle \Lambda^{-1/2}(\z-\muv), \Lambda^{-1/2}(-2(S\Lambda+\Lambda S)\Lambda^{-1}(S\Lambda + \Lambda S) + 2S\Lambda S)\Lambda^{-1/2}\Lambda^{-1/2}(\z-\muv)\rangle \\
    &- \frac{1}{2} \text{tr}(\Lambda^{-1}(S\Lambda + \Lambda S)\Lambda^{-1}(S\Lambda + \Lambda S)) +  \text{tr}(\Lambda^{-1}S\Lambda S) \\
    \implies \abs{u''(0)} & \le \norm{\Lambda^{-1}}\norm{\nuv}^2 + 2 \norm{ \Lambda^{-1/2}(S\Lambda + \Lambda S)\Lambda^{-1} \nuv} \norm{\Lambda^{-1/2}(\z-\muv)} \nonumber \\
    & +  \norm{\Lambda^{-1/2}(-(S\Lambda+\Lambda S)\Lambda^{-1}(S\Lambda + \Lambda S) + S\Lambda S)\Lambda^{-1/2}} \norm{\Lambda^{-1/2}(\z-\muv)}^2 \nonumber \\ 
    & + \frac{1}{2}\norm{\Lambda^{-1}}^2_F \norm{S\Lambda + \Lambda S}^2_F + \norm{\Lambda^{-1}}_F\text{tr}(S\Lambda S) \\
  \implies \abs{u''(0)}  & \le \norm{\Lambda^{-1}}_F\norm{\nuv}^2 + 4 \norm{\Lambda^{-1/2}}^2_F\norm{\Lambda^{1/2}}_F\norm{S}_F \norm{ \nuv} \norm{\Lambda^{-1/2}(\z-\muv)} \nonumber \\
    & + \bigg(\norm{\Lambda^{-1/2}(S\Lambda+\Lambda S)\Lambda^{-1/2}}^2_F + \norm{\Lambda^{-1/2} S \Lambda^{1/2}}^2_F \bigg) \norm{\Lambda^{-1/2}(\z-\muv)}^2 \nonumber \\ 
    & + 2\norm{\Lambda^{-1}}^2_F \norm{\Lambda}^2_F \norm{S}^2_F + \norm{\Lambda^{-1}}_F\norm{\Lambda}_F \norm{S}^2_F
\end{align*}
Then for $ \Lambda^{-1/2}(\z-\muv) \mapsto \z $ where $ \z \sim \mathcal{N}(\mathbf{0}_n, I)$ and using Jensen's inequality for $p > 1$ we get :
    \begin{align*}
           \abs{u''(0)}^p & \leq 6^{p-1}\Bigg(\norm{\Lambda^{-1}}^p_F\norm{\nuv}^{2p} + 4^p \norm{\Lambda^{-1/2}}^{2p}_F\norm{\Lambda^{1/2}}^p_F\norm{S}^p_F \norm{ \nuv}^p \norm{\z}^p \nonumber \\
    & + 4^p \norm{\Lambda^{-1/2}}^{4p}_F \norm{\Lambda}^{2p}_F \norm{S}^{2p}_F \norm{\z}^{2p} + \norm{\Lambda^{-1/2} }^{2p}_F \norm{\Lambda^{1/2} }^{2p}_F \norm{S}^{2p}_F \norm{\z}^{2p} \nonumber \\ 
    & + 2^p\norm{\Lambda^{-1}}^{2p}_F \norm{\Lambda}^{2p}_F \norm{S}^{2p}_F + \norm{\Lambda^{-1}}^p_F\norm{\Lambda}^p_F \norm{S}^{2p}_F \Bigg)
    \end{align*}
     Taking expectation operator both sides in the above bound w.r.t. the random vector $ \z \sim \mathcal{N}(\mathbf{0}_n, I)$ and again using the fact that $\E[\norm{\z}^{q}] = 2^{q/2} \frac{\Gamma(\frac{n+q}{2}) }{\Gamma(\frac{n}{2}) } $ for any $q>-n$ (see \cite{vershynin2018high}) yields:
       \begin{align}
          \E[\abs{u''(0)}^p] & \leq 6^{p-1}\Bigg(\norm{\Lambda^{-1}}^p_F\norm{\nuv}^{2p} + 4^p \norm{\Lambda^{-1/2}}^{2p}_F\norm{\Lambda^{1/2}}^p_F\norm{S}^p_F \norm{ \nuv}^p 2^{p/2} \frac{\Gamma(\frac{n+p}{2}) }{\Gamma(\frac{n}{2}) } \nonumber \\
    & + 4^p\norm{\Lambda^{-1/2}}^{4p}_F \norm{\Lambda}^{2p}_F \norm{S}^{2p}_F 2^{p} \frac{\Gamma(\frac{n+2p}{2}) }{\Gamma(\frac{n}{2}) } + \norm{\Lambda^{-1/2} }^{2p}_F \norm{\Lambda^{1/2} }^{2p}_F \norm{S}^{2p}_F 2^{p} \frac{\Gamma(\frac{n+2p}{2}) }{\Gamma(\frac{n}{2}) } \nonumber \\ 
    & + 2^p\norm{\Lambda^{-1}}^{2p}_F \norm{\Lambda}^{2p}_F \norm{S}^{2p}_F + \norm{\Lambda^{-1}}^p_F\norm{\Lambda}^p_F \norm{S}^{2p}_F \Bigg) \label{bwhessianestimate1b}
    \end{align}
    Further, using the last bound \eqref{bwhessianestimate1b} for $p=q >1$ followed by the subadditivity of nonnegative concave function $g(t) := t^{1/q}$ for $q>1$ and $t\ge 0$, i.e. $ g(\sum_i t_i ) \le \sum_i g(t_i)$ for $t_i \ge 0$, we have the following estimate:
    \begin{align}
         \sup_{\nuv \in \mathbb{R}^n, \norm{S}_F < 1} \frac{ (\E[\abs{u''(0)}^{q}])^{1/q}}{ \norm{\nuv}^2 +  \norm{S }^2_F  \norm{ \Lambda^{-1/2} }^{-2}_F}  &\le  6^{1-\frac{1}{q}}\Bigg( \norm{\Lambda^{-1}}_F\sup_{\nuv \in \mathbb{R}^n, \norm{S}_F < 1} \frac{\norm{\nuv}^{2}}{ \norm{\nuv}^2 +  \norm{S }^2_F  \norm{ \Lambda^{-1/2} }^{-2}_F} \nonumber \\
    &  \hspace{-4cm} + 4 \sqrt{2}  \norm{\Lambda^{-1/2}}^{2}_F\norm{\Lambda^{1/2}}_F   \bigg(\frac{\Gamma(\frac{n+q}{2}) }{\Gamma(\frac{n}{2}) } \bigg)^{1/q} \sup_{\nuv \in \mathbb{R}^n, \norm{S}_F < 1} \underbrace{\frac{\norm{S}_F \norm{ \nuv}}{ \norm{\nuv}^2 +  \norm{S }^2_F  \norm{ \Lambda^{-1/2} }^{-2}_F}}_{ \le \frac{\norm{S}_F \norm{ \nuv}}{ 2 \norm{\nuv} \norm{S }_F  \norm{ \Lambda^{-1/2} }^{-1}_F} \textbf{ by AM-GM}} \nonumber \\
    & \hspace{-4cm} + \bigg(2\norm{\Lambda^{-1/2} }^{2}_F \norm{\Lambda^{1/2} }^{2}_F +  8\norm{\Lambda^{-1/2}}^{4}_F \norm{\Lambda}^2_F \bigg)   \bigg(\frac{\Gamma(\frac{n+2q}{2}) }{\Gamma(\frac{n}{2}) }\bigg)^{1/q}  \sup_{\nuv \in \mathbb{R}^n, \norm{S}_F < 1} \frac{ \norm{S}^2_F}{ \norm{\nuv}^2 +  \norm{S }^2_F  \norm{ \Lambda^{-1/2} }^{-2}_F} \nonumber \\
    &  \hspace{-4cm} + \bigg( 2\norm{\Lambda^{-1}}^{2}_F \norm{\Lambda}^{2}_F  + \norm{\Lambda^{-1}}_F\norm{\Lambda}_F  \bigg)  \sup_{\nuv \in \mathbb{R}^n, \norm{S}_F < 1} \frac{\norm{S}^{2}_F}{ \norm{\nuv}^2 +  \norm{S }^2_F  \norm{ \Lambda^{-1/2} }^{-2}_F} \Bigg) \\
         & \hspace{-4cm} \le  6^{1-\frac{1}{q}}\Bigg( \norm{\Lambda^{-1}}_F  + 2 \sqrt{2}  \norm{\Lambda^{-1/2}}^{3}_F\norm{\Lambda^{1/2}}_F   \bigg(\frac{\Gamma(\frac{n+q}{2}) }{\Gamma(\frac{n}{2}) } \bigg)^{1/q}  \nonumber \\
    & \hspace{-4cm} + \bigg(2\norm{\Lambda^{-1/2} }^{4}_F \norm{\Lambda^{1/2} }^{2}_F +  8\norm{\Lambda^{-1/2}}^{6}_F \norm{\Lambda}^2_F \bigg)   \bigg(\frac{\Gamma(\frac{n+2q}{2}) }{\Gamma(\frac{n}{2}) }\bigg)^{1/q}  \nonumber \\
    &  \hspace{-4cm} + \bigg( 2\norm{\Lambda^{-1}}^{2}_F \norm{\Lambda}^{2}_F  + \norm{\Lambda^{-1}}_F\norm{\Lambda}_F  \bigg)  \norm{ \Lambda^{-1/2} }^{2}_F \Bigg)  \label{bwhessianestimate1b0}
    \end{align}
}

{Then the Hessian of the expected loss is given by:
\begin{align}
\frac{d^2}{dt^2} f(\w; \Lambda_t, \muv_t) \bigg\vert_{t=0} &= \int_{\z \in \mathbb{R}^n} \ell(\w;\z)  \frac{d^2}{dt^2}\bigg\vert_{t=0} e^{-u(t)} d\z \\
&= \int_{\z \in \mathbb{R}^n} \ell(\w;\z)  \bigg( e^{-u(0)}(-u''(0)+(u'(0))^2)\bigg) d\z \\
& = \E [ \ell(\w;\z) (u'(0))^2] - \E [ \ell(\w;\z) u''(0)] \\
\implies \abs{\frac{d^2}{dt^2} f(\w; \Lambda_t, \muv_t) \bigg\vert_{t=0}} & \leq (\E [ |\ell(\w;\z)|^p])^{1/p}  (\E [ |u'(0)|^{2q}])^{1/q} \nonumber \\ &  + (\E [ |\ell(\w;\z)|^p])^{1/p}  (\E [ |u''(0)|^{q}])^{1/q} \label{bwhessianestimate1c}
\end{align}
where expectation is with respect to a Gaussian measure with parameters $\muv, \Lambda$ and the last step follows from triangle and H\"{o}lder inequality for $ 1/p + 1/q =1$, $p,q \in [1,\infty]$.
}
{Using the estimates \eqref{bwhessianestimate1a0}, \eqref{bwhessianestimate1b0} in \eqref{bwhessianestimate1c} for $q \ge 1$ yields the following estimate:
\begin{align}
  \sup_{\nuv \in \mathbb{R}^n, \norm{S}_F < 1} \frac{ \abs{\frac{d^2}{dt^2}\Big|_{t=0}f(\tau(t))}}{ \norm{\nuv}^2 +  \norm{S }^2_F  \norm{ \Lambda^{-1/2} }^{-2}_F}  &\leq  (\E [ |\ell(\w;\z)|^p])^{1/p} \Bigg( \sup_{\nuv \in \mathbb{R}^n, \norm{S}_F < 1} \frac{ (\E [ |u'(0)|^{2q}])^{1/q}}{ \norm{\nuv}^2 +  \norm{S }^2_F  \norm{ \Lambda^{-1/2} }^{-2}_F}  \nonumber \\ & + \sup_{\nuv \in \mathbb{R}^n, \norm{S}_F < 1} \frac{ (\E [ |u''(0)|^{q}])^{1/q}}{ \norm{\nuv}^2 +  \norm{S }^2_F  \norm{ \Lambda^{-1/2} }^{-2}_F} \Bigg) \\
  & \hspace{-5cm}\leq  (\E [ |\ell(\w;\z)|^p])^{1/p} \Bigg[ 3^{2-\frac{1}{q}} \Bigg( 2 \norm{ \Lambda^{-1/2} }^{2}_F\bigg(\frac{\Gamma(\frac{n+2q}{2}) }{\Gamma(\frac{n}{2}) } \bigg)^\frac{1}{q}  + 4 \norm{ \Lambda }^{2}_2\norm{ \Lambda^{-1} }^{2}_2\bigg(\frac{\Gamma(\frac{n+4q}{2}) }{\Gamma(\frac{n}{2}) }\bigg)^\frac{1}{q} + n  \Bigg)   \nonumber \\ & \hspace{-5cm} +  6^{1-\frac{1}{q}}\Bigg( \norm{\Lambda^{-1}}_F  + 2 \sqrt{2}  \norm{\Lambda^{-1/2}}^{3}_F\norm{\Lambda^{1/2}}_F   \bigg(\frac{\Gamma(\frac{n+q}{2}) }{\Gamma(\frac{n}{2}) } \bigg)^{1/q}  \nonumber \\
    & \hspace{-2cm} + \bigg(2\norm{\Lambda^{-1/2} }^{4}_F \norm{\Lambda^{1/2} }^{2}_F +  8\norm{\Lambda^{-1/2}}^{6}_F \norm{\Lambda}^2_F \bigg)   \bigg(\frac{\Gamma(\frac{n+2q}{2}) }{\Gamma(\frac{n}{2}) }\bigg)^{1/q}  \nonumber \\
    &  \hspace{-2cm} + \bigg( 2\norm{\Lambda^{-1}}^{2}_F \norm{\Lambda}^{2}_F  + \norm{\Lambda^{-1}}_F\norm{\Lambda}_F  \bigg)  \norm{ \Lambda^{-1/2} }^{2}_F \Bigg) \Bigg] \\
    & \hspace{-2cm}  = C_{\Lambda, q, n} (\E [ |\ell(\w;\z)|^p])^{1/p} \\
    & \hspace{-4cm} \underbrace{\le}_{\eqref{Taylor13co} \, , 1/p +1/q =1 \, , \, p \ge 1}  C_{\Lambda, q, n} (C''_{p,\Lambda^*,n} \times \big( \mathbb{E}_{\z \sim \probP(\muv^*,\Lambda^*)}\left[\abs{\ell(\w;\z)}^{pq}\right] \big)^{1/q}  )^{1/p}
    \underbrace{<}_{\textbf{A2}} \infty  \label{bwhessianestimate1d}
\end{align}
where 
\begin{align}
   C_{\Lambda, q, n}  & = \Bigg[ 3^{2-\frac{1}{q}} \Bigg( 2 \norm{ \Lambda^{-1/2} }^{2}_F\bigg(\frac{\Gamma(\frac{n+2q}{2}) }{\Gamma(\frac{n}{2}) } \bigg)^\frac{1}{q}  + 4 \norm{ \Lambda }^{2}_2\norm{ \Lambda^{-1} }^{2}_2\bigg(\frac{\Gamma(\frac{n+4q}{2}) }{\Gamma(\frac{n}{2}) }\bigg)^\frac{1}{q} + n  \Bigg)   \nonumber \\ & \hspace{0cm} +  6^{1-\frac{1}{q}}\Bigg( \norm{\Lambda^{-1}}_F  + 2 \sqrt{2}  \norm{\Lambda^{-1/2}}^{3}_F\norm{\Lambda^{1/2}}_F   \bigg(\frac{\Gamma(\frac{n+q}{2}) }{\Gamma(\frac{n}{2}) } \bigg)^{1/q}  \nonumber \\
    & \hspace{0cm} + \bigg(2\norm{\Lambda^{-1/2} }^{4}_F \norm{\Lambda^{1/2} }^{2}_F +  8\norm{\Lambda^{-1/2}}^{6}_F \norm{\Lambda}^2_F \bigg)   \bigg(\frac{\Gamma(\frac{n+2q}{2}) }{\Gamma(\frac{n}{2}) }\bigg)^{1/q}  \nonumber \\
    &  \hspace{0cm} + \bigg( 2\norm{\Lambda^{-1}}^{2}_F \norm{\Lambda}^{2}_F  + \norm{\Lambda^{-1}}_F\norm{\Lambda}_F  \bigg)  \norm{ \Lambda^{-1/2} }^{2}_F \Bigg) \Bigg]   \label{bwhessianestimate1d0}
\end{align}
and for $ \norm{\Lambda}_F \ge 1 \, , \, \norm{\Lambda^{-1/2} }_F \ge 1$, $ C_{\Lambda, q, n} \sim_{q} \bigg(\frac{\Gamma(\frac{n+4q}{2}) }{\Gamma(\frac{n}{2}) }\bigg)^\frac{1}{q}\norm{\Lambda}^2_F \norm{\Lambda^{-1/2} }^{6}_F $ is uniformly bounded by using the estimates \eqref{Taylor13ao}, \eqref{Taylor13bo} as follows:
$$ C_{\Lambda, q, n} \sim_{q} \bigg(\frac{\Gamma(\frac{n+4q}{2}) }{\Gamma(\frac{n}{2}) }\bigg)^\frac{1}{q}\norm{\Lambda}^2_F \norm{\Lambda^{-1/2} }^{6}_F \sim_{q} \bigg(\frac{\Gamma(\frac{n+4q}{2}) }{\Gamma(\frac{n}{2}) }\bigg)^\frac{1}{q}\norm{(\Lambda^*)^{1/2}}^4_F \bigg( \frac{\sqrt{n} \norm{(\Lambda^*)^{-1}}_F}{(1-r \norm{(\Lambda^*)^{-1/2}}_F)}\bigg)^{6} \,  $$
and 
\begin{align*}
C''_{p,\Lambda^*,n}  & \lesssim_p n \max\{ 1, r^{2}\} \norm{(\Lambda^*)^{-1/2}}_F^{4} \norm{(\Lambda^*)^{1/2}}_F^{4} \big(\frac{\Gamma(\frac{n+4p}{2}) }{\Gamma(\frac{n}{2}) }\big)^{1/p}
\end{align*}
 for $ \norm{\Lambda^*}_F \ge 1 \, , \, \norm{(\Lambda^*)^{-1/2} }_F \ge 1 $ from \eqref{Taylor13co}.
Thus, from \eqref{holderestimate001} we have shown that 
\begin{align}
  \sup_{\nuv \in \mathbb{R}^n, \norm{S}_F < 1} \frac{ \abs{\text{Hess}_{(\muv,\Lambda)}f((\nuv,  S\Lambda +\Lambda S),(\nuv,S\Lambda +\Lambda S))}}{ \norm{\nuv}^2 +  \norm{S }^2_F  \norm{ \Lambda^{-1/2} }^{-2}_F} =  \sup_{\nuv \in \mathbb{R}^n, \norm{S}_F < 1} \frac{ \abs{\frac{d^2}{dt^2}\Big|_{t=0}f(\tau(t))}}{ \norm{\nuv}^2 +  \norm{S }^2_F  \norm{ \Lambda^{-1/2} }^{-2}_F}  \leq C < \infty   .
\end{align}
}  
\end{proof}

\subsection{Proof of Theorem \ref{distancehessianestimatethm}}
\begin{proof}
    From Lemma \ref{secboundlemma} for $r = (1-\gamma)\sqrt{\lambda_n}$ , $\gamma \in (0,1)$, where $\lambda_n = \norm{\Lambda^{-1}}^{-1}_2 $ and the ball given by
    \[
    \bar{\mathcal{B}}_r(\mathbf{0}) = \{ V\in T_{\Lambda}\texttt{SPD}_n | g_{\Lambda}(V,V) \leq r^2\}, \text{ and let } \bar{\mathcal{B}}_r(\Lambda) = \exp_{\Lambda}(\bar{\mathcal{B}}_r(\mathbf{0})),
    \] 
     the sectional curvature $\mathbf{K}_{BW}$ of the BW metric satisfies 
    \[
        0\leq \mathbf{K}_{BW}(\Lambda') \leq \frac{3}{4\gamma^4\lambda_n^2} = \frac{3 \norm{\Lambda^{-1}}^{2}_2}{4\gamma^4}
    \]
    for all $\Lambda' \in \bar{\mathcal{B}}_r(\Lambda)$. {Let $\gamma$ be such that $ (1-\gamma)\norm{\Lambda^{-1}}^{-1/2}_2 < \frac{  \gamma^2 \pi \norm{\Lambda^{-1}}^{-1}_2}{\sqrt{3} } $ . This inequality in $\gamma \in (0,1)$ is not vacuous and will be easily satisfied for any fixed positive definite $\Lambda$ whenever $\gamma$ is arbitrary close to $1$. As a consequence we get that $ \sqrt{ \frac{3 \norm{\Lambda^{-1}}^{2}_2}{4\gamma^4}} \times r < \pi/2 $.}
Then from Theorem \ref{jostdistthm}, for the distance squared function $k(x) :=\frac{1}{2}d^2(x,p) = \frac{1}{2} r^2(x)$ for $\mu = \frac{3 \norm{\Lambda^{-1}}^{2}_2}{4\gamma^4} $, $r(x) \le (1-\gamma)\sqrt{\lambda_n} = (1-\gamma) \norm{\Lambda^{-1}}^{-1/2}_2$, and for any $ x \in
\bar{\mathcal{B}}_{\rho}(p) := \{ q\in M | d(p,q) \leq \rho < \frac{\pi}{2 \sqrt{\mu}}\}$ we have that for any $ v \in T_{x}M$ :
\begin{align}
  \sqrt{\mu} r(x) \cot(\sqrt{\mu}r(x))  &\leq  \frac{\text{Hess }k(v,v)}{\norm{v}_x^2 } \leq \sqrt{-\lambda}r(x) \coth (\sqrt{-\lambda}r(x)) \\
  \implies  (1- \gamma)\norm{\Lambda^{-1}}^{1/2}_2\frac{\sqrt{3} }{2\gamma^2}  \cot \Bigg((1- \gamma)\norm{\Lambda^{-1}}^{1/2}_2\frac{\sqrt{3} }{2\gamma^2} \Bigg)   &\leq  \frac{\text{Hess }k(v,v)}{\norm{v}_x^2 } \leq \lim_{\lambda \to 0-} \sqrt{-\lambda}r(x) \coth (\sqrt{-\lambda}r(x))  \\
  \implies 0 < (1- \gamma)\norm{\Lambda^{-1}}^{1/2}_2\frac{\sqrt{3} }{2\gamma^2}  \cot \Bigg((1- \gamma)\norm{\Lambda^{-1}}^{1/2}_2\frac{\sqrt{3} }{2\gamma^2} \Bigg)   &\leq  \frac{\text{Hess }k(v,v)}{\norm{v}_x^2 } \leq  1
\end{align}
where the lower bound in the second last step follows from the strict monotonicity (strictly decreasing) of $g(t) := t \cot{t}$ on $(0, \pi/2) $.
{Then for $ d_{BW}(\Lambda,\Lambda^*) \le  (1-\gamma)\norm{(\Lambda^*)^{-1}}^{-1/2}_2 $ and $\gamma \in (0,1)$ such that $(1-\gamma)\norm{(\Lambda^*)^{-1}}^{-1/2}_2 < \frac{  \gamma^2 \pi \norm{(\Lambda^*)^{-1}}^{-1}_2}{\sqrt{3} }$ , the squared distance function in BW metric between $\Lambda, \Lambda^* $ given by $d_{BW}^2(\Lambda,\Lambda^*)$ satisfies :
\begin{align}
    0 < (1- \gamma)\norm{(\Lambda^*)^{-1}}^{1/2}_2\frac{\sqrt{3} }{\gamma^2}  \cot \Bigg((1- \gamma)\norm{(\Lambda^*)^{-1}}^{1/2}_2\frac{\sqrt{3} }{2\gamma^2} \Bigg)   &\leq  \frac{\text{Hess}_{(\muv,\Lambda)} d_{BW}^2(\Lambda,\Lambda^*)(V,V)}{\text{tr} (L_{\Lambda}(V) \Lambda L_{\Lambda}(V)) } \leq  2  \label{bwhessianestimate1e}
\end{align}
for any $V \in T_{\Lambda}(\texttt{SPD}_n )$.
}
\end{proof}

\subsection{Proof of Theorem \ref{lagrangehessianestimatethm}}
\begin{proof}
    For any fixed $\w$ the Lagrangian $\mathcal{L}(\muv,\Lambda;\w) $  admits a Hessian whose scaled quadratic form in the direction $ (\nuv,S\Lambda +\Lambda S)$ on the joint tangent space is given by :
\begin{align}
    \frac{ {\text{Hess}_{(\muv,\Lambda)}\mathcal{L}((\nuv,  S\Lambda +\Lambda S),(\nuv,S\Lambda +\Lambda S))}}{\langle (\nuv,  S\Lambda +\Lambda S) , (\nuv,  S\Lambda +\Lambda S)\rangle_{l_2 \oplus BW}} = \frac{ {\text{Hess}_{(\muv,\Lambda)}\mathcal{L}((\nuv,  S\Lambda +\Lambda S),(\nuv,S\Lambda +\Lambda S))}}{ \norm{\nuv}^2 + \tr(S \Lambda S)} .
\end{align}
Let $ f(\w; \Lambda , \muv) =  \mathbb{E}_{\z \sim \probP(\muv,\Lambda)} [\ell(\w;\z) ]$ and thus from \eqref{bwhessianestimate1d}, \eqref{bwhessianestimate1e} we get
\begin{align}
      \sup_{\nuv \in \mathbb{R}^n, \norm{S}_F < 1} \frac{ {\text{Hess}_{(\muv,\Lambda)}\mathcal{L}((\nuv,  S\Lambda +\Lambda S),(\nuv,S\Lambda +\Lambda S))}}{ \norm{\nuv}^2 + \tr(S \Lambda S)} & =   \sup_{\nuv \in \mathbb{R}^n, \norm{S}_F < 1} \frac{ {\text{Hess}_{(\muv,\Lambda)}f((\nuv,  S\Lambda +\Lambda S),(\nuv,S\Lambda +\Lambda S))}}{ \norm{\nuv}^2 + \tr(S \Lambda S)} \nonumber \\ 
      & \hspace{-5cm}- \frac{1}{\epsilon}\inf_{\nuv \in \mathbb{R}^n, \norm{S}_F < 1} \frac{ {\text{Hess}_{(\muv,\Lambda)}\norm{\muv - \muv^*}^2(  \nuv,\nuv)} + {\text{Hess}_{(\muv,\Lambda)}d_{BW}^2(\Lambda,\Lambda^*)(  S\Lambda +\Lambda S,S\Lambda +\Lambda S)}}{ \norm{\nuv}^2 + \tr(S \Lambda S)} \\
      & \hspace{-5cm} \le C_{\Lambda, q, n} (\E [ |\ell(\w;\z)|^p])^{1/p} \nonumber \\ & \hspace{-6cm} - \frac{1}{\epsilon}\inf_{\nuv \in \mathbb{R}^n, \norm{S}_F < 1} \frac{ 2 \norm{\nuv}^2 + (1- \gamma)\norm{(\Lambda^*)^{-1}}^{1/2}_2\frac{\sqrt{3} }{\gamma^2}  \cot \Bigg((1- \gamma)\norm{(\Lambda^*)^{-1}}^{1/2}_2\frac{\sqrt{3} }{2\gamma^2} \Bigg) \tr(S \Lambda S)}{ \norm{\nuv}^2 + \tr(S \Lambda S)} \\
         & \hspace{-6cm} \le C_{\Lambda, q, n} (\E [ |\ell(\w;\z)|^p])^{1/p}  - \frac{1}{\epsilon}(1- \gamma)\norm{(\Lambda^*)^{-1}}^{1/2}_2\frac{\sqrt{3} }{\gamma^2}  \cot \Bigg((1- \gamma)\norm{(\Lambda^*)^{-1}}^{1/2}_2\frac{\sqrt{3} }{2\gamma^2} \Bigg)  \label{bwhessianestimate1f} \\
          & \hspace{-6cm} \underbrace{\le}_{\eqref{Taylor13co} \, , 1/p +1/q =1 \, , \, p \ge 1}  C_{\Lambda, q, n} (C''_{p,\Lambda^*,n} \times \big( \mathbb{E}_{\z \sim \probP(\muv^*,\Lambda^*)}\left[\abs{\ell(\w;\z)}^{pq}\right] \big)^{1/q}  )^{1/p} \nonumber \\ & \hspace{-6cm} - \frac{1}{\epsilon}(1- \gamma)\norm{(\Lambda^*)^{-1}}^{1/2}_2\frac{\sqrt{3} }{\gamma^2}  \cot \Bigg((1- \gamma)\norm{(\Lambda^*)^{-1}}^{1/2}_2\frac{\sqrt{3} }{2\gamma^2} \Bigg) \label{bwhessianestimate1fx}
\end{align}
for $\gamma \in (0,1)$ and $(1-\gamma)\norm{(\Lambda^*)^{-1}}^{-1/2}_2 < \frac{  \gamma^2 \pi \norm{(\Lambda^*)^{-1}}^{-1}_2}{\sqrt{3} }$.
Then suppose if $$ 0 < \epsilon \le \frac{(1- \gamma)\norm{(\Lambda^*)^{-1}}^{1/2}_2\frac{\sqrt{3} }{\gamma^2}  \cot \Bigg((1- \gamma)\norm{(\Lambda^*)^{-1}}^{1/2}_2\frac{\sqrt{3} }{2\gamma^2} \Bigg)}{2C_{\Lambda, q, n}  (C''_{p,\Lambda^*,n} \times \big( \mathbb{E}_{\z \sim \probP(\muv^*,\Lambda^*)}\left[\abs{\ell(\w;\z)}^{pq}\right] \big)^{1/q})^{1/p}} $$ we have that the Lagrangian $\mathcal{L}$ is at least $C_{\Lambda, q, n}   (C''_{p,\Lambda^*,n} \times \big( \mathbb{E}_{\z \sim \probP(\muv^*,\Lambda^*)}\left[\abs{\ell(\w;\z)}^{pq}\right] \big)^{1/q})^{1/p}  $ geodesically strongly concave in $ (\muv,\Lambda)$ on the {product of closed balls of radius $r = (1- \gamma) \norm{(\Lambda^*)^{-1}}_2^{-1/2}$} centered at $(\muv^*,\Lambda^*)$ for any $ {\w \in \mathbb{R}^d}$ from assumption \textbf{A2}. Similarly using the upper bound from \eqref{bwhessianestimate1e} we get
\begin{align}
      \sup_{\nuv \in \mathbb{R}^n, \norm{S}_F < 1} \frac{ \abs{\text{Hess}_{(\muv,\Lambda)}\mathcal{L}((\nuv,  S\Lambda +\Lambda S),(\nuv,S\Lambda +\Lambda S))}}{ \norm{\nuv}^2 + \tr(S \Lambda S)} & \le   \sup_{\nuv \in \mathbb{R}^n, \norm{S}_F < 1} \frac{ \abs{\text{Hess}_{(\muv,\Lambda)}f((\nuv,  S\Lambda +\Lambda S),(\nuv,S\Lambda +\Lambda S))}}{ \norm{\nuv}^2 + \tr(S \Lambda S)} \nonumber \\ 
      & \hspace{-6cm} + \frac{1}{\epsilon}\sup_{\nuv \in \mathbb{R}^n, \norm{S}_F < 1} \frac{ \abs{\text{Hess}_{(\muv,\Lambda)}\norm{\muv - \muv^*}^2(  \nuv,\nuv) + \text{Hess}_{(\muv,\Lambda)}d_{BW}^2(\Lambda,\Lambda^*)(  S\Lambda +\Lambda S,S\Lambda +\Lambda S)}}{ \norm{\nuv}^2 + \tr(S \Lambda S)} 
      \end{align}
      \begin{align}
      &  \le C_{\Lambda, q, n} (\E [ |\ell(\w;\z)|^p])^{1/p}  + \frac{1}{\epsilon}\sup_{\nuv \in \mathbb{R}^n, \norm{S}_F < 1} \frac{ 2 \norm{\nuv}^2 + 2 \tr(S \Lambda S)}{ \norm{\nuv}^2 + tr(S \Lambda S)} \\
         &  \le C_{\Lambda, q, n} (\E [ |\ell(\w;\z)|^p])^{1/p}  +  \frac{2}{\epsilon} \label{bwhessianestimate1g} \\
         &  \underbrace{\le}_{\eqref{Taylor13co} \, , 1/p +1/q =1 \, , \, p \ge 1} C_{\Lambda, q, n} (C''_{p,\Lambda^*,n} \times \big( \mathbb{E}_{\z \sim \probP(\muv^*,\Lambda^*)}\left[\abs{\ell(\w;\z)}^{pq}\right] \big)^{1/q})^{1/p}  + \frac{2}{\epsilon}  \label{bwhessianestimate1gx}\\ 
         &   < \infty 
\end{align}
for any $\epsilon >0$ from assumption \textbf{A2}. Thus $\mathcal{L}$ is at most $C_{\Lambda, q, n}  (C''_{p,\Lambda^*,n} \times \big( \mathbb{E}_{\z \sim \probP(\muv^*,\Lambda^*)}\left[\abs{\ell(\w;\z)}^{pq}\right] \big)^{1/q})^{1/p} + \frac{2}{\epsilon}  $ gradient Lipschitz continuous in $ (\muv,\Lambda)$ on the {product of closed balls of radius $r = (1- \gamma) \norm{(\Lambda^*)^{-1}}_2^{-1/2}$} centered at $(\muv^*,\Lambda^*)$ for any $ {\w \in \mathbb{R}^d}$. Finally $C_{\Lambda, q, n} $ can be further bounded uniformly in $\Lambda$ by using the estimates \eqref{Taylor13ao}, \eqref{Taylor13bo} as follows:
$$ C_{\Lambda, q, n} \sim_{q} \bigg(\frac{\Gamma(\frac{n+4q}{2}) }{\Gamma(\frac{n}{2}) }\bigg)^\frac{1}{q} \norm{\Lambda}^2_F \norm{\Lambda^{-1/2} }^{6}_F \sim_{q} \bigg(\frac{\Gamma(\frac{n+4q}{2}) }{\Gamma(\frac{n}{2}) }\bigg)^\frac{1}{q}\norm{(\Lambda^*)^{1/2}}^4_F \bigg( \frac{\sqrt{n} \norm{(\Lambda^*)^{-1}}_F}{(1-r \norm{(\Lambda^*)^{-1/2}}_F)}\bigg)^{6}  := C_{\Lambda^*, q, n}  \,  $$
and hence the constant $C_{\Lambda, q, n}  $ can be replaced by the universal constant $C_{\Lambda^*, q, n} $ everywhere from \eqref{bwhessianestimate1fx} onward to \eqref{bwhessianestimate1gx}.
\end{proof}

\subsection{Proof of Theorem \ref{interiormaximaexistencethm}}
\begin{proof}
    We consider the Lagrangian $\hat{\mathcal{L}} := -\epsilon \mathcal{L} $ where $\mathcal{L}$ follows from the definition \ref{Lagrangiandefn} so that:
$$ \hat{\mathcal{L}}(\muv,\Lambda;\w) :=  \bigg(  \norm{\muv - \muv^*}^2 + d^2_{BW}(\Lambda ,\Lambda^*)\bigg) -\epsilon \mathbb{E}_{\z \sim \probP(\muv,\Lambda)} [\ell(\w;\z) ]   $$ 
and suppose in the context of Lemma \ref{localmaxexistencelem}, 
$$F(\muv,\Lambda) :=   \bigg(  \norm{\muv - \muv^*}^2 + d^2_{BW}(\Lambda ,\Lambda^*)\bigg) \, ,$$ 
$$H(\muv,\Lambda) :=  -\mathbb{E}_{\z \sim \probP(\muv,\Lambda)} [\ell(\w;\z) ]  \, \quad , \quad G := F + \epsilon H := \hat{\mathcal{L}} $$ 
Let $(\muv, \Lambda) $ be any arbitrary point inside the {product of closed balls of radius $r = (1- \gamma) \norm{(\Lambda^*)^{-1}}_2^{-1/2}$} centered at $(\muv^*, \Lambda^*)$, for $\gamma \in (0,1)$ and $(1-\gamma)\norm{(\Lambda^*)^{-1}}^{-1/2}_2 < \frac{  \gamma^2 \pi \norm{(\Lambda^*)^{-1}}^{-1}_2}{\sqrt{3} }$. Then by local $ L := C_{\Lambda^*, q, n} (C''_{p,\Lambda^*,n} \times \big( \mathbb{E}_{\z \sim \probP(\muv^*,\Lambda^*)}\left[\abs{\ell(\w;\z)}^{pq}\right] \big)^{1/q}  )^{1/p} $ -gradient Lipschitz continuity of $H$ (applying Theorem \ref{losshessianestimatethm} and then replacing $C_{\Lambda, q, n} $ by the universal constant $C_{\Lambda^*, q, n}$ as in the proof of Theorem \ref{lagrangehessianestimatethm}) , local $\mu $ -geodesic strong convexity (Theorem \ref{lagrangehessianestimatethm}) of $G, F$ where 
\begin{align*}
    \mu := & (1- \gamma)\norm{(\Lambda^*)^{-1}}^{1/2}_2\frac{\sqrt{3} }{\gamma^2}  \cot \Bigg((1- \gamma)\norm{(\Lambda^*)^{-1}}^{1/2}_2\frac{\sqrt{3} }{2\gamma^2} \Bigg)  - \epsilon C_{\Lambda^*, q, n} (C''_{p,\Lambda^*,n}  \big( \mathbb{E}_{\z \sim \probP(\muv^*,\Lambda^*)}\left[\abs{\ell(\w;\z)}^{pq}\right] \big)^{1/q}  )^{1/p} \nonumber \\  & > (1- \gamma)\norm{(\Lambda^*)^{-1}}^{1/2}_2\frac{\sqrt{3} }{2\gamma^2}  \cot \Bigg((1- \gamma)\norm{(\Lambda^*)^{-1}}^{1/2}_2\frac{\sqrt{3} }{2\gamma^2} \Bigg) > 0
\end{align*}
for
$$ 0 < \epsilon < \frac{(1- \gamma)\norm{(\Lambda^*)^{-1}}^{1/2}_2\frac{\sqrt{3} }{\gamma^2}  \cot \Bigg((1- \gamma)\norm{(\Lambda^*)^{-1}}^{1/2}_2\frac{\sqrt{3} }{2\gamma^2} \Bigg)}{2C_{\Lambda^*, q, n}  (C''_{p,\Lambda^*,n} \times \big( \mathbb{E}_{\z \sim \probP(\muv^*,\Lambda^*)}\left[\abs{\ell(\w;\z)}^{pq}\right] \big)^{1/q})^{1/p}}   $$
and $L >\mu >0$ are constants independent of $\w$ by \textbf{A2}, we conclude from Lemma \ref{localmaxexistencelem} that the unique minimizer of $ \hat{\mathcal{L}}$, or equivalently the unique maximizer $(\muv^{\#}, \Lambda^{\#})$  of $ \mathcal{L} $ exists inside the {product of closed balls of radius $r = (1- \gamma) \norm{(\Lambda^*)^{-1}}_2^{-1/2}$} centered at $(\muv^*, \Lambda^*)$, with
\[ d_{l_2 \oplus BW}\bigg((\muv^{\#}, \Lambda^{\#}), (\muv^*, \Lambda^*)\bigg) \leq 2\epsilon L / \mu < \frac{4C_{\Lambda^*, q, n}  (C''_{p,\Lambda^*,n} \times \big( \mathbb{E}_{\z \sim \probP(\muv^*,\Lambda^*)}\left[\abs{\ell(\w;\z)}^{pq}\right] \big)^{1/q})^{1/p}} {(1- \gamma)\norm{(\Lambda^*)^{-1}}^{1/2}_2\frac{\sqrt{3} }{\gamma^2}  \cot \Bigg((1- \gamma)\norm{(\Lambda^*)^{-1}}^{1/2}_2\frac{\sqrt{3} }{2\gamma^2} \Bigg)} \epsilon \, ,\]
provided 
\begin{align}
    \epsilon &< \min \bigg\{ \frac{\mu r}{2 L} , \frac{(1- \gamma)\norm{(\Lambda^*)^{-1}}^{1/2}_2\frac{\sqrt{3} }{\gamma^2}  \cot \Bigg((1- \gamma)\norm{(\Lambda^*)^{-1}}^{1/2}_2\frac{\sqrt{3} }{2\gamma^2} \Bigg)}{2C_{\Lambda^*, q, n}  (C''_{p,\Lambda^*,n} \times \big( \mathbb{E}_{\z \sim \probP(\muv^*,\Lambda^*)}\left[\abs{\ell(\w;\z)}^{pq}\right] \big)^{1/q})^{1/p}}   \bigg \}  \\
    \impliedby  \epsilon &< \frac{(1- \gamma)\norm{(\Lambda^*)^{-1}}^{1/2}_2\frac{\sqrt{3} }{\gamma^2}  \cot \Bigg((1- \gamma)\norm{(\Lambda^*)^{-1}}^{1/2}_2\frac{\sqrt{3} }{2\gamma^2} \Bigg)}{2C_{\Lambda^*, q, n}  (C''_{p,\Lambda^*,n} \times \big( \mathbb{E}_{\z \sim \probP(\muv^*,\Lambda^*)}\left[\abs{\ell(\w;\z)}^{pq}\right] \big)^{1/q})^{1/p}} \times \min \{\tfrac{r}{2} ,1\}
\end{align}
and the last step follows by substituting the $\mu $ lower bound and $L$.  
\end{proof}

\section{Gradient Calculations}\label{sectiongradientcalc}
\subsection{Gradient of the statistical loss}
{For gradient computation in BW metric, simplify $u'(0)$ (where $u(t)$ is given by \eqref{utdef}), get $\frac{d}{dt}\Big|_{t=0}f(\tau(t)) $ (where $\tau(t)$ is the curve \eqref{taudef}), use $$\langle \nabla_{(\muv,\Lambda)} f(\muv,\Lambda), (\nuv, S \Lambda + \Lambda S)\rangle_{l_2 \oplus BW} = \frac{d}{dt}\Big|_{t=0}f(\tau(t))  $$ to identify gradient $\nabla_{\Lambda} f(\muv,\Lambda)$ by setting $\nuv = 0$ in:
\begin{align*}
    \langle \nabla_{(\muv, \Lambda)} f(\muv,\Lambda), (\nuv, S \Lambda + \Lambda S)\rangle_{l_2 \oplus BW} &=  \langle \nabla_{\muv} f(\muv,\Lambda), \nuv\rangle + \langle \nabla_{\Lambda} f(\muv,\Lambda), S \Lambda + \Lambda S\rangle_{BW} \\ &=  \langle \nabla_{\muv} f(\muv,\Lambda), \nuv\rangle + \text{tr}(L_{\Lambda}(\nabla_{\Lambda} f(\muv,\Lambda)) \Lambda S) 
\end{align*}
}

Since $ \langle \nabla_{(\muv, \Lambda)} f(\muv,\Lambda), (\nuv, S\Lambda + \Lambda S)\rangle_{l_2 \oplus BW} = \frac{d}{dt}\Big|_{t=0}f(\tau(t)) $ or equivalently 
$$\langle \nabla_{(\muv, \Lambda)} f(\muv,\Lambda), (\nuv, S\Lambda + \Lambda S)\rangle_{l_2 \oplus BW} =  \langle \nabla_{\muv} f(\muv,\Lambda), \nuv\rangle + \underbrace{\langle \nabla_{\Lambda} f(\muv,\Lambda),  S\Lambda + \Lambda S\rangle_{BW}}_{ = \text{tr}(L_{\Lambda}(\nabla_{\Lambda} f(\muv,\Lambda)) \Lambda S)} = \frac{d}{dt}\Big|_{t=0}f(\tau(t)) . $$
Using the substitution $u(t) =\bigg(\frac{1}{2}\langle(\z -\muv_t),\Lambda_t^{-1}(\z - \muv_t)\rangle +\frac{1}{2}\log \det(2\pi\Lambda_t)\bigg)$ we first simplify $ \frac{d}{dt}\Big|_{t=0}f(\tau(t)) $ as follows:
\begin{align}
    \frac{d}{dt}\Big|_{t=0}f(\tau(t)) &=    \frac{d}{dt}\Big|_{t=0} \int_{\z \in \mathbb{R}^n} \ell(\w;\z)   \exp\bigg(-\frac{1}{2}\langle(\z -\muv_t),\Lambda_t^{-1}(\z - \muv_t)\rangle\bigg) \frac{d\z}{\sqrt{\det(2\pi \Lambda_t)}} \\
     & \hspace{-1.5cm}\underbrace{=}_{\substack{\textbf{Dominated convergence} \\ \textbf{theorem}}}     \int_{\z \in \mathbb{R}^n}    \frac{d}{dt}\Big|_{t=0}\exp\bigg(-\frac{1}{2}\langle(\z -\muv_t),\Lambda_t^{-1}(\z - \muv_t)\rangle -\frac{1}{2}\log \det(2\pi\Lambda_t)\bigg) { \ell(\w;\z) d\z} \\
     & =     \int_{\z \in \mathbb{R}^n}    -\exp (-u(0)) u'(0) { \ell(\w;\z) d\z} \\
     & =     \int_{\z \in \mathbb{R}^n}    -\exp \bigg(-\frac{1}{2}\langle(\z -\muv),\Lambda^{-1}(\z - \muv)\rangle\bigg)  u'(0) { \ell(\w;\z) } \frac{d\z}{\sqrt{\det(2\pi \Lambda)}}
\end{align}
For $\nuv = 0$ we have:
\begin{align}
    \frac{d}{dt}\Big|_{t=0}f(\tau(t)) &= \nonumber \\ & \hspace{-2.5cm}      \int_{\z \in \mathbb{R}^n}    -\exp \bigg(-\frac{1}{2}\langle(\z -\muv),\Lambda^{-1}(\z - \muv)\rangle\bigg)  \bigg(- \frac{1}{2}\langle (\z-\muv),(\Lambda^{-1} S + S\Lambda^{-1})(\z-\muv)\rangle + \text{tr}(S)\bigg) { \ell(\w;\z) } \frac{d\z}{\sqrt{\det(2\pi \Lambda)}}\\
    & \hspace{-2.5cm}   =   \int_{\z \in \mathbb{R}^n}    -\exp \bigg(-\frac{1}{2}\langle(\z -\muv),\Lambda^{-1}(\z - \muv)\rangle\bigg)  \bigg(- \frac{1}{2} \text{tr}\bigg(\langle (\z-\muv),\Lambda^{-1} S (\z-\muv)\rangle\bigg) - \frac{1}{2} \text{tr}\bigg(\langle (\z-\muv),S\Lambda^{-1} (\z-\muv)\rangle\bigg)  \nonumber \\ & + \text{tr}(S)\bigg) { \ell(\w;\z) } \frac{d\z}{\sqrt{\det(2\pi \Lambda)}}\\
    & \hspace{-2.5cm}   =   \int_{\z \in \mathbb{R}^n}    -\exp \bigg(-\frac{1}{2}\langle(\z -\muv),\Lambda^{-1}(\z - \muv)\rangle\bigg)  \bigg(- \frac{1}{2}\text{tr}\bigg( (\z-\muv)(\z-\muv)^T\Lambda^{-1} S \bigg) - \frac{1}{2}\text{tr}\bigg( (\z-\muv)(\z-\muv)^T S\Lambda^{-1} \bigg) \nonumber \\ & + \text{tr}(S)\bigg) { \ell(\w;\z) } \frac{d\z}{\sqrt{\det(2\pi \Lambda)}}\\
     & \hspace{-2.5cm}   =   \int_{\z \in \mathbb{R}^n}    -\exp \bigg(-\frac{1}{2}\langle(\z -\muv),\Lambda^{-1}(\z - \muv)\rangle\bigg)  \text{tr}\bigg( (I - (\z-\muv)(\z-\muv)^T\Lambda^{-1} ) S\bigg)  { \ell(\w;\z) } \frac{d\z}{\sqrt{\det(2\pi \Lambda)}}\\
       & \hspace{-2.5cm}   =   \int_{\z \in \mathbb{R}^n}    -\exp \bigg(-\frac{1}{2}\langle(\z -\muv),\Lambda^{-1}(\z - \muv)\rangle\bigg)  \text{tr}\bigg( (\Lambda^{-1} - \Lambda^{-1}(\z-\muv)(\z-\muv)^T\Lambda^{-1} ) \Lambda S\bigg)  { \ell(\w;\z) } \frac{d\z}{\sqrt{\det(2\pi \Lambda)}}\\
        & \hspace{-2.5cm}   =     \text{tr}\bigg( \bigg(\int_{\z \in \mathbb{R}^n}    -\exp \bigg(-\frac{1}{2}\langle(\z -\muv),\Lambda^{-1}(\z - \muv)\rangle\bigg) (\Lambda^{-1} - \Lambda^{-1}(\z-\muv)(\z-\muv)^T\Lambda^{-1} )  { \ell(\w;\z) } \frac{d\z}{\sqrt{\det(2\pi \Lambda)}} \bigg)\Lambda S\bigg) \\
         & \hspace{-2.5cm}   =     \text{tr}\bigg( X \Lambda S\bigg)
\end{align}
where $ \nabla_{\Lambda} f(\muv; \Lambda) =  X \Lambda + \Lambda X$
and 
{
\begin{align}
    X = \int_{\z \in \mathbb{R}^n}    \exp \bigg(-\frac{1}{2}\langle(\z -\muv),\Lambda^{-1}(\z - \muv)\rangle\bigg) (  \Lambda^{-1}(\z-\muv)(\z-\muv)^T\Lambda^{-1} -\Lambda^{-1})  { \ell(\w;\z) } \frac{d\z}{\sqrt{\det(2\pi \Lambda)}} \, .\label{gradcalc}
\end{align}
We note that $\nabla_{\muv} f(\muv; \Lambda) $ can be calculated similarly, the calculation is simpler so we omit it. 
}

\subsection{Gradient of quadratic regularizer}
{Consider the squared distance in BW metric between $\Lambda, \Lambda^*  $ as follows:
$$ d_{BW}^2(\Lambda,\Lambda^*) := {\text{tr}\bigg(\Lambda + \Lambda^* - 2( \Lambda^{1/2} \Lambda^*  \Lambda^{1/2})^{1/2}  \bigg)} .$$ So to find gradient of $d_{BW}^2(\Lambda,\Lambda^*)$ at $\Lambda$, we set $ u(t) := d_{BW}^2(\Lambda_t,\Lambda^*)$ and find 
\[
\frac{d}{dt}\Big|_{t=0} u(t), \quad \text{ where }\Lambda_t = (I+tS)\Lambda(I+tS) \quad \text{and}
\]
\[
\frac{d}{dt}\Big|_{t=0} u(t) = g_{BW}(\nabla_{\Lambda}d_{BW}^2(\Lambda,\Lambda^*), S\Lambda + \Lambda S) = \text{tr}(L_{\Lambda}(\nabla_{\Lambda}d_{BW}^2(\Lambda,\Lambda^*)) \Lambda S) \quad \forall \quad S \in \mathbb{S}^n
\]
}
 Simplifying $ \frac{d}{dt}\Big|_{t=0} u(t)$ and using $d_{BW}(\Lambda_t,\Lambda^*) = d_{BW}(\Lambda^*, \Lambda_t)$ yields:
\begin{align}
    u'(0) & = \frac{d}{dt}\Big|_{t=0}  \text{tr}\bigg(\Lambda^* + \Lambda_t - 2( \sqrt{\Lambda^*} (I +tS)\sqrt{\Lambda}\sqrt{\Lambda} (I +tS)\sqrt{\Lambda^*})^{1/2}  \bigg) \\
     & = \frac{d}{dt}\Big|_{t=0}  \text{tr}\bigg(\Lambda^* + \Lambda_t - 2\bigg( (\sqrt{\Lambda^*} (I +tS)\sqrt{\Lambda})(\sqrt{\Lambda^*} (I +tS)\sqrt{\Lambda})^T\bigg)^{1/2}  \bigg) \\
          & \underbrace{=}_{\textbf{Lemma } \ref{lyapunovlem1}} \text{tr}(S \Lambda + \Lambda S) - \text{tr}\bigg( \bigg(\sqrt{\Lambda^*} {\Lambda} \sqrt{\Lambda^*} \bigg)^{-1/2} \bigg( \sqrt{\Lambda^*} S {\Lambda} \sqrt{\Lambda^*} + \sqrt{\Lambda^*} {\Lambda} S\sqrt{\Lambda^*}\bigg)\bigg)\\
     & =  \text{tr}\bigg(2\Lambda S - {\Lambda} \sqrt{\Lambda^*}\bigg(\sqrt{\Lambda^*} {\Lambda} \sqrt{\Lambda^*} \bigg)^{-1/2} \sqrt{\Lambda^*} S   -  \sqrt{\Lambda^*}  \bigg(\sqrt{\Lambda^*} {\Lambda} \sqrt{\Lambda^*} \bigg)^{-1/2}\sqrt{\Lambda^*} {\Lambda} S \bigg) \\
      & =  \text{tr}\bigg(2\Lambda S -\sqrt{\Lambda^*}\bigg(\sqrt{\Lambda^*} {\Lambda} \sqrt{\Lambda^*} \bigg)^{-1/2} \sqrt{\Lambda^*}  {\Lambda}  S   -  \sqrt{\Lambda^*}  \bigg(\sqrt{\Lambda^*} {\Lambda} \sqrt{\Lambda^*} \bigg)^{-1/2}\sqrt{\Lambda^*} {\Lambda} S \bigg) \\
     & =  \text{tr}\bigg(\bigg(2I  -  \sqrt{\Lambda^*}\bigg(\sqrt{\Lambda^*} {\Lambda} \sqrt{\Lambda^*} \bigg)^{-1/2} \sqrt{\Lambda^*}     -  \sqrt{\Lambda^*}  \bigg(\sqrt{\Lambda^*} {\Lambda} \sqrt{\Lambda^*} \bigg)^{-1/2}\sqrt{\Lambda^*}  \bigg) {\Lambda}  S \bigg) \\
     & =  \text{tr}\bigg( L_{ \Lambda}(\nabla_{\Lambda}   d_{BW}^2(\Lambda,\Lambda^*)  ) \Lambda S \bigg) 
\end{align}
and hence
\begin{align}
    L_{ \Lambda}(\nabla_{\Lambda}   d_{BW}^2(\Lambda,\Lambda^*)  ) & = 2 \bigg(I  -  \sqrt{\Lambda^*}\bigg(\sqrt{\Lambda^*} {\Lambda} \sqrt{\Lambda^*} \bigg)^{-1/2} \sqrt{\Lambda^*} \bigg) \\
    & = 2(I-J_1)  \quad , \quad J_1 = \sqrt{\Lambda^*}\bigg(\sqrt{\Lambda^*} {\Lambda} \sqrt{\Lambda^*} \bigg)^{-1/2} \sqrt{\Lambda^*} = J^{-1}_{\Lambda^*, \Lambda} \\
    \implies \nabla_{\Lambda}   d_{BW}^2(\Lambda,\Lambda^*) & = \Lambda  L_{ \Lambda}(\nabla_{\Lambda}   d_{BW}^2(\Lambda,\Lambda^*)  ) +  L_{ \Lambda}(\nabla_{\Lambda}   d_{BW}^2(\Lambda,\Lambda^*)  ) \Lambda \\
    & = 2\Lambda  ( I - J_1) + 2( I - J_1)  \Lambda
\end{align}
Next, observe that 
\[
\Lambda J_1 =  (\sqrt{\Lambda^*})^{-1}  \sqrt{\Lambda^*}  \Lambda \sqrt{\Lambda^*}\bigg(\sqrt{\Lambda^*} {\Lambda} \sqrt{\Lambda^*} \bigg)^{-1/2} \sqrt{\Lambda^*} (\Lambda^*)^{-1} \Lambda^* = (\sqrt{\Lambda^*})^{-1}\bigg(\sqrt{\Lambda^*} {\Lambda} \sqrt{\Lambda^*} \bigg)^{1/2} (\sqrt{\Lambda^*})^{-1} \Lambda^* = J_{\Lambda^*, \Lambda} \Lambda^* ,
\]
therefore
\[
\Lambda J_{\Lambda^*, \Lambda}^{-1} =  J_{\Lambda^*, \Lambda} \Lambda^* \implies  J_{\Lambda^*, \Lambda}^{-1} \Lambda J_{\Lambda^*, \Lambda}^{-1} =   \Lambda^* 
\]
and since
\[
    J^{-1}_{\Lambda, \Lambda^*} = \sqrt{\Lambda}\bigg(\sqrt{\Lambda} {\Lambda^*} \sqrt{\Lambda} \bigg)^{-1/2} \sqrt{\Lambda} 
\]
we get that
\[
J^{-1}_{\Lambda, \Lambda^*} = \sqrt{\Lambda}\bigg(\sqrt{\Lambda} { J_{\Lambda^*, \Lambda}^{-1} \Lambda J_{\Lambda^*, \Lambda}^{-1}} \sqrt{\Lambda} \bigg)^{-1/2} \sqrt{\Lambda} = \sqrt{\Lambda}\bigg(\sqrt{\Lambda} { J_{\Lambda^*, \Lambda}^{-1} } \sqrt{\Lambda} \bigg)^{-1} \sqrt{\Lambda} =  J_{\Lambda^*, \Lambda} = J_1^{-1}
\]
or equivalently
\[
J_1 = J_{\Lambda, \Lambda^*} = J^{-1}_{\Lambda^*, \Lambda} = (\sqrt{\Lambda})^{-1}\bigg(\sqrt{\Lambda} {\Lambda^*} \sqrt{\Lambda} \bigg)^{1/2} (\sqrt{\Lambda})^{-1}.
\]
Then
\begin{align}
    \nabla_{\Lambda}   d_{BW}^2(\Lambda,\Lambda^*)  &=  2\Lambda  ( I - J_1) + 2( I - J_1)  \Lambda   \\
      & =  2\bigg(  \Lambda ( I - J_{\Lambda, \Lambda^*}) + ( I - J_{\Lambda, \Lambda^*}) \Lambda  \bigg) 
\end{align}
where
\[
    J_{\Lambda, \Lambda^*} = (\sqrt{\Lambda})^{-1}\bigg(\sqrt{\Lambda} {\Lambda^*} \sqrt{\Lambda} \bigg)^{1/2} (\sqrt{\Lambda})^{-1}
\]
is the standard OT map from $\Lambda $ to $\Lambda^*$.

{
\begin{lem}\label{lyapunovlem1}
    Let $X(t)$ be a matrix function such that $X(t) \in \texttt{SPD}_n$ in some open interval around $t=0$. Then the following hold for any $t$ in that open interval:
    \begin{itemize}
        \item $ \frac{d}{dt} \text{tr} (\sqrt{ X(t)}) = \frac{1}{2} \text{tr}\bigg((X(t))^{-1/2} \frac{d}{dt} X(t) \bigg) $
        \item $ \frac{d^2}{dt^2} \text{tr} (\sqrt{ X(t)}) = \frac{1}{2}  \text{tr}\bigg(-(X(t))^{-1/2}  L_{\sqrt{X}} \bigg(\frac{d}{dt}X \bigg)  (X(t))^{-1/2}  \frac{d}{dt} X(t)  + (X(t))^{-1/2}\frac{d^2}{dt^2}X(t)  \bigg)   $
    \end{itemize}
\end{lem}
\begin{proof}
The first part follows from the calculation below: 
    \begin{align}
        \frac{d}{dt}(\sqrt{X}\sqrt{X}) =  \sqrt{X}\frac{d}{dt}\sqrt{X} +  \frac{d}{dt}\sqrt{X} \sqrt{X} &= \frac{d}{dt}X    \\
        \implies  \frac{d}{dt}\sqrt{X}  &=  L_{\sqrt{X}} \bigg(\frac{d}{dt}X \bigg) \\
        \implies  \frac{d}{dt}\sqrt{X} + \sqrt{X}^{-1} \frac{d}{dt}\sqrt{X} \sqrt{X} &=  \sqrt{X}^{-1} \frac{d}{dt}X \\
        \implies \text{tr}\bigg(\frac{d}{dt}\sqrt{X} + \sqrt{X}^{-1} \frac{d}{dt}\sqrt{X} \sqrt{X}\bigg) = 2\frac{d}{dt} \text{tr} (\sqrt{ X}) &= \text{tr} \bigg(\sqrt{X}^{-1} \frac{d}{dt}X\bigg)
    \end{align}
    which proves the first part.
    The second part follows similarly:
    \begin{align}
        \frac{d^2}{dt^2} \text{tr} (\sqrt{ X(t)})  & = \frac{1}{2}  \frac{d}{dt} \text{tr}\bigg((X(t))^{-1/2} \frac{d}{dt} X(t) \bigg) \\
        & = \frac{1}{2}  \text{tr}\bigg(-(X(t))^{-1/2}  \frac{d}{dt} (X(t))^{1/2} (X(t))^{-1/2}  \frac{d}{dt} X(t)  + (X(t))^{-1/2}\frac{d^2}{dt^2}X(t)  \bigg) \\
        & = \frac{1}{2}  \text{tr}\bigg(-(X(t))^{-1/2}  L_{\sqrt{X}} \bigg(\frac{d}{dt}X \bigg)  (X(t))^{-1/2}  \frac{d}{dt} X(t)  + (X(t))^{-1/2}\frac{d^2}{dt^2}X(t)  \bigg) 
    \end{align}
\end{proof}
}

\section{Proofs for section \ref{RGAconvergencesection} }\label{sectionRGAconvergenceappendix}

\subsection{Proof of Lemma \ref{suplemPLbound0}}

\begin{proof}
    Let $ f(\w; \probP) := \mathbb{E}_{\z \sim \probP} [\ell(\w;\z) ] $ and suppose $\w \in \mathcal{W}$. Then 
\[
\mathcal{S}^*(\probP) \cap \overline{\mathcal{W}} : = \arg \inf_{\w \in \mathcal{W} } f(\w, \probP)
\]
The inclusion $\mathcal{S}^*(\probP) \cap \mathcal{W}   \subseteq \arg \inf_{\w \in \mathcal{W} } f(\w, \probP) $ is straightforward. To show that $\mathcal{S}^*(\probP) \cap  \overline{\mathcal{W}}   \supseteq \arg \inf_{\w \in \mathcal{W} } f(\w, \probP) $, let $ \w^* \in \mathcal{S}^*(\probP) \cap \partial \mathcal{W} $ be arbitrary. Take a converging sequence $\{\w_j\}$ in $ \mathcal{S}^*(\probP) \cap  {\mathcal{W}}$ with limit $\w^*$, then the inequality \eqref{QGtypegrowth1a} holds along that converging sequence $\{\w_j\}$ and by passing the limit and using continuity of $f(\cdot, \probP)$, $\nabla_{\w} f(\cdot, \probP)$ we will get that $f(\w^*; \probP) = \inf_{\w \in \mathcal{W} } f(\w; \probP)  $.
In the other direction first notice that 
$$\arg \inf_{\w \in \mathcal{W} } f(\w, \probP) =  \arg \inf_{\w \in \mathcal{W} } f(\w, \probP) \cap \mathrm{int}(\mathcal{W})  \bigcup \arg \inf_{\w \in \mathcal{W} } f(\w, \probP) \cap \partial \mathcal{W} $$ and thus $ \arg \inf_{\w \in \mathcal{W} } f(\w, \probP) \cap \mathrm{int}(\mathcal{W})  \subseteq \mathcal{S}^*(\probP) \cap \mathcal{W} $ due to the fact that $f$ is $\mathcal{C}^1$ and $\arg \inf_{\w \in \mathcal{W} } f(\w, \probP) \cap \mathrm{int}(\mathcal{W}) \subset \mathcal{W} $. If $\w^* \in \arg \inf_{\w \in \mathcal{W} } f(\w, \probP) \cap \partial \mathcal{W}$ then $\w^*$ is a limit of some convergent sequence in $ \arg \inf_{\w \in \mathcal{W} } f(\w, \probP) \cap \mathrm{int}(\mathcal{W})  $, thereby giving $ \nabla_{\w} f(\w^*, \probP)  = \mathbf{0} $ by continuity of $\nabla_{\w} f(\cdot, \probP)$. Hence we get $  \arg \inf_{\w \in \mathcal{W} } f(\w, \probP) \cap \partial \mathcal{W} \subseteq \mathcal{S}^*(\probP) \cap \overline{\mathcal{W}} $.

Next, we write the inequality in \textbf{A2'} as:
\[
\norm{\nabla_{\w} f(\w;\probP) }^{\beta} \geq C_\beta (f(\w; \probP) - \inf_{\w \in \mathcal{W}} f(\w, \probP))
\]
 Then under assumptions \textbf{A1-A2}, \textbf{A2'} for any $\w \in \mathcal{W} $ we derive the appropriate growth rate of $f(\w; \probP) $ in $\w$ for any $ (\w, \probP) \in \mathcal{W} $.  Consider the function:
\[
\wt g(\w) := (f(\w; \probP) - f_*(\probP))^{(\beta -1)/\beta}
\]
where $\probP$ has been suppressed in $\wt g$ for convenience and $ f_*(\probP) = \inf_{\w \in \mathcal{W}} f(\w, \probP)$. Observe that $\wt g \geq 0$ on $\mathcal{W}$ and $\wt g(\w) = 0$ on $\mathcal{W}$ iff $\w \in \mathcal{S}^*(\probP) $. By chain rule of derivative and using the fact that $(\beta -1)/\beta < 1 $, we get for any $\w \in \mathcal{W}$ from \textbf{A2'}, that 
{
\begin{align}
   \norm{\nabla_{\w} \wt g(\w)}^{\beta}  =  \bigg(\frac{\beta-1}{\beta}\bigg)^{\beta}\frac{\norm{  \nabla_{\w} f(\w; \probP)  }^{\beta}}{(f(\w; \probP) - f_*(\probP))} \geq C_{\beta}\left(\frac{\beta-1}{\beta} \right)^{\beta}. \label{PLtypebound1a}
\end{align}
}
{
Consider the first order initial value problem
\[
  \dot{\w} = -\left(\frac{\beta}{\beta-1} \right)^2C_{\beta}^{-2/\beta} \nabla_{\w} \wt g(\w(t)) \quad; \quad \w(0) := \w_0 \in \mathcal{V} \backslash \mathcal{S}^*(\probP) \quad \mathcal{V}  \Subset \mathcal{W} 
\]
where $\mathcal{V}$ is one of the connected components of some sub-level set of $g$ given by $\{ \w : \wt g(\w) \leq R\}$ and $\mathcal{W}$ separates $\mathcal{V} $ from other connected components of this sub-level set. 
}

Recall that $\wt g$ is coercive in $\w$ for all $\probP$ since $f$ is corcive by \textbf{A1.} 
 The above ODE corresponds to a gradient flow on $\wt g$ and it has a stationary solution in $ \mathcal{S}^*(\probP)$. A simple calculation along the solution curve of this ODE yields:
 \[
 \frac{d }{dt} \wt g(\w(t)) = \langle \nabla_{\w} \wt g(\w(t)) , \dot \w(t) \rangle = - {\left(\frac{\beta}{\beta-1} \right)^2C_{\beta}^{-2/\beta}} \norm{\nabla_{\w} \wt g(\w(t))}^2 < 0 
 \]
 and hence $\wt g$ decreases monotonically along the solution curve of the ODE. Since $\wt g$ is lower bounded on $\mathcal{W}$, i.e., $\wt g \geq 0$, and the flow of $\w(t)$ always stays in the sub-level set component $\mathcal{V} $ by monotonicity of $\wt g(\w(t))$, it must be that $\wt g(\w(t))$ converges by Monotone convergence theorem. In fact, from \eqref{PLtypebound1a} and Gronwall's inequality we get that:
 \begin{align}
     \frac{d }{dt} \wt g(\w(t)) & = - {\left(\frac{\beta}{\beta-1} \right)^2C_{\beta}^{-2/\beta}}\norm{\nabla_{\w} \wt g(\w(t))}^2  < -1 \label{continuity0} \\
     \implies 0 \leq \wt g(\w(t)) & \leq e^{-t} \quad \forall  \quad t \geq 0   \\
     \implies \lim_{t \to \infty} \wt g(\w(t)) &= 0  \label{continuity1} \\
     \implies \lim_{t \to \infty}  f(\w(t); \probP) & =  f_*(\probP) = \inf_{\w \in \mathcal{W}} f(\w, \probP)
 \end{align}
 Let $\mathcal{T}$ be the set of omega limit points (occurring infinitely often) of the flow curve $ \{\w(t) : t \in [0, \infty)\}$ defined as follows:
 \[
 \mathcal{T} = \{ \v : \quad \forall \quad T \quad \text{ and } \quad \forall \quad  \epsilon > 0 \quad \exists \quad t> T \quad \text{such that} \quad \norm{\w(t) - \v} < \epsilon \}
 \]
 Then $ \wt g(\mathcal{T}) = 0$ from \eqref{continuity1} and continuity of $\wt g$, implying $\mathcal{T} = \mathcal{S}^*(\probP) $ and thus $ \mathrm{dist}(\mathcal{T}, \mathcal{S}^*(\probP)) = 0$. 
 
Next, suppose the flow of $\w(t)$ hits $ \mathcal{S}^*(\probP)$ at some time $T \in (0, \infty]$.
By fundamental theorem of calculus on $\wt g$ and using the bound ${\left(\frac{\beta}{\beta-1} \right)^2C_{\beta}^{-2/\beta}}\norm{\nabla_{\w} \wt g(\w(t))}^2 \geq 1$ for any $\w \in \mathcal{W}$  we can write,
\begin{align}
    \wt g(\w(0)) - \wt g(\w(T)) &= -\int_{0}^{T} \langle \nabla_{\w}\wt g(\w) , \dot\w\rangle dt = \int_{0}^{T} {\left(\frac{\beta}{\beta-1} \right)^2C_{\beta}^{-2/\beta}} \norm{\nabla_{\w} \wt g(\w(t))}^2 dt\geq T \\
    \implies \wt g(\w(0)) &= \wt g(\w(0)) - \inf_{\w \in \mathcal{W}} \wt g(\w) = \wt g(\w(0)) - \wt g(\w(T))  \geq T
\end{align}
Hence, the time for the ODE curve to hit $ \mathcal{S}^*(\probP)$ is upper bounded. The length of this flow curve from $\w(0)$ to $\w(T)$ can be lower bounded as:
\begin{align}
 {\left(\frac{\beta}{\beta-1} \right)^2C_{\beta}^{-2/\beta}}\int_{0}^T \norm{\nabla_{\w} \wt g(\w(t))} dt   = \int_{0}^T \norm{\dot \w} dt  \geq \mathrm{dist}(\w(0), \mathcal{S}^*(\probP))
\end{align}
Then integrating \eqref{continuity0}, using the above bound in \eqref{continuity0} along with $ \left(\frac{\beta}{\beta-1} \right)C_{\beta}^{-1/\beta} \norm{\nabla_{\w} \wt g(\w)} \geq 1 $ from \eqref{PLtypebound1a} we get :
{
\begin{align}
 \wt g(\w(0)) =\wt g(\w(0)) - \wt g(\w(T)) & = \int_{0}^{T} \left(\frac{\beta}{\beta-1} \right)^2C_{\beta}^{-2/\beta}\norm{\nabla_{\w} \wt g(\w(t))}^2 dt \nonumber \\
  &\geq \left(\frac{\beta}{\beta-1} \right)C_{\beta}^{-1/\beta} \int_{0}^{T} \norm{\nabla_{\w} \wt g(\w(t))} dt \nonumber \\
  &\geq \left(\frac{\beta}{\beta-1} \right)^{-1}C_{\beta}^{1/\beta} \mathrm{dist}(\w(0), \mathcal{S}^*(\probP)) 
\end{align}
}
and since $\w(0) \in \mathcal{V} \backslash \mathcal{S}^*(\probP)  $ was arbitrary, we get that
\[
\wt g(\w) := (f(\w; \probP) - f_*(\probP))^{(\beta -1)/\beta} \geq {\left(\frac{\beta}{\beta-1} \right)^{-1}C_{\beta}^{1/\beta}} \mathrm{dist}(\w, \mathcal{S}^*(\probP))
\]
or equivalently
\begin{align}
 f(\w; \probP) - f_*(\probP) \geq {\left(\frac{\beta}{\beta-1} \right)^{\frac{-\beta}{\beta-1}}C_{\beta}^{\frac{1}{\beta-1}}}\bigg( \mathrm{dist}(\w, \mathcal{S}^*(\probP))\bigg)^{\frac{\beta}{\beta -1}}   \label{QGtypegrowth0}
\end{align}
Combining this with the inequality from \textbf{A2'} we immediately have the required gradient bound:
{
\begin{align}
  \norm{\nabla_{\w} f(\w;\probP) }^{\beta}  \geq C_{\beta}\bigg(f(\w; \probP) - f_*(\probP)\bigg) \geq \left(\frac{\beta}{\beta-1} \right)^{\frac{-\beta}{\beta-1}}C_{\beta}^{\frac{\beta}{\beta-1}}\bigg( \mathrm{dist}(\w, \mathcal{S}^*(\probP))\bigg)^{\frac{\beta}{\beta -1}}  \nonumber \\
   \implies \norm{\nabla_{\w} f(\w;\probP) } \geq C_{\beta}^{\frac{1}{\beta-1}}\bigg(\frac{\beta}{\beta-1}\bigg)^{\frac{-1}{\beta-1}}\bigg( \mathrm{dist}(\w, \mathcal{S}^*(\probP))\bigg)^{\frac{1}{\beta -1}} \label{QGtypegrowth1a*}
\end{align}
which completes the proof.
}
\end{proof}

\section{Second order Taylor expansions on \texorpdfstring{$\mathbb{R}^n \times BW$}{} }\label{taylormetricappendix}
Recall that the exponential map for the BW metric is given by
\begin{equation}\label{expBW}
\exp:\Rn^n\times U(\Lambda) \to \Rn^n\times \texttt{SPD}_n, \quad \exp_{(\muv,\Lambda)}(\nuv, S) = (\muv + \nuv, (I+L_{\Lambda}(S))\Lambda(I+L_{\Lambda}(S)))
\end{equation}
where $U(\Lambda) = \{ S\in \mathbb{S}^n| I+L_{\Lambda}(S) \in \texttt{SPD}_n\}$.

Now let 
    \[
    g(\muv, \Lambda) = \exp\left(- \frac{1}{2}\langle(\z -\muv),\Lambda^{-1}(\z - \muv)\rangle -\frac{1}{2}\log \det(2\pi\Lambda)\right). 
    \]
    Then it follows from \eqref{gradcalc} that
    \begin{align}
        \nabla_{(\muv,\Lambda)} g &= \left( \nabla_{\muv}g, \nabla_{\Lambda}g\right)  \nonumber \\
        &:=g(\muv,\Lambda) \left(\Lambda^{-1}(\z-\muv),   (\Lambda^{-1}(\z-\muv)(\z-\muv)^{T} +(\z-\muv)(\z-\muv)^{T}\Lambda^{-1} )-2I\right) \label{tayf}.
    \end{align}

For the sake of brevity, we write 
\begin{align}
\z-\muv = \zeta, &\quad \wt{S} = L_{\Lambda}(S) \text{ and }    \nonumber \\
P_{\Lambda}(S) &= S\Lambda + \Lambda S \text{ for all } S\in \mathbb{S}^n, \Lambda \in \texttt{SPD}_n \label{Pdef}.
\end{align}
To calculate the Taylor series of the gradient of statistical loss, we write equation \eqref{tayf} in exponential coordinates. Writing $\z-\muv = \zeta$ and $\wt{S} = L_{\Lambda}(S)$ for the sake of brevity we have 
\begin{align*}
    \frac{(\nabla_{\muv} g)\circ \exp_{(\muv,\Lambda)}(\nuv, S)}{g(\exp_{(\muv,\Lambda)}(\nuv, S))} &=  (I+\wt{S})^{-1}\Lambda^{-1}(I+\wt{S})^{-1}(\zeta-\nuv)    \\
     \frac{(\nabla_{\Lambda} g)\circ \exp_{(\muv,\Lambda)}(\nuv, S)}{g(\exp_{(\muv,\Lambda)}(\nuv, S))}&=  (I+\wt{S})^{-1}\Lambda^{-1}(I+\wt{S})^{-1}(\zeta-\nuv)(\zeta-\nuv)^{T} \\
     &\quad \quad \qquad +(\zeta-\nuv)(\zeta-\nuv)^{T}(I+\wt{S})^{-1}\Lambda^{-1}(I+\wt{S})^{-1}- 2I.
\end{align*}
We have for $ \norm{\wt{S}}_F < 1$ that
\begin{align*}
    (I+\wt{S})^{-1}\Lambda^{-1}(I+\wt{S})^{-1} &= \sum_{j,k=0}^{\infty} (-1)^{j+k}\wt{S}^j\Lambda^{-1}\wt{S}^k \\
    &= \Lambda^{-1} - \wt{S}\Lambda^{-1} - \Lambda^{-1} \wt{S} + (\wt{S}^2\Lambda^{-1} + \Lambda^{-1}\wt{S}^2 + \wt{S}\Lambda^{-1}\wt{S}) + o(2) \\
    &=\Lambda^{-1} - \Lambda^{-1}S \Lambda^{-1} + \underset{Q_e}{\underbrace{(\wt{S}^2\Lambda^{-1} + \Lambda^{-1}\wt{S}^2 + \wt{S}\Lambda^{-1}\wt{S})}} + \underbrace{o(2)}_{= o( \norm{\Lambda^{-1}}_F \norm{\wt{S}}_F^2)}.
\end{align*}
Consequently\footnote{We note that the $ o(2)$ term from here onward implies the $ o\big( \norm{\wt{S}}_F^2\big)$ term and henceforth we will always refer to it as $o(2)$ for notational brevity.} ,
\begin{align*}
    &\quad\langle (\zeta-\nuv), (I+\wt{S})^{-1}\Lambda^{-1}(I+\wt{S})^{-1}(\zeta -\nuv) \rangle \\ 
    &=\langle (\zeta-\nuv),\left(\Lambda^{-1} - \Lambda^{-1}S \Lambda^{-1} + Q_e\right)(\zeta -\nuv) \rangle + o(2)\\
    &= \langle \zeta, \left(\Lambda^{-1} - \Lambda^{-1}S \Lambda^{-1} + Q_e\right)\zeta\rangle - 2\langle \zeta, \left(\Lambda^{-1} - \Lambda^{-1}S \Lambda^{-1} + Q_e\right)\nuv\rangle + \langle \nuv, \left(\Lambda^{-1} - \Lambda^{-1}S \Lambda^{-1} + Q_e\right)\nuv \rangle + o(2)\\
    &= \langle\zeta, \Lambda^{-1}\zeta\rangle - \langle \Lambda^{-1}\zeta, S\Lambda^{-1}\zeta\rangle + \langle\zeta, Q_e\zeta\rangle - 2\langle \zeta,\Lambda^{-1}\nuv\rangle + 2\langle \Lambda^{-1}\zeta, S\Lambda^{-1}\nuv\rangle + \langle\nuv, \Lambda^{-1}\nuv\rangle + o(2) \\
    &= \langle\zeta, \Lambda^{-1}\zeta\rangle - \langle \Lambda^{-1}\zeta, S\Lambda^{-1}\zeta\rangle -  2\langle \zeta,\Lambda^{-1}\nuv\rangle + 2\langle \Lambda^{-1}\zeta, S\Lambda^{-1}\nuv\rangle + \langle\nuv, \Lambda^{-1}\nuv\rangle +\langle\zeta, Q_e\zeta\rangle +o(2) \\
    &= \langle \zeta, \Lambda^{-1}\zeta\rangle + \underset{T(\nuv,S)}{\underbrace{\langle \zeta, -2\Lambda^{-1}\nuv - \Lambda^{-1}S\Lambda^{-1}\zeta\rangle}} + \, \, \, \underset{\cq(\nuv,S)}{\underbrace{2\langle \zeta, \Lambda^{-1}S\Lambda^{-1}\nuv\rangle + \langle\nuv, \Lambda^{-1}\nuv\rangle + \langle \zeta, Q_e\zeta\rangle}} + o(2). 
\end{align*}
Thus, 
\begin{align*}
    \exp\left(- \frac{1}{2}\langle (\zeta-\nuv), (I+\wt{S})^{-1}\Lambda^{-1}(I+\wt{S})^{-1}(\zeta -\nuv) \rangle \right) &= \exp\left(- \frac{1}{2}\langle \zeta, \Lambda^{-1}\zeta\rangle -\frac{1}{2}T(\nuv,S) -\frac{1}{2}\cq(\nuv, S) +o(2) \right) \\
    & \hspace{-3cm}= {\exp\left(- \frac{1}{2}\langle \zeta, \Lambda^{-1}\zeta\rangle \right)}\left(1 -\frac{1}{2}T(\nuv,S) -\frac{1}{2}\cq(\nuv, S) +\frac{1}{8}T^2(\nuv,S) +o(2)\right).
\end{align*}
The last step was obtained by second order Taylor expansion of $ \exp\left( -\frac{1}{2}T(\nuv,S) -\frac{1}{2}\cq(\nuv, S) +o(2) \right) $ in $(\nuv,S)$ about the origin $(\mathbf{0},\mathbf{0})$.
The Taylor expansion of determinant is  
\begin{align*}
    \det(I+\wt{S}) &= 1 + \tr(\wt{S}) + \frac{1}{2}(\tr(\wt{S})^2 - \tr(\wt{S}^2)) + o(2)\\
    \frac{1}{\det(I+\wt{S})} &= 1- \tr(\wt{S}) - \frac{1}{2}(\tr(\wt{S})^2 - \tr(\wt{S}^2)) + \tr(\wt{S})^2 +o(2) \\
    &= 1- \tr(\wt{S}) + \frac{1}{2}(\tr(\wt{S})^2 + \tr(\wt{S}^2)) +o(2). 
\end{align*}
Putting all these together, we get
\begin{align}
    &\quad g(\exp_{(\muv,\Lambda)}(\nuv,S)) \\
    &= g(\muv,\Lambda) \left(1 -\frac{1}{2}T(\nuv,S) -\frac{1}{2}\cq(\nuv, S) +\frac{1}{8}T^2(\nuv,S) +o(2)\right)\left(  1- \tr(\wt{S}) + \frac{1}{2}(\tr(\wt{S})^2 + \tr(\wt{S}^2)) +o(2)\right) \\
    &= g(\muv,\Lambda) \left(1 - \frac{1}{2}T(\nuv,S)-\tr(\wt{S}) +\underset{Q_g}{\underbrace{\frac{1}{2}(\tr(\wt{S})^2+\tr(\wt{S}^2))+ \frac{1}{8}T^2(\nuv,S)-\frac{1}{2}\cq(\nuv,S) +\frac{1}{2}T(\nuv,S)\tr(\wt{S})}}\right) +o(2). \label{Taylor01}
\end{align}
We now write the Taylor expansion of $(\nabla_{\muv} g)\circ \exp_{(\muv,\Lambda)}$ as 
\begin{align}
    &\quad (\nabla_{\muv} g)\circ \exp_{(\muv,\Lambda)}(\nuv, S) \nonumber \\
    &= g(\exp_{(\muv,\Lambda)}(\nuv, S))  (I+\wt{S})^{-1}\Lambda^{-1}(I+\wt{S})^{-1}(\zeta-\nuv) \nonumber \\
    &=g(\muv,\Lambda) \left(1 - \frac{1}{2}T(\nuv,S)-\tr(\wt{S}) +\underset{Q_g}{\underbrace{\frac{1}{2}(\tr(\wt{S})^2+\tr(\wt{S}^2))+ \frac{1}{8}T^2(\nuv,S)-\frac{1}{2}\cq(\nuv,S) +\frac{1}{2}T(\nuv,S)\tr(\wt{S})}}\right) \nonumber \\
    &\times \left(\Lambda^{-1} - \Lambda^{-1}S \Lambda^{-1} + \underset{Q_e}{\underbrace{(\wt{S}^2\Lambda^{-1} + \Lambda^{-1}\wt{S}^2 + \wt{S}\Lambda^{-1}\wt{S})}} \right)(\zeta - \nuv) +o(2) \nonumber \\
    &\hspace{-1cm} = g(\muv,\Lambda) \left(\Lambda^{-1} -\frac{1}{2}T(\nuv, S)\Lambda^{-1} - \tr(\wt{S})\Lambda^{-1} - \Lambda^{-1}S\Lambda^{-1}+Q_g\Lambda^{-1}+Q_e +\left(\frac{1}{2}T(\nuv, S)+ \tr(\wt{S}) \right)\Lambda^{-1}S\Lambda^{-1}\right)(\zeta -\nuv) \nonumber\\
    &\qquad \qquad \qquad \qquad \qquad \qquad \hspace{10cm} +o(2)\nonumber \\
    &= g(\muv,\Lambda)\left[\Lambda^{-1}\zeta - \frac{1}{2}T(\nuv,S) \Lambda^{-1}\zeta - \tr(\wt{S})\Lambda^{-1}\zeta -\Lambda^{-1}S\Lambda^{-1}\zeta -\Lambda^{-1}\nuv\right] \nonumber \\
    &+ g(\muv, \Lambda)\left[\left(Q_g\Lambda^{-1}+Q_e +\left(\frac{1}{2}T(\nuv, S)+ \tr(\wt{S}) \right)\Lambda^{-1}S\Lambda^{-1} \right)\zeta +\left(\frac{1}{2}T(\nuv,S)\Lambda^{-1}+\tr(\wt{S})\Lambda^{-1}S\Lambda^{-1}\right)\nuv \right]+ o(2). \label{Taylor12}
\end{align}
Next, 
\begin{align}
    &\quad (I+\wt{S})^{-1}\Lambda^{-1}(I+\wt{S})^{-1}(\zeta-\nuv)(\zeta-\nuv)^{T} 
     +(\zeta-\nuv)(\zeta-\nuv)^{T}(I+\wt{S})^{-1}\Lambda^{-1}(I+\wt{S})^{-1}  \nonumber\\
     &= (\Lambda^{-1}-\Lambda^{-1}S\Lambda^{-1}+Q_e)(\zeta\zeta^T - \zeta\nuv^T -\nuv\zeta^T+\nuv\nuv^T) + (\zeta\zeta^T - \zeta\nuv^T -\nuv\zeta^T+\nuv\nuv^T)(\Lambda^{-1}-\Lambda^{-1}S\Lambda^{-1}+Q_e) +o(2) \nonumber\\
     &= \Lambda^{-1}\zeta\zeta^T +\zeta\zeta^T\Lambda^{-1} - \Lambda^{-1}S\Lambda^{-1}\zeta\zeta^T - \zeta\zeta^T\Lambda^{-1}S\Lambda^{-1}-\Lambda^{-1}\zeta\nuv^T-\Lambda^{-1}\nuv\zeta^T -\zeta\nuv^T\Lambda^{-1}-\nuv\zeta^T\Lambda^{-1}  \nonumber\\
     &+\Lambda^{-1}S\Lambda^{-1}(\zeta\nuv^T+\nuv\zeta^T)+(\zeta\nuv^T + \nuv\zeta^T)\Lambda^{-1}S\Lambda^{-1} +\Lambda^{-1}\nuv\nuv^T + \nuv \nuv^T\Lambda^{-1} + \Lambda^{-1}Q_e +Q_e\Lambda^{-1}+ o(2)\nonumber \\
     &= P_{\Lambda^{-1}}(\zeta\zeta^T) - \Lambda^{-1}S\Lambda^{-1}\zeta\zeta^T - \zeta\zeta^T\Lambda^{-1}S\Lambda^{-1}-P_{\Lambda^{-1}}(\zeta\nuv^T+\nuv\zeta^T)  \nonumber\\
     &\underset{Q_l}{\underbrace{+\Lambda^{-1}S\Lambda^{-1}(\zeta\nuv^T+\nuv\zeta^T)-(\zeta\nuv^T + \nuv\zeta^T)\Lambda^{-1}S\Lambda^{-1} +P_{\Lambda^{-1}}(\nuv\nuv^T +Q_e)}} + o(2)\nonumber
\end{align}
where $P_{\Lambda^{-1}}$ is the operator given by \eqref{Pdef}.
The Taylor expansion of $(\nabla_{\Lambda}g)\circ \exp_{(\muv, \Lambda)}$ is
\begin{align}
         &\quad (\nabla_{\Lambda} g)\circ \exp_{(\muv,\Lambda)}(\nuv, S) \nonumber \\
         &=g(\muv,\Lambda) \left(1 - \frac{1}{2}T(\nuv,S)-\tr(\wt{S}) +Q_g\right) \nonumber \\
         &\times \left[P_{\Lambda^{-1}}(\zeta\zeta^T) - \Lambda^{-1}S\Lambda^{-1}\zeta\zeta^T - \zeta\zeta^T\Lambda^{-1}S\Lambda^{-1}-P_{\Lambda^{-1}}(\zeta\nuv^T+\nuv\zeta^T) +Q_l-2I\right] \nonumber \\
         &= g(\muv,\Lambda)\left[P_{\Lambda^{-1}}(\zeta\zeta^T)- 2I - \Lambda^{-1}S\Lambda^{-1}\zeta\zeta^T -\zeta\zeta^T\Lambda^{-1}S\Lambda^{-1} - P_{\Lambda^{-1}}(\zeta\nuv^T+\nuv\zeta^T) \right] \nonumber\\
         &+ g(\muv, \Lambda)\left[\left(2I - P_{\Lambda^{-1}}(\zeta\zeta^T)\right)\left(\frac{1}{2}T(\nuv,S) +\tr(\wt{S}) \right)+\left( P_{\Lambda^{-1}}(\zeta\zeta^T) -2I \right)Q_g  +Q_l \right] \nonumber \\
         &+ g(\muv, \Lambda)\left[\left( \Lambda^{-1}S\Lambda^{-1}\zeta\zeta^T + \zeta\zeta^T\Lambda^{-1}S\Lambda^{-1}+P_{\Lambda^{-1}}(\zeta\nuv^T+\nuv\zeta^T)\right)\left(\frac{1}{2}T(\nuv, S) + \tr(\wt{S}) \right)\right] +o(2)\label{Taylor22}
\end{align}

Integrating equations \eqref{Taylor12} and \eqref{Taylor22} w.r.t. $ \z \in \mathbb{R}^n$, after multiplying the loss $\ell(\w;\z) $, we get
\begin{align}
    \nabla_{(\muv,\Lambda)}\mathbb{E}_{\z \sim \probP(\exp_{(\muv,\Lambda)}(\nuv, S))} [\ell(\w;\z) ]  = \left(\nabla_{\muv}\mathbb{E}_{\z \sim \probP(\exp_{(\muv,\Lambda)}(\nuv, S))} [\ell(\w;\z) ], \nabla_{\Lambda}\mathbb{E}_{\z \sim \probP(\exp_{(\muv,\Lambda)}(\nuv, S))} [\ell(\w;\z) ] \right)  \label{Taylor007}
\end{align}
where\footnote{We emphasize that the expectation operator $ \mathbb{E}_{\z \sim \probP}$ in the last step \eqref{Taylor007} is w.r.t. the push forward probability measure $ \probP(\exp_{(\muv,\Lambda)}(\nuv, S))$. The Taylor expansions then allow us to work with the expectation operator $ \mathbb{E}_{\z \sim \probP}$ at origin, i.e., $\probP(\exp_{(\muv,\Lambda)}(\mathbf{0}, \mathbf{0})) $ which is simply the Gaussian measure with parameters $\muv, \Lambda$. } the Taylor series of the components of the gradient are
\begin{align}
    & \nabla_{\muv}\mathbb{E}_{\z \sim \probP(\exp_{(\muv,\Lambda)}(\nuv, S))} [\ell(\w;\z) ](\nuv,\wt{S}) \nonumber \\
    &= \mathbb{E}_{\z \sim \probP(\muv,\Lambda)}\left[\ell(\w;\z)\left(\left[\Lambda^{-1}\zeta - \frac{1}{2}T(\nuv,S) \Lambda^{-1}\zeta - \tr(\wt{S})\Lambda^{-1}\zeta -\Lambda^{-1}S\Lambda^{-1}\zeta -\Lambda^{-1}\nuv\right] \right)\right] \nonumber \\
    &+ \mathbb{E}_{\z \sim \probP(\muv,\Lambda)}\left[\ell(\w;\z)\left(\left(Q_g\Lambda^{-1}+Q_e +\left(\frac{1}{2}T(\nuv, S)+ \tr(\wt{S}) \right)\Lambda^{-1}S\Lambda^{-1} \right)\zeta \right)\right] \nonumber \\ & + \mathbb{E}_{\z \sim \probP(\muv,\Lambda)}\left[\ell(\w;\z)\left(\left(\frac{1}{2}T(\nuv,S)\Lambda^{-1}+\tr(\wt{S})\Lambda^{-1}S\Lambda^{-1}\right)\nuv +  o(2) \right)\right], \label{Taylor13}
\end{align}
and 
\begin{align}
    &\quad \nabla_{\Lambda}\mathbb{E}_{\z \sim \probP(\exp_{(\muv,\Lambda)}(\nuv, S))} [\ell(\w;\z) ](\nuv,\wt{S}) \nonumber \\
    &= \mathbb{E}_{\z \sim \probP(\muv,\Lambda)}\left[\ell(\w;\z)\left(P_{\Lambda^{-1}}(\zeta\zeta^T) - \Lambda^{-1}S\Lambda^{-1}\zeta\zeta^T - \zeta\zeta^T\Lambda^{-1}S\Lambda^{-1}-P_{\Lambda^{-1}}(\zeta\nuv^T+\nuv\zeta^T) -2I\right)\right] \nonumber \\
     &+ \mathbb{E}_{\z \sim \probP(\muv,\Lambda)}\left[\ell(\w;\z) \left(\left(2I - P_{\Lambda^{-1}}(\zeta\zeta^T)\right)\left(\frac{1}{2}T(\nuv,S) +\tr(\wt{S}) \right)+\left(P_{\Lambda^{-1}}(\zeta\zeta^T)-2I \right)Q_g +Q_l \right)   \right] \nonumber \\
         &+ \mathbb{E}_{\z \sim \probP(\muv,\Lambda)}\left[\ell(\w;\z) \left(\left( \Lambda^{-1}S\Lambda^{-1}\zeta\zeta^T + \zeta\zeta^T\Lambda^{-1}S\Lambda^{-1}+P_{\Lambda^{-1}}(\zeta\nuv^T+\nuv\zeta^T)\right)\left(\frac{1}{2}T(\nuv, S) + \tr(\wt{S}) \right) +o(2) \right)\right]  \label{Taylor23}
\end{align}

\section{Uniform estimate computations in a compact neighborhood of \texorpdfstring{$ (\muv^*,\Lambda^*)$}{}  and for all $\w \in \mathbb{R}^d$}\label{sectionuniflipestimation}
In this section we show that various quantities crucial to our analysis have bounded norm on the {product of closed balls of radius $r = (1- \gamma) \norm{(\Lambda^*)^{-1}}_2^{-1/2}$ with $\gamma \in (0,1)$} centered\footnote{Without loss of generality to keep things compact, we take a bounded perturbation of $ \muv^*$ even though the Euclidean space has an infinite injectivity radius and $\norm{\nuv}$ can be arbitrary large even when $\norm{\wt{S}}_F < 1$. For simplicity we can take $\norm{\nuv} \le 1 $, i.e. a larger ball for $\muv$, by assuming small enough $r$ without loss of generality. } at $(\muv^*,\Lambda^*)$ for any given $\w$. \\
\subsection{Uniform estimates of $ \norm{\Lambda^{1/2}}_F$ , $ \norm{\Lambda^{-1/2}}_F$ :}\label{appendixunifestimatenormlambda}
Since $\Lambda = (I+L_{\Lambda^*}(S))\Lambda^*(I+L_{\Lambda^*}(S)) = \Lambda^* + S + L_{\Lambda^*}(S)\Lambda^* L_{\Lambda^*}(S) $ and $ \wt{S} = L_{\Lambda^*}(S)$ is symmetric with $ \norm{\wt{S}}_F < 1$ it must be that:
\begin{align}
    \norm{\Lambda^{1/2}}_F & = \norm{(I+L_{\Lambda^*}(S)) (\Lambda^*)^{1/2}}_F \le (1 + \norm{L_{\Lambda^*}(S)}_F) \norm{(\Lambda^*)^{1/2}}_F \le 2 \norm{(\Lambda^*)^{1/2}}_F \, . \label{Taylor13ao}
\end{align}
Next, $$ d_{BW}(\Lambda, \Lambda^*) = \norm{S}_{BW} = \sqrt{\text{tr}(L_{\Lambda^*}(S)\Lambda^* L_{\Lambda^*}(S))} = \norm{(\Lambda^*)^{1/2} L_{\Lambda^*}(S) }_F $$ and thus 
\[  r > d_{BW}(\Lambda, \Lambda^*) = \norm{(\Lambda^*)^{1/2} L_{\Lambda^*}(S) }_F \ge    \norm{(\Lambda^*)^{-1/2}  }^{-1}_F  \norm{ L_{\Lambda^*}(S) }_F \]
Hence, for $r <  \norm{(\Lambda^*)^{-1/2}}^{-1}_F $ we will get $ \norm{\wt{S}}_F \le r \norm{(\Lambda^*)^{-1/2}}_F < 1$ and then
\begin{align}
    \norm{\Lambda^{-1/2}}_F  = \norm{(I+L_{\Lambda^*}(S))^{-1} (\Lambda^*)^{-1/2}}_F \le \norm{(I+L_{\Lambda^*}(S))^{-1}}_F \norm{(\Lambda^*)^{-1/2}}_F &\le \frac{\sqrt{n} \norm{(\Lambda^*)^{-1}}_F}{(1-\norm{L_{\Lambda^*}(S)}_F)}  \\
    & \le     \frac{\sqrt{n} \norm{(\Lambda^*)^{-1}}_F}{(1-r \norm{(\Lambda^*)^{-1/2}}_F)} \, . \label{Taylor13bo}
\end{align}
\begin{rem}
    We remind that the uniform estimates in this section are derived using Taylor expansions that hold for $ \norm{\wt S}_F <1$, where $ \wt{S} = L_{\Lambda^*}(S)$, so that $\Lambda = (I+L_{\Lambda^*}(S))\Lambda^*(I+L_{\Lambda^*}(S)) = \Lambda^* + S + L_{\Lambda^*}(S)\Lambda^* L_{\Lambda^*}(S) $ is invertible. But $\Lambda$ may be invertible even when $ \norm{\wt S}_F  \ge 1$. We can however show that such $\Lambda$ will not be inside the ball of radius $r = (1- \gamma) \norm{(\Lambda^*)^{-1}}_2^{-1/2}$ centered at $\Lambda^*$. Since $d_{BW}(\Lambda, \Lambda^*) = \sqrt{ \text{tr}(\wt S \Lambda^* \wt S)} $, for $ \norm{\wt S}_F \ge 1$, $\wt S \in \mathbb{S}^n$ we must have 
    \[ r^2 \le \norm{(\Lambda^*)^{-1}}^{-1}_2 \le \norm{\wt S}_F^2  \norm{(\Lambda^*)^{-1}}^{-1}_2  \le \text{tr}((\wt S)^2 \Lambda^* ) =  \text{tr}(\wt S \Lambda^* \wt S) =  d^2_{BW}(\Lambda, \Lambda^*) \]
    and hence $d_{BW}(\Lambda, \Lambda^*) <r$ implies that $ \norm{\wt S}_F <1$. Therefore the uniform estimates derived in this section for $ \norm{\wt S}_F <1$ will hold uniformly on the closed ball $ \{ \Lambda \, \vert \, d_{BW}(\Lambda, \Lambda^*) \le r \}$.
\end{rem}

\subsection{Uniform estimates of \texorpdfstring{ $\norm{ \nabla_{\w}  \nabla_{(\muv,\Lambda)} \mathbb{E}_{\z \sim \probP(\muv,\Lambda)} [\ell(\w ;\z) ]}_{op}$ , $\norm{   \nabla_{(\muv,\Lambda)} \mathbb{E}_{\z \sim \probP(\muv,\Lambda)} [\ell(\w ;\z) ]}_{op}$}{} :}\label{mixedderivativeappendix}
 By the dominated convergence theorem (along with assumption \textbf{A2}) and joint $\mathcal{C}^2$ smoothness of the loss $\ell(\cdot \, ; \, \cdot)$ we can write:
\[   \nabla_{\w}\nabla_{(\muv,\Lambda)} \mathbb{E}_{\z \sim \probP(\muv,\Lambda)} [\ell(\w;\z) ] = \nabla_{(\muv,\Lambda)} \mathbb{E}_{\z \sim \probP(\muv,\Lambda)} [\nabla_{\w}\ell(\w;\z) ]\]
and thus we have
\begin{align}
    & \big(\nabla_{\muv}\mathbb{E}_{\z \sim \probP(\exp_{(\muv,\Lambda)}(\nuv, S))} [\nabla_{\w}\ell(\w;\z) ](\nuv,\wt{S}) \big)\nonumber \\
    &= \mathbb{E}_{\z \sim \probP(\muv,\Lambda)}\left[\left(\left[\Lambda^{-1}\zeta - \frac{1}{2}T(\nuv,S) \Lambda^{-1}\zeta - \tr(\wt{S})\Lambda^{-1}\zeta -\Lambda^{-1}S\Lambda^{-1}\zeta -\Lambda^{-1}\nuv\right] \right)   (\nabla_{\w}\ell(\w;\z))^T\right] \nonumber \\
    &+ \mathbb{E}_{\z \sim \probP(\muv,\Lambda)}\left[\left(\left(Q_g\Lambda^{-1}+Q_e +\left(\frac{1}{2}T(\nuv, S)+ \tr(\wt{S}) \right)\Lambda^{-1}S\Lambda^{-1} \right)\zeta \right)   (\nabla_{\w}\ell(\w;\z))^T \right] \nonumber \\ & + \mathbb{E}_{\z \sim \probP(\muv,\Lambda)}\left[\left(\left(\frac{1}{2}T(\nuv,S)\Lambda^{-1}+\tr(\wt{S})\Lambda^{-1}S\Lambda^{-1}\right)\nuv +  o(2) \right)  (\nabla_{\w}\ell(\w;\z))^T \right], \label{Taylor13a}
\end{align}
where taking norm\footnote{Operator norms $\norm{ \nabla_{\w}  \nabla_{(\muv,\Lambda)} \mathbb{E}_{\z \sim \probP(\muv,\Lambda)} [\ell(\w ;\z) ]}_{op}$ , $\norm{   \nabla_{(\muv,\Lambda)} \mathbb{E}_{\z \sim \probP(\muv,\Lambda)} [\ell(\w ;\z) ]}_{op}$ are induced by the Riemannian metric at $(\muv,\Lambda) $.} both sides in \eqref{Taylor13a}, applying Cauchy-Schwarz, triangle and H\"older inequalities repeatedly with $ 1/p + 1/q = 1$ , $p \ge 1 $ and simplifying gives:
\begin{align}
    & \norm{\big(\nabla_{\muv}\mathbb{E}_{\z \sim \probP(\exp_{(\muv,\Lambda)}(\nuv, S))} [\nabla_{\w}\ell(\w;\z) ](\nuv,\wt{S}) \big)}_F\nonumber \\
    &  \le \big( \mathbb{E}_{\z \sim \probP}\left[\norm{\nabla_{\w}\ell(\w;\z)}^q\right] \big)^{1/q} \Bigg(\norm{\Lambda^{-1/2}}_F  \big(\mathbb{E}_{\z \sim \probP}\left[ \norm{\Lambda^{-1/2}\zeta}^p \right]\big)^{1/p}  \nonumber \\ & + \frac{1}{2} \norm{\Lambda^{-1/2}}_F \big(\mathbb{E}_{\z \sim \probP}\left[\abs{T(\nuv,S)}^{2p} \right]\big)^{1/2p} \big(\mathbb{E}_{\z \sim \probP}\left[  \norm{\Lambda^{-1/2}\zeta}^{2p} \right]\big)^{1/2p} \nonumber \\ & + \big(\sqrt{n} \norm{\wt{S}}_F\norm{\Lambda^{-1/2}}_F  +\norm{\Lambda^{-1}S\Lambda^{-1/2}}_F\big) \big(\mathbb{E}_{\z \sim \probP}\left[ \norm{\Lambda^{-1/2}\zeta}^p \right] \big)^{1/p} + \norm{\Lambda^{-1}\nuv} \nonumber \\
    &+ \Bigg( \norm{\Lambda^{-1/2}}_F \big(\mathbb{E}_{\z \sim \probP(\muv,\Lambda)}\left[ \abs{Q_g}^{2p}  \right] \big)^{1/2p}  + \norm{\Lambda^{1/2}}_F\big(\mathbb{E}_{\z \sim \probP(\muv,\Lambda)}\left[ \norm{Q_e}_F^{2p}  \right] \big)^{1/2p} \nonumber \\ & + \norm{\Lambda^{-1}S\Lambda^{-1/2}}_F\Bigg( \frac{1}{2}\big(\mathbb{E}_{\z \sim \probP(\muv,\Lambda)}\left[ \abs{T(\nuv, S)}^{2p}  \right] \big)^{1/2p} + \sqrt{n} \norm{\wt{S}}_F\Bigg)  \Bigg) \big(\mathbb{E}_{\z \sim \probP}\left[  \norm{\Lambda^{-1/2}\zeta}^{2p} \right]\big)^{1/2p} \nonumber \\ & + \Bigg( \frac{1}{2}\big(\mathbb{E}_{\z \sim \probP(\muv,\Lambda)}\left[ \abs{T(\nuv, S)}^{p}  \right] \big)^{1/p} \norm{\Lambda^{-1} \nuv} + \sqrt{n} \norm{\wt{S}}_F \norm{\Lambda^{-1}S\Lambda^{-1} \nuv} + o(2)\Bigg)  \Bigg) , \label{Taylor13ab}
\end{align}
and using the estimates from section \ref{auxillaryestimatessec} in \eqref{Taylor13ab} along with the bounds $$ \norm{\wt{S}}_F<1 \iff  \norm{L_{\Lambda}(S)}_F<1 \implies \norm{{S}}_F< 2 \norm{\Lambda}_F \norm{\wt{S}}_F < 2 \norm{\Lambda}_F , \norm{\nuv} < 1$$ for $(\muv, \Lambda) = (\muv^*, \Lambda^*) $ we will get that:
\begin{align}
   \norm{\big(\nabla_{\muv}\mathbb{E}_{\z \sim \probP(\exp_{(\muv^*,\Lambda^*)}(\nuv, S))} [\nabla_{\w}\ell(\w;\z) ](\nuv,\wt{S}) \big)}_F \le C_{p,\Lambda^*,n} \times \big( \mathbb{E}_{\z \sim \probP(\muv^*, \Lambda^*) }\left[\norm{\nabla_{\w}\ell(\w;\z)}^q\right] \big)^{1/q}    \label{Taylor23ax}
\end{align}
for some universal constant $C_{p,\Lambda^*,n} $ where $p \ge 1 $. 

A similar calculation by replacing $\nabla_{\w}\ell(\w;\z) $ with $ \ell(\w;\z)$ will yield:
\begin{align}
   \norm{\big(\nabla_{\muv}\mathbb{E}_{\z \sim \probP(\exp_{(\muv^*,\Lambda^*)}(\nuv, S))} [\ell(\w;\z) ](\nuv,\wt{S}) \big)} \le C_{p,\Lambda^*,n} \times \big( \mathbb{E}_{\z \sim \probP(\muv^*, \Lambda^*) }\left[\abs{\ell(\w;\z)}^q\right] \big)^{1/q}    \label{Taylor23ay}
\end{align}
for the same universal constant $C_{p,\Lambda^*,n} $ where $p \ge 1 $. 

Next, adopting the numerator layout notation of $ \text{vec} ( \nabla_{\w} f ) = \frac{\partial \text{vec} (f)}{\partial \w^T} $ from \cite{magnus1985matrix} we will have
\begin{align}
    &\quad \text{vec}\big(\nabla_{\Lambda}\mathbb{E}_{\z \sim \probP(\exp_{(\muv,\Lambda)}(\nuv, S))} [\nabla_{\w}\ell(\w;\z) ](\nuv,\wt{S})\big) \nonumber \\
    &= \mathbb{E}_{\z \sim \probP(\muv,\Lambda)}\left[ \text{vec}\left(P_{\Lambda^{-1}}(\zeta\zeta^T) - \Lambda^{-1}S\Lambda^{-1}\zeta\zeta^T - \zeta\zeta^T\Lambda^{-1}S\Lambda^{-1}-P_{\Lambda^{-1}}(\zeta\nuv^T+\nuv\zeta^T) -2I\right)  \otimes  (\nabla_{\w}\ell(\w;\z))^T  \right] \nonumber \\
     &+ \mathbb{E}_{\z \sim \probP(\muv,\Lambda)}\left[\text{vec} \left(\left(2I - P_{\Lambda^{-1}}(\zeta\zeta^T)\right)\left(\frac{1}{2}T(\nuv,S) +\tr(\wt{S}) \right)+\left(P_{\Lambda^{-1}}(\zeta\zeta^T)-2I\right)Q_g +Q_l\right)   \otimes  (\nabla_{\w}\ell(\w;\z))^T   \right] \nonumber \\
         & \hspace{-1cm}+ \mathbb{E}_{\z \sim \probP(\muv,\Lambda)}\left[ \text{vec}\left(\left( \Lambda^{-1}S\Lambda^{-1}\zeta\zeta^T + \zeta\zeta^T\Lambda^{-1}S\Lambda^{-1}+P_{\Lambda^{-1}}(\zeta\nuv^T+\nuv\zeta^T)\right)\left(\frac{1}{2}T(\nuv, S) + \tr(\wt{S}) \right) +o(2) \right)  \otimes  (\nabla_{\w}\ell(\w;\z))^T \right]  \label{Taylor23a}
\end{align}
and by similar norm calculations\footnote{We omit norm computations in \eqref{Taylor23a} due to their tedious nature and also for brevity. The computations are similar to those in \eqref{Taylor13ab}.} as in \eqref{Taylor13ab}, using the estimates from section \ref{auxillaryestimatessec} along with the bounds $$ \norm{\wt{S}}_F<1 \iff  \norm{L_{\Lambda}(S)}_F<1 \implies \norm{{S}}_F< 2 \norm{\Lambda}_F \norm{\wt{S}}_F < 2 \norm{\Lambda}_F , \norm{\nuv} < 1$$ for $(\muv, \Lambda) = (\muv^*, \Lambda^*) $ we will get that
\begin{align}
    \norm{\text{vec}\big(\nabla_{\Lambda}\mathbb{E}_{\z \sim \probP(\exp_{(\muv^*,\Lambda^*)}(\nuv, S))} [\nabla_{\w}\ell(\w;\z) ](\nuv,\wt{S})\big) }_F &\le C'_{p,\Lambda^*,n} \times \big( \mathbb{E}_{\z \sim \probP(\muv^*, \Lambda^*) }\left[\norm{\nabla_{\w}\ell(\w;\z)}^q\right] \big)^{1/q} \label{Taylor23az}
\end{align}
for some universal constant $C'_{p,\Lambda^*,n} $ and $1/p + 1/q = 1$ where $p \ge 1 $.

Last, we need an estimate (in the BW metric) at the point $\Lambda = (I+L_{\Lambda^*}(S))\Lambda^*(I+L_{\Lambda^*}(S))$ , where $ \wt{S} = L_{\Lambda^*}(S) $, for the following norm: 
{
\begin{align*}
    \norm{\text{vec}\big(\nabla_{\Lambda}\mathbb{E}_{\z \sim \probP(\exp_{(\muv^*,\Lambda^*)}(\nuv, S))} [\nabla_{\w}\ell(\w;\z) ](\nuv,\wt{S})\big) }_{\text{BW}} &:= \sqrt{ \sum_{i=1}^d \text{tr}(L_{\Lambda}(V_i) \Lambda L_{\Lambda}(V_i) )  } 
\end{align*}
}
 where $V_i = \nabla_{\Lambda}\mathbb{E}_{\z \sim \probP(\exp_{(\muv^*,\Lambda^*)}(\nuv, S))} \big[ \frac{\partial\ell(\w;\z)}{\partial{\w_i}} \big](\nuv,\wt{S}) \in  T_{\Lambda} \texttt{SPD}_n $. Using the standard estimate \cite{horn2012matrix} for the operator norm of $L_{V} $ where $V \succ \mathbf{0}$ from the Sylvester equation $$ X V + VX = G \iff L_{V}(G) = X \implies \norm{X}_F \le \norm{ L_{V}}_{op} \norm{G}_F  \le \frac{1}{2}\norm{V^{-1}}_2 \norm{G}_F $$  
 gives the estimate
 $$ \norm{L_{\Lambda}(V_i)}_F   \le \frac{1}{2}\norm{\Lambda^{-1}}_2 \norm{V_i}_F $$
 and invoking the trivial upper bound  
 $$ \norm{V_i}_F \le \norm{\text{vec}\big(\nabla_{\Lambda}\mathbb{E}_{\z \sim \probP(\exp_{(\muv^*,\Lambda^*)}(\nuv, S))} [\nabla_{\w}\ell(\w;\z) ](\nuv,\wt{S})\big) }_F \quad \forall \quad 1 \le i \le d \, ,$$
 we get the following estimate:
 {
\begin{align}
    \norm{\text{vec}\big(\nabla_{\Lambda}\mathbb{E}_{\z \sim \probP(\exp_{(\muv^*,\Lambda^*)}(\nuv, S))} [\nabla_{\w}\ell(\w;\z) ](\nuv,\wt{S})\big) }_{\text{BW}} & = \sqrt{ \sum_{i=1}^d \text{tr}(L_{\Lambda}(V_i) \Lambda^{1/2} \Lambda^{1/2} L_{\Lambda}(V_i) )  } \\
    & =  \sqrt{ \sum_{i=1}^d \norm{\Lambda^{1/2} L_{\Lambda}(V_i)}^2_F } \\
    & \le \norm{\Lambda^{1/2}}_F\sqrt{  \sum_{i=1}^d \norm{ L_{\Lambda}(V_i)}^2_F } \\
    & \le \frac{1}{2}\norm{\Lambda^{-1}}_2  \norm{\Lambda^{1/2}}_F\sqrt{  \sum_{i=1}^d \norm{V_i}^2_F } 
        \end{align}
    \begin{align}
    &  \le \frac{\sqrt{d}}{2}\norm{\Lambda^{-1/2} \Lambda^{-1/2}}_F  \norm{\Lambda^{1/2}}_F  \norm{\text{vec}\big(\nabla_{\Lambda}\mathbb{E}_{\z \sim \probP(\exp_{(\muv^*,\Lambda^*)}(\nuv, S))} [\nabla_{\w}\ell(\w;\z) ](\nuv,\wt{S})\big) }_F \\
    &  \underbrace{\le}_{\eqref{Taylor23az}}  \frac{\sqrt{d}}{2}\norm{\Lambda^{-1/2}}^2_F  \norm{\Lambda^{1/2}}_F C'_{p,\Lambda^*,n} \times \big( \mathbb{E}_{\z \sim \probP(\muv^*, \Lambda^*) }\left[\norm{\nabla_{\w}\ell(\w;\z)}^q\right] \big)^{1/q} \\
     & \underbrace{\le}_{\eqref{Taylor13ao}, \eqref{Taylor13bo}}  {\sqrt{d}\norm{(\Lambda^*)^{1/2}}_F} \bigg(\frac{\sqrt{n} }{(1-r \norm{(\Lambda^*)^{-1/2}}_F)}\bigg)^2  C'_{p,\Lambda^*,n} \times \big( \mathbb{E}_{\z \sim \probP(\muv^*, \Lambda^*) }\left[\norm{\nabla_{\w}\ell(\w;\z)}^q\right] \big)^{1/q} 
         \end{align}
    \begin{align}
 \implies \norm{\text{vec}\big(\nabla_{\Lambda}\mathbb{E}_{\z \sim \probP(\exp_{(\muv^*,\Lambda^*)}(\nuv, S))} [\nabla_{\w}\ell(\w;\z) ](\nuv,\wt{S})\big) }_{\text{BW}}   & \le C'_{p,\Lambda^*,n, d, r} \times \big( \mathbb{E}_{\z \sim \probP(\muv^*, \Lambda^*) }\left[\norm{\nabla_{\w}\ell(\w;\z)}^q\right] \big)^{1/q} \label{Taylor23bz}
\end{align}
for some universal constant $C'_{p,\Lambda^*,n,d,r} $ and $1/p + 1/q = 1$ where $p \ge 1 $. \\
A similar calculation by replacing $\nabla_{\w}\ell(\w;\z) $ with $ \ell(\w;\z)$ will yield:
\begin{align}
\norm{\text{vec}\big(\nabla_{\Lambda}\mathbb{E}_{\z \sim \probP(\exp_{(\muv^*,\Lambda^*)}(\nuv, S))} [\ell(\w;\z) ](\nuv,\wt{S})\big) }_{\text{BW}}   & \le \tilde{C}'_{p,\Lambda^*,n, r} \times \big( \mathbb{E}_{\z \sim \probP(\muv^*, \Lambda^*) }\left[\abs{\ell(\w;\z)}^q\right] \big)^{1/q}     \label{Taylor23by}
\end{align}
for the universal constant\footnote{Notice that $ \norm{\text{vec}\big(\nabla_{\Lambda}\mathbb{E}_{\z \sim \probP(\exp_{(\muv^*,\Lambda^*)}(\nuv, S))} [\ell(\w;\z) ](\nuv,\wt{S})\big) }_{\text{BW}}$ estimate will not require a coordinate wise splitting over $d$ dimensions as done in the steps leading to \eqref{Taylor23bz}.} $ \tilde{C}'_{p,\Lambda^*,n, r} = \frac{1}{\sqrt{d}} C'_{p,\Lambda^*,n, d, r} $ where $C'_{p,\Lambda^*,n,d,r} $ is the constant from \eqref{Taylor23bz} , $p \ge 1 $ and $1/p + 1/q = 1$ . 
}
 \subsection{Uniform estimates of \texorpdfstring{$ {   \mathbb{E}_{\z \sim \probP(\muv,\Lambda)} [\abs{\ell(\w ;\z)}^p ]}$}{} for any $p \geq 1$ :}
From the Taylor expansion \eqref{Taylor01} 
\begin{align*}
  g(\exp_{(\muv,\Lambda)}(\nuv,S)) 
    &= g(\muv,\Lambda) \left(1 - \frac{1}{2}T(\nuv,S)-\tr(\wt{S}) +{Q_g} +o(2) \right) 
\end{align*}
 we have that :
\begin{align}
     \mathbb{E}_{\z \sim \probP(\exp_{(\muv,\Lambda)}(\nuv,S))} [\abs{\ell(\w ;\z)}^p ] & =  \int_{\z} \abs{\ell(\w ;\z)}^p g(\exp_{(\muv,\Lambda)}(\nuv,S)) d\z \\
     & =  \int_{\z} \abs{\ell(\w ;\z)}^p g(\muv,\Lambda) \left(1 - \frac{1}{2}T(\nuv,S)-\tr(\wt{S}) +{Q_g} +o(2) \right) d\z \\
     &  \le \mathbb{E}_{\z \sim \probP(\muv,\Lambda)} \bigg[\abs{\ell(\w ;\z)}^p \left(1 + \frac{1}{2} \abs{T(\nuv,S)} + \sqrt{n}\norm{\wt{S}}_F +\abs{Q_g} +o(2) \right) \bigg] \\
     & \hspace{-4.5cm} \underbrace{\le}_{\textbf{H\"older}} \big( \mathbb{E}_{\z \sim \probP(\muv,\Lambda)}\left[\abs{\ell(\w;\z)}^{pq}\right] \big)^{1/q}  \left(1 + \frac{1}{2} (\mathbb{E}_{\z \sim \probP(\muv,\Lambda)} [ \abs{T(\nuv,S)}^p ])^{1/p}  + \sqrt{n}\norm{\wt{S}}_F + (\mathbb{E}_{\z \sim \probP(\muv,\Lambda)} [\abs{Q_g}^p])^{1/p} +o(2) \right) .\label{Taylor010}
\end{align}
Then using the estimates from section \ref{auxillaryestimatessec} along with the bounds $$ \norm{\wt{S}}_F<1 \iff  \norm{L_{\Lambda}(S)}_F<1 \implies \norm{{S}}_F< 2 \norm{\Lambda}_F \norm{\wt{S}}_F < 2 \norm{\Lambda}_F , \norm{\nuv} < 1$$ for $(\muv, \Lambda) = (\muv^*, \Lambda^*) $ in \eqref{Taylor010} we will get that
\begin{align}
    \mathbb{E}_{\z \sim \probP(\exp_{(\muv^*,\Lambda^*)}(\nuv,S))} [\abs{\ell(\w ;\z)}^p ] \le C''_{p,\Lambda^*,n} \times \big( \mathbb{E}_{\z \sim \probP(\muv^*,\Lambda^*)}\left[\abs{\ell(\w;\z)}^{pq}\right] \big)^{1/q}  \label{Taylor13co}
\end{align}
for some universal constant $C''_{p,\Lambda^*,n} $ and $1/p + 1/q = 1$ where $p \ge 1 $ and
\begin{align*}
C''_{p,\Lambda^*,n}  & \lesssim_p n \max\{ 1, r^{2}\} \norm{(\Lambda^*)^{-1/2}}_F^{4} \norm{(\Lambda^*)^{1/2}}_F^{4} \big(\frac{\Gamma(\frac{n+4p}{2}) }{\Gamma(\frac{n}{2}) }\big)^{1/p}
\end{align*}
for $ \norm{\Lambda}_F \ge 1 \, , \, \norm{\Lambda^{-1/2} }_F \ge 1$ from section \ref{auxillaryestimatessec}.

 \subsection{Uniform estimates of \texorpdfstring{$\norm{ \text{Hess}_{(\muv,\Lambda)} \big( \mathcal{L}(\muv,\Lambda; \w) \big)  }_{op}$}{} :}
Recall that 
\[
\text{Hess}_{(\muv,\Lambda)} \big( \mathcal{L}(\muv,\Lambda; \w) \big)  :=\text{Hess}_{(\muv,\Lambda)} \bigg( \mathbb{E}_{\z \sim \probP(\muv,\Lambda)} [\ell(\w ;  \z) ] - \frac{1}{\epsilon} \bigg(  \norm{\muv - \muv^*}^2 + d^2_{BW}(\Lambda ,\Lambda^*)\bigg) \bigg) 
\]
and for any $(\muv,\Lambda) $ in the {product of closed balls of radius $r$} centered at $(\muv^*,\Lambda^*)$, from prior Hessian estimates \eqref{bwhessianestimate1g} we have:
\begin{align}
      \sup_{\nuv \in \mathbb{R}^n, \norm{S}_F < 1} \frac{ \abs{\text{Hess}_{(\muv,\Lambda)}\mathcal{L}((\nuv,  S\Lambda +\Lambda S),(\nuv,S\Lambda +\Lambda S))}}{ \norm{\nuv}^2 + tr(S \Lambda S)} &  \le C_{\Lambda, q, n} (\mathbb{E}_{\z \sim \probP(\muv,\Lambda)} [ |\ell(\w;\z)|^p])^{1/p}  + \frac{2}{\epsilon} 
\end{align}
where $ C_{\Lambda, q, n} \sim_{n,q} \norm{\Lambda}^2_F \norm{\Lambda^{-1/2} }^{6}_F $ from \eqref{bwhessianestimate1d0} , $1/p +1/q =1$ , $q \ge 1$ and $ C_{\Lambda, q, n} $ is uniformly bounded above for any compact set where $\Lambda \succ \mathbf{0} $. Using the estimates \eqref{Taylor13ao}, \eqref{Taylor13bo} and \eqref{Taylor13co} we get that:
\begin{align*}
    C_{\Lambda, q, n} (\mathbb{E}_{\z \sim \probP(\muv,\Lambda)} [ |\ell(\w;\z)|^p])^{1/p}  &\lesssim_{n,q} \norm{\Lambda^{1/2} }^4_F \norm{\Lambda^{-1/2} }^{6}_F   (\mathbb{E}_{\z \sim \probP(\muv,\Lambda)} [ |\ell(\w;\z)|^p])^{1/p} \\
    & \hspace{-2.5cm} \lesssim_{n,q} \norm{\Lambda^{1/2}}^4_F \norm{\Lambda^{-1/2} }^{6}_F   \bigg( C''_{p,\Lambda^*,n} \times \big( \mathbb{E}_{\z \sim \probP(\muv^*,\Lambda^*)}\left[\abs{\ell(\w;\z)}^{pq}\right] \big)^{1/q} \bigg)^{1/p} \\
     & \hspace{-2.5cm} \lesssim_{n,q} (C''_{p,\Lambda^*,n})^{1/p} \times { \norm{ 2 (\Lambda^*)^{1/2}}^4_F} \bigg(\frac{\sqrt{n} }{(1-r \norm{(\Lambda^*)^{-1/2}}_F)}\bigg)^6 \big( \mathbb{E}_{\z \sim \probP(\muv^*,\Lambda^*)}\left[\abs{\ell(\w;\z)}^{pq}\right] \big)^{1/{pq}} \\
     & \hspace{-2.5cm} \le C''_{p,\Lambda^*,n, r} \times \big( \mathbb{E}_{\z \sim \probP(\muv^*,\Lambda^*)}\left[\abs{\ell(\w;\z)}^{pq}\right] \big)^{1/{pq}}
\end{align*}
and\footnote{Observe that the constant $C''_{p,\Lambda^*,n, r} $ omits $q$ dependence since it already depends on $p$ and $1/p +1/q =1$.} hence we get that
\begin{align}
    \sup_{\nuv \in \mathbb{R}^n, \norm{S}_F < 1} \frac{ \abs{\text{Hess}_{(\muv,\Lambda)}\mathcal{L}((\nuv,  S\Lambda +\Lambda S),(\nuv,S\Lambda +\Lambda S))}}{ \norm{\nuv}^2 + tr(S \Lambda S)} &  \le C''_{p,\Lambda^*,n, r} \times \big( \mathbb{E}_{\z \sim \probP(\muv^*,\Lambda^*)}\left[\abs{\ell(\w;\z)}^{pq}\right] \big)^{1/{pq}}  + \frac{2}{\epsilon} \label{Taylor13do}
\end{align}
for some universal constant $C''_{p,\Lambda^*,n,r} $ and $1/p + 1/q = 1$ where $p \ge 1 $.
\\ \\
\subsection{Uniform estimates of \texorpdfstring{$ \norm{   \mathbb{E}_{\z \sim \probP(\muv,\Lambda)} [\nabla^2_{\w}\ell(\w ;\z) ]}_{F}$ , $ \norm{\mathbb{E}_{\z \sim \probP(\muv,\Lambda)} [\nabla_{\w}\ell(\w ;\z) ]}$}{} :}
From the Taylor expansion \eqref{Taylor01} 
\begin{align*}
  g(\exp_{(\muv,\Lambda)}(\nuv,S)) 
    &= g(\muv,\Lambda) \left(1 - \frac{1}{2}T(\nuv,S)-\tr(\wt{S}) +{Q_g} +o(2) \right) 
\end{align*}
 we have that :
\begin{align}
      \norm{\mathbb{E}_{\z \sim \probP(\exp_{(\muv,\Lambda)}(\nuv,S))} [\nabla^2_{\w} {\ell(\w ;\z)} ]}_F & \le  \int_{\z} \norm{\nabla^2_{\w} \ell(\w ;\z)}_F  g(\exp_{(\muv,\Lambda)}(\nuv,S)) d\z \\
     & =  \int_{\z} \norm{\nabla^2_{\w} \ell(\w ;\z)}_F  g(\muv,\Lambda) \left(1 - \frac{1}{2}T(\nuv,S)-\tr(\wt{S}) +{Q_g} +o(2) \right) d\z \\
     &  \le \mathbb{E}_{\z \sim \probP(\muv,\Lambda)} \bigg[\norm{\nabla^2_{\w} \ell(\w ;\z)}_F \left(1 + \frac{1}{2} \abs{T(\nuv,S)} + \sqrt{n}\norm{\wt{S}}_F +\abs{Q_g} +o(2) \right) \bigg] \\
     & \hspace{-6.5cm} \underbrace{\le}_{\textbf{H\"older}} \big( \mathbb{E}_{\z \sim \probP(\muv,\Lambda)}\left[\norm{\nabla^2_{\w} \ell(\w ;\z)}_F^{q}\right] \big)^{1/q}  \left(1 + \frac{1}{2} \mathbb{E}_{\z \sim \probP(\muv,\Lambda)} [ \abs{T(\nuv,S)}^p ]^{1/p}  + \sqrt{n}\norm{\wt{S}}_F + \mathbb{E}_{\z \sim \probP(\muv,\Lambda)} [\abs{Q_g}^p]^{1/p} +o(2) \right) .\label{Taylor010z}
\end{align}
Then using the estimates from section \ref{auxillaryestimatessec} along with the bounds $$ \norm{\wt{S}}_F<1 \iff  \norm{L_{\Lambda}(S)}_F<1 \implies \norm{{S}}_F< 2 \norm{\Lambda}_F \norm{\wt{S}}_F < 2 \norm{\Lambda}_F , \norm{\nuv} < 1$$ for $(\muv, \Lambda) = (\muv^*, \Lambda^*) $ in \eqref{Taylor010z} we will get that
\begin{align}
    \norm{\mathbb{E}_{\z \sim \probP(\exp_{(\muv,\Lambda)}(\nuv,S))} [\nabla^2_{\w} {\ell(\w ;\z)} ]}_F \le C'''_{p,\Lambda^*,n} \times \big( \mathbb{E}_{\z \sim \probP(\muv^*,\Lambda^*)}\left[\norm{\nabla^2_{\w} \ell(\w ;\z)}_F^q \right] \big)^{1/q}  \label{Taylor13eo}
\end{align}
for some universal constant $C'''_{p,\Lambda^*,n} \sim_p C''_{p,\Lambda^*,n} $ and $1/p + 1/q = 1$ where $p \ge 1 $. 

By similar calculations, replacing $ \nabla^2_{\w} {\ell(\w ;\z)}$ with $\nabla_{\w} {\ell(\w ;\z)}$ we also have the estimate:
\begin{align}
    \norm{\mathbb{E}_{\z \sim \probP(\exp_{(\muv,\Lambda)}(\nuv,S))} [\nabla_{\w} {\ell(\w ;\z)} ]} \le C'''_{p,\Lambda^*,n} \times \big( \mathbb{E}_{\z \sim \probP(\muv^*,\Lambda^*)}\left[\norm{\nabla_{\w} \ell(\w ;\z)}^q \right] \big)^{1/q}  \label{Taylor13fo}
\end{align}
for the same universal constant $C'''_{p,\Lambda^*,n} $ as in \eqref{Taylor13eo} and $1/p + 1/q = 1$ where $p \ge 1 $. 

\subsection{Uniform estimates of \texorpdfstring{$ \norm{  \nabla_{(\muv, \Lambda)} \mathcal{L}(\muv,\Lambda; \w) }_{(\muv, \Lambda)}$}{}   :}
By MVT on the manifold $\mathbb{R}^n \times BW$ for the function $\nabla_{(\muv, \Lambda)}  \mathcal{L}(\cdot \, , \, \cdot \, ; \w) $ we write : 
\[
\nabla_{(\muv, \Lambda)} \mathcal{L}(\muv,\Lambda; \w)  = {\Psi}_{0 \to 1} (\nabla_{(\muv, \Lambda)}  \mathcal{L}(\muv^*,\Lambda^* ; \w))  + \int_{t=0}^1 {\Psi}_{t \to 1}\big( \text{Hess}_{(\muv,\Lambda)} \big( \mathcal{L}(\muv_t,\Lambda_t \, ; \, \w) \big) [ c'(t)]  \big) dt
\]
where $c(t) = (\muv_t, \Lambda_t)$ is the geodesic with $c(0) = (\muv^*,\Lambda^*)$ , $c(1) = (\muv,\Lambda)$ , ${\Psi}$ is the parallel transport operator. Then from MVT, $ {\nabla_{(\muv, \Lambda)} \mathcal{L}(\muv^*,\Lambda^* ; \w)} ={\nabla_{(\muv, \Lambda)} \mathbb{E}_{\z \sim \probP(\muv,\Lambda)} [\ell(\w ;\z) ] \vert_{(\muv^*,\Lambda^*)} } $ and the fact that parallel transport has unit operator norm we have the following estimate :
\begin{align}
 \norm{  \nabla_{(\muv, \Lambda)} \mathcal{L}(\muv,\Lambda; \w)  }_{(\muv, \Lambda)} & \le   \norm{{\Psi}_{0 \to 1}}_{op}\norm{  {\nabla_{(\muv, \Lambda)} \mathcal{L}(\muv^*,\Lambda^* ; \w)} }_{(\muv^*, \Lambda^*)}  \nonumber \\ & + \int_{t=0}^1 \norm{{\Psi}_{t \to 1}}_{op} \norm{ \text{Hess}_{(\muv,\Lambda)} \big( \mathcal{L}(\muv_t,\Lambda_t \, ; \, \w) \big)  }_{op}\norm{c'(t)}_{c(t)} dt \\
    \norm{  \nabla_{(\muv, \Lambda)} \mathcal{L}(\muv,\Lambda; \w)  }_{(\muv, \Lambda)} & \underbrace{\le}_{\textbf{Jensen's}}  \norm{  \nabla_{\muv} \mathbb{E}_{\z \sim \probP(\muv,\Lambda)} [\ell(\w ;\z) ] \vert_{\muv^*}} + \norm{  \nabla_{ \Lambda} \mathbb{E}_{\z \sim \probP(\muv,\Lambda)} [\ell(\w ;\z) ] \vert_{\Lambda^*}}_{BW}  \nonumber \\ & + \sup_{t \in [0,1]}\norm{ \text{Hess}_{(\muv,\Lambda)} \big( \mathcal{L}(\muv_t,\Lambda_t \, ; \, \w) \big)  }_{op}\text{dist}((\muv, \Lambda),(\muv^*,\Lambda^*)) \\
    & \hspace{-3cm} \underbrace{\le}_{\eqref{Taylor23ay} , \eqref{Taylor23by}, \eqref{Taylor13do} } C_{p,\Lambda^*,n} \times \big( \mathbb{E}_{\z \sim \probP(\muv^*, \Lambda^*) }\left[\abs{\ell(\w;\z)}^q\right] \big)^{1/q}   + \tilde{C}'_{p,\Lambda^*,n, r} \times \big( \mathbb{E}_{\z \sim \probP(\muv^*, \Lambda^*) }\left[\abs{\ell(\w;\z)}^q\right] \big)^{1/q}  \nonumber \\ & + \bigg(  C''_{p,\Lambda^*,n, r} \times \big( \mathbb{E}_{\z \sim \probP(\muv^*,\Lambda^*)}\left[\abs{\ell(\w;\z)}^{pq}\right] \big)^{1/{pq}}  + \frac{2}{\epsilon} \bigg) \times \sqrt{2} r  \label{Taylor13dot}
\end{align}
for $p \ge 1$, $ 1/p +1/q =1$.
\subsection{Uniform estimates of \texorpdfstring{$ \norm{ d\exp_{(\muv, \Lambda),\v} }_{op}$}{}   :}
From the proof of Proposition \ref{expder}, for the map $\alpha:(-\epsilon,\epsilon)\times [0,1] \to \mathbb{R}^n \times BW$ given by $$  \alpha(s,t) = \exp_{c(s)}(t {\Psi}_{0 \to s} (\v)) = (\muv + s  \nuv ,  (I+ s L_{\Lambda}(S))\Lambda(I+ s L_{\Lambda}(S))) \quad , \quad (\nuv , S ) = t \v  $$  the Jacobi field is given by :
   \[
    J_1(t) = \frac{\partial}{\partial s}\Big|_{s=0} \alpha(s,t)  = \big(  \nuv , S   \big) = t\v \, .
    \]
    Similarly, for the map $\beta :(-\epsilon,\epsilon)\times [0,1] \to \mathbb{R}^n \times BW$ given by 
    $$  \beta(s,t) = \exp_{(\muv, \Lambda)}(t(\v+s\u))  = (\muv + t (  \nuv + s \tilde{\nuv}) ,  (I+ t L_{\Lambda}(S + s \wt{S}))\Lambda(I+ t L_{\Lambda}(S + s \wt{S}))) \quad , \v+ s \u:= (\nuv , S ) + s(\wt{\nuv},\wt{S})   $$ with $ \norm{L_{\Lambda}(S)  }_F < 1 $ , the Jacobi field is given by :
    \begin{align*}
        J_2(t) = \frac{\partial}{\partial s}\Big|_{s=0} \beta(s,t) &= \bigg (t \tilde{\nuv} , \frac{\partial}{\partial s}\Big|_{s=0}\big( I + t (S + s \wt{S}) + t^2  L_{\Lambda}(S + s \wt{S})\Lambda  L_{\Lambda}(S + s \wt{S}) \big) \bigg) \\
        &= \bigg (t \tilde{\nuv} , t \wt{S} + t^2 (L_{\Lambda}( \wt{S})\Lambda  L_{\Lambda}(S)  + L_{\Lambda}(S)\Lambda  L_{\Lambda}( \wt{S}))  \bigg)
    \end{align*}
Then $ d\exp_{(\muv, \Lambda),\v}  = (d_{(\muv, \Lambda)}\exp_{(\muv, \Lambda), \v}  ; d_{\v}\exp_{(\muv, \Lambda) , \v}  )$ and
 \[
    d_{(\muv, \Lambda)}\exp_{(\muv, \Lambda), \v} ({\nuv},{S}) = J_1(1) \quad , \quad  d_{\v}\exp_{(\muv, \Lambda) , \v} (\wt{\nuv},\wt{S}) =  J_2(1).
    \]
    Hence the following uniform estimates hold:
    \begin{align}
       \norm{ d_{(\muv, \Lambda)}\exp_{(\muv, \Lambda), \v} }_{op} = \sup_{{(\nuv,S )}} \frac{ \norm{d_{(\muv, \Lambda)}\exp_{(\muv, \Lambda), \v} ({\nuv},{S})}_{(\muv, \Lambda)}}{\norm{(\nuv,S )}_{(\muv, \Lambda)}} =\sup_{{(\nuv,S )}} \frac{\norm{J_1(1)}_{(\muv, \Lambda)}}{\norm{(\nuv,S )}_{(\muv, \Lambda)}} = 1  \label{Taylor13dota}
    \end{align}
        \begin{align}
       \norm{ d_{\v}\exp_{(\muv, \Lambda), \v} }_{op} = \sup_{{(\wt\nuv,\wt S )}} \frac{\norm{J_2(1)}_{(\muv, \Lambda)}}{\norm{(\wt\nuv,\wt S )}_{(\muv, \Lambda)}} &= \sup_{{(\wt\nuv,\wt S )}} \frac{1}{\norm{(\wt\nuv,\wt S )}_{(\muv, \Lambda)}} \norm{\bigg ( \tilde{\nuv} ,  \wt{S} + \underbrace{ (L_{\Lambda}( \wt{S})\Lambda  L_{\Lambda}(S)  + L_{\Lambda}(S)\Lambda  L_{\Lambda}( \wt{S}))}_{= V}  \bigg)}_{(\muv, \Lambda)} \\
       & \le 1 + \sup_{{(\wt\nuv,\wt S )}} \frac{\sqrt{  \text{tr}(L_{\Lambda}(V) \Lambda L_{\Lambda}(V)) }}{\norm{(\wt\nuv,\wt S )}_{(\muv, \Lambda)}} \\
        & = 1 + \sup_{{(\wt\nuv,\wt S )}} \frac{\norm{L_{\Lambda}(V) (\Lambda)^{1/2} }_F }{\sqrt{ \norm{\wt{\nuv}}^2 + \text{tr}(L_{\Lambda}(\wt S) \Lambda L_{\Lambda}(\wt S)) } } \\
        & \le 1 + \sup_{{\wt S }} \frac{\norm{L_{\Lambda}(V) (\Lambda)^{1/2} }_F }{\norm{L_{\Lambda}(\wt S) (\Lambda)^{1/2} }_F } \le 1 + \sup_{{\wt S }} \frac{\norm{L_{\Lambda}(V)  }_F \norm{ (\Lambda)^{1/2} }_F }{\norm{L_{\Lambda}(\wt S)  }_F \norm{(\Lambda)^{-1/2} }^{-1}_F }
    \end{align}
    Then using the following standard estimate from Sylvester equation \cite{horn2012matrix} for any $W$ and $\Lambda \succ \mathbf{0}$
    $$  \norm{L_{\Lambda}(W)  }_F \le  \frac{1}{2}\norm{\Lambda^{-1}}_F \norm{W}_F  $$  in the last step along with $ \norm{L_{\Lambda}(S)  }_F < 1 $ we get:
    \begin{align}
        \norm{ d_{\v}\exp_{(\muv, \Lambda), \v} }_{op} & \le  1 + \sup_{{\wt S }} \frac{2\norm{L_{\Lambda}(S)  }_F \norm{L_{\Lambda}(\wt S)  }_F \norm{ (\Lambda) }_F\norm{ (\Lambda)^{1/2} }_F \norm{ (\Lambda)^{-1} }_F }{2 \norm{L_{\Lambda}(\wt S)  }_F \norm{(\Lambda)^{-1/2} }^{-1}_F } \\
        & \le  1 + {  \norm{ (\Lambda)^{1/2} }^3_F \norm{ (\Lambda)^{-1/2} }^3_F } \\
         & \underbrace{\le}_{\eqref{Taylor13ao}, \eqref{Taylor13bo}} { 1 +  8  \bigg( \frac{\sqrt{n} \norm{(\Lambda^*)^{-1}}_F}{(1-r \norm{(\Lambda^*)^{-1/2}}_F)}\bigg)^3 \norm{ (\Lambda^*)^{1/2} }^3_F  }  \, . \label{Taylor13dotb}
    \end{align}
\subsection{Verification of uniform Lipschitz estimates} \label{uniformlipschitzfinalappendix}
On the {product of closed balls of radius $r = (1- \gamma) \norm{(\Lambda^*)^{-1}}_2^{-1/2}$ with $\gamma \in (0,1)$}, centered at $(\muv^*,\Lambda^*)$, and any given compact set $\mathcal{K} \Subset \mathbb{R}^d$ with $\w \in \mathcal{K}$, the following hold under assumptions \textbf{A1-A2} :
\begin{itemize}
    \item The function $\mathcal{L}(\muv,\Lambda; \w) $ is uniformly Lipschitz continuous with respect to $(\muv,\Lambda) \in \bar{\mathcal{B}}_r(\muv^*) \times \bar{\mathcal{B}}_r(\Lambda^*)$ for any $\w \in \mathcal{K}$ from \eqref{Taylor13dot} and is uniformly Lipschitz continuous in $\w$ for any $(\muv,\Lambda) \in \bar{\mathcal{B}}_r(\muv^*) \times \bar{\mathcal{B}}_r(\Lambda^*)$ from \eqref{Taylor13fo}
    \item The function $  \nabla_{\w} \mathcal{L}(\muv,\Lambda; \w) $ is uniformly Lipschitz continuous with respect to $(\muv,\Lambda) \in \bar{\mathcal{B}}_r(\muv^*) \times \bar{\mathcal{B}}_r(\Lambda^*)$ for any $\w \in \mathcal{K}$ from \eqref{Taylor23ax}, \eqref{Taylor23bz} and is uniformly Lipschitz continuous in $\w \in \mathcal{K}$ for any $(\muv,\Lambda) \in \bar{\mathcal{B}}_r(\muv^*) \times \bar{\mathcal{B}}_r(\Lambda^*)$ from \eqref{Taylor13eo}
    \item The function $  \nabla_{(\muv,\Lambda)} \mathcal{L}(\muv,\Lambda; \w) $ is uniformly Lipschitz continuous in $\w \in \mathcal{K}$ for any $(\muv,\Lambda) \in \bar{\mathcal{B}}_r(\muv^*) \times \bar{\mathcal{B}}_r(\Lambda^*)$ from \eqref{Taylor23ax}, \eqref{Taylor23bz} and is uniformly Lipschitz continuous with respect to $(\muv,\Lambda) \in \bar{\mathcal{B}}_r(\muv^*) \times \bar{\mathcal{B}}_r(\Lambda^*)$ for any $\w \in \mathcal{K}$ from \eqref{Taylor13do}
      \item The functions $ \nabla_{\w} \nabla_{(\muv,\Lambda)} \mathcal{L}(\muv,\Lambda; \w) $ , $  \nabla_{(\muv,\Lambda)} \mathcal{L}(\muv,\Lambda; \w) $ have bounded operator norm uniformly for any $\w \in \mathcal{K}$ and any $(\muv,\Lambda) \in \bar{\mathcal{B}}_r(\muv^*) \times \bar{\mathcal{B}}_r(\Lambda^*)$ from \eqref{Taylor23ax}, \eqref{Taylor23bz} , \eqref{Taylor13dot} 
      \item The function $ \exp_{(\muv, \Lambda)}(\v) $ is uniformly Lipschitz continuous for any $\v \in T_{(\muv, \Lambda)}(\mathbb{R}^n \times BW)$ with $\exp_{(\muv, \Lambda)}(\v) \in \bar{\mathcal{B}}_r(\muv^*) \times \bar{\mathcal{B}}_r(\Lambda^*)$ from \eqref{Taylor13dota}, \eqref{Taylor13dotb}
\end{itemize}
As a consequence we can define a uniform Lipschitz parameter $L := L_{n,d,p,\Lambda^*, r, \epsilon}$ by taking the max over the uniform Lipschitz constants from \eqref{Taylor13dot}, \eqref{Taylor13fo}, \eqref{Taylor23ax}, \eqref{Taylor23bz}, \eqref{Taylor13eo}, \eqref{Taylor13do}, \eqref{Taylor13dota}, \eqref{Taylor13dotb} and the uniform operator norms from \eqref{Taylor23ax}, \eqref{Taylor23bz} , \eqref{Taylor13dot}.

\subsection{Moment computations for the terms \texorpdfstring{$ T(\nuv,S), \, Q_e, \, Q_g, \, P_{\Lambda^{-1}}(\zeta\nuv^T+\nuv\zeta^T), \,  P_{\Lambda^{-1}}(\zeta\zeta^T), $ and $ \Lambda^{-1}S\Lambda^{-1}\zeta\zeta^T + \zeta\zeta^T\Lambda^{-1}S\Lambda^{-1} $}{} }\label{auxillaryestimatessec}
For the random vector $ \z \sim \mathcal{N}(\mathbf{0}_n, I)$ we have that $\E[\norm{\z}^{q}] = 2^{q/2} \frac{\Gamma(\frac{n+q}{2}) }{\Gamma(\frac{n}{2}) } $ for any $q>-n$ (see \cite{vershynin2018high}). Thus\footnote{In the following calculations we suppress the notation $\probP(\muv,\Lambda) $ with $\probP$ for brevity.} using the facts that $ \Lambda^{-1/2}\zeta \sim \mathcal{N}(\mathbf{0}_n, I)$ and $ \norm{\wt{S}}_F<1 \iff  \norm{L_{\Lambda}(S)}_F<1 \implies \norm{{S}}_F< 2 \norm{\Lambda}_F \norm{\wt{S}}_F < 2 \norm{\Lambda}_F , \norm{\nuv} \le r$, we have the following estimates for any $ p \ge 1$ :
\begin{align}
    \mathbb{E}_{\z \sim \probP}\left[ \norm{\Lambda^{-1/2}\zeta}^p \right] & = 2^{p/2} \frac{\Gamma(\frac{n+p}{2}) }{\Gamma(\frac{n}{2}) } .
\end{align}
\begin{align}
    \mathbb{E}_{\z \sim \probP}\left[ \abs{T(\nuv,S)}^p \right] & =    \mathbb{E}_{\z \sim \probP}\left[ \abs{\langle \zeta, 2\Lambda^{-1}\nuv + \Lambda^{-1}S\Lambda^{-1}\zeta\rangle}^p \right] \\
    & \underbrace{\le}_{\textbf{Jensen's}} 2^{p-1}\mathbb{E}_{\z \sim \probP}\left[ \abs{\langle \zeta, 2\Lambda^{-1}\nuv \rangle}^p + \abs{\langle \zeta,\Lambda^{-1}S\Lambda^{-1}\zeta\rangle}^p \right] \\
    & \le 2^{2p-1} \norm{\Lambda^{-1/2}\nuv}^p\mathbb{E}_{\z \sim \probP}\left[ \norm{\Lambda^{-1/2}\zeta}^p \right] + 2^{p-1}  \norm{\Lambda^{-1/2} S \Lambda^{-1/2}}^p\mathbb{E}_{\z \sim \probP} \left[  \norm{\Lambda^{-1/2}\zeta}^{2p} \right] \\
        & = 2^{ \frac{5p}{2}-1} \norm{\Lambda^{-1/2}\nuv}^p \frac{\Gamma(\frac{n+p}{2}) }{\Gamma(\frac{n}{2}) }  + 2^{2p-1}  \norm{\Lambda^{-1/2} S \Lambda^{-1/2}}^p  \frac{\Gamma(\frac{n+2p}{2}) }{\Gamma(\frac{n}{2}) } \\
        & \le 2^{ \frac{5p}{2}-1} \norm{\Lambda^{-1/2}}_F^p \norm{\nuv}^p \frac{\Gamma(\frac{n+p}{2}) }{\Gamma(\frac{n}{2}) }  + 2^{2p-1} \norm{\Lambda^{-1/2}}_F^{2p} \norm{ S }_F^p  \frac{\Gamma(\frac{n+2p}{2}) }{\Gamma(\frac{n}{2}) } \\
        & \lesssim_p \max\{ 1, r^p\} \norm{\Lambda^{-1/2}}_F^{2p} \norm{\Lambda}_F^{p} \frac{\Gamma(\frac{n+2p}{2}) }{\Gamma(\frac{n}{2}) }
\end{align}
for $ \norm{\Lambda}_F \ge 1 \, , \, \norm{\Lambda^{-1/2} }_F \ge 1$.
\begin{align}
    \mathbb{E}_{\z \sim \probP}\left[ \norm{Q_e}_F^p \right] & =   \mathbb{E}_{\z \sim \probP}\left[ \norm{(\wt{S}^2\Lambda^{-1} + \Lambda^{-1}\wt{S}^2 + \wt{S}\Lambda^{-1}\wt{S})}_F^p \right]  \\
    & = \norm{(\wt{S}^2\Lambda^{-1} + \Lambda^{-1}\wt{S}^2 + \wt{S}\Lambda^{-1}\wt{S})}_F^p =  \norm{Q_e}_F^p   \quad, \quad \wt{S} = L_{\Lambda}(S) \\
    & \hspace{-1cm} \underbrace{\le}_{\textbf{Jensen's}}  3^{p-1}\norm{\wt{S}^2}_F^p \norm{\Lambda^{-1} }_F^p\le 3^{p-1} \norm{\Lambda^{-1/2} }_F^{2p} .
\end{align}
\begin{align}
     \mathbb{E}_{\z \sim \probP}\left[ \abs{ \cq(\nuv,S)}^p \right]    & =  \mathbb{E}_{\z \sim \probP}\left[ \abs{2\langle \zeta, \Lambda^{-1}S\Lambda^{-1}\nuv\rangle + \langle\nuv, \Lambda^{-1}\nuv\rangle + \langle \zeta, Q_e\zeta\rangle}^p \right] \\
      & \hspace{-1cm} \underbrace{\le}_{\textbf{Jensen's}} 3^{p-1}\mathbb{E}_{\z \sim \probP}\left[ 2^p\abs{\langle \zeta, \Lambda^{-1}S\Lambda^{-1}\nuv\rangle}^p + \abs{\langle\nuv, \Lambda^{-1}\nuv\rangle}^p + \abs{\langle \Lambda^{-1/2} \zeta, \Lambda^{1/2} Q_e \Lambda^{1/2}\Lambda^{-1/2}\zeta\rangle}^p \right] \\
      & \le 2 \cdot 6^{p-1} \norm{\Lambda^{-1/2}S\Lambda^{-1}\nuv}^p\mathbb{E}_{\z \sim \probP}\left[ \norm{\Lambda^{-1/2}\zeta}^p \right]  + 3^{p-1}\abs{\langle\nuv, \Lambda^{-1}\nuv\rangle}^p \nonumber \\ & + 3^{p-1} \norm{\Lambda^{1/2} Q_e \Lambda^{1/2}}^p\mathbb{E}_{\z \sim \probP}\left[\norm{ \Lambda^{-1/2}\zeta}^{2p} \right] \\
       & \le 2^{p/2 + 1} \cdot 6^{p-1} \norm{\Lambda^{-1/2}}_F^p \norm{\Lambda^{-1}}_F^p\norm{S}_F^p \norm{\nuv}^p \frac{\Gamma(\frac{n+p}{2}) }{\Gamma(\frac{n}{2}) }   + 3^{p-1} \norm{\Lambda^{-1}}_F^p \norm{\nuv}^{2p}   \nonumber \\ & + 2 \cdot 6^{p-1} \norm{\Lambda^{1/2}}_F^{2p}  \norm{Q_e}_F^p  \frac{\Gamma(\frac{n+2p}{2}) }{\Gamma(\frac{n}{2}) } \\
       & \lesssim_{p} \max\{1,r^{2p}\} \frac{\Gamma(\frac{n+2p}{2}) }{\Gamma(\frac{n}{2}) }  \norm{\Lambda^{-1/2}}_F^{3p} \norm{\Lambda^{1/2}}_F^{2p}
\end{align}
for $ \norm{\Lambda}_F \ge 1 \, , \, \norm{\Lambda^{-1/2} }_F \ge 1$.
\begin{align}
    \mathbb{E}_{\z \sim \probP}\left[ \abs{Q_g}^p \right] & =       \mathbb{E}_{\z \sim \probP}\left[ \abs{\frac{1}{2}(\tr(\wt{S})^2+\tr(\wt{S}^2))+ \frac{1}{8}T^2(\nuv,S)-\frac{1}{2}\cq(\nuv,S) +\frac{1}{2}T(\nuv,S)\tr(\wt{S})}^p \right] \\
     & \hspace{-1cm} \underbrace{\le}_{\textbf{Jensen's}} 5^{p-1}  \mathbb{E}_{\z \sim \probP}\left[ \frac{1}{2^p}\abs{tr(\wt{S})^2}^p+ \frac{1}{2^p}\abs{\tr(\wt{S}^2)}^p + \frac{1}{8^p}\abs{T^2(\nuv,S)}^p + \frac{1}{2^p}\abs{\cq(\nuv,S)}^p +\frac{1}{2^p}\abs{T(\nuv,S)\tr(\wt{S})}^p \right] 
     \end{align}
     \begin{align}
     & =  5^{p-1} \bigg(\frac{1}{2^p}\abs{tr(\wt{S})^2}^p+ \frac{1}{2^p}\abs{\tr(\wt{S}^2)}^p \bigg) +   \frac{5^{p-1}}{8^p} \mathbb{E}_{\z \sim \probP}\left[  \abs{T(\nuv,S)}^{2p} \right] +  \frac{5^{p-1}}{2^p} \mathbb{E}_{\z \sim \probP}\left[ \abs{\cq(\nuv,S)}^p  \right] \nonumber \\
     & +  \frac{5^{p-1}}{2^p} \abs{\tr(\wt{S})}^p \mathbb{E}_{\z \sim \probP}\left[   \abs{T(\nuv,S)}^p \right] \\
          & \le  5^{p-1} \norm{\wt{S}}_F^{2p} \frac{n^p + 1}{2^p}  +   \frac{5^{p-1}}{8^p} \mathbb{E}_{\z \sim \probP}\left[  \abs{T(\nuv,S)}^{2p} \right] +  \frac{5^{p-1}}{2^p} \mathbb{E}_{\z \sim \probP}\left[ \abs{\cq(\nuv,S)}^p  \right] \nonumber \\
     & +  \frac{5^{p-1} n^{p/2}}{2^p} \norm{\wt{S}}_F^{p}  \mathbb{E}_{\z \sim \probP}\left[   \abs{T(\nuv,S)}^p \right] \\
     & \lesssim_p n^p \max\{ 1, r^{2p}\} \norm{\Lambda^{-1/2}}_F^{4p} \norm{\Lambda^{1/2}}_F^{4p} \frac{\Gamma(\frac{n+4p}{2}) }{\Gamma(\frac{n}{2}) }
\end{align}
for $ \norm{\Lambda}_F \ge 1 \, , \, \norm{\Lambda^{-1/2} }_F \ge 1$.
\begin{align}
    \mathbb{E}_{\z \sim \probP}\left[ \norm{P_{\Lambda^{-1}}(\zeta\nuv^T+\nuv\zeta^T)}_F^p \right] & \underbrace{\le}_{\textbf{Jensen's}}      2^{p} \norm{\Lambda^{-1/2}}^p_F\mathbb{E}_{\z \sim \probP}\left[ \norm{(\Lambda^{-1/2}\zeta\nuv^T+\Lambda^{-1/2}\nuv\zeta^T)}^p_F \right] \\
    & \underbrace{\le}_{\textbf{Jensen's}}    \frac{1}{2} \cdot  4^{p} \norm{\Lambda^{-1/2}}^p_F\mathbb{E}_{\z \sim \probP}\left[ \norm{\Lambda^{-1/2}\zeta\nuv^T}^p_F + \norm{\nuv ( \Lambda^{-1/2}\zeta)^T}^p_F\right] \\
    & \le 4^{p} \norm{\Lambda^{-1/2}}^p_F \norm{\nuv}\mathbb{E}_{\z \sim \probP}\left[ \norm{\Lambda^{-1/2}\zeta}^p_F \right] \\
    & = 4^{p} \norm{\Lambda^{-1/2}}^p_F \norm{\nuv} 2^{p/2} \frac{\Gamma(\frac{n+p}{2}) }{\Gamma(\frac{n}{2}) } .
\end{align}
\begin{align}
    \mathbb{E}_{\z \sim \probP}\left[ \norm{P_{\Lambda^{-1}}(\zeta\zeta^T)}_F^p \right] & \underbrace{\le}_{\textbf{Jensen's}}      2^{p-1} \mathbb{E}_{\z \sim \probP}\left[ \norm{\Lambda^{-1/2} \Lambda^{-1/2}\zeta\zeta^T \Lambda^{-1/2} \Lambda^{1/2}}^p_F + \norm{\Lambda^{1/2} \Lambda^{-1/2}\zeta\zeta^T\Lambda^{-1/2} \Lambda^{-1/2}}^p_F\right]\\
    & \le 2^{p} \norm{\Lambda^{-1/2}}^p_F \norm{\Lambda^{1/2}}^p_F \mathbb{E}_{\z \sim \probP}\left[ \norm{\Lambda^{-1/2}\zeta}^{2p}_F \right] \\
    & = 4^{p} \norm{\Lambda^{-1/2}}^p_F \norm{\Lambda^{1/2}}^p_F \frac{\Gamma(\frac{n+2p}{2}) }{\Gamma(\frac{n}{2}) } .
\end{align}
\begin{align}
   \mathbb{E}_{\z \sim \probP}\left[ \norm{\Lambda^{-1}S\Lambda^{-1}\zeta\zeta^T + \zeta\zeta^T\Lambda^{-1}S\Lambda^{-1}}_F^p \right] & \underbrace{\le}_{\textbf{Jensen's}} 2^{p-1}      \mathbb{E}_{\z \sim \probP}\left[ \norm{\Lambda^{-1}S\Lambda^{-1}\zeta\zeta^T}_F^p + \norm{\zeta\zeta^T\Lambda^{-1}S\Lambda^{-1}}_F^p \right]  \\
   & =  2^{p}   \,  \mathbb{E}_{\z \sim \probP}\left[ \norm{\Lambda^{-1}S\Lambda^{-1}\zeta\zeta^T}_F^p \right]  \\
   & = 2^{p}   \,  \mathbb{E}_{\z \sim \probP}\left[ \norm{\Lambda^{-1}S\Lambda^{-1/2} \Lambda^{-1/2}\zeta \Lambda^{1/2} \Lambda^{-1/2}\zeta^T}_F^p \right] \\
   & \le 2^{p} \norm{\Lambda^{-1}S\Lambda^{-1/2}}_F^p  \,  \mathbb{E}_{\z \sim \probP}\left[ \norm{ \Lambda^{-1/2}\zeta}_F^{2p} \norm{\Lambda^{1/2}}_F^p \right] \\
   & \le 2^{p} \norm{\Lambda^{-1/2}}_F^{3p}  \norm{\Lambda^{1/2}}_F^p \norm{S}^p_F   \,  \mathbb{E}_{\z \sim \probP}\left[ \norm{ \Lambda^{-1/2}\zeta}_F^{2p} \right] \\
   & \le 8^{p} \norm{\Lambda^{-1/2}}_F^{3p}  \norm{\Lambda^{1/2}}_F^{3p}   \, \frac{\Gamma(\frac{n+2p}{2}) }{\Gamma(\frac{n}{2}) } \, .
\end{align}

\section{Examples validating assumptions}

\subsection{On geodesic nonconvexity of the statistical risk with respect to probability measure}\label{appendixcountereg1}

Let $z \sim  \mathcal{N}(m,1)$ where $\mathcal{N}(m,1)$ is Gaussian distribution on $\mathbb{R}$ with mean $m$ and unit variance. Consider the loss function $f(w,z, y) := (wz^2-y)^2$ with quartic dependence on $z$ where $w$ is the parameter and $y$ is the label.
The population risk for $f$ is
\[
  \mathcal{R}(m,w) \;=\; \mathbb E \big[(wz^2-y)^2\big]
  \;=\; w^2\,\mathbb E[z^4] \;-\; 2wy\,\mathbb E[z^2] \;+\; y^2 .
\]
Using the Gaussian moments (variance $1$)
\[
  \mathbb E[z^2] = m^2+1, \qquad \mathbb E[z^4] = m^4+6 m^2+3,
\]
we obtain the closed form
\begin{equation}
  {\,\mathcal{R}(m,w) \;=\; w^2\big(m^4+6m^2+3\big) \;-\; 2wy\big(m^2+1\big) \;+\; y^2\,}
  \label{eq:risk}
\end{equation}

Collecting \eqref{eq:risk} by powers of $m$,
\[
 \mathcal{R}(m,w) \;=\; w^2m^4 \;+\; \big(6w^2-2wy\big)m^2 \;+\; \big(3w^2-2wy+y^2\big),
\]
is a quartic $am^4+bm^2+c$ with $a=w^2\ge 0$ and $b = 2w(3w-y)$. Its second
derivative in $m$ is
\[
  \mathcal{R}''(m) \;=\; 12w^2m^2 \;+\; 2b, \qquad
  \mathcal{R}''(0) \;=\; 4w(3w-y).
\]
For $w>0$, the map $m\mapsto \mathcal{R}(m,w)$ is convex if $y\le 3w$, and is a
strict double well \textup{(}nonconvex, with a spurious critical point at
$m=0$\textup{)} if $y>3w$. In the nonconvex regime the two global
minimizers are located at
\[
  m^* \;=\; \pm\sqrt{\frac{y-3w}{w}}.
\]
Further adding any sufficiently small quadratic growth ($\epsilon \ll 1$) to the loss $\mathcal{R}$ as $ \mathcal{R}(m,w) + \frac{\epsilon}{2} |m - m^*|$ will still preserve nonconvexity of $\mathcal{R}$ in $m$ by continuity.
\subsection{A solvable toy model: critical manifold and its Lipschitz stability}\label{appendixcountereg2}

Let $z$ be a random variable with mean $\mathbb{E}[z] = m$ and
variance $\operatorname{Var}(z) = s^2$, and let $(x_1,x_2)\in\mathbb{R}^2$ be
non-random. Define
\[
f(x_1,x_2;m,s^2) \;:=\; (x_1-m)^2 + s^2 .
\]

Then
\begin{align*}
f(x_1,x_2;m,s^2)
&= x_1^2 + m^2 - 2x_1m + s^2 \\
&= \mathbb{E}[x_1^2] + \big(\mathbb{E}[z]\big)^2 - 2\,\mathbb{E}[x_1 z] + s^2 \\
&= \mathbb{E}[x_1^2] + \Big(\mathbb{E}[z^2] - \operatorname{Var}(z)\Big)
   - 2\,\mathbb{E}[x_1 z] + s^2 \\
&= \mathbb{E}\big[(x_1-z)^2\big] \;-\; \operatorname{Var}(z) \;+\; s^2 \\
&= \mathbb{E}\big[(x_1-z)^2\big],
\end{align*}
so
\[
f(x_1,x_2;m,s^2) \;=\; \mathbb{E}\big[(x_1-z)^2\big] \, .
\]

The critical set of $f$ is
\[
\mathcal{S}^*(m) \;=\; \big\{(x_1,x_2)\in\mathbb{R}^2 : x_1 = m\big\}
\;=\; m\times\mathbb{R},
\]
a full affine line in $\mathbb{R}^2$, not an isolated point.

For any
$m_0,m_1\in\mathbb{R}$,
\[
\text{dist}_H\big(\mathcal{S}^*(m_0),\,\mathcal{S}^*(m_1)\big) \;=\; |m_0-m_1|,
\]
where $\text{dist}_H(\cdot, \cdot)$ denotes Hausdorff distance in $\mathbb{R}^2$.

Let $f_* := \inf_{(x_1,x_2)} f(x_1,x_2;m,s^2) = s^2$ which is attained exactly on $\mathcal{S}^*(m)=m\times\mathbb{R}$, so
\[
f(x_1,x_2;m,s^2) - f_* = (x_1-m)^2+s^2 - s^2 = (x_1-m)^2.
\]
Further,
\[
\|\nabla f(x_1,x_2;m,s^2)\|^2
= \big(2(x_1-m)\big)^2 + 0^2
= 4(x_1-m)^2.
\]
Hence we have:
\[
\|\nabla f(x_1,x_2;m,s^2)\|^2 \;=\; 4(x_1-m)^2 \;=\; 4\big(f(x_1,x_2;m,s^2)-f_*\big),
\]
which implies the Polyak--{\L}ojasiewicz inequality ($\beta = 2$) :
\[
\|\nabla f(x_1,x_2;m,s^2)\|^2 \;\ge\; 2\mu_{PL}\big(f(x_1,x_2;m,s^2)-f_*\big)
\]
So $f$ satisfies the PL inequality globally on $\mathbb{R}^2$, with sharp constant $\mu_{PL}=2$.

\end{document}